\documentclass[11pt]{article}
\usepackage[margin=1in]{geometry}
\usepackage[export]{adjustbox}
\usepackage{amsmath}
\usepackage{amsthm}
\usepackage{wrapfig}

\usepackage{titletoc}
\usepackage{thm-restate}
\usepackage[font=small]{caption}%make figure captions a bit smaller
\usepackage[hidelinks]{hyperref}
\usepackage[capitalise]{cleveref}
\usepackage{amsfonts}
\usepackage{amssymb}
\usepackage[textsize=tiny]{todonotes}
\usepackage[inline]{enumitem}
\usepackage{mathtools}
\usepackage{cases}
\usepackage[only,llbracket,rrbracket]{stmaryrd}
\newcommand{\smallblacksquare}{{\vcenter{\hbox{\rule{.6ex}{.6ex}}}}}
\usepackage{array}
\usepackage{multirow}
\usepackage{multicol}
\newtheorem{theorem}{Theorem}[section]
\newtheorem{corollary}[theorem]{Corollary}
\newtheorem{lemma}[theorem]{Lemma}
\newtheorem{claim}[theorem]{Claim}
\newtheorem{proposition}[theorem]{Proposition}
\newtheorem{conjecture}[theorem]{Conjecture}
\newtheorem{property}[theorem]{Property}
\newtheorem{fact}[theorem]{Fact}
\newtheorem{observation}[theorem]{Observation}
\theoremstyle{definition}
\newtheorem{definition}[theorem]{Definition}
\newtheorem{example}[theorem]{Example}
\DeclareMathOperator{\fw}{fw}
\usepackage{enumitem}
\setlist[itemize]{topsep=4pt,itemsep=3pt,parsep=0pt} 
\setlist[enumerate]{topsep=4pt,itemsep=3pt,parsep=0pt} 

\Crefname{figure}{Figure}{Figures}

\renewcommand{\cal}{\mathcal}

\newcommand{\eps}{\varepsilon}

\renewcommand{\int}{\mathrm{int}}
\newcommand{\const}{\mathrm{const}}

\newcommand{\rk}{\mathrm{rk}}

\newcommand{\wh}[1]{#1^+}

\newcommand{\NN}[0]{\mathrm{\mathbb{N}}}
\newcommand{\N}[0]{\mathrm{\mathbb{N}}}

\newcommand{\otp}{\mathrm{otp}}

\newcommand{\atp}{\mathrm{atp}}
\newcommand{\ex}{\mathrm{ex}}
\renewcommand{\subset}{\subseteq}

\newcommand{\col}{\textnormal{col}}

\newcommand{\dist}{\textnormal{dist}}

\renewcommand{\phi}{\varphi}

\newcommand{\Start}{\mathrm{start}}
\newcommand{\End}{\mathrm{end}}

\newcommand{\LL}{\mathcal{L}}
\newcommand{\lc}{\mathrm{lc}}

\newcommand{\flip}{\mathrm{flip}}
\newcommand{\prep}{\mathrm{prep}}

\newcommand{\DD}{\mathcal{D}}
\newcommand{\BB}{\mathcal{B}}
\newcommand{\TT}{\mathcal{T}}
\newcommand{\II}{\mathcal{I}}

\newcommand{\Bs}{\BB^\star}
\newcommand{\Bc}{\BB^\bullet}
\newcommand{\Bo}{\BB^{\scriptscriptstyle \blacktriangleleft}}
\newcommand{\Bg}{\BB^{\smallblacksquare}}
\newcommand{\twins}{\mathrm{twins}}
\newcommand{\counttwins}{\mathrm{\#twins}}

\newcommand{\dropbounds}[1]{\llbracket#1\rrbracket}

\newcommand{\effective}[1]{\smallskip\noindent#1}

\newcommand{\set}[1]{\{#1\}}
\newcommand{\setof}[2]{\{#1 : #2\}}

\newcommand{\gc}[1]{\mathcal{#1}}

\newcommand{\CC}{\mathcal{C}}
\newcommand{\EE}{\mathcal{E}}

\newcommand{\PP}{\mathcal{P}}

\renewcommand{\le}{\leqslant}
\renewcommand{\leq}{\le}
\renewcommand{\ge}{\geqslant}
\renewcommand{\geq}{\ge}
\newcommand{\KK}{\mathcal{K}}

\newenvironment{claimproof}[1][\proofname]{%
  \begin{proof}[#1]%
}{%
  \end{proof}%
}

\newcommand{\from}{\colon}

\begin{document}

\newcommand{\funding}{
N.M. was supported by the German Research Foundation (DFG) with grant agreement No. 444419611.
S.T. was supported by the project BOBR that is funded from the European Research Council (ERC) under the European Union’s Horizon 2020 research and innovation programme with grant agreement No. 948057.
The BOBR project also supported a stay of N.M. at the University of Warsaw, where part of this work was done.
}

\title{Flip-Breakability: A Combinatorial Dichotomy \\ for Monadically Dependent Graph Classes\thanks{\funding}}
\date{}
\author{
  Jan Dreier \\
  \small{TU Wien} \\
  \small{\texttt{dreier@ac.tuwien.ac.at}}
  \and
  Nikolas M\"ahlmann \\
  \small{University of Bremen} \\
  \small{\texttt{maehlmann@uni-bremen.de}}
  \and
  Szymon Toru\'nczyk \\
  \small{University of Warsaw} \\
  \small{\texttt{szymtor@mimuw.edu.pl}}
   %Anonymous Authors
}
\maketitle

\titlecontents{section}
[5pt]                                               % left margin
{}%
{\contentsmargin{0pt}                               % numbered entry format
    \thecontentslabel\enspace%
    \large}
{\contentsmargin{0pt}\large}                        % unnumbered entry format
{\titlerule*[.5pc]{.}\contentspage}                 % filler-page format (e.g dots)
[]                                                  % below code (e.g vertical space)

%let page numbers only start at introduction
\pagenumbering{gobble}
\begin{abstract}
A conjecture in algorithmic model theory predicts 
that the model-checking problem for first-order logic is fixed-parameter tractable on a hereditary graph class if and only if the class is \emph{monadically dependent}.
Originating in model theory, this notion is defined in terms of logic, and encompasses 
nowhere dense classes, monadically stable classes, and classes of  bounded twin-width.
Working towards this conjecture, we provide the first two combinatorial characterizations
of monadically dependent graph classes. This yields the following dichotomy.

On the structure side, we characterize 
monadic dependence by a Ramsey-theoretic property called \emph{flip-breakability}.
This notion generalizes the notions of uniform quasi-wideness, flip-flatness, and bounded grid rank, which characterize nowhere denseness,  monadic stability, and bounded twin-width, 
respectively, and played a key role in their respective model checking algorithms.
Natural restrictions of flip-breakability additionally characterize bounded treewidth and cliquewidth and bounded treedepth and shrubdepth.

On the non-structure side, we characterize 
monadic dependence by explicitly listing few families of \emph{forbidden induced subgraphs}.
This result is analogous to the characterization of nowhere denseness via forbidden subdivided cliques, 
and allows us to resolve one half of the motivating conjecture: First-order model checking is AW[$*$]-hard on every hereditary graph class that is monadically independent.
The result moreover implies that hereditary graph classes which are small, have almost bounded twin-width, or have almost bounded flip-width, are monadically dependent.

Lastly, we lift our result to also obtain a combinatorial dichotomy in the more general setting of monadically dependent classes of binary structures.
\end{abstract}

\paragraph*{Acknowledgements.}
N.M.\ thanks Micha{\l} Pilipczuk for hosting a stay at the University of Warsaw, where part of this work was done.

\clearpage

% ignore subsections in the table of contents
\setcounter{tocdepth}{1} 
\tableofcontents

\clearpage
\setcounter{page}{1}
\pagenumbering{arabic}

\newpage
\section{Introduction}

Algorithmic model theory 
studies the interplay between 
the computational complexity 
of computational problems defined 
using logic,  and the
structural properties of the considered 
instances.
 In this context, \emph{algorithmic meta-theorems}~\cite{Kreutzer_2011,grohe2008logic} are results that establish the tractability of {entire families of  computational problems}, which are defined in terms of logic, while imposing structural restrictions on the input instances.
The archetypical example is Courcelle's theorem, which states that every problem that can be expressed in monadic second order logic (MSO), can be solved in linear time, whenever the considered input graphs have bounded treewidth~\cite{courcelle1990monadic,ArnborgLS91}.
This implies that the model checking problem for MSO is \emph{fixed-parameter tractable} (fpt) on every class $\CC$ of bounded treewidth. That is, 
there is an algorithm 
which determines whether a given input  graph $G\in \CC$ satisfies a given  MSO formula $\phi$ in time
$f(|\phi|)\cdot |V(G)|^c$ for some 
function $f\from \N\to \N$ and some constant $c$ (in this case $c=1$).
More generally, the model checking problem for MSO is fpt on all classes of bounded cliquewidth~\cite{courcelle2000linear}.

In this paper, we focus on 
the model-checking problem for 
\emph{first-order logic} (FO), which allows to relax the structure of the input graphs greatly, at the cost of restricting the logic.
Graph classes for which this problem is fpt 
include classes of bounded degree \cite{seese1996linear}, the class of planar graphs \cite{frick2001deciding}, classes which exclude a minor \cite{flum2001fixed}, classes of bounded expansion \cite{DvorakKT13-journal}, and more generally, nowhere dense classes \cite{grohe2017deciding}, all of which are \emph{sparse} graph classes 
(specifically, every $n$-vertex graph in such a class has $O(n^{1+\varepsilon})$ edges, for every fixed $\varepsilon>0$).
The problem is moreover fpt on
proper hereditary classes of permutation graphs, some classes of bounded twin-width \cite{twwI}, structurally nowhere dense classes \cite{ssmc}, monadically stable classes \cite{wip}, and others~\cite{bonnet2022model}.
The central question in the area, first phrased in~\cite[Sec.~8.2]{grohe2008logic}, is the following.
\begin{quote}
What are the structural properties that exactly characterize the hereditary\footnote{A graph class is \emph{hereditary} if it is closed under vertex removal.} graph classes with fpt first-order model checking?
\end{quote}

\emph{Sparsity theory}~\cite{nevsetvril2012sparsity}, initiated by Ne\v{s}et\v{r}il and Ossona de Mendez,
has provided solid structural foundations, and a very general notion of structural  tameness for sparse graph classes.
More recently, \emph{twin-width theory}~\cite{twwI} provides
another notion of  graph classes 
which are structurally tame, and not necessarily sparse.
While \emph{twin-width theory} and \emph{sparsity theory} bear striking similarities, they are fundamentally incomparable in scope.
The community eagerly anticipates a theory that unifies both frameworks, and answers the central question.
 \emph{Stability theory} -- 
an area in model theory developed initially by Shelah --
provides very general notions of \emph{logical} tameness of classes of graphs (or other structures), called (monadic) \emph{stability} and \emph{dependence}\footnote{In the literature, \emph{dependence} is also referred to as \emph{NIP}, which stands for \emph{negation of the independence property}.}, which subsume the notions studied in sparsity theory and in twin-width theory,
but are not easily amenable to combinatorial or algorithmic treatment.
Let us briefly review \emph{sparsity theory}, \emph{twin-width theory}, 
and \emph{stability theory}, which we build upon. 
\Cref{fig:classes} shows the relationships between the properties of graph classes discussed in this paper.

\begin{figure}[h]
    \vspace{-0.3cm}
      \begin{center}
       \includegraphics[width=0.9\textwidth]{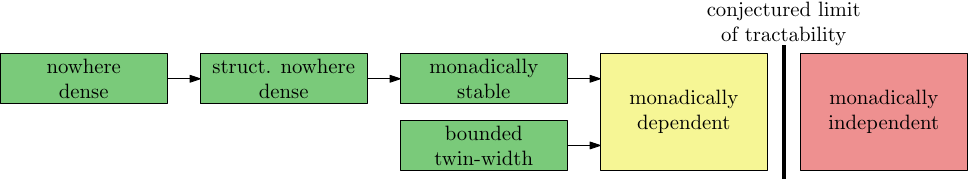}
      \end{center}
      \vspace{-0.6cm}
      \caption{Hierarchy of selected properties of graph classes.
          Classes in green boxes are known do admit fpt model checking algorithms
          (with additional assumptions required in the case of bounded twin-width).
          \emph{Monadically dependent} classes are known to generalize all these notions. 
          It is conjectured that a hereditary class is monadically dependent if and only if its model checking problem is tractable.
      }\label{fig:classes}
      \vspace{-0.3cm}
\end{figure}

\paragraph{Sparsity.}
The central notion of sparsity theory~\cite{nevsetvril2012sparsity} is that of a \emph{nowhere dense} graph class.
This is a very general notion of structural tameness of sparse graph classes, and encompasses all classes with bounded degree, bounded treewidth, the class of planar graphs, and classes that exclude some graph as a minor.
After a long sequence of prior work~\cite{seese1996linear,flum2001fixed,frick2001deciding,dawar2007locally,DvorakKT13-journal},
the celebrated result of Grohe, Kreutzer and Siebertz~\cite{grohe2017deciding}
established that the model checking problem is fpt on every nowhere dense graph class.
For classes that are \emph{monotone} (closed under removal of vertices and edges) this is optimal~\cite{DvorakKT13-journal, Kreutzer_2011},
yielding the following milestone result.
\begin{quote}
A monotone graph class admits fpt model checking if and only if it is \emph{nowhere dense}
(assuming \(\textnormal{FPT} \neq \textnormal{AW}[*]\))~\cite{grohe2017deciding}.
\end{quote}

Here, \(\textnormal{FPT} \neq \textnormal{AW}[*]\) is a standard complexity assumption in parameterized complexity, equivalent to the statement that 
model checking is not fpt on the class of all graphs.
The major shortcoming of this result 
is that it only captures monotone classes, 
and thus \emph{sparse} graph classes, while there are dense graph classes which are not monotone, and have fpt model checking.
A trivial example is the class of cliques.
More generally, all classes of bounded cliquewidth have fpt model checking.
For many years, not many examples of 
 graph classes for which the model checking problem is fpt -- beyond nowhere dense classes and classes of bounded cliquewidth -- were known.

A recent line of research
extends sparsity theory beyond the sparse setting using \emph{transductions}.
We say a graph class \(\mathcal{C}\) \emph{transduces} a graph class \(\mathcal{D}\) if
there exists a first-order formula \(\phi(x,y)\) such that
every graph in \(\mathcal{D}\) can be obtained by the following four steps:
(1) taking a graph \(G \in \mathcal{C}\), (2) coloring the vertices of \(G\),
(3) replacing the edge set of \(G\) with \(\{ uv \mid u, v \in V(G), u \neq v, G \models \phi(u,v) \lor \phi(v,u)\}\), 
and (4) taking an induced subgraph.
Step 3 can construct, for example, the \emph{edge-complement} of \(G\) with \(\phi(x,y)=\neg E(x,y)\)
or the \emph{square} of \(G\) with \(\phi(x,y)= \exists z~E(x,z) \land E(z,y)\).
Note that the formula \(\phi\) has access to the colors of \(G\) from step~2.
Therefore, for example, the class of {all graphs} is transduced by the class of \(1\)-subdivided cliques 
(by coloring the subdivision vertices that should be turned into edges),
but is \emph{not} transduced by the class of {cliques}.
It was recently shown that model checking is fpt on \emph{transductions of nowhere dense classes}~\cite{ssmc},
extending the work of Grohe, Kreutzer, and Siebertz~\cite{grohe2017deciding} beyond the sparse setting. This result has been  generalized very recently in \cite{wip}, as described below.

\begin{wrapfigure}{r}{0.17\textwidth}
    \vspace{-0.5cm}
   \begin{center}
       \includegraphics{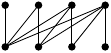}
   \end{center}
\vspace{-0.5cm}
\caption{A half-graph of order 4.}\label{fig:ladder}
\vspace{-0.1cm}
\end{wrapfigure}

\paragraph{Monadic Stability.}
Stability, and its variants, are logical tameness conditions arising in  Shelah's classification program~\cite{shelah1986monadic}.
A graph class is \emph{monadically}\footnote{\emph{Monadically} refers to the coloring step of a transduction, which is an expansion of the graph with unary/monadic predicates. In classical model theory monadically stable (monadically dependent) classes are equivalently defined as those classes that remain stable (dependent) under unary expansions.}
\emph{stable}
if it does not transduce the class of all \emph{half-graphs} (see \Cref{fig:ladder}).
Monadically stable classes include all structurally nowhere dense classes~\cite{adler2014interpreting,stable_graphs}.
In~\cite{wip} it is shown that FO model checking is fpt on all monadically stable classes,
and moreover obtains matching hardness bounds via a new characterization of such classes in terms of forbidden induced subgraphs.
A class is \emph{orderless}
if it avoids some half-graph as a semi-induced%
\footnote{Semi-induced half-graphs look similar to \Cref{fig:ladder}, 
but the connections within the top row and within the bottom row may be arbitrary.}
subgraph.
This establishes monadically stable classes as the \emph{limit of tractability} among orderless classes.
\begin{quote}
An orderless, hereditary graph class admits fpt model checking if and only if it is \emph{monadically stable}
(assuming \(\textnormal{FPT} \neq \textnormal{AW}[*]\))~\cite{wip}.
\end{quote}
\paragraph{Twin-Width.}
Twin-width~\cite{twwI} is a recently introduced notion which has its roots in enumerative combinatorics and the Stanley-Wilf conjecture/Marcus-Tardos theorem.
Graph classes with bounded twin-width include the class of 
planar graphs, all classes of bounded clique\-width, the class of unit interval graphs, and every proper hereditary class of permutation graphs.
It was shown that model checking is fpt on graph classes of bounded twin-width,
assuming an appropriate decomposition of the graph (in form of a so-called \emph{contraction sequence}) 
is given as part of the input~\cite{twwI}.
On \emph{ordered graphs} (that is, graphs equipped with a total order on the vertex set, which can be accessed by the formulas)
a recent breakthrough result has lifted this additional proviso, and also characterized classes of bounded twin-width as the \emph{limit of tractability}~\cite{twwIV}.
\begin{quote}
A hereditary class of ordered graphs admits fpt model checking if and only if it has bounded \emph{twin-width}
(assuming \(\textnormal{FPT} \neq \textnormal{AW}[*]\))~\cite{twwIV}.
\end{quote}

\paragraph{Monadic Dependence.}
As discussed above, an exact characterization 
of classes with fpt model-checking 
has been established in three settings:
for monotone graph classes, for the more general hereditary orderless graph classes, and for hereditary classes of ordered graphs, in terms of the notions of nowhere denseness, monadic stability, and bounded twin-width, respectively.
Those notions turn out to be three facets of a single notion, again originating in stability theory.
A graph class $\CC$ is \emph{monadically independent} 
if it  transduces the class of all graphs~\cite{baldwin1985second}. 
Otherwise $\CC$ is \emph{monadically dependent}.
In all three settings where we have a complete characterization of fpt model checking,
monadic dependence precisely captures the limit of tractability.
\begin{itemize}
    \item On monotone classes, nowhere denseness is equivalent to monadic dependence~\cite{stable_graphs, adler2014interpreting}.
    \item On orderless classes, monadic stability is equivalent to monadic dependence~\cite{nevsetvril2021rankwidth}.
    \item On classes of ordered graphs, bounded twin-width is equivalent to monadic dependence~\cite{twwIV}.
\end{itemize}
This suggests that the known tractability limits~\cite{grohe2017deciding,twwIV,ssmc,wip} are fragments of a larger picture, 
where monadically dependent classes unify sparsity theory, twin-width theory, and stability theory into a single theory of tractability.

\begin{conjecture}[e.g., \cite{warwick-problems,twwIV,ssmc}]\label{conj:warwick}%
 Let \(\mathcal{C}\) be a hereditary class of graphs. Then the model checking
problem for first-order logic is fpt on \(\mathcal{C}\) if and only if \(\mathcal{C}\) is monadically dependent.\footnote{The conjecture stated in~\cite{warwick-problems} mentions dependence instead of monadic dependence, but for hereditary classes, those notions are equivalent \cite{braunfeld2022existential}. 
Furthermore, the conjecture in \cite{warwick-problems} states only one implication, but the other one was also posed at the same workshop as an open problem.}
\end{conjecture}

Originally stated in 2016~\cite{warwick-problems}, the above conjecture is now the central open problem in the area.
Both directions have been open.

\subsection*{Contribution}

To approach Conjecture~\ref{conj:warwick}, 
a combinatorial characterization of monadic dependence is sought.
Based 
on the development in the sparse, orderless, and ordered cases discussed above,
it appears that what is needed are
 \emph{combinatorial dichotomy results}, stating that all monadically dependent graph classes
exhibit structure which can be used to design efficient algorithms,
while all other graph classes exhibit a sufficient lack of structure,
which can be used to prove hardness results.
The three known restricted classifications of classes with fpt model checking 
were enabled in part (see discussion below) thanks to such combinatorial dichotomies
for nowhere dense, monadically stable, and bounded twin-width graph classes.
However, for monadically dependent classes not a single combinatorial characterization has been known.
Previous characterizations of monadic dependence (for example, via indiscernible sequences or existentially defined canonical configurations~\cite{braunfeld2021characterizations,braunfeld2022existential}) all have a logical, rather than combinatorial, aspect. This limits their algorithmic usefulness. 

In this paper, we provide the first two \emph{purely combinatorial} characterizations of monadically dependent classes,
which together constitute a combinatorial dichotomy theorem for these classes.
\begin{itemize}
    \item On the one hand, we show that monadically dependent graph classes have a Ramsey-like property called
        \emph{flip-breakability}, which guarantees that any large set of vertices \(W\) contains
        two still-large subsets \(A,B\) that in a certain sense are strongly separated.
    \item 
        On the other hand, we show that graph classes that are monadically independent
        contain certain highly regular patterns as induced subgraphs, which are essentially two highly interconnected sets.
\end{itemize}
As argued below, \emph{flip-breakability}
might be a crucial step towards establishing fixed-parameter tractability of the model checking problem for monadically dependent classes,
and settling the tractability side of \Cref{conj:warwick}.
Moreover, we use the patterns of the second characterization to prove
the hardness side of \Cref{conj:warwick}: we show that first-order model checking is  AW[$*$]-hard on every hereditary graph class that is monadically independent (\cref{thm:hardness-main}).
We now present our two combinatorial characterizations of monadically dependent classes in more detail.

\subsubsection*{Flip-Breakability}
\emph{Flips} are a central emerging mechanism in the study of well-behaved graph classes~(e.g., \cite{shrubdepth-journal,bonnet2022model,flipper-game,flipwidth}).
A \emph{\(k\)-flip} of a graph \(G\) is a graph $H$ obtained from $G$ by partitioning the vertex set into $k$ parts and,
for every pair $X,Y$ of parts, either leaving the adjacency between $X$ and $Y$ intact, or complementing it (see also \Cref{sec:prelims}).

We use {flips} to measure the interaction between vertex sets \(A,B\) at a fixed distance \(r\) in a graph $G$.
Intuitively, if for some small value $k$ there is a $k$-flip $H$ of $G$  
in which the distance \(\dist_{H}(A,B)\) is larger than $r$,
then this witnesses that the sets $A$ and $B$ are ``well-separated'' at distance $r$.

\begin{figure}[h]
    \begin{center}
    \includegraphics[scale=1]{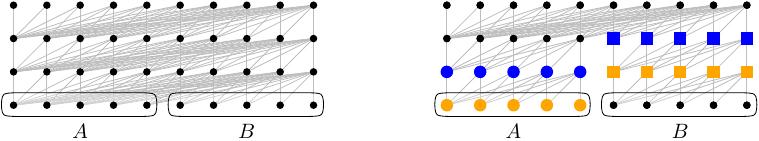}%
    \end{center}%
    \vspace{-0.5cm}
    \caption{
        The sets \(A\) and \(B\) are far away after a \(5\)-flip shown on the right side.
        We flip between the blue squares and the blue circles,
        and between the orange squares and the orange circles.
    }\label{fig:stack}
\end{figure}
Consider for example the graph $G$ on the left side of \Cref{fig:stack}, consisting of stacked half-graphs.
Applying a flip between the blue squares and the blue circles (as depicted to the right of Fig.~\ref{fig:stack}),
and between the orange squares and the orange circles,
we obtain a $5$-flip of $G$ in which the sets \(A\) and \(B\) have distance at least $6$.

Roughly, our first main result expresses that for graphs from a monadically dependent class,
in every sufficiently large vertex set
\(W\), we can find two still-large subsets \(A\) and \(B\)
whose distance in some \(O(1)\)-flip 
is larger than a given constant $r$.
This is formally expressed as follows.

\begin{restatable}[Flip-Breakability]{definition}{defFlipBreakableOne}\label{def:flip-breakability1}
A class of graphs $\CC$ is $\emph{flip-breakable}$
if for every radius \(r \in \N\) there exists a function
\(N_r : \N \to \N\) and a constant \(k_r \in \N\) such that for all \(m \in \N\), \(G \in \CC\) and \(W \subseteq V(G)\) with
\(|W| \ge N_r(m)\) there exist subsets $A,B\subset W$ with \(|A|,|B| \ge m\)
and a \(k_r\)-flip \(H\) of \(G\) such that:
\[
    \dist_{H}(A,B) > r.
\]
\end{restatable}

\begin{restatable}{theorem}{thmFlipbreakable}\label{thm:mainflipbreakable}
    A class of graphs is monadically dependent if and only if it is flip-breakable.
\end{restatable}

Our proof is algorithmic: For a fixed monadically dependent class $\CC$ and radius $r$, the subsets $A$ and $B$ and the witnessing flip $H$ can be computed in time $O_{\CC,r}(|V(G)|^2)$ (\cref{thm:alg-flip-breakability}).

As we discuss in~\Cref{sec:technicalOverview},
our results are based on a new notion called the \emph{insulation property} (\Cref{def:insulator,def:insulationProperty}),
which trivially implies flip-breakability, but imposes an even more regular structure on monadically dependent graph classes.

\subsubsection*{Relation of Flip-Breakability to Other Work}
The model checking algorithms for nowhere dense classes~\cite{grohe2017deciding}
and monadically stable classes~\cite{ssmc,wip} 
are respectively based on winning strategies for \emph{pursuit-evasion games} called the \emph{splitter game}~\cite{grohe2017deciding} and the \emph{flipper game}~\cite{flipper-game}.
These winning strategies were obtained from
a characterization of nowhere dense classes in terms of \emph{uniform quasi-wideness}~\cite{dawar2010homomorphism,nevsetvril2011nowhere}
and a characterization of monadically stable classes in terms of \emph{flip-flatness}~\cite{dreier2022indiscernibles}.
Both are Ramsey-like properties
which are very similar to flip-breakability.
The definitions of these three notions match the same template:
\begin{quote}
A class of graphs $\CC$ is \dots\  
if for every radius \(r \in \N\) there exists a function
\(N_r : \N \to \N\) and a constant \(k_r \in \N\) such that for all \(m \in \N\), \(G \in \CC\) and \(W \subseteq V(G)\) with
\(|W| \ge N_r(m)\) there exist \dots
\end{quote}
In all three cases we continue the definition by stating the existence of a large number of vertices that are in a certain sense ``well-separated''.
However, the sentence is completed in an increasingly more general way for the three notions.
The crucial differences are highlighted in bold.

\begin{itemize}
    \item uniform quasi-wide~\cite{dawar2010homomorphism,nevsetvril2011nowhere}: \quad \dots
        a set \(\boldsymbol S \subseteq V(G)\) of at most \(\boldsymbol{k_r}\) \textbf{vertices} and
        \(A \subseteq W\setminus S\) with \(|A| \ge m\),
        such that in \(G \setminus S\), all vertices from \(A\) have \textbf{pairwise distance greater than~\(\boldsymbol{r}\)}.

    \item flip-flat~\cite{dreier2022indiscernibles}: \quad \dots
        a \textbf{\(\boldsymbol{k_r}\)-flip $\boldsymbol H$} of $G$ and
    \(A \subseteq W\) with \(|A| \ge m\),
    such that in \(H\), all vertices from \(A\) have \textbf{pairwise distance greater than \(\boldsymbol r\)}.

\item 
    flip-breakable: \quad \dots 
    a \textbf{\(\boldsymbol {k_r}\)-flip} $\boldsymbol H$ of $G$ and
    \(A,B \subseteq W\) with \(|A|,|B| \ge m\),
    such that in \(H\), \textbf{all vertices in \(\boldsymbol A\) have distance greater than \(\boldsymbol r\) from all vertices in \(\boldsymbol B\)}.
\end{itemize}

In the context of ordered graphs, the model checking algorithm for classes of bounded twin-width
crucially relied on a characterization of these classes in terms of \emph{bounded grid rank}~\cite[Sec. 3.4]{twwIV}.
Rephrasing this characterization in the language of this paper, it reads as follows.

\begin{quote}\label{def:grid-rank}
A class of ordered graphs $\CC$ has \emph{bounded grid rank}
if there exists \(k \in \N\) such that for all \(G \in \mathcal{C}\)
and ordered sequences of vertices
\(a_1 < \dots < a_k \in V(G)\), \(b_1 < \dots < b_k \in V(G)\),
there exists 
a $k$-flip $H$ of $G$
and indices \(i,j \in [k-1]\) defining ranges 
\(A = \{ a \in V(G) \mid a_i \le a \le a_{i+1}\}\),
\(B = \{ b \in V(G) \mid b_j \le b \le b_{j+1}\}\),
such that there are no edges incident to both \(A\) and \(B\) in $H$.
\end{quote}

Flip-breakability combines the distance-\(r\)-based aspects of uniform quasi-wideness and flip-flatness with
the ``no edges incident to both \(A\) and \(B\)'' criterion of grid rank.
Using our \emph{insulation property}, one can easily reprove the characterization of nowhere dense classes
in terms of uniform quasi-wideness~\cite{nevsetvril2011nowhere}, of monadically stable classes in terms of flip-flatness~\cite{dreier2022indiscernibles},
and of classes of ordered graphs with bounded twin-width in terms of grid rank~\cite{twwIV}.

Given that flip-breakability characterizes monadically dependent graph classes
and naturally generalizes the previous notions,
we believe that it will also play a crucial role in a future model checking algorithm for monadically dependent classes:
Mirroring the situation for nowhere dense and monadically stable classes,
flip-breakability might lead to a characterization of monadically dependent classes
in terms of a \emph{pursuit-evasion game}.
Just like the previous two algorithms~\cite{grohe2017deciding,ssmc,wip}, 
winning strategies for this game might then be useful for model checking,
and settling the backward direction of \Cref{conj:warwick}.
(Note however that in each of these cases, 
the games where only one of the main ingredients
of the model checking algorithm).

\paragraph{Binary Structures.}
Our results apply in the more general setting of binary structures, rather than graphs,
that is, of structures equipped with one or more binary relation.
In particular, we prove that monadically dependent classes of binary structures can be equivalently characterized in terms of flip-breakability, defined suitably for binary structures.
As an application, we derive a key result of \cite{twwIV}: that monadically dependent classes of ordered graphs have bounded grid rank (see definition above).

\begin{theorem}[\cite{twwIV}]\label{thm:grid-rank-intro}
    Let $\CC$ be a monadically dependent class of ordered graphs.
    Then $\CC$ has bounded grid rank.
\end{theorem}

\paragraph{Variants of Flip-Breakability.}

The previously highlighted similarities between flip-breakability, flip-flatness, and uniform quasi-wideness 
suggest the following natural variations of flip-breakability.
We allow to modify the graph using either
\emph{(1)~flips or vertex deletions} 
and demand that the resulting subset is either 
\emph{(2)~flat or broken},    
that is, either pairwise separated or separated into two large sets.
We further additionally parameterize the type of separation to be either 
\emph{(3)~\mbox{distance-$r$} or \mbox{distance-$\infty$}}.   
While \emph{distance-\(r\)} separation corresponds to the usual kind given in \Cref{def:flip-breakability1},
\emph{distance-\(\infty\)} separation demands sets to be in different connected components of the graph.
This is formalized by the following definition.

\begin{definition}
    A class of graphs $\CC$ is \emph{distance-$\infty$ flip-breakable},
    if there exists a function
    \(N : \N \to \N\) and a constant \(k \in \N\) such that for all \(m \in \N\), \(G \in \CC\) and \(W \subseteq V(G)\) with
    \(|W| \ge N(m)\) there exist subsets $A,B\subset W$ with \(|A|,|B| \ge m\) 
    and a $k$-flip $H$ of $G$ such that in $H$, no two vertices $a\in A$ and $b\in B$ are in the same connected component.
\end{definition}

Using this terminology, uniform quasi-wideness, for example, becomes \emph{distance-\(r\) deletion-flatness}.
For formal definitions of these variants, we refer to \Cref{sec:variants}.
As summarized in the following \cref{tab:variants},
each of the eight possible combinations of (1), (2) and (3) characterizes a well-studied property of graph classes.

\begin{table}[htbp]
    \centering

    \begin{tabular}{|c|c|>{\centering}p{5.7cm}|>{\centering}p{6cm}|}
        \cline{3-4} \cline{4-4} 
        \multicolumn{1}{c}{} &  & flatness & breakability\tabularnewline
        \hline 
        \multirow{2}{*}{dist-$r$} & flip- & monadic stability \hfill \cite{dreier2022indiscernibles} & monadic dependence \hfill (Thm.\ \ref{thm:main-circle})\tabularnewline
        \cline{2-4} \cline{3-4} \cline{4-4} 
         & deletion- & nowhere denseness \hfill \cite{dawar2010homomorphism,nevsetvril2011nowhere} & nowhere denseness \hfill (Thm.\ \ref{thm:nd}) \tabularnewline
        \hline 
        \multirow{2}{*}{dist-$\infty$} & flip- & bd. shrubdepth \hfill (Thm.\ \ref{thm:shrubdepth}) & bd. cliquewidth \hfill ~ (Thm.\ \ref{thm:cw})\tabularnewline
        \cline{2-4} \cline{3-4} \cline{4-4} 
         & deletion- & bd. treedepth \hfill (Thm.\ \ref{thm:treedepth}) & bd. treewidth \hfill (Thm.\ \ref{thm:tw})\tabularnewline
        \hline 
        \end{tabular}
    \caption{Variants of flip-breakability.}
    \label{tab:variants}
\end{table}
Notably, the right column of \cref{tab:variants} corresponds to conjectured tractability limits of the model checking problem, where the distance-$r$ and distance-$\infty$ variants correspond to, respectively, first-order and monadic second-order logic, and the flip and deletion variants correspond to hereditary and monotone classes.
Furthermore, it is interesting to see that the seemingly more general deletion-breakability collapses to deletion-flatness, as both properties characterize nowhere denseness.

\subsubsection*{Forbidden Induced Subgraphs}
Imposing bounds on the size of certain patterns is a common and powerful mechanism for defining well-behaved graph classes.
For example, by definition, a class $\CC$ is \emph{nowhere dense} if and only if for every \(r \in \N\),
$\CC$ avoids some \(r\)-subdivided clique as a subgraph.
In similar spirit, it is shown \cite{wip} that an orderless class $\CC$ is \emph{monadically stable} if and only if for every \(r \in \N\),
the class $\CC$ avoids all $r$-flips of the \(r\)-subdivided clique, and its line graph, as an induced subgraph.
A similar characterization for monadically dependent classes has so far been elusive.
In this paper, we show that a class $\CC$ is \emph{monadically dependent} if and only if for every \(r \in \N\),
the class $\CC$ avoids certain variations of the \(r\)-subdivision of some complete bipartite graph, as an induced subgraph.
Let us start by defining these patterns.

For \(r \ge 1\), the \emph{star \(r\)-crossing} of order \(n\) is the \(r\)-subdivision of \(K_{n,n}\).
More precisely, it consists of \emph{roots} \(a_1,\dots,a_n\) and \(b_1,\dots,b_n\)
together with \(r\)-vertex paths \(\{ \pi_{i,j} \mid i,j \in [n] \}\) that are pairwise vertex-disjoint (see \Cref{fig:pattern}).
We denote the two endpoints of a path \(\pi_{i,j}\) by \(\Start(\pi_{i,j})\) and \(\End(\pi_{i,j})\).
We require that roots appear on no path, that each root \(a_i\) is adjacent to \(\{ \Start(\pi_{i,j}) \mid j \in [n] \}\),
and that each root \(b_j\) is adjacent to \(\{ \End(\pi_{i,j}) \mid i \in [n] \}\).
The \emph{clique \(r\)-crossing} of order $n$ is the graph obtained from the star \(r\)-crossing of order $n$
by turning the neighborhood of each root into a clique.
Moreover, we define the \emph{half-graph \(r\)-crossing} of order $n$ similarly to the star \(r\)-crossing of order $n$,
where each root \(a_i\) is instead adjacent to \(\{ \Start(\pi_{i',j}) \mid i',j \in [n], i \le i' \}\),
and each root \(b_j\) is instead adjacent to \(\{ \End(\pi_{i,j'}) \mid i,j' \in [n], j \le j' \}\).
Each of the three $r$-crossings contains no edges other than the ones described.
At last, the \emph{comparability grid} of order \(n\) consists of vertices
\(\{ a_{i,j} \mid i,j \in [n] \}\) and edges between vertices \(a_{i,j}\) and \(a_{i',j'}\) if and only if either $i=i'$, or $j=j'$,
or $i<i'\Leftrightarrow j<j'$.

\begin{figure}[h]
    \begin{center}
    \includegraphics[width=\textwidth]{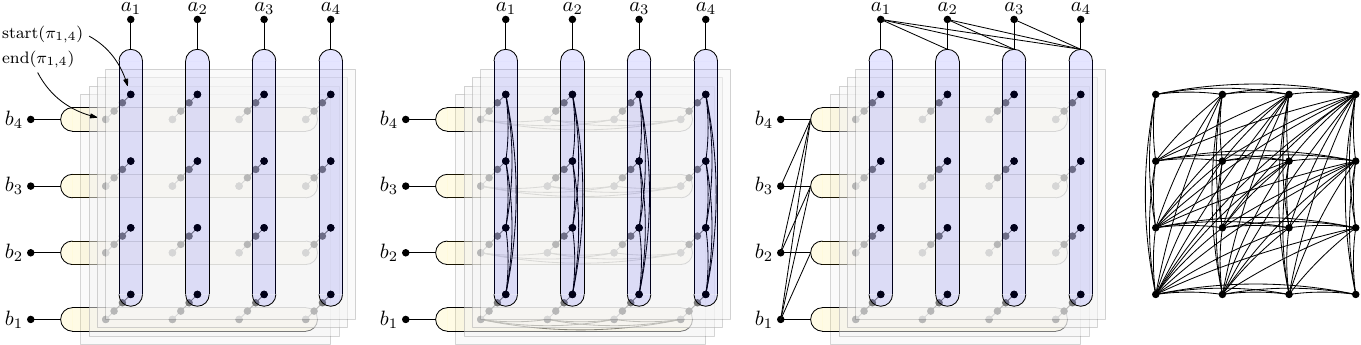}%
    \end{center}%
    \vspace{-0.5cm}
    \caption{
    (\emph{i}) star 4-crossing of order 4. 
    (\emph{ii}) clique 4-crossing of order 4. 
    (\emph{iii}) half-graph 4-crossing of order 4. 
    (\emph{iv}) comparability grid of order 4.
    In (\emph{i}), (\emph{ii}), (\emph{iii}), the roots are adjacent to all vertices in their respective colorful strip.
    }\label{fig:pattern}
\end{figure}
It is easy to see that these patterns are logically complicated:
For every fixed~\(r\), the four graph classes containing these four types of patterns are monadically independent.

We also need to consider certain flips of the above patterns.
To this end, we partition the vertices of star, clique, and half-graph \(r\)-crossings into \emph{layers}:
The 0th layer consists of the vertices \(\{a_1,\dots,a_n\}\).
The \(l\)th layer, for \(l \in [r]\), consists of the \(l\)th vertices of the paths \(\{ \pi_{i,j} \mid i,j \in [n] \} \)
(that is, the 1st and \(r\)th layer, respectively, are
\(\{ \Start(\pi_{i,j}) \mid i,j \in [n] \}\) and
\(\{ \End(\pi_{i,j}) \mid i,j \in [n] \}\)).
Finally, the \((r+1)\)th layer consists of the vertices \(\{b_1,\dots,b_n\}\). 
A \emph{flipped} star/clique/half-graph \(r\)-crossing
is a graph obtained from a 
star/clique/half-graph \(r\)-crossing
by performing a flip where the parts of the specifying partition are the layers of the $r$-crossing.
Note that while there is only one star/clique/half-graph  $r$-crossing of order $n$, there are multiple flipped star/clique/half-graph  $r$-crossings of order $n$. Their number is however bounded by $2^{(r+2)^2}$: an upper bound for the number of possible flips specified by a single partition of size~$(r+2)$.

\begin{restatable}{theorem}{thmMainforbiddenpatterns}\label{thm:mainforbiddenpatterns}
    Let \(\mathcal{C}\) be a graph class.
    Then \(\mathcal{C}\) is monadically dependent if and only if for every \(r \geq 1\) there exists \(k \in \N\)
    such that \(\mathcal{C}\) excludes as induced subgraphs
    \begin{itemize}
        \item all flipped star $r$-crossings of order $k$, and
        \item all flipped clique $r$-crossings of order $k$, and
        \item all flipped half-graph $r$-crossings of order $k$, and
        \item the comparability grid of order \(k\).
    \end{itemize}
\end{restatable}

This characterization by forbidden induced subgraphs opens the door for various algorithmic, logical, and combinatorial hardness results.
On the algorithmic side, we prove the hardness part of \cref{conj:warwick}.

\begin{restatable}{theorem}{thmHardnessMain}\label{thm:hardness-main}
    The first-order model checking problem is $\mathrm{AW}[*]$-hard on every hereditary, monadically independent graph class.
\end{restatable}

The result builds on the following logical hardness result, which is of independent interest.

\begin{restatable}{theorem}{hardnessInterprets}\label{thm:hardness-interprets}
    Every hereditary, monadically independent graph class efficiently interprets the class of all graphs.
\end{restatable}

Recall that monadically independent classes were defined as those which transduce the class of all graphs. For hereditary classes, the above result strengthens this definition: \emph{interpretations} are a more restrictive version of transductions, that do not include a nondeterministic coloring step. Moreover, \emph{efficient} interpretations come with a polynomial time algorithm which,
given an arbitrary input graph $G$, outputs a graph in $\CC$
in which $G$ is encoded. This allows to reduce the model checking problem for all graphs to the model checking problem on the class $\CC$, thus proving \cref{thm:hardness-main}.

On the combinatorial side, we show that no hereditary, monadically independent graph class
is \emph{small}, has \emph{almost bounded twin-width}, or has \emph{almost bounded flip-width}.
Let us quickly explain the three notions.

A graph class $\CC$ is \emph{small}
if it contains at most $n!c^n$ distinct labeled $n$-vertex graphs, for some constant $c$.
This notion has been studied in enumerative combinatorics \cite{MCDIARMID2005187,NORINE2006754}.
It is known that all classes of bounded twin-width are small \cite{twwII}. 
The converse implication was conjectured,
and was subsequently refuted \cite{twwVII}.
In the context of ordered graphs, 
it was shown that all small classes of ordered graphs have bounded twin-width~\cite{twwIV}.
We prove the following.

\begin{restatable}{theorem}{thmSmall}\label{thm:small}
    Every hereditary, small graph class is monadically dependent.
\end{restatable}

Say that a class $\CC$ of graphs has \emph{almost bounded twin-width}
if for every $\eps>0$ 
the twin-width of every $n$-vertex graph $G\in\CC$ is bounded by  $O_{\eps,\CC}(n^{\eps})$.
We prove the following.

\begin{restatable}{theorem}{almostBoundedTww}\label{thm:almost-bounded-tww}
    Every hereditary, almost bounded twin-width graph class is monadically dependent.
\end{restatable}    

In \cite{flipwidth}, a family of graph-width parameters called \emph{flip-width of radius} $r$, for $r\ge 1$, is introduced,
together with the ensuing notion of classes of \emph{almost bounded flip-width}.
Classes of almost bounded flip-width include all nowhere dense classes, all classes of bounded twin-width, and more generally, and all classes of almost bounded twin-width.
It is conjectured that for hereditary graph classes, almost bounded flip-width coincides with monadic dependence:

\begin{restatable}[\cite{flipwidth}]{conjecture}{conjFW}\label{conj:fw}
    A hereditary graph class $\CC$ has almost bounded flip-width if and only if it is monadically dependent.
\end{restatable}

We prove one implication of this conjecture.
\begin{theorem}\label{thm:fw-intro}
    Every hereditary, almost bounded flip-width graph class is monadically dependent.
\end{theorem}

\subsubsection*{Relation of the Patterns to Other Work}
Our forbidden patterns characterization~\Cref{thm:mainforbiddenpatterns}
 generalizes similar previous characterizations.
 Recall that monadic dependence is captured by
 monadic stability, 
 for hereditary orderless graph classes,
and by bounded twin-width, for hereditary classes of ordered graphs.
The results \cite{wip,twwIV} 
characterizing monadic dependence in those two settings
can be restated as follows.
\begin{itemize}
    \item
        For orderless graph classes, a class \(\mathcal{C}\) is monadically dependent 
        if and only if for every \(r \in \N\) there exists \(k \in \N\)
        such that no graph in \(\mathcal{C}\) contains a flipped star \(r\)-crossing or clique \(r\)-crossing
        of order \(k\) as an induced subgraph~\cite{wip}.
    \item
        For classes of ordered graphs, a class \(\mathcal{C}\) is monadically dependent 
        if and only if there exists \(k \in \N\)
        and a specific, finite family of  ordered graphs with $2k$ vertices (similar to a matching  
on $2k$ vertices ordered suitably) 
        which are avoided by \(\mathcal{C}\) as  semi-induced ordered subgraphs~\cite{twwIV}
        (see \cref{thm:tww-patterns} for a formulation).
\end{itemize}
Our approach towards proving the forbidden patterns characterization and model checking hardness originates in \cite{wip} and \cite{twwIV}.
As a result, we reprove some of their results.
A subset of the patterns identified in this paper are sufficient to characterize monadic dependence 
in the setting of orderless graph classes.
As orderless classes cannot contain large half-graph crossings or comparability grids,
\Cref{thm:mainforbiddenpatterns} implies the result of~\cite{wip} characterizing monadically stable classes in terms of forbidden induced subgraphs.
It is worth noting that in \cite{wip}, for hereditary, orderless, monadically independent graph classes, hardness is shown even for \emph{existential} model checking.
Adapting our results to the setting of binary structures, we show how to derive the characterization from \cite{twwIV} of monadically dependent classes of ordered graphs in terms of forbidden patterns.

A recent paper by Braunfeld and Laskowski shows that a hereditary class of relational structures is monadically dependent if and only if it is dependent~\cite{braunfeld2022existential}.
Dependence is a generalization of monadic dependence studied in model theory.
For the special case of graph classes, our \cref{thm:hardness-interprets} reproves that result.
Similarly to our paper, the proof of Braunfeld and Laskowski exhibits certain large configurations (called pre-coding configurations)
in classes that are monadically independent.
As pre-coding configurations are defined in terms of formulas with tuples of free variables, these results are not a purely combinatorial characterization of monadically dependent graph classes. In particular, they seem insufficient for obtaining algorithmic hardness results.
However, we believe that one could also prove \cref{thm:small}, stating that small, hereditary classes are dependent, using the results of \cite{braunfeld2022existential}.

In \cite{DBLP:journals/corr/abs-2306-12611},
Eppstein and McCarty
prove that
many different types of geometric graphs have unbounded flip-width, and in fact \cite[App. A]{DBLP:journals/corr/abs-2306-12611}, form monadically independent graph classes.
These include interval graphs, permutation graphs, circle graphs, and others. It is shown that these classes contain large \emph{interchanges}, which 
are similar to $1$-subdivided cliques,
 and to our patterns.
Containing interchanges of arbitrarily large order is a sufficient, but not necessary condition to monadic independence.

\section{Technical Overview}\label{sec:technicalOverview}

All our proofs are fully constructive and only use elementary tools such as Ramsey's theorem.

\paragraph{Insulators.}
In order to prove flip-breakability for a graph \(G\) and a large set \(W \subseteq V(G)\),
we try to enclose a sizeable subset $W_\star\subset W$
in a structure $\gc A$ that ``insulates'' the elements of $W_\star$ from each other, and from external vertices that are not enclosed in $\gc A$.
We call this structure $\gc A$ an \emph{insulator} (\Cref{def:insulator}). It is a grid-like partition of a subset of $V(G)$.
An example is shown on the left side of \Cref{fig:technicalOverview} (in general, each cell may contain more than one vertex).

Roughly speaking, there is a coloring of the vertex set, using a bounded number of colors, such that
the adjacency between a pair of vertices in non-adjacent rows of the grid is determined only by their colors.
For a pair of vertices in adjacent rows, the adjacency is determined by their colors and the order of their columns.
Finally, the adjacencies between vertices in the exterior of the grid 
and the vertices inside the grid (without the first and last column, and last row) are determined by their colors. 
Exceptionally, the connections within the top row and between adjacent cells of the grid may be arbitrary.

The large subset $W_\star\subset W$ is distributed in the bottom row of the insulator, in a way such that every vertex of $W_\star$ is contained in a different column.
Assume the insulator has height at least~$r$.
We can choose two large subsets $A$ and $B$ of $W_\star$, such that 
for any two vertices $v\in A$ and $w\in B$, it takes at least $r$ steps along edges which are not controlled by the insulator to get from $v$ to $w$.

\begin{figure}[h]
    \begin{center}
    \includegraphics[width=\textwidth]{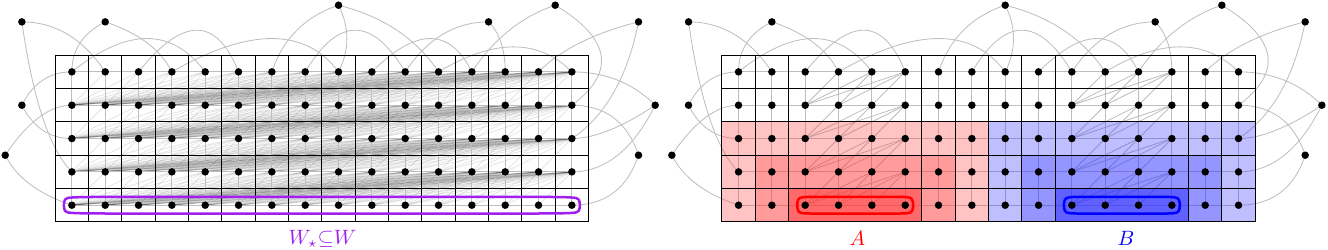}%
    \end{center}%
    \vspace{-0.5cm}
    \caption{
        Left: An insulator surrounding the set \(W_\star\). 
        The shown stacked half-graph layout is easily described by a bounded number of colors and the column order.
        A bounded number of flips are sufficient to ensure the highlighted subsets \(A \subseteq W_\star\) and \(B \subseteq W_\star\)
        have distance larger than \(4\). The shades of red and blue highlight upper bounds on the distance of other vertices to \(A\) and \(B\).
        In particular, all unshaded vertices have distance larger than two to these sets.
    }\label{fig:technicalOverview}
\end{figure}

Given $r \in \N$,
we strive for insulators of height $r$, where the number of colors is bounded by some constant, depending on $r$.
Thus, if we embed large sets \(A,B\) in such an insulator and coarsen the columns as shown on the right side of \Cref{fig:technicalOverview},
we can ensure with a bounded number of flips (depending on the number of colors) that edges only go between adjacent cells, and thus \(A\) and \(B\) have distance at least \(r\).
Hence, we can use the insulators to obtain flip-breakability.

\paragraph{Constructing Insulators or Prepatterns.}
One can trivially construct an insulator of height one that embeds a set \(W\):
Build a single row by placing every vertex of \(W\) into a distinct cell.
The central step of our construction, presented in \Cref{part:structure}, takes a large insulator of height \(r\),
and adds another row on top to obtain a still-large insulator of height \(r+1\).
To this end, we build upon and significantly extend 
techniques developed in the context of flip-flatness~\cite{dreier2022indiscernibles}, based on \emph{indiscernible sequences}, a fundamental tool in model theory.

Given a large  insulator of height $r$ in any graph \(G\),
the key step in our construction
shows that there is either
\emph{enough disorder} to connect the topmost row of the grid into a preliminary pattern, called \emph{prepattern} (\Cref{def:prepattern}), resembling a crossing (\Cref{fig:gridOrRow}, left), or
\emph{enough order} to add an additional row \(r+1\) on top (\Cref{fig:gridOrRow}, right).
While both cases reduce the number of columns in the insulator, there still remains a large number of them.

\begin{figure}[h]
    \begin{center}
    \includegraphics[scale=1]{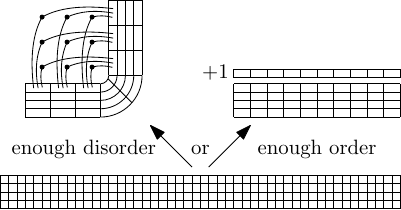}%
    \end{center}%
    \vspace{-0.5cm}
    \caption{
        In every graph, we can either increase the height of an insulator of find a large prepattern.
    }\label{fig:gridOrRow}
\end{figure}

We define \emph{prepattern-free} classes (\Cref{def:prepatternfree}) 
by excluding certain prepatterns.
In such classes, we therefore always find \emph{enough order} to construct \emph{large insulators} row-by-row (\Cref{prop:prepatternImpliesInsulation}).
We say that such classes have the \emph{insulation property} (\Cref{def:insulationProperty}).
By the arguments outlined in \Cref{fig:technicalOverview}, such classes are easily shown to be \emph{flip-breakable} (\Cref{prop:ge-implies-fb}).
Lastly, using the locality property of first-order logic, it is easy to show that \emph{flip-breakable} classes are monadically dependent (\Cref{prop:fb-implies-mnip}).
Thus, any graph class satisfies the following chain of implications.
\newcommand{\superarrow}{\ {=\joinrel=\joinrel=\joinrel\Longrightarrow}\ }
\[
    \textnormal{prepattern-free}~\overset{\textnormal{Prop.\,\ref{prop:prepatternImpliesInsulation}}}{\superarrow}~
    \textnormal{insulation property}~\overset{\textnormal{Prop.\,\ref{prop:ge-implies-fb}}}{\superarrow}~
    \textnormal{flip-breakable}~\overset{\textnormal{Prop.\,\ref{prop:fb-implies-mnip}}}{\superarrow}~
    \textnormal{mon.\ dependent}
\]

\paragraph{Cleaning Up Prepatterns.}
Afterwards, in \Cref{part:nonstructure}, we clean up the prepatterns.
This involves heavy use of Ramsey's theorem to regularize the patterns
until we either obtain a flipped star/clique/half-graph crossing, or a comparability grid.
From such patterns one can trivially transduce all graphs.
This is summarized by the following chain of implications.
\[
    \neg\textnormal{prepattern-free} ~\overset{\,\textnormal{Def.}}{\Leftarrow\joinrel=\joinrel\Rightarrow}~
    \textnormal{large prepatterns}~\overset{\textnormal{Prop.\,\ref{prop:patterns-main}}}{\superarrow}~
    \textnormal{large patterns}~\overset{\textnormal{Prop.\,\ref{lem:transduceAll}}}{\superarrow}~
    \textnormal{mon.\ independent} 
\]
Hence, by contrapositive, monadic dependence implies prepattern-freeness. Together with the previous chain of implications,
we obtain the desired equivalences.
%Both chains of implications together imply the desired equivalences.
\[
    \textnormal{insulation property}~\Longleftrightarrow~
    \textnormal{flip-breakable}     ~\Longleftrightarrow~
    \textnormal{no large patterns}  ~\Longleftrightarrow~
    \textnormal{mon.\ dependent}
\]

\paragraph{Hardness.}
Consider a hereditary class $\CC$ that is monadically independent.
As discussed, $\CC$ contains large patterns witnessing this.
To obtain hardness of model checking, we reduce from the model checking problem in general graphs.
This requires an encoding of an arbitrary graph into a large pattern, using a first-order formula.
We know that when adding colors, such an encoding is possible: classes that are monadically independent transduce the class of all graphs.
However, for our reduction we require encodings that do not use colors: we want to show that every hereditary class that is monadically independent \emph{interprets} the class of all graphs.
Here, an \emph{interpretation} is a transduction that does not use colors and where instead of taking arbitrary induced subgraphs, the vertex set of the interpreted graph must be definable in the original graph by a formula $\delta(x)$.
Heavily relying on the fact that in hereditary classes we can take induced subgraphs \emph{before} applying the interpretation, it is not too hard to show that for each $r \geq 1$, the hereditary closures of the class of all comparability grids and the classes of all \emph{non-flipped} star/clique/half-graph $r$-crossings interpret the class of all graphs.
The main challenge is to ``reverse'' the flips using first-order formulas in the case of \emph{flipped} $r$-crossings. This is achieved by using sets of twins to mark layers of the $r$-crossing.

\section{Conclusions and Future Work}
In this paper, we obtain the first combinatorial characterizations of monadically dependent graph classes, 
which open the way to generalizing the results of sparsity theory to the setting of hereditary graph classes.
Our main results, \Cref{thm:mainflipbreakable} and \Cref{thm:mainforbiddenpatterns}, may be seen as analogues of the 
result \cite{nevsetvril2011nowhere} characterizing nowhere dense graph classes 
as exactly the \emph{uniformly quasi-wide} classes \cite{dawar2010homomorphism}, and 
the result \cite{adler2014interpreting,stable_graphs}
characterizing nowhere dense classes 
as exactly those, whose monotone closure is monadically dependent.

Central results in sparsity theory, characterizing 
nowhere dense classes, can be usually grouped into two 
types: \emph{qualitative characterizations} and \emph{quantitative characterizations}.
Qualitative characterizations typically say that 
for every radius $r$, some quantity is bounded by a constant depending on $r$ for all graphs in the class.
%  Examples 
%  include the characterizations of nowhere denseness in terms of the bounded size of $r$-subdivisions of cliques as subgraphs, or the characterization in terms of uniform quasi-wideness.
 Our two main results fall within this category.

 On the other hand, quantitative characterizations of nowhere dense classes
are phrased in terms of a fine analysis of densities of some parameters,
such as degeneracy, minimum degree, weak coloring numbers,
neighborhood complexity, or VC-density;
those results almost always involve  bounds of the form $n^\eps$ or $n^{1+\eps}$, where 
$n$ is the number of vertices of the considered graph, and \(\eps > 0\) can be chosen arbitrarily small.
% Many of the key resuts of sparsity theory,
% including the algorithmic results concerning model checking \cite{grohe2017deciding},
% rely on \emph{quantitative} characterizations of nowhere dense classes, allowing to measure the density of graphs in the class.
For instance, given a nowhere dense graph class \(\CC\),
for every fixed $r\in\N$ and \(\eps > 0\), all graphs $G$ whose $r$-subdivision is a subgraph of a graph in $\CC$
satisfy $$|E(G)|\le O_{\CC,r,\eps}(|V(G)|^{1+\eps}).$$
Similarly, a parameter called the \emph{weak $r$-coloring number} (for any fixed $r\in \N$, \(\eps > 0\)),
is bounded by $O_{\CC,r,\eps}(n^\eps)$, 
for every $n$-vertex graph $G$ in a nowhere dense class.

Both the qualitative and quantitative results are of fundamental importance in sparsity theory, 
and lifting them to the setting of monadically dependent graph classes is therefore desirable. 
Some of the most elaborate results of sparsity theory combine both aspects of the theory.
In particular for nowhere dense model checking  \cite{grohe2017deciding},
the \emph{splitter game} \cite{grohe2017deciding} -- a qualitative characterization -- is combined with the quantitative characterization in terms of weak coloring numbers.
The characterization in terms of the splitter game relies on uniform quasi-wideness, and 
has been extended to the setting of monadically stable graph classes, in terms of the flipper game \cite{flipper-game}, which in turn relies on the characterization in terms of flip-flatness.
Basing on this, we believe that flip-breakability may be a first step towards obtaining a game characterization of monadic dependence.

\medskip
As we observe now, our results also provide a first quantitative characterization of monadically dependent classes.
As an analogue of the notion of \emph{containing 
the $r$-subdivision of a graph as a subgraph}, we introduce the following concept of \emph{radius-$r$ encodings}.
Fix an integer $r\ge 1$.
Let $G=(A,B,E)$ be a bipartite graph with $|A|=|B|=n$ for some $n$,
and let $A=\set{a_1,\ldots,a_n}$ and $B=\set{b_1,\ldots,b_n}$.
Consider a graph $H$ which is a star/clique/half-graph $r$-crossing
with roots $a_1,\ldots,a_n$ and $b_1,\ldots,b_n$.
Recall that $V(H)$ can be partitioned into $r+2$ \emph{layers}, and there are $n^2$ distinguished $r$-vertex \emph{paths} $\pi_{i,j}$
connecting \(a_i\) and \(b_j\), for $i,j\in[n]$.
Let $H'$ be a graph  obtained from $H$ by:
\begin{enumerate}
    \item adding arbitrary edges within each layer of $H$,
    \item removing all vertices of the paths $\pi_{i,j}$ for $i,j\in[n]$ such that $\set{a_i,b_j}\notin E(G)$,
    \item flipping pairs of layers arbitrarily.
\end{enumerate}
We call $H'$ a \emph{radius-$r$ encoding of $G$}.
In particular, every flipped star/clique/half-graph $r$-crossing of order $n$ 
is a radius-$r$ encoding of the complete bipartite graph $K_{n,n}$.
Moreover, the comparability grid of order $n+1$ contains as an induced subgraph a radius-$1$ encoding of the complete bipartite graph $K_{n,n}$,
which can be obtained from the half-graph $1$-crossing of order $n+1$
by adding edges within the three layers.

The following result follows easily from our \cref{thm:main-circle},
and from results of Dvo{\v r}{\'a}k \cite{dvorak-thesis,dvovrak2018induced}.
\begin{restatable}{theorem}{thmQuant}\label{thm:quant}
    Let  $\CC$ be a hereditary graph class.
    The following conditions are equivalent:
    \begin{enumerate}[label=(\roman*)]
        \item\label{it:quant-mdep} $\CC$ is monadically dependent, 
        \item\label{it:quant-eps} for every real $\eps>0$ and integer $r\ge 1$, 
        for every bipartite graph $G$
        such that $\CC$ contains some radius-$r$ encoding of $G$,
        we have that $|E(G)|\le O_{\CC,r,\eps}(|V(G)|^{1+\eps})$,
        \item\label{it:quant-eps-1/2} for every integer $r\ge 1$
        there is an integer $N\ge 1$ such that for every bipartite graph $G$ with $|V(G)|>N$
        such that $\CC$ contains some radius-$r$ encoding of $G$,
        we have that $|E(G)|< \frac{|V(G)|^2}4$.
    \end{enumerate}
\end{restatable}

\cref{thm:quant} may be a first step towards developing a quantitative theory of monadically dependent classes. 
A challenging goal here is to generalize the characterization of nowhere denseness in terms of weak coloring numbers, to monadically dependent classes.

\paragraph{Flip-width.}
The \emph{flip-width} parameters were introduced \cite{flipwidth} 
with the aim of obtaining quantitative characterizations of monadic dependence.
The flip-width at radius $r\ge 1$, denoted 
$\fw_r(\cdot)$, is an analogue of the weak $r$-coloring number, and the notion of classes of bounded flip-width,
and almost bounded flip-width,  generalize classes of bounded expansion and nowhere dense classes to the dense setting.
A graph class $\CC$ has \emph{bounded flip-width} 
if for every integer $r\ge 1$ 
there is a constant $k_r$ such that  $\fw_r(G)\le k_r$ for all graphs $G\in\CC$. 
A hereditary graph class $\CC$ has \emph{almost bounded flip-width}
if for every integer $r\ge 1$ and real $\eps>0$
we have  $\fw_r(G)\le O_{\CC,r,\eps}(n^\eps)$ for all $n$-vertex graphs $G\in\CC$.
\cref{conj:fw} -- equating monadic dependence with almost bounded flip-width -- therefore postulates a quantitative characterization of monadic dependence, parallel to the characterization of nowhere denseness in terms of weak coloring numbers.

As mentioned, our results imply the forward implication of \Cref{conj:fw}:
that every hereditary class of almost bounded flip-width is monadically dependent.
 \Cref{thm:quant} might be an initial step towards resolving the backwards direction in \Cref{conj:fw}, as 
by \Cref{thm:quant},
\Cref{conj:fw} is equivalent to the following:
\begin{conjecture}\label{conj:fw-restated}
    Let $\CC$ be a hereditary graph class. Then the following conditions are equivalent:
    \begin{enumerate}[label=(\roman*)]
         \item for every real $\eps>0$ and integer $r\ge 1$, 
        for every bipartite graph $G$
        such that $\CC$ contains some radius-$r$ encoding of $G$,
        we have: $$\frac{|E(G)|}{|V(G)|}\le O_{\CC,r,\eps}(|V(G)|^{\eps}),$$
        (equivalently, by \cref{thm:quant}, $\CC$ is monadically dependent)
        \item $\CC$ has almost bounded flip-width: For every real $\eps>0$ and integer $r\ge 1$, 
        and graph $G\in \CC$,
        we have:  $$\fw_r(G)\le O_{\CC,r,\eps}(|V(G)|^{\eps}).$$
    \end{enumerate}
\end{conjecture}
Note that the implication (\(ii\))$\rightarrow$(\(i\)) in \cref{conj:fw-restated} holds, by \cref{thm:quant} and \cref{thm:fw-intro}.
The following conjecture would imply the converse implication.

\begin{conjecture}\label{conj:fw-density}
    For every $r\ge 1$ there are integers  $s,k\ge 1$ such that for every graph $G$ we have:
    $$\fw_r(G)\le \max_{H} \left(\frac{|E(H)|}{|V(H)|}\right)^k,$$
    where the maximum ranges over all bipartite graphs $H$ such that $G$ contains some radius-$s$ encoding of $H$ as an induced subgraph.
\end{conjecture}

An analogue of \cref{conj:fw-density} holds in the sparse setting. There, the flip-width parameters are replaced by weak coloring numbers, and the maximum ranges over all  graphs $H$ such that $G$ contains some $s$-subdivision of $H$ as a subgraph.
\Cref{conj:fw-density} would furthermore imply the following characterization of classes of bounded flip-width,
analogous to a known characterization of classes with bounded expansion in terms of weak coloring numbers~\cite{ZHU20095562}.

\begin{conjecture}\label{conj:fw-quant}
    Let $\CC$ be a hereditary graph class. Then the following conditions are equivalent:
    \begin{enumerate}[label=(\roman*)]
        \item for every integer $r\ge 1$ there is a constant $k_r$
        such that 
        for every bipartite graph $G$
        such that $\CC$ contains some radius-$r$ encoding of $G$,
        we have that $|E(G)|\le k_r\cdot |V(G)|$,

        \item $\CC$ has bounded flip-width: For every integer $r\ge 1$ there is a constant $k_r$ such that for every 
     graph $G\in \CC$,
        we have that $\fw_r(G)\le k_r$.
    \end{enumerate}
\end{conjecture}

Note that the implication (\(ii\))$\rightarrow$(\(i\))
follows from the results 
of \cite{flipwidth} (that every weakly sparse 
transduction of a class of bounded flip-width 
has bounded expansion).
The converse implication would resolve
Conjecture 11.4 from \cite{flipwidth}, 
which predicts that if a class $\CC$ has unbounded flip-width,
then $\CC$ transduces a weakly sparse class of unbounded expansion.

\paragraph{Near-twins.}
A specific goal, not involving flip-width, 
which would be implied by the above conjectures, 
can be phrased in terms of \emph{near-twins}. 
Say that two distinct vertices $u,v$ in a graph $G$ 
are $k$-near-twins, where $k\in\N$, if the symmetric difference 
of the neighborhoods of $u$ and of $v$ consists of at most $k$ vertices.
It is known \cite{flipwidth} that every graph $G$ with more than one vertex 
contains a pair of $(2\fw_1(G))$-near-twins.
Consequently, for every class $\CC$ of bounded flip-width
there is a constant $k$ such that every graph $G\in\CC$ 
with more than one vertex contains a pair of $k$-near-twins.
Similarly, if $\CC$ has almost bounded flip-width then 
every $n$-vertex graph $G\in\CC$ with $n>1$
contains a pair of $O_{\CC,\eps}(n^\eps)$-near-twins.

Therefore, a first step towards \cref{conj:fw} is to 
prove that for all monadically dependent classes, every $n$-vertex graph $G\in\CC$ with $n>1$
contains a pair of $O_{\CC,\eps}(n^\eps)$-near-twins.
Similarly, a step towards \cref{conj:fw-quant} 
is to prove the following consequence of Conjecture~\ref{conj:fw-density} 
(in the case $r=1$),
stated below.

\begin{conjecture}
    There is an integer $s\ge 1$ and an unbounded function $f\from\N\to\N$ such that every graph $G$ with 
    more than one vertex 
    and no pair of $d$-near twins, for some $d\in\N$, contains as an induced subgraph
    a radius-$s$ encoding of some bipartite graph $H$ 
    with $|E(H)|/|V(H)|\ge f(d)$.        
\end{conjecture}

% Indeed, as every graph $G$ has a pair of $(2fw_1(G))$-near-twins, 
% Conjecture~\ref{conj:fw-density}, in the case $r=1$,
% implies that for every graph $G$,
% if $t$ is the maximum of the expression in the right-hand side 
% of the inequality in Conjecture~\ref{conj:fw-density},
% then $G$ contains a pair of $2t$-near-twins.
% In particular, if $G$ contains no pair of $d$-near-twins,
% then $2t>d$, proving the existence of a bipartite graph $H$ 
% with $|E(H)|

\paragraph{VC-density and Neighborhood Complexity.}
Another, related conjecture \cite[Conj.\ 2]{wip} bounds the 
\emph{neighborhood complexity}, or \emph{VC-density} of set systems defined by neighborhoods in graphs from a monadically dependent graph class, and is phrased as follows.
\begin{conjecture}[\cite{wip}]
    Let $\CC$ be a monadically dependent graph class and let $\eps>0$ be a real.
    Then for every graph $G\in\CC$ and set $A\subset V(G)$,
    we have that
    $$|\setof{N(v)\cap A}{v\in V(G)}|\le O_{\CC,\eps}(|A|^{1+\eps}).$$
\end{conjecture}
This conjecture is confirmed for all nowhere dense classes \cite{EickmeyerGKKPRS17, PilipczukST18a}, 
for all monadically stable classes \cite{wip},
and for all classes of bounded twin-width \cite{DBLP:journals/algorithmica/BonnetKRTW22, DBLP:conf/lics/Przybyszewski23}.
% has a pair of $d$-near-twins, for some $d\le \max_{H}
% (|E(H)|/|V(H)|)^k$, where $H$ ranges over all bipartite graphs such that $G$ contains some radius-$s$ encoding of $H$ as an induced subgraph.

\clearpage
\section{Preliminaries}\label{sec:prelims}

\newcommand{\tpred}{\mathrm{pred}}
\newcommand{\tsucc}{\mathrm{succ}}
\newcommand{\core}{\mathrm{core}}
\newcommand{\tail}{\mathrm{tail}}

\paragraph*{Sequences.}
To address and order the combinatorial objects of this paper,
we use \emph{indexing sequences}.
These are sequences (usually denoted by \(I,J\)) of elements without duplicates.
We denote the sequence \((1,\dots,n)\) also sometimes by \([n]\).
We write \(I \subseteq J\) if \(I\) is a subsequence of \(J\).
We use the usual comparison operators \(<,>\) to indicate the order of elements within a sequence.
Given a sequence $I$, and an element $i \in I$, we denote by $\tpred_I(i)$ and $\tsucc_I(i)$ the predecessor and successor of $i$ in $I$.
Moreover, if $I = (a_{1}, \ldots, a_{n})$, we define $\tail(I) := (a_{2}, \ldots, a_{n})$.

\paragraph*{Graphs.}
All graphs are simple and undirected. Unless a graph $G$ is considered to be an input to an algorithm, we do not need to assume that $G$ is a finite graph.

The \emph{length} of a path equals its number of edges.
The \emph{distance} between two vertex sets \(A\) and \(B\) in a graph \(G\), denoted by \(\dist_G(A,B)\),
is the length of a shortest path with endpoints in \(A\) and \(B\).
Two sets $A,B$ are \emph{non-adjacent} if \(\dist(A,B) > 1\).
The \emph{open} and \emph{closed \(r\)-neighborhoods} of a vertex \(v\) are denoted, respectively, by
\[
    N^G_r[u] = \{ v \in V(G) \mid \dist_G({\{u\},\{v\}}) \le r \},
    \quad\quad
    \quad\quad
    N^G_r(u) = N^G_r[u] \setminus \{v\}.
\]
The \emph{complement graph} of a graph \(G\) is denoted by \(\bar G\).
By default, graphs have no colors, but we speak of \emph{colored graphs} when we allow vertex-colors.
In this case we treat the colors as a fixed set of unary predicates which partition the vertex set, that is, each vertex has exactly one color.
We call \(G^+\) an \emph{\(s\)-coloring of \(G\)}, if \(G^+\) is obtained by coloring $G$ with $s$ many colors.

A \emph{graph class} is a set $\CC$ of graphs. 
(In Section~\ref{sec:small} we additionally assume that $V(G)\subset \N$ for all $G\in \CC$, and that 
$\CC$ is closed under isomorphism, but this assumption is not essential 
in other sections.)
A graph class $\CC$ is \emph{hereditary} if for every graph $G\in \CC$ and set $W\subset V(G)$ we have that  $G[W]\in\CC$, where $G[W]$ denotes the subgraph of $G$ induced by $W$.

\paragraph*{Flips.}
Fix a graph $G$ and a partition $\KK$ of its vertices.
We will think of $\KK$ as a coloring of the vertices of $G$. 
For every vertex $v \in V(G)$ we denote by $\KK(v)$ the unique color $X \in \KK$ satisfying $v \in X$.
Let $F \subseteq \KK^2$ be a symmetric relation.
The \emph{flip} $G \oplus_\KK F$ of $G$ is defined as the (undirected) graph with vertex set $V(G)$,
and edges defined by the following condition, for distinct $u,v\in V(G)$:
\[
    \set{u,v} \in E(G \oplus_\KK F) \Leftrightarrow \begin{cases}
        \set{u,v} \notin E(G) & \text{if } (\KK(u), \KK(v)) \in F,\\
        \set{u,v} \in E(G) & \text{otherwise.}
    \end{cases}
\]

We call $G \oplus_\KK F$ a \emph{$\KK$-flip} of $G$.
If $\KK$ has at most $k$ parts, we say that $G \oplus_\KK F$ a \emph{$k$-flip} of $G$.
A crucial property of flips is that they are reversible using first order-logic.
We can recover the edges of the original graph in a coloring of its flip as follows.
Let $H := G \oplus_\KK F$ and $H^+$ be the coloring of $H$ where each part of $\KK$ is assigned its own color.
Define the symmetric binary formula
\[
    \phi_{\KK,F}(x,y) := x \neq y \wedge \bigvee_{X,Y \in \KK} x \in X \wedge y \in Y \wedge \bigl( E(x,y) \texttt{ XOR } (X,Y) \in F\bigr).
\]
We now have $G \models E(u,v) 
    \Leftrightarrow
    H^+ \models \phi_{\KK,F}(u,v).$

\paragraph*{Logic.}
All formulas are over the signature of (possibly colored) graphs.
Let $\alpha(x;y_1,\ldots,y_k)$ be a formula, with free variables partitioned into $x$ and $y_1,\ldots,y_k$, as indicated by the colon.
Given a graph $G$, vertices $v_1,\ldots,v_k$, and a set $U\subset V(G)$,
we denote
$$\alpha(U;v_1,\ldots,v_k):=\setof{u\in U}{G\models \alpha(u;v_1,\ldots,v_k)}.$$

Let $G^+$ be a graph with colors $U_1,\ldots,U_l$.
The \emph{atomic type} of a tuple $\bar v=(v_1,\ldots,v_k)$ of vertices 
in $G^+$ 
is the quantifier-free formula $\alpha(x_1,\ldots,x_k)$ defined as the conjunction of all literals  $\beta(x_1,\ldots,x_k)$ 
(that is, formulas $x_i=x_j$, $E(x_i,x_j)$, $U_1(x_i), \ldots, U_l(x_i)$, or their negations)
such that 
 $$G^+ \models\beta(\bar v).$$
 We write $\atp_{G^+}(v_1,\ldots,v_k)$ to denote the atomic type of $\bar v$ in $G^+$. 

\paragraph*{Transductions and Monadic Dependence.}
In the Introduction we defined the notions 
that a graph class $\CC$ \emph{transduces} a graph class $\DD$,
and that a graph class $\CC$ is monadically dependent.
See \cref{sec:notions} for a more formal treatment.

\paragraph*{Computation Model.}
Our proofs are effective. In the algorithmic statements, we assume that 
the input graph is represented by its adjacency matrix. We assume the standard word RAM model of computation, with machine words of length $O(\log n)$, where $n$ is the size of the input. In this model, it is possible 
to store each vertex of the input graph in a single machine word,
and to determine the adjacency of two vertices in time $O(1)$.

\paragraph*{Asymptotic Notation.}
We introduce two new notations that simplify our statements and proofs.

\begin{quote}\itshape
    \(\const(p_1,\ldots,p_k)\) denotes a natural number, only depending on the parameters $p_1,\ldots,p_k$. 
\end{quote}
We allow parameters of any kind, in particular, graph classes.
Moreover, as an analogue of the $O$-notation,

\begin{quote}\itshape
    $U_{p_1,\ldots,p_k}(n)$ denotes an anonymous function that is non-negative and unbounded in $n$, and only depends on the parameters $p_1,\ldots,p_k$.
\end{quote}
More precisely, the $i$th occurrence of the notation $U_{p_1,\ldots,p_k}(n)$ in the text should be interpreted as
$f^i_{p_1,\ldots,pk}(n)$, for some fixed unbounded function $f^i_{p_1,\ldots,p_k}:\N\rightarrow\N$ that depends only on the parameters $p_1,\ldots,p_k$ and $i$.
To illustrate the use, we state Ramsey's theorem with this notation.

\begin{example}
    For every $k$-coloring of the edges of the complete graph on vertex set $[n]$, 
    there exists a set $X \subset [n]$ of size $U_k(n)$ which induces a monochromatic clique.
\end{example}

\paragraph*{Ramsey Theory.}
For $\ell\in \N$ and a set $I$, let $I^{(\ell)}$ denote the set of 
subsets  $J\subset I$ of size $\ell$.

\begin{lemma}[Ramsey's Theorem, \cite{ramsey}]\label{lem:ramsey}
    For every $k,\ell,n\in\N$ there exists 
    $N\in \N$ such that for every $$c\from [N]^{(\ell)}\to [k]$$ 
    there is some $I\in [N]^{(n)}$
    such that $c(J)=c(J')$ for all $J,J'\in I^{(\ell)}$.

    \effective{Moreover, there is a function $f\from \N^3\to\N$
    such that for every $k,\ell\in\N$
    the function $m\mapsto f(k,\ell,m)$ is monotone and unbounded,
    and there is an algorithm that, given 
    numbers $k,\ell,m\in \N$ and 
    a function $c\from [m]^{(\ell)}\to[k]$, 
    computes in time $O_{k,\ell}(m^\ell)$
    a set $I\subset [m]$ of size $f(k,\ell,m)$
    such that $c(J)=c(J')$ for all $J,J'\in I^{(\ell)}$.}
\end{lemma}

See, e.g., \cite[Thm. C]{ramsey} or 
\cite[Thm. 5.4]{gasarch2012ramsey-proofs}
for a (standard) proof of \cref{lem:ramsey}.
The ``moreover'' part of the statement above follows by tracing the construction,
which proceeds by induction on $\ell$, where in each stage of the construction we iterate over some subset of $[m]$.

For a pair $(a,b)$ of elements of a linearly ordered set $(A,\le)$,
let $\otp(a,b)\in\set{<,=,>}$ indicate whether $a<b$, $a=b$, or $a>b$ holds.
For $\ell\ge 1$ and an $\ell$-tuple of elements $a_1,\ldots,a_\ell$ of a linearly ordered set $(A,\le)$,
define the \emph{order type} of $(a_1,\ldots,a_\ell)$, denoted $\otp(a_1,\ldots,a_\ell)$ as the tuple $(\otp(a_i,a_j))_{1\le i<j\le \ell}$.

We now reformulate Ramsey's theorem 
using the $U$-notation, and also so that it talks about 
$\ell$-tuples, rather than $\ell$-element subsets.

\begin{lemma}[Reformulation of Ramsey's Theorem]
    \label{lem:reramsey}
    For every $k,\ell,n$ and coloring 
    $$c\from [n]^\ell\to [k]$$ there is 
    a subset $I\subset [n]$ of size $U_{k,\ell}(n)$
    such that
     $c(a_1,\ldots,a_\ell)$ depends only on 
    $\otp(a_1,\ldots,a_\ell)$,  for all $(a_1,\ldots,a_\ell)\in I^\ell$.

    \effective{Moreover, there is an algorithm that, given $c$, computes $I$ in time $O_{k,\ell}(n^\ell)$.}
\end{lemma}
The conclusion of the lemma means that there is a function $f$ 
such that $c(a_1,\ldots,a_\ell)=f(\otp(a_1,\ldots,a_\ell))$,
for all $(a_1,\ldots,a_\ell)\in I^\ell$.

\begin{lemma}[Bipartite Ramsey Theorem]\label{gridramsey}
    For every $k,\ell_1,\ell_2,n$
    and coloring 
    $$c\from [n]^{\ell_1}\times[n]^{\ell_2}\to [k]$$
    there are subsets $I_1,I_2\subset [n]$ of size $U_{k,\ell_1,\ell_2}(n)$
    such that 
    $c(\bar a,\bar b)$ depends only on 
    $\otp(\bar a)$ and $\otp(\bar b)$,
    for all $\bar a\in I_1^{\ell_1}$ and $\bar b\in I_2^{\ell_2}$.
\end{lemma}
The conclusion of the lemma means that there is a function $f$ 
    such that $c(\bar a,\bar b)=f(\otp(\bar a),\otp(\bar b))$,
    for all $\bar a\in I_1^{\ell_1}$ and $\bar b\in I_2^{\ell_2}$.

\begin{proof}
    Set $\ell=\ell_1+\ell_2$.
    The coloring $c$ can be viewed as a coloring  $c\from [n]^{\ell}\to [k]$.
     By Lemma~\ref{lem:reramsey}, there is a subset $I\subset [n]$
     of size $U_{k,\ell}(n)$
     such that the restriction of $c$ to $I$ is homogeneous,
     that is, 
     $c(\bar a)$ depends only on $\otp(\bar a)$ for $\bar a\in I^\ell$.
     We can assume that $|I|$ is even.
     Let $I_1$ be the first $|I|/2$ elements of $I\subset [n]$,
     and $I_2$ be the remaining $|I|/2$ elements of $I$.
     Then $|I_1|=|I_2|\ge U_{k,\ell}(n)$.
     Let $(\bar a,\bar b),(\bar a',\bar b')\in I_1^{\ell_1}\times I_2^{\ell_2}$ 
     be two pairs such that 
     $\otp(\bar a)=\otp(\bar a')$ and 
     $\otp(\bar b)=\otp(\bar b').$
     Then $\otp(\bar a,\bar b)=\otp(\bar a',\bar b')$,
     since 
$\otp(a,b)={<}$
 for all $a\in I_1$ and $b\in I_2$.
     Therefore, $c(\bar a,\bar b)=c(\bar a',\bar b')$,
     by homogeneity of $c$ restricted to $I$.
     \end{proof}

\newpage
\part{Structure}\label{part:structure}
\label{part1}
\newcommand{\spacedRightarrow}{\quad\Rightarrow\quad}

In this part we define three tameness conditions for graph classes:
\begin{itemize}
    \item \emph{Prepattern-freeness}: the absence of certain combinatorial obstructions called \emph{prepatterns}.
    \item \emph{Insulation property}:
    the ability to 
    guard any vertex set
    using so-called \emph{insulators}.
    \item \emph{Flip-breakability}: the ability to break any vertex set into two distant parts using few flips.
\end{itemize}
We show that for any graph class the following implications hold:
\[
    \text{prepattern-free}
    \spacedRightarrow
    \text{insulation-property}
    \spacedRightarrow
    \text{flip-breakable}
    \spacedRightarrow
    \text{mon.\ dependent}
\]
We later close the circle of implications in \cref{part:nonstructure}, where we show:
\[
    \text{not prepattern-free}
    \spacedRightarrow
    \text{large flipped crossings/comparability grids}
    \spacedRightarrow
    \text{mon.\ independent}
\]

\section{Grids and Insulators}

\paragraph{Grids.}
The most basic building block of this section is the notion of \emph{grids}.
\begin{definition}[Grids]\label{def:grids}
    Fix a non-empty sequence $I$ and an integer $h \ge 1$.
    A \emph{grid} $A$ \emph{indexed by} $I$ and of \emph{height} $h$ in a graph $G$ is a collection of pairwise disjoint sets $A[i,r] \subseteq V(G)$, for $i\in I$ and $r \in [h]$, called \emph{cells}.
    Each grid is either tagged as \emph{orderless} or \emph{ordered}.
\end{definition}
To facilitate notation, we often assume, up to renaming, the indexing sequence to be   $I= (1, \ldots, n)$.
We do so in the following.
For subsets \(J \subseteq I\) and \(R \subseteq [h]\), we write
\(A[J,R] = \bigcup_{i \in J, r \in R} A[i,r]\).
We often use implicitly defined sets via wildcards and comparisons. For example \(A[{\le}i,*]\) stands for \(A\bigl[\{1,\dots,i\},[h]\bigr]\).
In particular, we use $A[i,*]:=\bigcup_{r\in [h]} A[i,r]$ and $A[*,r]:=\bigcup_{i\in I}A[i,r]$,
and refer to those sets as to \emph{columns} and \emph{rows} of \(A\), respectively.
In slight abuse of notation, we often write \(A\) instead of \(A[*,*]\) to denote the set of all elements inside the grid.
We define the \emph{interior} of $A$ as $\mathrm{int}(A) := A \setminus (A[1,*] \cup A[n,*] \cup A[*,h])$.
Moreover, we say two columns $A[i,*]$ and $A[j,*]$  are \emph{close}, if $|i-j| \le 1$
and two rows $A[*,r]$ and $A[*,t]$ are \emph{close}, if $|r-t| \le 1$.
Two cells are \emph{close} if their respective columns and rows are close.

Note that unlike standard matrix notation, to highlight the hierarchical relationship between columns and rows,
our notation \(A[i,r]\) first mentions the column index \(i \in I\) and then the row index~\(r \in [h]\).

\paragraph{Insulators.}
The following notion of \emph{insulators} serves a twofold purpose.
On the one hand, insulators enforce the necessary structure to obtain \emph{flip-breakability} (see \Cref{sec:flip-breakability}).
On the other hand, in \Cref{part2}, we use insulators to build the patterns presented in~\Cref{thm:mainforbiddenpatterns}.

\newcommand{\explanation}{$\blacktriangleright$\ }

\begin{definition}[Insulators]\label{def:insulator}
    An \emph{insulator} $\gc{A} = (A, \KK, F, R)$ \emph{indexed by a sequence $I$} of \emph{height} $h$ and \emph{cost} \(k\) in a graph $G$ consists of
    \begin{itemize}
        \item a grid $A$ indexed by $I$ and of height $h$,
        \item a partition $\KK$ of $V(G)$ into at most \(k\) color classes,
        \item a symmetric relation $F\subseteq \KK^2$ specifying a flip $G' := G \oplus_\KK F$,
        \item a relation $R \subseteq \KK^2$.
    \end{itemize}
    We furthermore say that $\gc A$ is \emph{orderless} (\emph{ordered}), if $A$ is \emph{orderless} (\emph{ordered}).

    \medskip\noindent
    If $\gc A$ is \emph{orderless}, we demand:
    
    \begin{enumerate}[leftmargin= 4em, label={(U.$\arabic*$)}]
    \item \label{itm:orderless} For all $i \in I$ there exists $a_i \in V(G)$ such that
    \[
        A[i,1] = \{a_i\}  \quad \text{and} \quad A[i,{\leq} r] = N^{G'}_{r-1}[a_i] \text{ for all } r \in [h].
    \]

    \explanation In particular, a column of an orderless insulator is just the radius $h-1$ ball in a $k$-flip around single vertex sitting in the bottom cell of that column.

    \end{enumerate}

    \medskip\noindent
    If $\gc A$ is \emph{ordered}, we demand:
    \begin{enumerate}[leftmargin= 4em, label={(O.$\arabic*$)}] 
        \item \label{itm:consistent-rows} If two vertices are in different rows of $A$, then they have different colors in $\KK$. 

        \smallskip
        \explanation This technical property ensures that the rows of the insulator are sufficiently distinguishable. We will use this to argue that certain modifications on the insulator can be performed without increasing its cost.

        \item\label{itm:rootedness} Every vertex $v \in A[i,r]$ with $r>1$, \(i \in I\) has a neighbor in the cell $A[i,r-1]$ in $G'$. 

        \smallskip
        \explanation The mandatory downward edge, together with \ref{itm:adj-bot-left} and \ref{itm:ordered}, ensures that each column is cohesive: we will later observe that the columns of the insulator are first-order definable.
        This property is crucial to obtain hardness results.

        \item\label{itm:outside} 
            For every $v \notin A[*,*]$ and $X \in \KK$ we require that $v$ is \emph{homogeneous} to $X \cap \mathrm{int}(A)$ in $G$ \\
            (that is, either all or no vertices in $X \cap \mathrm{int}(A)$ are adjacent to \(v\)).

            \smallskip
            \explanation The inside of the insulator is ``insulated'' from its outside:
            the adjacencies between the two are described using only colors.

        \item\label{itm:adjacency} 
            For every $u \in A[i,r]$ with $r < h$, \(i \in I\) and $v \in A$ we have the following: \\
            (Up to renaming, we assume $I = (1,\ldots,n)$.)
            \begin{enumerate}[leftmargin= 5em, label={(O.4$.\arabic*$)}]
                \item \label{itm:adj-different-rows} If $u$ and $v$ are in rows that are not close and $u\in\mathrm{int}(A)$, \\
                    then they are non-adjacent in $G'$.

                \item \label{itm:adj-bot-left} If $v \in A[{<}i,r-1] \cup A[{>}i,r+1]$, then $u$ and $v$ are non-adjacent in $G'$.
                \item \label{itm:adj-left} If $v \in A[{>}i+1, \{r,r-1\}]$, then $G \models E(u,v) \Leftrightarrow (\KK(u),\KK(v))\in R$.
                \item \label{itm:adj-right} If $v \in A[{<}i-1, \{r,r+1\}]$, then $G \models E(u,v) \Leftrightarrow (\KK(v),\KK(u))\in R$.
            \end{enumerate}
            Otherwise, we make no claims regarding the adjacency of $u$ and $v$.

            \smallskip
            \explanation 
            Properties \ref{itm:adj-different-rows}, \ref{itm:adj-left}, and \ref{itm:adj-right} provide vertical and horizontal insulation inside the insulator.
            Property \ref{itm:adj-bot-left} helps to keep each column cohesive.
            See \Cref{fig:adjacency} for an illustration.
            \begin{figure}[h]
                \begin{center}
                    \includegraphics[scale = 1]{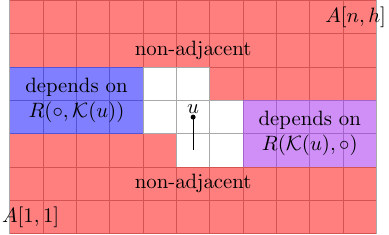}%
                \end{center}%
                \vspace{-0.5cm}
                \caption{Illustration of how \ref{itm:adjacency} controls the adjacency of a vertex \(u\) within the insulator.}\label{fig:adjacency}
            \end{figure}

        \item \label{itm:ordered}  There exists a bound $r < h$ and a $k$-flip $H$ of $G$
        such that for every two distinct vertices $u,v \in A[*,1]$.
        \[
            N_r^{H}[u] \cap N_r^{H}[v] = \varnothing. 
        \]
        Moreover, there are vertices   
        $\{b(v) \in N_r^{H}[v] : v\in A[*,1]\}$ and 
        $\{c_i \in V(G): i\in I\}$
        and a symbol ${\sim} \in \{{\leq},{\geq}\}$
        such that for every $i,j \in I$, $v \in A[j,1]$ and 
        \[
            G \models E(c_i, b(v))
            \quad
            \text{if and only if}
            \quad 
            i \sim j.
        \]
        \explanation This property orders the columns. It will later be used to first-order define the columns as intervals in the order.
\end{enumerate}
\end{definition}

\begin{observation}\label{obs:orderless-just-as-good}
    Every orderless insulator $\gc A = (A, \KK, F,R)$ where $F = R$ also satisfies the properties
    \ref{itm:rootedness}, \ref{itm:outside}, and \ref{itm:adjacency} of an ordered insulator. 
    Since the requirements of an orderless insulator put no restrictions on $R$, in the orderless case we can always assume $F=R$.
\end{observation}

Two example insulators are depicted in \cref{fig:insulator-examples}.

\begin{figure}[htbp]
    \centering
    \includegraphics[scale=1.2]{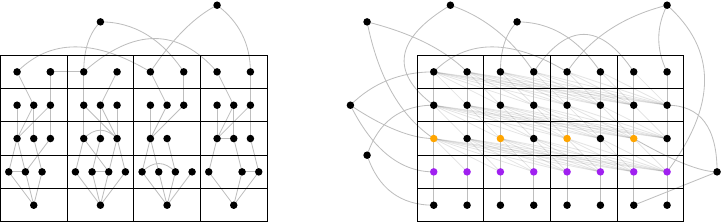}
    \caption{On the left: an orderless insulator. Each column is just the ball (in a flip) around the single vertex of its bottom cell. Apart from the boundary, there are no connections in between columns. 
    On the right: an ordered insulator. It contains connections between columns, but they are well controlled by property \ref{itm:adjacency}. Property \ref{itm:ordered} is witnessed by the highlighted vertices.
    The purple $b(\cdot)$ vertices are contained in disjoint radius $r=1$ balls around the vertices of the bottom cell.
    They are preordered by the yellow $c(\cdot)$ vertices.
    }
    \label{fig:insulator-examples}
\end{figure}

We will often identify an insulator \(\gc A\) and its underlying grid \(A\),
and write, for example \(v \in \gc A\) to indicate that \(v \in A[*,*]\).
We start by observing basic properties of insulators.
The following property will be crucial to obtain the model checking hardness result.

\begin{lemma}\label{lem:instructive-columns}
    The columns of an insulator are definable in first-order logic in a coloring of $G$.

    \smallskip\noindent
    More precisely:
    Let $\gc A$ be an insulator with grid $A$ of cost $k$ and height $h$ indexed by $I$ in a graph~$G$.
    There exists a formula $\alpha(x,y)$ (depending only on \(k\) and \(h\)), a $\const(k,h)$-coloring $G^+$ of $G$, and vertices $\{ a_i : i \in I \}$ such that for each $i \in I$
    \[
        A[i,*] = \{ v \in V(G) : G^+ \models \alpha(v, a_i)  \}.
    \]
\end{lemma}

We do not rely directly on this lemma, as we will later further process insulators before proving hardness. However, we sketch a proof for instructive purposes.
\begin{proof}[Proof sketch for \cref{lem:instructive-columns}] 
    All the formulas we write will work in a $\const(k,h)$-coloring $G^+$ of $G$, but we omit the details of the coloring to streamline the presentation.
    If $\gc A$ is orderless we use property \ref{itm:orderless}. 
    The $a_i$ vertices are the singleton vertices in the bottom row of $\gc A$ and the formula $\alpha(x,y)$ is defined as 
    \begin{center}
        ``$x$ and $y$ are at distance at most $h-1$ in the flip~$G'$ of \(G\)''.    
    \end{center}
    If $\gc A$ is ordered, we use \ref{itm:ordered}.
    There is a formula $\beta(x,y)$ that defines $b(v)$ for each $v \in A[i,*]$: 
    \begin{center}
        ``$x$ is a $b$-vertex and contained in the $r$-ball around $y$ in the flip $H$ of \(G\)''.
    \end{center}
    Building on $\beta$, there is a formula $\gamma(x,y)$ which, given $v \in A[j,*]$, defines $\{ c_i : R(i,j), i \in I\}$:
    \begin{center}
        ``$x$ is a $c$-vertex and adjacent to $b(y)$ in $G$''.
    \end{center}
    As $R \in \{\leq,\geq\}$, we can now define a preorder \(\prec\) on the vertices of the bottom row $A[*,1]$ which respects the column order by comparing their $\gamma$-neighborhoods:
    \[
        x \prec y := \exists z : \gamma(z,x) \wedge \neg \gamma(z,y)
        \qquad
        \text{or}
        \qquad
        x \prec y := \exists z : \neg\gamma(z,x) \wedge \gamma(z,y)
    \]
    depending on the choice of $R$.
    We choose
    an arbitrary vertex $a_i$ of each cell $A[i,1]$ of the bottom row.
    Using the preorder, we can write a formula $\alpha_1(x,y)$ that defines $A[i,1]$ from $a_i$ as its equivalence class in the preorder.
    For $1 < r \leq h$, we can now inductively write a formula $\alpha_r(x,y)$ that defines $A[i,{\leq}r]$ from $a_i$ as all the vertices that are already in $A[i,{<}r]$, or adjacent to $A[i,{<}r]$ but not to $A[i - 1,{<}r]$ in $G'$. Here we again use the preorder to define $A[i - 1,{<}r]$ from $a_i$.
    The correctness for vertices inside $\gc A$ follows from the properties \ref{itm:rootedness} and \ref{itm:adj-bot-left}.
    Vertices from outside $\gc A$ can be marked in the coloring and ignored.
    Setting $\alpha := \alpha_h$ finishes the proof sketch.
\end{proof}

\begin{definition}[Subgrids and subinsulators]\label{def:subgrids}
    Let $A$ be a grid indexed by a sequence $J$ and of height $h$ in a graph~$G$.
    For every subsequence $I \subseteq J$ of length at least two we define the \emph{subgrid} $A\vert_I$ as the grid indexed by $\tail(I)$ and of height $h$, 
    containing the following cells.
    For all $i \in \tail(I)$ and $r \in [h]$,
    depending on whether $A$ is \emph{orderless} or \emph{ordered}, we respectively set
    \[
        A\vert_I[i,r] := A[i,r]
        \quad
        \text{or}
        \quad
        A\vert_I[i,r] := \bigcup\{A[m,r] : m \in I \text{ and } \tpred_I(i) < m \leq i\}.
    \]
    $A\vert_I$ is ordered (orderless) if and only if \(A\) is ordered (orderless). See \cref{fig:subgrid} for a depiction.
    For every insulator \((A\vert_I,\KK,F,R)\) indexed by \(J\), we moreover define the \emph{subinsulator} $\gc{A}\vert_I := (A\vert_I,\KK,F,R)$.
    The upcoming \Cref{lem:subinsulator} will prove the validity of this definition.
\end{definition}

\begin{figure}[htbp]
    \centering
    \includegraphics[]{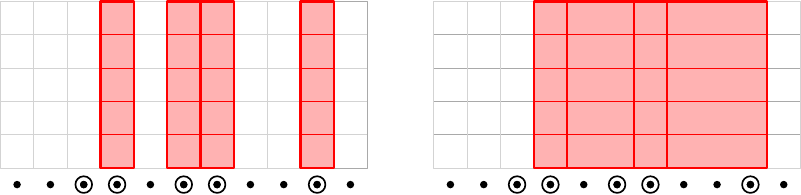}
    \caption{On the left/right: a subgrid of an orderless/ordered grid. The original grid $A$ is depicted in gray. 
    The dots at the bottom represent the sequence $J$ indexing $A$. 
    The subsequence $I \subseteq J$ is marked with circles.
    The subgrid $A\vert_I$ is overlaid in red.
    It is indexed by $\tail(I)$.}
    \label{fig:subgrid}
\end{figure}

For an ordered insulator $\gc A$, in the definition of a subinsulator $\gc A\vert_I$, it is necessary that the indexing sequence $\tail(I)$ of $\gc A\vert_I$ excludes the first element of $I$: each $i \in I$ represents the interval $(\tpred_I(i), i]$ which is undefined for the first element of $I$.
In orderless insulators this problem does not arise, but we choose to also exclude the first element of $I$ to allow for uniform proofs which do not distinguish between the two. 
The following observation about subgrids is crucial to their definition and will later be used to build well-behaved subgrids using Ramsey-arguments.

\begin{lemma}\label{lem:subgrid-ramseyness}
    Let $A$ be a grid indexed by $J$ and of height $h$ in a graph $G$.
    For every subsequence $I\subset J$ of length at least two, in the subgrid $A\vert_I$, the content of the column $A\vert_I[i,*]$ depends only on $A$, $i$ and $\tpred_I(i)$, instead of the whole sequence~$I$.
    More precisely, there exists a function $M_A : J \times J \rightarrow 2^{V(G)}$ such that for every $I \subseteq J$ and $i\in \tail(I)$ we have
    \[
        A\vert_I[i,*] = M_A(\tpred_I(i),i).    
    \]
\end{lemma}

\begin{proof}
    We define the function $\mu_A : J \times J \times [h] \rightarrow 2^{V(G)}$ as follows.
    For all $i<j \in J$ and every $r\in [h]$, depending on whether $A$ is \emph{orderless} or \emph{ordered}, we respectively set
    \[
        \mu_A(i,j,r) := A[j,r] 
        \quad
        \text{or}
        \quad
        \mu_A(i,j,r) := \bigcup\{A[m,r] : m \in I \text{ and } i < m \leq j\}.
    \]
    We can now define for all $i<j \in J$
    \[
        M_A(i,j) := \bigcup_{r \in [h]} \mu_A(i,j,r). \qedhere
    \]
\end{proof}

\begin{observation}[Transitivity]\label{obs:transitivity}

    Let $A$ be a grid indexed by $I_0$, $I_1$ be a subsequence of $I_0$, and $I_2$ be a subsequence of $\tail(I_1)$. 
    If \(B = A\vert_{I_1}\)  and \(C = B\vert_{I_2}\)
    then \(C = A\vert_{I_2}\).
\end{observation}

\begin{observation}[Monotonicity and Coverability]\label{obs:mon-and-cov}
    Let $A$ be a grid indexed by $J$ of height $h$ and $I$ be a subsequence of $J$.
    For all $i \in \tail(I)$ and $r \in [h]$ we have 
    \[
        A[i,r] ~
        \subseteq  ~
        A\vert_I[i,r] ~
        \subseteq ~
        \bigcup\{A[m,r] : m \in I \text{ and } \tpred_I(i) < m \leq i\}
        .
    \]
\end{observation}

We finally show that taking a subinsulator preserves its good properties without increasing its~cost.

\begin{lemma}\label{lem:subinsulator}
    Let $\gc{A} = (A,\KK,F,R)$ be an insulator indexed by $J$ on a graph $G$. For every subsequence $I \subseteq J$ we have that 
    $\gc{A}\vert_I := (A\vert_I,\KK,F,R)$ is also an insulator on $G$.
\end{lemma}
\begin{proof}
    Up to renaming, we assume $J = (1,\ldots,n)$.
    Let $B:={A}\vert_{I}$ and $G' := G \oplus_{\KK} F$.
    %As we leave $\KK$ untouched, it still has $k$ parts.
    In the orderless case, as the graph $G'$ remains the same and $B$ is obtained from $A$ by just dropping columns, it is easy to see that \ref{itm:orderless} still holds.
    It remains to check the ordered case.
    \begin{itemize}

        \item To prove \ref{itm:consistent-rows} and \ref{itm:adjacency}, we observe that for every $i \in I$ and $r \in [h]$
        \[
            B[*,r] \subseteq A[*,r] \quad \text{and} \quad B[<i,r] \subseteq A[<i,r] \quad \text{and} \quad B[>i,r] \subseteq A[>i,r].
        \]
        Property \ref{itm:consistent-rows} follows directly.
        To prove, for example, \ref{itm:adj-right}, assume \(u \in B[i,r]\) and $v \in B[<i-1, \{r,r+1\}]$.
        Then also \(u \in B[>\tpred_I(i),r]\) and $v \in B[<\tpred_I(i), \{r,r+1\}]$.
        As argued above, \(u \in A[> \tpred_{I}(i),r]\) and $v \in A[<\tpred_{I}(i), \{r,r+1\}]$.
        By property \ref{itm:adj-right} of \(\gc A\),
        we have $G \models E(u,v) \Leftrightarrow (\KK(v),\KK(u))\in R$, as desired.
        The remaining statements of \ref{itm:adjacency} follow similarly.

        \item To prove \ref{itm:rootedness}, let $u \in B[i,r]$ for $r > 1$ and let us show that $u$ has a neighbor in $B[i,r-1]$ in $G$.
        By construction, we have $u \in A[i',r]$ for some $\tpred_I(i) < i' \leq i$. By \ref{itm:rootedness} of $\gc A$, $u$ has a neighbor $v$ in $A[i',r-1]$ in $G'$.
        Again by construction, $A[i',r-1] \subseteq B[i,r-1]$, so also $v \in B[i,r-1]$.

        \item 
        For \ref{itm:outside} to hold we must check,
        for every $u \notin B$ and $X \in \KK$, that $u$ is homogeneous to $X_B := X \cap \mathrm{int}(B)$ in $G$.
        As \(G'\) is a \(\KK\)-flip of \(G\), we can check the property in $G'$ instead.
        If $u \notin A$ this holds as $X_B \subseteq \mathrm{int}(A)$ and \ref{itm:outside} was already true in $\gc A$.
        
        Assume now $u \in A[i,r] \setminus B$ for some $i\in J$ and $r\in[h]$.
        As we already established \ref{itm:consistent-rows}, we know that all vertices from $X_B$ are in the same row $B[*,r'] \subseteq A[*,r']$ for some $r' \in [h]$.
        If $|r - r'| > 1$ then $u$ and $X_B$ are non-adjacent in $G'$ as $\gc A$ satisfies \ref{itm:adjacency}.
        Now assume $|r - r'| \leq 1$.
        Let $i_0,i_1, \ldots, i_n$ be the continuous subsequence of $J$ where $i_0$ and $i_n$ are the first and last elements of $I$.
        By construction, we have
        \[
            B[*,*] = A[\{i_1,\ldots,i_n\},*]
            \quad
            \text{ and }
            \quad
            \mathrm{int}(B) = 
            A[\{i_2,\ldots,i_{n-1}\}, {<}h].
        \]
        Since $u \notin B$ we have $i < i_1$ or $i_n < i$.
        Assume $i < i_1$. 
        Then we have $X_B \subseteq A[{\leq} i_1,r'] \subseteq A[{>}i,r']$, and we can again use \ref{itm:adjacency} of $\gc A$.
        Now if $r' = r + 1$ then $u$ is non-adjacent to $X_B$ in $G'$. If $r' \in \{r,r-1\}$ then $u$ is adjacent to all of $X_B$ if $(\KK(u),X) \in R$ and non-adjacent to all of $X_B$ otherwise.
        In each case $u$ is homogeneous to $X_B$.
        The case where $i_n < j$ follows by a symmetric argument.

        \item
        The property \ref{itm:ordered} of $\gc A$ is witnessed by
        \begin{itemize}
            \item a symbol ${\sim} \in \{ \leq, \geq \}$, 
            \item a $k$-flip $H$, and 
            \item vertices $\{b(v) \in V(G): v\in A[*,1]\}$ and 
            $\{c_i \in V(G): i\in J\}$.
        \end{itemize}
        To witness \ref{itm:ordered} of $\gc B$, we use $\sim$, $H$, $\{b(v) \in V(G): v\in B[*,1]\}$, and
        \begin{itemize}
            \item 
            $\{c_{i'} : i'= \tsucc_J(\tpred_I(i)), i\in \tail(I)\}$ if $(\sim)=(\leq)$, or
            \item
            $\{c_i : i\in \tail(I)\}$ if $(\sim)= (\geq)$.
            \qedhere
        \end{itemize}
    \end{itemize}
\end{proof}

\section{Prepatterns}

Throughout the following sections, we will either make progress constructing large insulators,
or will obtain the following kind of preliminary patterns, which are then processed further in \Cref{part2}.

\begin{definition}\label{def:bi-pattern}
    Let $\gc A$ be an insulator indexed by a sequence $K$ with grid $A$ in a graph $G$.
    Say that  $G$ contains a \emph{bi-prepattern} of order $t$ on $\gc A$ if there exist
    \begin{itemize}
        \item index sequences $I,J\subset K$ of size $t$ and with $|I|=|J|=t$,
        \item vertices $c_{i,j} \in V(G)$ for all $i\in I,j \in J$,
        \item quantifier-free formulas $\alpha_1(x;y,s_1)$ and  $\alpha_2(x;y,s_2)$ with parameters $s_1, s_2 \in V(G)$,
        \item symbols $\sim_1,\sim_2\in\set{=,\neq}$,
    \end{itemize}
        such that for all $i\in I,j\in J$,
        \begin{align*}
            i&= \min \big\{ i' \in I : \alpha_1(A[i',*];c_{i,j},s_1)\sim_1\emptyset \big\},\\
            j&= \min \big\{ j' \in J : \alpha_2(A[j',*];c_{i,j},s_2)\sim_2\emptyset \big\}.
        \end{align*}
\end{definition}
Let us give some intuition for the above definition.
The bi-prepattern consists of two sequences of columns indexed by $I$ and $J$. 
For each pair of columns $(i,j) \in I \times J$, there exists a private vertex $c_{i,j}$ pairing them up in the following sense. Column $i$ is the first column in $I$, in which
\begin{itemize}
    \item $c_{i,j}$ has no $\alpha_1$-neighbor (if ${\sim_1}$  is $=$), or
    \item $c_{i,j}$ has an $\alpha_1$-neighbor (if ${\sim_1}$ is $\neq$).
\end{itemize}
Similarly, $j$ is the first column in $J$ in which has an $\alpha_2$-neighbor (resp.\ no $\alpha_2$-neighbor).
\cref{fig:bi-prepattern} (\emph{left}) illustrates this column pairing.

\begin{figure}[htbp]
    \centering
    \includegraphics[scale=1.0]{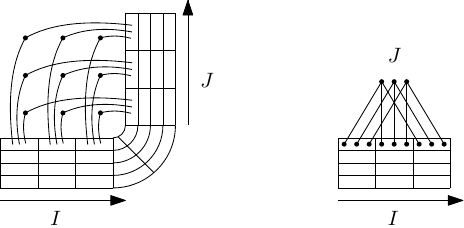}
    \caption{Schematic depiction of a bi-prepattern (left) and mono-prepattern (right).}
    \label{fig:bi-prepattern}
\end{figure}

This pairing aspect will be used later to argue that large bi-prepatterns witness monadic independence (or equivalently: they are obstructions for monadic dependence). 
By \cref{lem:instructive-columns}, each column of the insulator can be first-order defined from a single representative vertex.
Bi-prepatterns therefore logically resemble $1$-subdivided bicliques, where the column representatives are the principle vertices and the vertices  $c_{i,j}$ form the subdivision vertices.
Subdivided bicliques of unbounded size are witnesses for monadic independence: they encode arbitrary bipartite graphs by coloring the subdivision vertices.
Therefore, bi-prepatterns may be thought of as obstructions to monadic dependence.
Moreover, by the strong structure properties of insulators and since $\alpha_1$ and $\alpha_2$ are quantifier-free, we can use bi-prepatterns to later extract from them our concrete forbidden induced subgraph characterization.

In addition to the bi-prepatterns, our analysis produces a second kind of obstruction called a \emph{mono-prepattern}.

\begin{definition}\label{def:mono-pattern}
    Let $\gc A$ be an insulator indexed by a sequence \(K\) with grid $A$ in a graph $G$.
    Say that $G$ contains a \emph{mono-prepattern} of order $t$ on $\gc A$ if there exist
    \begin{itemize}
        \item index sequences $I,J$ with \(I \subseteq K\) and $|I|=|J|=t$, 
        \item vertices $c_j \in V(G)$ for all $j \in J$,
        \item vertices $b_{i,j} \in A[i,*]$ for all $i\in I,j \in J$,
        \item a symbol ${\sim} \in \{=, \neq, \leq, <, \geq, >\}$,
    \end{itemize}
    such that
    for all $i\in I$ and $j,j'\in J$,
        \[
            G \models E(c_j,b_{i,j'}) \Leftrightarrow j \sim j'.
        \]
\end{definition}

Similarly as in a bi-prepattern, in a mono-prepattern the $c_j$ vertices can be used to pair up columns $(i,j) \in I\times I$.
While the bi-prepattern logically resembles a subdivided biclique (we pair elements from two sequences $I$ and $J$), the mono-prepattern corresponds to a subdivided clique (we pair elements from the same sequence $I$).
\cref{fig:bi-prepattern} (\emph{right}) illustrates mono-prepatterns.

\begin{definition}\label{def:prepattern}
    $G$ contains a \emph{prepattern} of order $t$ on an insulator $\gc A$ if it either contains a bi-prepattern of order $t$ on $\gc A$ or a mono-prepattern of order $t$ on $\gc A$.
\end{definition}

\begin{definition}\label{def:prepatternfree}
    A class of graphs $\CC$ is \emph{prepattern-free}, if for every $k,r\in \N$, there exists $t\in \N$ 
    such that every graph $G \in \mathcal{C}$ does not contain prepatterns of order $t$ on insulators of cost at most \(k\) and height at most \(r\) in \(G\).
\end{definition}

We later show in \Cref{thm:main-circle} that these  obstructions are exhaustive: prepattern-freeness is equivalent to monadic dependence.
In the following sections, we 
start by deriving structural properties of prepattern-free classes.

\section{The Insulation Property}\label{sec:insulation-property}
Towards proving that prepattern-free classes are flip-breakable, we first show that they possess a more fine-grained structure property which we call the \emph{insulation property}.

\begin{definition}
    Let $\gc A$ be an insulator with grid $A$ indexed by a sequence $I$ in a graph $G$ and let $W \subseteq V(G)$.
    We say that $\gc A$ \emph{insulates} $W$
    if there is a bijection $f: W \rightarrow I$, such that for all $v \in W$
    \[
        v \in A[f(v),1].
    \]
    A set $W$ is \emph{$(r,k)$-insulated} in $G$ if there is an insulator $\gc A$ of height $r$ and cost $k$ that insulates $W$.
\end{definition}

\begin{definition}\label{def:insulationProperty}
    A class of graphs $\CC$ has the \emph{insulation property}, if for every radius $r\in \N$ there exist a function $N_r:\N\rightarrow\N$ and a constant $k_r \in \N$ such that for every $m\in\N$, $G\in \CC$, $W\subseteq V(G)$ with $|W| \geq N_r(m)$, there is a subset $W_\star\subseteq W$ of size at least $m$ that is $(r,k_r)$-insulated in $G$.
\end{definition}
More briefly: $\CC$ has the insulation property if for every $r\in\N$,
$G\in\CC$ and $W\subset V(G)$ there is a subset $W_\star\subset W$ of size $U_{\CC,r}(|W|)$ that is $(r,\const(\CC,r))$-insulated in $G$.

\medskip
The goal of this section is to prove the following.

\begin{proposition}\label{prop:prepatternImpliesInsulationNoAlg}
    Every prepattern-free graph class has the insulation property.
\end{proposition}

We first notice that insulation for radius $r = 1$ is trivial.

\begin{lemma}\label{lem:insulate-base-case}
    Fix a graph $G$. Every set $W \subseteq V(G)$ is $(1,1)$-insulated in $G$.
\end{lemma}
\begin{proof}
    Fix any enumeration $(a_{1}, \ldots, a_{n})$ of $W$.
    Let $A$ be the orderless grid indexed by the sequence $I := (1, \ldots, n)$ and of height~$1$ such that $A[i,1] := \{ a_i \}$ for all $i \in I$.
    Now $\gc A := (A,\{V(G)\},\varnothing,\varnothing)$ is an orderless insulator with cost and height~$1$ that insulates $W$.
\end{proof}

In order to insulate sets with higher radii, we will grow insulators in height.

\begin{definition}[Row-Extensions]
    Let $A$ be a grid indexed by $I$ and of height $h$.
    We say a grid $B$ is a \emph{row-extension} of $A$, if it satisfies the following properties.
    \begin{itemize}
        \item $B$ is indexed by $I$ and has height $h + 1$.
        \item For all $i\in I$ and $j \in [h]$ we have $B[i,j] = A[i,j]$.
        \item Either $A$ and $B$ are both \emph{orderless} or both \emph{ordered}.
    \end{itemize}
    Similarly, we say an insulator $\gc B$ is a \emph{row-extension} of an insulator $\gc A$, if the grid of $\gc B$ is a row-extension of the grid of $\gc A$.
\end{definition}

Depending on whether the insulator at hand is orderless or ordered, we will create row-extensions using one of the following two insulator growing lemmas.
To keep the presentation streamlined, 
the (quite technical) proofs of the two lemmas are deferred to \cref{sec:growing}.

\begin{restatable}[Orderless Insulator Growing]{lemma}{lemOrderlessGrowing}
    \label{thm:orderless-grid-growing}
    Fix \(k,t \in \N\).
    For every graph \(G\) and orderless insulator~\(\mathcal{A}\) indexed by \(J\) of cost \(k\) in $G$,
    we can compute a subsequence \(I \subseteq J\) of length \(U_{t}(|J|)\) 
    such that either
    \begin{itemize}
        \item $G$ contains a prepattern of order $t$ on $\gc A\vert_I$
        \item $\gc A\vert_{I}$ is orderable, or
        \item there exists a row-extension of $\gc A\vert_I$ of cost $\const(k,t)$ in \(G\).
    \end{itemize}

    \effective{Moreover, there is an algorithm that, given $G$ and $\gc A$, computes the sequence $I$ and one of the three outcomes (a prepattern, a witnesses for $\gc A\vert_I$ being orderable, or a row-extension) in time $O_{k,t}(|V(G)|^2)$.
    } 
\end{restatable}

The definition of an \emph{orderable} insulator will be given shortly after the following lemma.

\begin{restatable}[Ordered Insulator Growing]{lemma}{lemOrderedGrowing}
    \label{lem:order-growing}
    Fix \(k,t \in \N\).
    For every graph \(G\) and ordered insulator \(\mathcal{A}\) with cost \(k\), indexed by \(J\) in $G$, we can compute a subsequence \(I \subseteq J\)
    of length \(U_{t}(|J|)\) 
    such that either
    \begin{itemize}
        \item $G$ contains a prepattern of order $t$ on $\gc A\vert_I$, or
        \item $G$ contains a row-extension of $\gc A\vert_I$ with cost $\const(k,t)$.
    \end{itemize}

    \effective{Moreover, there is an algorithm that, given $G$ and $\gc A$, computes the sequence $I$ and one of the two outcomes (a prepattern or a row-extension) in time $O_{k,t}(|V(G)|^2)$.
    } 
\end{restatable}

Both insulator growing lemmas follow the same scheme. 
Given an insulator we find a large subinsulator that either witnesses a lot of non-structure (in form of a prepattern), or improves the structural guarantees of the original insulator (in form of a row-extension).
In the orderless case a third outcome may appear: the subinsulator may be \emph{orderable}.
Orderable insulators are orderless insulators which can be converted into ordered ones, as made precise by the following definition and lemma.

\begin{definition}[Orderable Insulators]\label{def:orderable}
    Let $\gc A$ be an orderless insulator with grid $A$ indexed by $I$ in a graph $G$. 
    We say that $\gc A$ is \emph{orderable} if 
    there exist vertices $\{b_i \in A[i,*] : i \in I\}$ and $\{c_i \in V(G) : i \in I\}$ and a symbol ${\sim} \in \{\leq,{\ge}\}$ such that for all $i,j \in I$
    \[
        G \models E(c_i,b_j)    
        \quad
        \text{if and only if}
        \quad
        i \sim j.
    \]
\end{definition}

\begin{lemma}\label{lem:orderless-to-ordered-insulator}
    Let $\gc A$ be an orderless insulator of cost $k$ and height $h$ in a graph $G$ that insulates a set $W \subseteq V(G)$.
    If $\gc A$ is orderable, then there exists an ordered insulator $\gc B$ of cost $k \cdot h$ and height~$h$ that also insulates $W$.

    \effective{Moreover, there is an algorithm that, given $G$ and $\gc A$, computes $\gc B$ in time $O_{k,h}(|V(G)|)$.}
\end{lemma}

\begin{proof}
    Let $\gc A = (A, \KK, F, F)$ be as in the statement.
    Let $B$ be the grid obtained by changing the tag of the grid $A$ from \emph{orderless} to \emph{ordered}.
    Towards ensuring \ref{itm:consistent-rows},
    let $\KK_\star$ be the size $k\cdot h$ coloring obtained by taking the common refinement of the $k$-coloring $\KK$ and an $h$-coloring that assigns vertices from different rows in $A$ different colors. 
    Let $F_\star$ be the corresponding refinement of $F$ such that $G \oplus_\KK F = G \oplus_{\KK_\star} F_\star$.
    We check that $\gc B := (B, \KK_\star, F_\star, F_\star)$ is the desired ordered insulator.
    The insulator property \ref{itm:consistent-rows} holds by construction.
    By \cref{obs:orderless-just-as-good}, \ref{itm:rootedness}, \ref{itm:outside}, and \ref{itm:adjacency} carry over from~$\gc A$. 
    Finally, \ref{itm:orderless} ensures
    for the \(k\)-flip \(H := G \oplus_\KK F\) and $v \in A[i,1]$
    that \(N_{h-1}^{H}[v] = A[i,*]\).
    As $\gc A$ moreover is orderable (cf.\ \cref{def:orderable}), property \ref{itm:ordered} follows.
    The bound on the running time is trivial.
\end{proof}

As we know how to grow orderless insulators,
turn orderless, orderable insulators into ordered ones, and
grow ordered insulators,
we can now grow arbitrarily high insulators. 
This yields a proof of \cref{prop:prepatternImpliesInsulationNoAlg}, which we restate below with an added algorithmic conclusion. 

\begin{proposition}\label{prop:prepatternImpliesInsulation}
    Every prepattern-free graph class $\CC$ has the insulation property.
    
    \effective{Moreover, there is an algorithm that, given a radius $r$, a graph $G\in \CC$ and a set $W$, computes the subset $W_\star \subseteq W$ and a witnessing insulator in time $O_{\CC,r}(|V(G)|^2)$.}
\end{proposition}

\begin{proof}
    Fix a prepattern-free class $\CC$
    and $r\ge 1$.
    We prove that for every $G\in \CC$ and $W\subset V(G)$ 
    there is a subset $W_r\subset W$  of size $U_{\CC,r}(|W|)$ that is $(r,k_r)$-insulated in $G$, for some $k_r\le \const(\CC,r)$.
This statement is proved by induction on $r\ge 1$.
The base case of $r=1$ follows immediately from Lemma~\ref{lem:insulate-base-case}.

In the inductive step, assume the statement holds for some $r\ge 1$;
we prove it for $r+1$.
Let $k_r\le \const(\CC,r)$ be the value obtained by inductive assumption.
As $\CC$ is prepattern-free, there is some number $t=\const(\CC,r)$ such that no graph $G\in\CC$ contains a pattern of order $t$ on insulators of cost at most $k_r \cdot r$ and height at most $r$ in $G$.

Let $G\in\CC$ and $W\subset V(G)$.
By the inductive assumption, there is 
a subset $W_r\subset W$ of size $U_{\CC,r}(|W|)$ 
and an insulator $\gc A_r$ 
of height $r$ and cost $\const(\CC,r)$ which 
insulates $W_r$. 
Let $J$ denote the indexing sequence of $\gc A_r$.
We prove that there is a subset $W_{r+1}\subset W_r$ of size $U_{\CC,r}(|W_r|)\ge U_{\CC,r}(|W|)$ and an insulator $\gc A_{r+1}$ of height $r+1$ and cost $\const(\CC,r)$ which insulates $W_{r+1}$.
We consider two cases, depending on whether $\gc A_r$ is orderless or ordered.

Assume first that $\gc A_r$ is orderless. 
We apply \cref{thm:orderless-grid-growing} to $\gc A_r$. This yields a sequence $I \subseteq J$ of length $U_t(|J|)$ such that one of the following three cases applies.
        \begin{enumerate}
            \item $G$ contains a prepattern of order $t$ on $\gc A_r\vert_I$.
            \item $G$ contains a row-extension $\gc A_r\vert_I$ with cost $\const(k_r,t)\le \const(\CC,r)$.
            \item $\gc A_r\vert_I$ is orderable.
        \end{enumerate}

        We set $W' := A[*,1]$ where $A$ is the grid of $\gc A_r\vert_I$.
        By the definition of an orderless subgrid we have $W' \subseteq W_r \subseteq W$ and 
        \[
            |W'|=|\tail(I)|=|I|-1\ge U_t(|J|).    
        \]
        As $\gc A_r$ is of height $r$ and cost $k_r$, the same holds for $\gc A_r\vert_I$.
        By our choice of $t$, the first case cannot apply.
        In the second case, the row extension of $\gc A_r\vert_I$ witnesses that $W'$ is $(r+1, \const(\CC,r))$-insulated, and we conclude by setting $W_{r+1} := W'$.
        It remains to handle the third case.
        We apply \cref{lem:orderless-to-ordered-insulator} to the orderable insulator $\gc A_r \vert_I$, yielding an ordered insulator $\gc B$ of cost $k_r \cdot r$ and height~$r$ that also insulates $W'$.
        Up to renaming, we can assume that $\gc B$ is indexed by $\tail(I)$, the same sequence that also indexes $\gc A_r \vert_I$.
        We apply \cref{lem:order-growing} to $\gc B$ with $k := k_r \cdot r$ and $t$.
        This yields a sequence $K \subseteq \tail(I)$ with $|K|=U_t(|J|)$ such that one of the following two cases applies.

        \begin{enumerate}
            \item $G$ contains a prepattern of order $t$ on $\gc B\vert_K$.
            \item $G$ contains a row-extension of $\gc B\vert_K$ with cost $\const(k,t)\le \const(\CC,r)$.
        \end{enumerate}
        By  our choice of $t$, the first case cannot apply.
        In the second case, by definition of an ordered subgrid, there exists a set $W_{r+1} \subseteq W' \subseteq W$ which is insulated by a row extension of $\gc B\vert_K$ and has size at least
        \[
            |W_{r+1}| = |\tail(K)| = |K| - 1 \geq U_t(|J|).
        \]
        $W_{r+1}$ is the desired $(r+1, k)$-insulated set. 
        This concludes the case where $\gc A_r$ is orderless.
        If $\gc A_r$ is ordered we can directly apply \cref{lem:order-growing}, which just improves the bounds.
        This completes the inductive proof.

        The induction can be easily turned into an algorithm.
        The trivial insulator from \cref{lem:insulate-base-case} can be computed in time $O(|V(G)|)$.
        The running time of the inductive step follows from the running times of \cref{thm:orderless-grid-growing}, \cref{lem:orderless-to-ordered-insulator}, and \cref{lem:order-growing}.
\end{proof}

\section{Flip-Breakability}
\label{sec:flip-breakability}

In this section we show the following implications: 
\[
    \text{insulation property}
    \quad
    \Rightarrow
    \quad
    \text{flip-breakable}
    \quad 
    \Rightarrow
    \quad
    \text{monadically dependent}
\]
For convenience, we restate the definition of flip-breakability.

\defFlipBreakableOne*

In short, $\CC$ is flip-breakable if for every $r\in\N$, graph $G\in\CC$ and $W\subset V(G)$,
there exist subsets $A,B\subset W$ with $|A|,|B|\ge U_{\CC,r}(|W|)$ and a $\const(\CC,r)$-flip $H$ of $G$ with $\dist_H(A,B)>r$.

\subsection{Insulation Property Implies Flip-Breakability}
In this section we prove the following.
\begin{proposition}\label{prop:ge-implies-fb}
    Every graph class $\CC$ with the insulation property is flip-breakable.

\effective{Moreover, if $\CC$ is also prepattern-free, then there is an algorithm that, given a radius $r$, graph $G$, and set $W$, computes in time $O_{\CC, r}(|V(G)|^2)$ subsets $A,B \subseteq W$, partition $\KK_\star$, and relation $F_\star \subseteq \KK_\star^2$ witnessing flip-breakability. 
(This means \(|A|,|B| \ge U_{\CC,r}(|W|)\), \(|\KK_\star| \le \const(\CC,r)\), and \(\dist_{H}(A,B) > r\) in the flip \(H\) defined by \(\KK_\star\) and \(F_\star\).)
}
\end{proposition}

Let $W$ be a set of vertices in a graph $G$.
We call $W$ \emph{$(r,k)$-flip-breakable}, if there exist two disjoint sets $A,B \subseteq W$, each of size at least $\frac{1}{3}|W|$, and a $k$-flip $H$ of $G$, such that
\[
    N^H_r(A) \cap N^H_r(B) = \varnothing.
\]
\cref{prop:ge-implies-fb} will be implied by the following.

\begin{lemma}\label{lem:insulated-to-fb}
    Let $W$ be a $(2r+1,k)$-insulated set of at least $18r$ vertices in a graph $G$.
    Then $W$ is $(r,2^{48r^2k})$-flip-breakable in \(G\).

    \effective{Moreover, there is an algorithm that, given a radius $r$, graph $G$, set $W$, and an insulator witnessing that $W$ is $(2r+1,k)$-insulated, computes in time $O_{r,k}(|V(G)|)$ the subsets $A,B \subseteq W$, partition $\KK_\star$, and relation $F_\star \subseteq \KK_\star^2$ that witness the $(r,2^{48r^2k})$-flip-breakability of $W$.
    }
\end{lemma}

\begin{proof}
    Let $\gc A =(A,\KK,F,R)$ be the insulator of cost \(k\) and height $2r + 1$ that insulates $W$.
    Since we will only be using the insulator properties \ref{itm:outside} and \ref{itm:adjacency} and by \cref{obs:orderless-just-as-good}, we do not have to distinguish whether $\gc A$ is orderless or ordered.
    Up to renaming, we can assume that $A$ is indexed by a sequence $I = (1,\ldots,n)$ for some $n\in\N$.
    Let $l_1, l_2, r_1,r_2 \in I$ be the indices such that $(l_1,l_1+1,\ldots, l_2)$ and $(r_1,r_1 + 1, \ldots,r_2)$ are the sequences containing the first and last $2r$ elements of $I$ respectively.
    Since $\gc A$ insulates $W$, we can choose indices $m_1 > \ell_2$ and $m_2 < r_1$ such that $(m_1, m_1 +1, \ldots, m_2)$ contains $2r$ elements and there are two disjoint sets $W_1 := W \cap M_1 \cap A[*,1]$ and $W_2 := W \cap M_2 \cap A[*,1]$, where
    \[
        M_1 := \bigcup_{l_2 < i < m_1} A[i,*]
        \quad 
        \text{and}
        \quad
        M_2 := \bigcup_{m_2 < i < r_1} A[i,*],
    \]
    such that $W_1$ and $W_2$ each contain at least $\frac{1}{2}(|W| - 6r) \geq \frac{1}{3}|W|$ vertices.
    The sets $W_1$ and $W_2$ will play the role of the sets $A$ and $B$ in the flip-breakability statement.
    We define a new grid $B$ of height $h$ indexed by $J := (1, \ldots, 3\cdot 2r + 2)$ on the same vertex set as $A$.
    The rows of $B$ are the same as the rows of $A$.
    The columns of $B$ are in order
    \[
        \underbrace{A[l_1,*],\ \ldots,\ A[l_2,*]}_{2r \text{ columns}},\ 
        M_1
        ,\ 
        \underbrace{A[m_1,*],\ \ldots,\ A[m_2,*]}_{2r \text{ columns}}
        ,\ 
        M_2
        ,\ 
        \underbrace{A[r_1,*],\ \ldots,\ A[r_2,*]}_{2r \text{ columns}}.
    \]
    See \cref{fig:coarsening} for a visualization.
    We observe that $B$ coarsens the columns of $A$ in the following sense.
    \begin{observation}\label{obs:gefb-column-coarsening}
        Let $u \in B[j,t]$ and $v \in B[j',t']$ such that $j < j'$ for some $j,j' \in J$ and $t,t' \in [h]$.
        The also $u \in A[i,t]$ and $v \in A[i',t']$ for some $i < i' \in I$.
    \end{observation}
    \noindent Additionally, the construction ensures that
    \begin{itemize}
        \item $B$ has $(3\cdot 2r + 2)\cdot (2r+1) \leq 24r^2$ cells,
        \item $W_1 \subseteq B[2r + 1, 1]$ and $W_2 \subseteq B[4r + 2,1]$, and
        \item $\int(A) = \int(B)$.
    \end{itemize}

    \begin{figure}[htbp]
        \centering
        \includegraphics[scale=0.9]{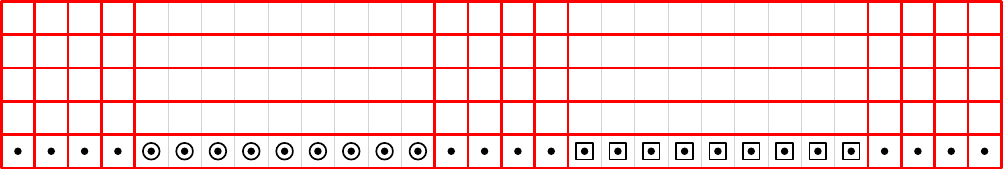}
        \caption{An example of a coarsening for radius $r = 2$. The coarsening $B$ (in bold red) overlays the grid $A$ (in gray).
        The sequence $W$ is located in the bottom row. The two subsequences $W_1$ and $W_2$ are marked with circles and squares, respectively.
        }
        \label{fig:coarsening}
    \end{figure}

    We build $\KK_\star$ as a refinement of $\KK$ by encoding into the color of every vertex $v \in V(G)$ for every color $X \in \KK$ and for every cell $B[i,j]$ of $B$ the information
    \begin{enumerate}[leftmargin= 3em, label={(C.$\arabic*$)}]
        \item \label{itm:gefb-col-x} whether $v \in X$,
        \item \label{itm:gefb-col-cell} whether $v \in B[i,j]$,
        \item \label{itm:gefb-col-outside} whether $v$ is adjacent in \(G\) to a vertex from $X \cap \int(B)$.
    \end{enumerate}

    $\KK_\star$ has at most $2^{k \cdot 24r^2 \cdot 2}$ colors.
    By \ref{itm:gefb-col-x}, we can define for every $X \in \KK_\star$ a color $\KK(X) \in \KK$ such that for all $v \in X$ we have $v\in\KK(X)$.
    We define $F_\star \subseteq \KK^2$ via the following four rules.

    \begin{enumerate}[leftmargin= 3em, label={(F.$\arabic*$)}]
        \item \label{itm:gefb-flip-hor-left}
        If $X \subseteq B[j,t]$ and $Y \subseteq B[j',t']$ and $j < j'$ and $t' \in \{t,t-1\}$, then
        \[
            (X,Y) \in F_\star 
            \Leftrightarrow 
            \big(\KK(X), \KK(Y)\big) \in R
            .
        \]
        \item \label{itm:gefb-flip-hor-right}
        If $X \subseteq B[j,t]$ and $Y \subseteq B[j',t']$ and $j > j'$ and $t' \in \{t,t+1\}$, then
        \[
            (X,Y) \in F_\star 
            \Leftrightarrow 
            \big(\KK(Y), \KK(X)\big) \in R
            .
        \]
        \item \label{itm:gefb-flip-outside}
        If $X \not \subseteq B$ and $Y \subseteq \int(B)$, or vice versa, then
        \[
            (X,Y) \in F_\star 
            \Leftrightarrow 
            \text{there is an edge between $X$ and $Y$ in $G$}
            .
        \]
        \item \label{itm:gefb-flip-else}
        Otherwise, $(X,Y) \in F_\star \Leftrightarrow \big(\KK(X), \KK(Y)\big) \in F$.
    \end{enumerate}
    By construction, $F_\star$ is symmetric and therefore describes a valid flip.
    Let $G' := G \oplus_\KK F$ and $G_\star := G \oplus_{\KK_\star} F_\star$.
    
    \begin{claim}\label{clm:gefb-far-rows}
        Let $u \in \int(B)$ and $v \in B$ be vertices from rows that are not close in $B$.
        Then \(u\) and \(v\) are non-adjacent in $G_\star$.
    \end{claim}
    
    \begin{claimproof}
        Let $u$ and $v$ be as in the claim.
        Then case \ref{itm:gefb-flip-else} applies and $u$ and $v$ are adjacent in $G_\star$ if an only if they are adjacent in $G'$.
        By construction of $B$, we have also $u \in \int(A)$ and $v \in A$ and $u$ and $v$ are in rows that are not close in $A$.
        It follows by property \ref{itm:adj-different-rows} of $\gc A$ that $u$ and $v$ are non-adjacent in both $G'$ and $G_\star$.
    \end{claimproof}

    \begin{claim}\label{clm:gefb-far-cols}
        Let $u \in \int(B)$ and $v \in B$ be vertices from columns that are not close in $B$.
        Then \(u\) and \(v\) are non-adjacent in $G_\star$.
    \end{claim}
    
    \begin{claimproof}
        Let $u \in B[j,t]$ and $v \in B[j',t']$ be as in the claim for some $j,j' \in J$ and $t, t'\in[h]$.
        If $u$ and $v$ are in rows that are not close, then we are done by \cref{clm:gefb-far-rows}.
        We can therefore assume $t' \in \{t-1,t,t+1\}$.

        Assume first that $j < j'$. Since $B[j,*]$ and $B[j',*]$ are not close, we even have $j+1 < j'$.
        If $t' = t + 1$,
        then case \ref{itm:gefb-flip-else} applies and $u$ and $v$ are adjacent in $G_\star$ if an only if they are adjacent in~$G'$.
        By \cref{obs:gefb-column-coarsening}, the adjacency between $u$ and $v$ in $G'$ can be determined using the insulator property \ref{itm:adj-bot-left}
        of $\gc A$: $u$ and $v$ are non-adjacent as desired.
        If $t' \in \{t,t-1\}$, then by \ref{itm:gefb-col-cell}, case \ref{itm:gefb-flip-hor-left} applies and the following are equivalent.
        Let $X := \KK_\star(u)$ and $Y := \KK_\star(v)$.
        \begin{enumerate}
            \item The adjacency between $u$ and $v$ got flipped when going from $G$ to $G_\star$.
            \item $\big(\KK(X), \KK(Y)\big) \in R$. 
            \hfill (by \ref{itm:gefb-flip-hor-left})
            \item $u$ and $v$ are adjacent in $G$.
            \hfill (by \cref{obs:gefb-column-coarsening} and \ref{itm:adj-left})
        \end{enumerate}
        The equivalence between the first and the last item establishes that $u$ and $v$ are non-adjacent in~$G_\star$, as desired.

        The proof for $j > j'$ works symmetrically.
        If $t' = t - 1$, then case \ref{itm:gefb-flip-else} applies, and we argue using \ref{itm:adj-bot-left}. If $t' \in \{t,t+1\}$, then case \ref{itm:gefb-flip-hor-right} applies, and we argue using \ref{itm:adj-right}.
    \end{claimproof}

    \begin{claim}\label{clm:befg-outside}
        Let $u \in \int(B)$ and $v$ be a vertex not in $B$. Then $u$ and $v$ are non-adjacent in $G_\star$.
    \end{claim}
    
    \begin{claimproof}
        Let $X := \KK_\star(u)$ and $Y := \KK_\star(v)$ be the colors of $u$ and $v$ in $\KK_\star$.
        By \ref{itm:gefb-col-cell}, we have $Y \cap A = \varnothing$ and $X \subseteq \int(A)$.
        It follows from property \ref{itm:outside} of $\gc A$ that every vertex from $Y$ is homogeneous to $X$ in $G$.
        Moreover, by \ref{itm:gefb-col-outside}, either every or no vertex in $Y$ has a neighbor in $X$ in $G$.
        It follows that the connection between $X$ and $Y$ is homogeneous in $G$.
        Also, by \ref{itm:gefb-col-outside}, case \ref{itm:gefb-flip-outside} applies, and the following are equivalent.

        \begin{enumerate}
            \item The adjacency between $u$ and $v$ got flipped when going from $G$ to $G_\star$. 
            \item There is an edge between $X$ and $Y$ in $G$.
            \hfill (by \ref{itm:gefb-flip-outside})
            \item There is an edge between $u$ and $v$ in $G$.
            \hfill (by homogeneity of $X \ni u$ and $Y \ni v$ in $G$)
        \end{enumerate}
        The equivalence between the first and the last item establishes that $u$ and $v$ are non-adjacent in~$G_\star$, as desired.
    \end{claimproof}

    Combining \cref{clm:gefb-far-rows}, \cref{clm:gefb-far-cols}, and \cref{clm:befg-outside} yields the following crucial observation.
    
    \begin{observation}
        In $G_\star$, if vertices $u \in \int(B)$ and $v \in V(G)$ are adjacent, then they are in cells that are close in $B$.
    \end{observation}

    A straightforward induction now yields that the $2r$-neighborhood of every vertex $u \in W_1 \subseteq B[2r+1,1]$ in $G_\star$ 
    satisfies $N^{G_\star}_{2r}(u) \subseteq B[{\le} 4r + 1,*]$.
    In particular, it contains no vertex from $W_2 \subseteq B[4r + 2,1]$. 
    We therefore have
    $N^{G_\star}_r(W_1) \cap N^{G_\star}_r(W_2) = \varnothing$
    as desired.
    This finishes the proof of the flip-breakability of $W$.
    To bound the running time, note that
    $W_1$, $W_2$, $\KK_\star$, and $F_\star$ can all be computed in time $O_{r,k}(|V(G)|)$.
\end{proof}

We can now prove \cref{prop:ge-implies-fb}.

\begin{proof}[Proof of \cref{prop:ge-implies-fb}]
    The non-algorithmic part of the statement immediately follows from \cref{lem:insulated-to-fb}.
    For the algorithmic part, assume that $\CC$ is also prepattern-free.
    Given any set $W$, we can compute an insulated subset $W' \subseteq W$ and a witnessing insulator $\gc A$ using \cref{prop:prepatternImpliesInsulation}.
    Applying the algorithm given by \cref{lem:insulated-to-fb} to $W'$ and $\gc A$ yields the desired result.
\end{proof}

\subsection{Flip-Breakability Implies Monadic Dependence}\label{sec:breakableImpliesNIP}
We prove the second main implication of the section.
\begin{proposition}\label{prop:fb-implies-mnip}
    Every flip-breakable class of graphs is monadically dependent.
\end{proposition}

Let $\phi(x,y)$ be a formula over the signature of colored graphs, $G^+$ be a colored graph, and $W$ be a set of vertices in $G^+$.
We say that $\phi$ \emph{shatters} $W$ in $G^+$, if there exists vertices
$(a_{R})_{R \subseteq W}$ such that for all $b \in P$ and $R \subseteq W$,
\[
    G^+ \models \phi(a_{R},b) \Leftrightarrow b \in R.
\]
Let $G$ be an (uncolored) graph, and $W$ be a set of vertices in $G$.
We say that $\phi$ \emph{monadically shatters} $W$ in $G$, if there exists a coloring $G^+$ of $G$ 
in which $\phi$ shatters $W$.

\begin{fact}[{\cite{baldwin1985second}}]
    A class of graphs $\CC$ is monadically dependent if and only if for every formula $\phi(x,y)$ over the signature of colored graphs, there exists a bound $m$ such that $\phi$ monadically shatters no set of size $m$ in any graph of $\CC$. 
\end{fact}

\cref{prop:fb-implies-mnip} will be implied by the following.

\begin{lemma}\label{lem:fb-shatter}
    Let \(\phi(x,y)\) be a formula and \(k \in \N\).
    There exist \(r_\phi, m_{\phi,k} \in \N\), where \(r_\phi\) depends only on \(\phi\),
    %For every formula $\phi(x,y)$ there exists $r_\phi$, such that for every integer $k$ there exists $m_{\phi,k} \in \N$
    such that no graph $G$ contains a set of at least $m_{\phi,k}$ vertices $W$ such that
    \begin{itemize}
        \item $\phi$ monadically shatters $W$ in $G$, and
        \item $W$ is $(r_\phi,k)$-flip-breakable in $G$.
    \end{itemize}
\end{lemma}

In order to prove \cref{lem:fb-shatter}, we use the following statement, which is an immediate consequence of Gaifman's locality theorem~\cite{gaifman82}.
For an introduction of the locality theorem see for example~\cite[Sec. 4.1]{grohe2008logic}.

\begin{corollary}[of {\cite[Main Theorem]{gaifman82}}]\label{lem:gaifman-coloring}
  Let $\phi(x,y)$ be a formula. Then there are numbers $r,t\in \NN$, where $r$ depends only on the quantifier-rank of $\phi$ and $t$ depends only on the signature and quantifier-rank of $\phi$,
  such that every (colored) graph $G$ can be vertex-colored using $t$ colors
  in such a way that for any two vertices $u,v\in V(G)$ with distance greater than $r$ in $G$,
  $G\models\phi(u,v)$ depends only on the colors of $u$ and $v$.
  We call $r$ the \emph{Gaifman radius} of $\phi$.
\end{corollary}

\begin{proof}[Proof of \cref{lem:fb-shatter}]
    We set $r_\phi$ to be the Gaifman radius of $\phi$.
    Let $s$ be the number of colors used by $\phi$.
    As stated in \cref{lem:gaifman-coloring}, let $t$ be the number of colors needed to determine the truth value of formulas in
    the signature of $(s\cdot k)$-colored graphs that have the same quantifier-rank as $\phi(x,y)$.
    Let $m_{\phi,k} := 3(t+1)$.

    Assume now towards a contradiction the existence of an $(r_\phi,k)$-flip-breakable in set \(W\) in \(G\) of size \(m_{\phi,k}\) such that 
    $\phi$ monadically shatters $W$ in $G$.
    Then there exists an $s$-coloring $G^+$ of $G$ in which $\phi$ shatters $W$.
    We apply flip-breakability which yields a $k$-flip $H$ of $G$ together with two disjoint sets $A,B \subseteq W$ each of size at least $t+1$, such that $N^H_r(A) \cap N^H_r(B) = \varnothing$.
    By using $k$ colors to encode the flip, we can rewrite $\phi$ to a formula $\psi$ with the same quantifier-rank as $\phi$, such that there exists a $(s \cdot k)$-coloring $H^+$ of $H$ where for all $u,v \in V(G)$,
    \[
        G^+ \models \phi(u,v) \Leftrightarrow H^+ \models \psi(u,v).
    \]
    In particular, $\psi$ shatters $W$ in $H^+$.
    Since $\psi$ has the same quantifier-rank as $\phi$ and is a formula over the signature of $s\cdot k$-colored graphs, by \cref{lem:gaifman-coloring} there exists a coloring 
    of $H^+$ with $t$ colors such that the truth of $\psi(u,v)$ in $H^+$ only depends on the colors of $u$ and $v$ for all vertices $u$ and $v$ with distance greater than $r$ in $H^+$.
    Recall that $A$ and $B$ each have size $t+1$.
    By the pigeonhole principle, there exist two distinct vertices $a_1,a_2 \in A$ that are assigned the same color and two distinct vertices $b_1,b_2 \in B$ that are assigned the same color. 
    Since $\psi$ shatters $W$ in $H^+$, there exists a vertex $v \in V(G)$ such that
    \[
        H^+ \models \psi(v,a_1) \wedge \neg \psi(v,a_2) \wedge \psi(v,b_1) \wedge \neg \psi(v,b_2).
    \]
    By \cref{lem:gaifman-coloring}, $v$ must be contained in $N^H_r(A) \cap N^H_r(B)$, 
    as the truth of \(\psi\) is inhomogeneous among both \(v\) and $\{a_1,a_2\}$ and among \(v\) and $\{b_1,b_2\}$.
    This is a contradiction to $N^H_r(A) \cap N^H_r(B)~=~\varnothing$.
\end{proof}

Combining the results of \cref{sec:insulation-property} and \cref{sec:flip-breakability}
yields the desired chain of implications.
\[
    \text{prepattern-free}
    \spacedRightarrow
    \text{insulation-property}
    \spacedRightarrow
    \text{flip-breakable}
    \spacedRightarrow
    \text{mon.\ dependent}
\]
In the remaining two sections of \cref{part1}, we provide the deferred proofs of the two insulator growing lemmas (Lemmas \ref{thm:orderless-grid-growing} and \ref{lem:order-growing}).

\section{Sample Sets}

We work towards proving 
Lemmas \ref{thm:orderless-grid-growing} and \ref{lem:order-growing}, which grow the height of an insulator.
In this section we show that in prepattern-free classes, given an insulator $\gc A$, we can extract a subinsulator $\gc B$ and a small \emph{sample set} of vertices which can be used to approximately represent the connections of all the vertices in the graph towards~$\gc B$.
We give some notation to make this statement precise.

\begin{definition}
    Given a graph $G$, a vertex $v \in V(G)$, and a set $A \subseteq V(G)$, we define the \emph{atomic type} of $v$ over $A$ as
    \[
        \atp(v/A) := \{(R,a) : G\models R(v,a), R \in \{E,=\}, a \in A\}.
    \]
\end{definition}

\begin{observation}\label{obs:monotone-disagreement}
    Let $G$ be a graph, $u,v \in V(G)$, and $A \subseteq B \subseteq V(G)$. Then
    \[
        \atp(u/A) \neq \atp(v/A)
        \quad
        \Rightarrow
        \quad
        \atp(u/B) \neq \atp(v/B).
    \]
\end{observation}

\begin{restatable}{definition}{defSampled}\label{def:sampled}
    Let $G$ be a graph containing an insulator $\gc A$ with grid $A$ indexed by $I$.
    Let $v,s_<,s_>$ be vertices from $G$, $i\in I$, and \(m \in \N\).
    We say \emph{$v$ is $(m,i,s_<,s_>)$-sampled on $\gc A$} if
    \[
        \atp(v/ A[{<} i,*]) = \atp(s_</ A[{<} i,*])
        \quad 
        \text{and}
        \quad
        \atp(v/ A[{\ge} i + m,*]) = \atp(s_>/ A[{\ge} i + m,*]).
    \]
    We call $m$ the \emph{margin}, \(i\) the \emph{exceptional index}, $s_<$ the \emph{left-sample}, and $s_>$ the \emph{right-sample}.
\end{restatable}

\begin{restatable}{definition}{defSamples}\label{def:samples}
    Fix $p\in \N$. Let $G$ be a graph containing an insulator $\gc A$ indexed by $I$ and let $S \subseteq V(G)$.
    We say \emph{$S$ samples $G$} on $\gc A$ with margin $m$ 
    if there exists functions 
    $\ex : V(G) \rightarrow I$ and $s_<,s_> : V(G) \rightarrow S$
    such that every $v \in V(G)$ 
    is $\big(m,\ex(v),s_<(v),s_>(v)\big)$-sampled on $\gc A$.
\end{restatable}

We are now ready to state the main result of this section.

\begin{restatable}{lemma}{lemSampling}\label{lem:sampling}
    Fix \(t \in \N\).
    For every graph \(G\) and insulator  \(\gc A\) indexed by \(J\) in $G$, there is a subsequence \(I \subseteq J\) of size $U_{t}(|J|)$ such that either
    \begin{itemize}
        \item $G$ contains a prepattern of order $t$ on $\gc A\vert_{I}$, or
        \item there is a set $S\subseteq V(G) \setminus \gc A\vert_I$ of size $\const(t)$ that samples $G$ on $\gc A\vert_I$ with margin $2$.
    \end{itemize}
    \effective{Moreover, there is an algorithm that, given $G$ and $\gc A$, computes $I$ and one of the two outcomes (a prepattern or a sampling set $S$) in time $O_{t}(|V(G)|^2)$.
    }
\end{restatable}

We will build the set $S$ iteratively by extracting single sampling vertices one by one.

\subsection{Extracting Single Sample Vertices}

Before we show how to extract a new sample vertex, we state some auxiliary lemmas about subinsulators and sampling sets.

\begin{lemma}\label{clm:disagree-subseq}
    Let $\gc A$ be an insulator indexed by $J$ in a graph $G$.
    Let $I$ be a subsequence of $J$, $u,v \in V(G)$, and $i \in \tail(I)$.
    Let $A$ and $B$ be the grids of $\gc A$ and $\gc A\vert_I$ respectively.
    Then
    \[
        \atp(u/A[i,*]) \neq \atp(v/A[i,*])
        \quad
        \Rightarrow
        \quad
        \atp(u/B[i,*]) \neq \atp(v/B[i,*]).
    \]
\end{lemma}
\begin{claimproof}
    Follows from \cref{obs:monotone-disagreement} and \cref{obs:mon-and-cov}.
\end{claimproof}

\begin{lemma}\label{lem:drop-elem}
    Let $\gc A$ be a grid indexed by a sequence $J$ of length at least four in a graph $G$ and let $v \in V(G)$. 
    There exists a subsequence $I \subseteq J$ of length $|I| \geq \frac{1}{2}|J|$ such that $v \notin \gc A\vert_I$.

    \effective{Moreover, there is an algorithm that, given $G$ and $\gc A$, computes $I$ in time $O(|V(G)|)$.}
\end{lemma}

\begin{proof}
    Up to renaming, assume $J = (1,\ldots, n)$.
    If $v \notin \gc A\vert_I$, we can just set $I := J$.
    Otherwise, let $i \in J$ be such that $v \in A[i,*]$ where $A$ is the grid of $\gc A$.
    We choose \(I\) as the larger of the two sequences \((1,\dots,i-1)\) and \((i,\dots,n)\).
    Since \(I\) has length at least two, this defines a subinsulator.
    By \cref{obs:mon-and-cov}, \(v \not \in A \vert_I\).
    The bound on the running time is obvious.
\end{proof}

\begin{lemma}\label{lem:subgrid-sample}
    Fix $m\in \N$. Let $G$ be a graph containing an insulator $\gc A$ indexed by $J$ and $S \subseteq V(G)$ be a set that samples $G$ on $\gc A$ with margin $m$. 
    For every $I \subseteq J$, $S$ also samples $G$ on $\gc A\vert_I$ with margin~$m$.
\end{lemma}
\begin{proof}
    Let $B:={A}\vert_{I}$ be the grid of $\gc A\vert_I$.
    As in the first item in the proof of \Cref{lem:subinsulator}, we observe for every \(i \in I\)
    \[
         B[{<}i,*] \subseteq A[{<}i,*] \quad \text{and} \quad B[{>}i,*] \subseteq A[{>}i,*].
    \]
    Hence, by contrapositive of \cref{obs:monotone-disagreement}, for all vertices \(u,s_<,s_>\),
    \[
        \atp(v/A[{<}i,*]) = \atp(s_</A[{<}i,*])
        \quad
        \Rightarrow
        \quad
        \atp(v/B[{<}i,*]) = \atp(s_</B[{<}i,*]),
    \]
    \[
        \atp(v/A[{>}i,*]) = \atp(s_>/A[{>}i,*])
        \quad
        \Rightarrow
        \quad
        \atp(v/B[{>}i,*]) = \atp(s_>/B[{>}i,*]).
    \]
    The lemma now follows from \Cref{def:samples}.
\end{proof}

We will now show how to extract a new sample vertex.
We start with a given set of sample vertices $S$. In the absence of large prepatterns, we either find a subinsulator on which $S$ samples $G$, or find a new vertex $v$ by which we will later use to extend the sample set.

\begin{lemma}\label{lem:new-sample-vtx2}
    Fix \(k,t \in \N\).
    For every graph \(G\) and insulator \(\mathcal{A}\) indexed by \(J\), and every vertex set $S$ of size at most $k$,
    there is a subsequence \(I \subseteq J\) of size $U_{k,t}(|J|)$
    such that either
    \begin{itemize}
        \item $G$ contains a bi-prepattern of order $t$ on $\gc A\vert_{I}$, or
        \item $S$ samples $G$ on $\gc A\vert_I$ with margin $2$, or
        \item there is a vertex $v \notin \gc A\vert_I$, such that for all $s \in S$ and every column $C$ in $\gc A\vert_{I}$
        \[
            \atp(v/C) \neq \atp(s/C).   
        \]
    \end{itemize}    
    \effective{Moreover, there is an algorithm that, given $G$, $\gc A$, and $S$, computes the sequence $I$ and one of the three outcomes (a bi-prepattern, the conclusion that $S$ samples $G$, or a vertex $v$) in time $O_{k,t}(|V(G)|^2)$.
    }
\end{lemma}

\begin{proof}
    We first show how to construct a sequence $I$ with the desired properties.
    Afterwards we analyze the running time. 
    Getting the desired bounds will then require a small preprocessing that reduces the size of $J$.
    The proof is split into multiple paragraphs.

    \paragraph*{Notation.}
    For vertices $v,s \in V(G)$ and a set $U \subseteq V(G)$, we say \emph{$v$ is $s$-connected to $U$} if 
    \[
        \atp(v/U) = \atp(s/U).
    \]
    We generalize this to sets $S_\star \subseteq S$ and say \emph{$v$ is $S_\star$-connected to $U$} if
    \[
        \{ s \in S \mid \text{$v$ is $s$-connected to $U$} \} = S_\star.
    \]

    \paragraph*{Ramsey.}
    We start by defining some coloring to which we will apply Ramsey's theorem.
    Let $A$ be the grid of $\gc A$ and $M_A$ be the function from \cref{lem:subgrid-ramseyness}.
    For \(t' \in \{6,4t\}\) and \(S_1,\dots,S_{t'-1} \subseteq S\), we label all \(t'\)-tuples \(i_1 < \dots < i_{t'} \in J\) with a color 
    indicating whether
    \begin{equation}\label{eq:swbit}
        \exists v \bigwedge_{l=1,\dots,t'-1} \text{ $v$ is $S_l$-connected to $M_A(i_l,i_{l+1})$}.
    \end{equation}
    Applying Ramsey's Theorem (\cref{lem:ramsey}) to this coloring gives us a subsequence \(I' \subseteq J\) such that for all 
    \(t' \in \{8,6t\}\) and \(S_1,\dots,S_{t'-1} \subseteq S\),
    the above equation (\ref{eq:swbit}) either holds for all or no \(t'\)-tuples \(i_1 < \dots < i_{t'} \in I'\).
    Note that the number of colors is bounded by $\const(k,t)$, which guarantees \(|I'| \geq U_{k,t}(|J|)\).
    We can assume without loss of generality that \(I := \tail(I')\) has length at least \(24t\).
    In order to simplify notation, we assume up to renaming that $I' = (0,\ldots, n)$ and $I = (1,\ldots,n)$.
    Let $B := A\vert_{I'}$ be the grid of $\gc B := \gc A\vert_{I'}$.
    By \cref{lem:subgrid-ramseyness}, we observe the following.

    \begin{observation}\label{obs:determine-columns}
        For all $i \in I$ we have $B[i,*] = M_A(i-1,i)$.
    \end{observation}

    We say a subsequence of a sequence \(K\) is \emph{\(1\)-spaced} if it contains no consecutive elements from~\(K\).
    
    \begin{claim}\label{claim:shift}
        Let \(s \in S\) and \(\bar w \in \{0,1\}^{2t}\).
        If there is a $1$-spaced subsequence \(i_1 < \dots < i_{2t}\) of $I$ such that
        \[
            \exists v \bigwedge_{l=1,\dots,2t} 
            \text{$v$ is $s$-connected to $B[i_l,*]$ if and only if $w_{l}=1$},
        \]
        then the above holds for all such $1$-spaced subsequences of the same length.
    \end{claim}
    \begin{claimproof}
        By \cref{obs:determine-columns}, the statement of the claim holds for a $1$-spaced subsequence \(i_1 < \dots < i_{2t}\)
        if and only if 
        equation \eqref{eq:swbit} holds for the corresponding \(4t\)-tuple 
        \[
            i_1 - 1 < i_1 < 
            i_2 - 1 < i_2 < 
            \dots < 
            i_{2t} - 1 < i_{2t}
            \in I'
        \] 
        of distinct elements
        and some \(S_1,\dots,S_{4t-1} \subseteq S\) with \(s \in S_{2l-1} \Leftrightarrow w_l = 1\).
        Note that we use $1$-spacedness to guarantee that all elements in the \(4t\)-tuple are distinct.

        Assume there is a $1$-spaced subsequence satisfying the claimed statement,
        and let \(S_1,\dots,S_{4t-1} \subseteq S\) be the sets witnessing 
        the truth of equation (\ref{eq:swbit}) for the corresponding \(4t\)-tuple.
        By Ramsey's theorem, equation (\ref{eq:swbit}) with \(S_1,\dots,S_{4t-1} \subseteq S\) holds for all \(4t\)-tuples,
        and thus, the claimed statement holds for all $1$-spaced subsequences.
    \end{claimproof}

    \begin{claim}\label{claim:shift2}
        If there is \(i,j \in I\) with \(i+2 < j\) and sets $P_i, Q_i, P_j, Q_j \subseteq S$ with $P_i \neq Q_i$ and $P_j \neq Q_j$ such that
        \begin{align*}
        \exists v~  
        &\text{ $v$ is $P_i$-connected to $B[i,*]$}\\
        \wedge &\text{ $v$ is $Q_i$-connected to $B[i+1,*]$}\\
        \wedge &\text{ $v$ is $P_j$-connected to $B[j,*]$}\\
        \wedge &\text{ $v$ is $Q_j$-connected to $B[j+1,*]$}
        \end{align*}
        then the above holds for all \(i,j \in I\) with \(i+2 < j\).
    \end{claim}
    \begin{claimproof}
        We again use \cref{obs:determine-columns}.
        Thus, the statement of the claim holds for a given \(i,j \in I\)
        if and only if 
        equation (\ref{eq:swbit})
        holds for
        the corresponding \(6\)-tuple 
        \[
            i-1 < i < i+1
            <
            j-1 < j < j+1
            \in I'
        \]
        of distinct elements
        and some \(S_1,\dots,S_{5} \subseteq S\) with \(S_1 = P_i\) and \(S_2 = Q_i\) and $S_4 =P_j$ and $S_5 =Q_j$.
        The rest follows as in \Cref{claim:shift}.
    \end{claimproof}
    
    \paragraph{Constructing a Prepattern.}
    We say an index \(i \in I\) is an \emph{alternation point} of a vertex \(v\) on \(\mathcal{B}\) if
    $v$ is $P$-connected to $B[i,*]$ and $Q$-connected to $B[i+1,*]$ for distinct sets $P \neq Q \subseteq S$.

    \begin{claim}\label{claim:alternationpoints}
        One of the following two conditions holds.
        \begin{enumerate}
            \item For every vertex $v\in V(G)$ with alternation points \(i,j \in I\), we have \(|i-j| \le 2\).
            \item $G$ contains a bi-prepattern of order $t$ on $\gc B$.
        \end{enumerate}

    \end{claim}

    \begin{claimproof}
        Assume the first condition fails and let \(v \in V(G)\) be a vertex with alternation points \(i,j \in I\) such that \(i+2<j\).
        By \Cref{claim:shift2},
        we can assume without loss of generality that \(v\)
        has alternation points \(6t\) and \(19t\).
        In the following, we find two $1$-spaced subsequences of $I$
        \begin{gather*}
            I_\star =
            \underbrace{i_1 < \dots < i_{t}}_{\subseteq (1,\ldots,4t)}
            \;\;<\;\;
            \underbrace{i_{\star\vphantom{t}}}_{\mathclap{\in \{6t,6t + 1\}}} 
            \;\;<\;\; 
            \underbrace{i_{t+1} < \dots < i_{2t}}_{\subseteq (9t,11t)},
            \\
            J_\star = \underbrace{j_1 < \dots < j_{t}}_{\subseteq (13t,\ldots,17t)}
            \;\;<\;\;
            \underbrace{j_{\star\vphantom{t}}}_{\mathclap{\in \{19t,19t + 1\}}} 
            \;\;<\;\;
            \underbrace{j_{t+1} < \dots < j_{2t}}_{\subseteq (22t,\ldots,24t)}         
        \end{gather*}
        and vertices $s_1,s_2 \in S$ satisfying the following property:
        Either
        \begin{itemize}
            \item $i_\star$ is the first index among $I_\star$, such that $v$ is $s_1$-connected to $B[i_\star,*]$, or
            \item $i_\star$ is the first index among $I_\star$, such that $v$ is \emph{not} $s_1$-connected to $B[i_\star,*]$,
        \end{itemize}
        and moreover, either
        \begin{itemize}
            \item $j_\star$ is the first index among $J_\star$, such that $v$ is $s_2$-connected to $B[j_\star,*]$, or
            \item $j_\star$ is the first index among $J_\star$, such that $v$ is \emph{not} $s_2$-connected to $B[j_\star,*]$.
        \end{itemize}
        We start by chosing $s_1 \in S$ to be an arbitrary vertex witnessing the alternation point \(6t\) of $v$, that is,
        \[
            \text{$v$ is $s_1$-connected to $B[6t,*]$}
            \quad
            \Leftrightarrow
            \quad
            \text{$v$ is not $s_2$-connected to $B[6t + 1,*]$}.
        \]
        By a simple majority argument, we can choose $i_1,\ldots,i_t$ to be a $1$-spaced subsequence of $(1,\ldots,4t)$ such that $v$ is $s_1$-connected either to none or to all of $B[i_1,*],\ldots,B[i_t,*]$.
        In the first case (respectively last case) we set $i_\star$ to be the index from $\{6t,6t+1\}$ such that $v$ is $s_1$-connected to $B[i_\star,*]$ (respectively \emph{not} $s_1$-connected to $B[i_\star,*]$).
        We can now choose $i_{t+1},\ldots,i_{2t}$ to be an arbitrary $1$-spaced subsequence of $(9t, 11t)$.
        This concludes the construction of $I_\star$.
        We further choose $s_2 \in S$ to be an arbitrary vertex witnessing the alternation point \(19t\) of $v$ and construct $J_\star$ analogously as a $1$-spaced subsequence of $(13t,\ldots,24t)$.
        
        For every $i \in [t]$ let $I_i$ be a subsequence of $I_\star$ of length $t$ such that $i_\star$ is the $i$th element of~$I_i$. 
        Such a sequence exists, since $i_\star$ has both $t$ successors and $t$ predecessors in $I_\star$. 
        Similarly, for every $j \in[t]$ let $J_j$ be a subsequence of $J_\star$ of length $t$ such that $j_\star$ is the $j$th element of $J_j$. 
        For every $i,j \in [t]$, by concatenating $I_i$ and $J_j$, we obtain a $1$-spaced subsequence $I_{i,j} = (p_1,\ldots,p_t,q_1,\ldots,q_t)$ of $I$ of length $2t$ such that 
        \begin{itemize}
            \item the $s_1$-connection from $v$ to $B[p_1,*], \ldots, B[p_{i-1},*]$ is homogeneous but switches at $B[p_i,*]$,
            \item the $s_2$-connection from $v$ to $B[q_1,*], \ldots, B[q_{j-1},*]$ is homogeneous but switches at $B[q_j,*]$.
        \end{itemize}
        By \cref{claim:shift} and witnessed by $v$ and all the $I_{i,j}$,
        we can fix an arbitrary $1$-spaced subsequence $(p_1,\ldots,p_t,q_1,\ldots,q_t)$ of $I$ of length $2t$ and  there exist vertices $c_{i,j}$ for all $i,j \in [t]$ such that either
        \begin{itemize}
            \item $\forall i,j\in[t]$: $c_{i,j}$ is not $s_1$-connected to $B[p_1,*], \ldots, B[p_{i-1},*]$ but $s_1$-connected to $B[p_i,*]$, or
            \item $\forall i,j\in[t]$: $c_{i,j}$ is $s_1$-connected to $B[p_1,*], \ldots, B[p_{i-1},*]$ but not $s_1$-connected to $B[p_i,*]$,
        \end{itemize}
        and similarly either
        \begin{itemize}
            \item $\forall i,j\in[t]$: $c_{i,j}$ is not $s_2$-connected to $B[q_1,*], \ldots, B[q_{j-1},*]$ but $s_2$-connected to $B[q_j,*]$, or
            \item $\forall i,j\in[t]$: $c_{i,j}$ is $s_2$-connected to $B[q_1,*], \ldots, B[q_{j-1},*]$ but not $s_2$-connected to $B[q_j,*]$.
        \end{itemize}

        Let \(\alpha_1(y;x,s_1)\)   be the quantifier-free formula checking whether $\atp(x/\{y\}) \neq \atp(s_1/\{y\})$. 
        Whenever $v$ is not $s_1$-connected to a column $B[i,*]$, then this is witnessed by an element $u \in B[i,*]$ such that 
        $G \models \alpha_1(u;v,s_1)$.
        If $v$ is $s$-connected to $B[i,*]$, then no such element exists in $B[i,*]$.
        Among $p_1,\ldots,p_t$, we have that either
        \begin{itemize}
            \item $\forall i,j\in[t]$: $p_i$ is the first index such that $\alpha_1(B[p_i,*];c_{i,j},s_1)$ is \emph{empty}
            
            (this happens in the case where the $c_{i,j}$ are $s_1$-connected to $B[{p_i},*]$), or

            \item $\forall i,j\in[t]$: $p_i$ is the first index such that $\alpha_1(B[p_i,*];c_{i,j},s_1)$ is \emph{non-empty}
            
            (this happens in the case where the $c_{i,j}$ are \emph{not} $s_1$-connected to $B[{p_i},*]$).
        \end{itemize}
        Similarly, there is a quantifier-free formula $\alpha_2(y,x,s_2)$ checking whether $\atp(x/\{y\}) \neq \atp(s_2/\{y\})$.
        Among $q_1,\ldots,q_t$, we have that either
        \begin{itemize}
            \item $\forall i,j\in[t]$: $q_j$ is the first index such that $\alpha_2(B[q_j,*];c_{i,j},s_2)$ is \emph{empty}, or

            \item $\forall i,j\in[t]$: $q_j$ is the first index such that $\alpha_2(B[q_j,*];c_{i,j},s_2)$ is \emph{non-empty}.
        \end{itemize}
        This proves that $G$ contains a bi-prepattern of order $t$ on $\gc B$ witnessed by the sequences $(p_1,\ldots,p_t)$, $(q_1,\ldots,q_t)$, the $c_{i,j}$ vertices, the parameters $s_1$, $s_2$, and the formulas $\alpha_1$ and $\alpha_2$.
    \end{claimproof}
    \paragraph*{Extracting a Sample Vertex.}
    If the second condition of \cref{claim:alternationpoints} holds, then 
    $G$ contains a bi-prepattern of order $t$ on $\gc A\vert_{I'}$, so $I'$ can play the role of the sequence $I$ in the statement of the lemma, and we are done. We therefore assume from now on that the first condition of \cref{claim:alternationpoints} holds.

    \begin{claim}\label{clm:nothings-are-close}
        One of the following two conditions holds.
        \begin{enumerate}
            \item For every vertex $u \in V(G)$, that is $\varnothing$-connected to columns $B[i,*]$ and $B[j,*]$ for $i,j \in I$, we have $|i - j| \leq 1$.
            In particular each vertex is $\varnothing$-connected to at most two columns in $B$.
            \item There is a sequence $K$ of length at least $\frac{1}{6}|I|$ and a vertex $v \notin \gc A\vert_K$, such that $v$ is $\varnothing$-connected to every column of $\gc A\vert_K$.
        \end{enumerate}
    \end{claim}
    \begin{claimproof}

        Assume the first condition fails. We have a vertex $u$ that is $\varnothing$-connected to columns $B[i,*]$ and $B[j,*]$ for $i + 1 < j \in I$.
        In order to show the second condition, 
        we will first find a vertex $v$ together with a sequence $K'\subseteq I$ of length at least $\frac{1}{3}|I|$, such that $v$ is $\varnothing$-connected to every column $B[p,*]$ with $p \in K$.
        As $(i,j)$ is a $1$-spaced subsequence of $I$,
        \cref{claim:shift} yields a vertex $v$ which is $\varnothing$-connected
        to $B[m, *]$ and $B[m + 2, *]$ where $m := \lfloor \frac{1}{2} |I| \rfloor$.
        Assume $v$ has at least one alternation point on $\gc B$ and its earliest alternation point $q$ satisfies $m \leq q$.
        Then $u$ has the same connection type to every $B[p,*]$ with $1 \leq p \leq m$.
        As $v$ is $\varnothing$-connected to $B[m, *]$, we can set $K' := (1, \ldots, m)$.
        Otherwise, either $v$ has no alternation points on $\gc B$, or the latest alternation point $q$ of $v$ satisfies $q \leq  m + 1$, by \cref{claim:alternationpoints}.
        Then $u$ has the same connection type to every $B[p,*]$ with $m + 2 \leq p \leq n$.
        As $v$ is $\varnothing$-connected to $B[m+2, *]$, we can set $K' := (m+2,\ldots,n)$.
        This finishes the construction of $v$ and~$K'$.
        We use \cref{lem:drop-elem} to obtain a sequence $K \subseteq K'$ of length $|K| \geq \frac{1}{2} |K'| \geq \frac{1}{6}|I|$ with $v \notin \gc A\vert_K$.
        By \cref{clm:disagree-subseq}, $v$ is $\varnothing$-connected to every column of $\gc A\vert_{K}$.
    \end{claimproof}

    \paragraph*{Establishing the Sampling Property.}
    If the second condition of \cref{clm:nothings-are-close}  holds, then there exists a sequence 
    $K$ that can play the role of $I$ in the statement of the lemma, and we are done. We therefore assume from now on that the first condition of \cref{clm:nothings-are-close}  holds.

    \begin{claim}\label{clm:establishing-sampling}
        $S$ samples \(G\) on $\gc A\vert_K$ with margin \(2\) for $K := (2,\ldots,n-2) \subseteq I$.
    \end{claim}
    \begin{claimproof}
        We have to choose for every vertex $v \in V(G)$ an exceptional index $i \in \tail(K) = (3\ldots,n-2)$ and two vertices $s_<, s_> \in S$ such that $v$ is $s_<$-connected to $A\vert_K[{<}i,*]$ and $s_>$-connected to $A\vert_K[{>}i + 1,*]$.
        By \cref{obs:transitivity}, we have $A\vert_K = B\vert_K$.
        Thus, by \cref{obs:mon-and-cov}, we have
        \[
            A\vert_K[{<}i,*] \subseteq B[\{3,\ldots, i-1\},*]
            \quad
            \text{and}
            \quad
            A\vert_K[{>} i + 1,*] \subseteq B[\{i + 2,\ldots, n-2\},*].
        \]
        Take any vertex $v$. By \cref{claim:alternationpoints}, 
        there exist successive indices $i_1 \in I$ and $i_2 := i_1 + 1$, and sets $S_1,S_2 \subseteq S$, such that $v$ is $S_1$-connected to all columns $B[{<}i_1,*]$ and $S_2$-connected to all columns $B[{>}i_2,*]$.
        
        \medskip\noindent
        Assume $S_1 \neq \varnothing \neq S_2$.
        Then we can arbitrarily choose $s_< \in S_1$ and $s_> \in S_2$ and set
        \[
            i :=
            \begin{cases}
                3& \text{if } i_1 < 3,\\
                n-2& \text{if } i_1 > n-2,\\
                i_1& \text{otherwise.}
            \end{cases}
        \]
        Now $v$ is $s_<$-connected to $B[\{3,\ldots, i-1\},*]$ and $s_>$-connected to $B[\{i + 2,\ldots, n-2\},*]$,
        proving the claim.
        
        \medskip\noindent
        Assume $S_1 = \varnothing \neq S_2$.
        If $i_1 > 3$ then $v$ is $\varnothing$-connected to the first three columns of $B$, contradicting \cref{clm:nothings-are-close}.
        So we have $i_1 \leq 3$. We set $i := 3$ and choose an arbitrary $s_> \in S_2$. 
        As desired, $v$ is $s_>$-connected to $B[\{i + 2,\ldots, n-2\},*]$.
        As $B[\{3,\ldots, i-1\},*]$ is empty any vertex from $S$ can take the role of~$s_<$.

        \medskip\noindent
        Assume $S_1 \neq \varnothing = S_2$.
        The proof is symmetric to the previous case.
        
        \medskip\noindent
        Assume $S_1 = \varnothing = S_2 $.
        Since $|I|>8$, we either find left of $i_1$ or right of $i_2$ at least three columns of $B$ to which $v$ is $\varnothing$-connected.
        This is a contradiction to \cref{clm:nothings-are-close}.
    \end{claimproof}
    We have successfully established the sampling property which proves that $K$ ($I$ in the statement) has the desired properties.

    \paragraph*{Running Time.}
    In the previous paragraphs, we have proven the existence of a sequence $I$ with the desired properties.
    Let us redefine $n := |V(G)|$.
    To show that $I$ can be constructed in time $O_{k,t}(n^2)$, we first consider the following preprocessing routine.
    Let $t_\star := \max(6,4t)$ and choose a subsequence \(J_0 \subseteq J\) of size \(\lfloor |J|^{1/t_\star} \rfloor\).
    Note that \(|J_0| \le n^{1/t_\star}\).
    By applying the construction to \(\gc A\vert_{J_0}\) and \(J_0\) instead of \(\gc A\) and \(J\),
    we obtain a subsequence \(I \subseteq J_0 \subseteq J\) with the desired properties that still has size \(|I| \ge U_{k,t}(|J_0|) \ge U_{k,t}(|J|)\).
    By this argument, and as we can build \(J_0\) and $\gc A\vert_{J_0}$ in time $O_{k,t}(n)$, 
    we can from now on assume without loss of generality that $|J| \leq n^{1/t_\star}$.

    Towards computing the coloring needed for the Ramsey application, we first compute for each $v \in V(G)$, $S_\star \subseteq S$, and $i \in J$, whether $v$ is $S_\star$-connected to the column $A[i,*]$ of $\gc A$.
    As the columns of \(A\) are disjoint and \(|S| \le k\),
    this takes a total time of $O(2^k \cdot n^2)$.
    Moreover, for each vertex $v \in V(G)$, subset $S_\star \subseteq S$ and pair $i < i' \in J$ we calculate whether $v$ is $S_\star$-connected to $M_A(i,i')$.
    By \cref{lem:subgrid-ramseyness},
    if $\gc A$ is orderless this amounts to checking whether $v$ is $S_\star$-connected to $A[i',*]$, which we have already computed in the previous step.
    If $\gc A$ is ordered, we instead check whether $v$ is $S_\star$-connected to each of $A[m,*]$ for $i < m \leq i'$.
    As $|J| \leq \sqrt{n}$ and with the data from the previous step we can do this check in time $O(\sqrt{n})$ for a single vertex $v$, set \(S_\star\) and pair $i,i'$.
    Since $|J| \leq n^{1/4}$, there are at most $\sqrt{n}$ pairs $i,i'$ that need to be checked.
    It follows that we can compute the desired data for all vertices $v$, sets $S_\star$, and pairs $i,i'$ in total time $O(2^k \cdot n \cdot \sqrt{n} \cdot \sqrt{n}) = O(2^k \cdot n^2)$. 
    We recall the construction of the coloring for the Ramsey application:
    For \(t' \in \{6,4t\}\) and \(S_1,\dots,S_{t'-1} \subseteq S\), the \(t'\)-tuples \(i_1 < \dots < i_{t'} \in J\) are labeled with a color 
    indicating whether
    \[
        \exists v \bigwedge_{l=1,\dots,t'-1} \text{ $v$ is $S_l$-connected to $M_A(i_l,i_{l+1})$}.
    \]
    Using our precomputed information, for a single $t'$-tuple, we can compute its colors in time $O_{k,t}(n)$ by iterating over all vertices $v \in V(G)$.    
    As $|J| \leq n^{1/t_\star}$, there are at most $n$ many $t'$-tuples from $J$, so we can compute the coloring in time $O_{k,t}(n^2)$.
    Due to the size bounds on $J$, applying Ramsey's Theorem (\cref{lem:reramsey}) to the coloring runs in time $O_{k,t}(n)$.
    This yields the sequence $I$.

    Having constructed $I$, we obtain the insulator $\gc B = \gc A\vert_{I'}$.
    By \cref{obs:determine-columns},
    our precomputed information can also be used to check whether a vertex $v \in V(G)$ is $S_\star$-connected to a column $B[i,*]$ of $\gc B$ for some $S_\star \subseteq S$ and $i \in I$. 
    Let us now show how to compute one of the three outcomes: a bi-prepattern, a new sample vertex $v$, or a large subsequence of $I$ on which $S$ samples $G$.

    \begin{itemize}
        \item We can check for every vertex, whether it contains two alternation points with distance bigger than $2$ on $\gc B$ in time $O_{k}(n \cdot |I|) \leq O_{k}(n^2)$.
        If this such a vertex exists, the proof of \cref{claim:alternationpoints} yields that there is a bi-prepattern on $\gc B$ of size $t$ on every $1$-spaced subsequence of length $2t$ of columns of $\gc B$. 
        We can choose any such sequence and search for the witnessing vertices in time $O_{k,t}(n)$.
    
        \item If the previous case does not apply, we can again search in time $O_{k}(n \cdot |I|) \leq O_{k}(n^2)$ for a new sample vertex $v$ and a corresponding subsequence of $I$ (cf. \cref{clm:nothings-are-close}).
        
        \item If none of the two previous cases apply, we immediately find the sampled subsequence by dropping the first and two last elements of $I$ (cf. \cref{clm:establishing-sampling}).
    \end{itemize}
    We have shown that each step of the construction can be carried out in time $O_{k,t}(n^2)$.
    This concludes the proof of \cref{lem:new-sample-vtx2}.
\end{proof}

\subsection{Extracting Small Sample Sets}

We use the following Ramsey-type result for set systems due to Ding, Oporowski, Oxley, and Vertigan~\cite[Cor. 2.4]{ding-unavoidable} (see also~\cite[Thm. 2]{gravier2004}).
We say a bipartite graph with sides $a_1,\ldots, a_\ell$ and $b_1,\ldots, b_\ell$ forms
\begin{itemize}
    \item a \emph{matching} of order $\ell$ if $a_i$ and $b_j$ are adjacent if and only if $i=j$ for all $i,j\in [\ell]$, 
    \item a \emph{co-matching} of order $\ell$ if $a_i$ and $b_j$ are adjacent if and only if $i\neq j$ for all $i,j\in [\ell]$, 
    \item a \emph{half-graph} of order $\ell$ if $a_i$ and $b_j$ are adjacent if and only if $i\leq j$ for all $i,j\in [\ell]$. 
\end{itemize}

We call two distinct vertices $u$ and $v$ \emph{twins} in a graph $G$ they have the same neighborhood, that is,  
$N(u) \setminus \{v\} =N(v) \setminus \{u\}$.

\begin{fact}[{\cite[Corollary 2.4]{ding-unavoidable}}]\label{thm:matching}
    There exists a function $Q:\NN\rightarrow\NN$ such that for every $\ell\in\NN$ and for every bipartite graph $G = (L,R,E)$, where $L$ has size at least $Q(\ell)$ and contains no twins, contains a matching, co-matching, or half-graph of order $\ell$ as an induced subgraph.
    
    \effective{Moreover, there is an algorithm that, given $G$, computes such an induced subgraph in time $O(|V(G)|^2).$}
\end{fact}

While the construction of \cite{ding-unavoidable} is algorithmic, no running time is stated for \cref{thm:matching}. 
To be self-contained, we instead deduce an algorithm a posteriori.

\begin{proof}[Proof of the running time of \cref{thm:matching}]
    Let $Q$ be the function given by the non-algorithmic part of \cref{thm:matching}.
    To prove an algorithmic version of the statement we weaken the bounds
    by demanding $L$ to have size
    at least $f^{-1}(Q(\ell))$ instead, where $f(x)= \lfloor\sqrt{\log(x)/2}\rfloor$.

    We first compute an induced subgraph $G'$ of $G$ on sets $L' \subseteq L$ and $R' \subseteq R$.
    To this end, choose $L'$ as an arbitrary subset of \(L\) of size $\lfloor\sqrt{\log(|L|)/2}\rfloor$.
    Choose $R' \subseteq R$ of size at most $|L'|^2 \leq \log(|L|)/2$ by picking for each pair of distinct vertices $u,v \in L'$
    a vertex from the symmetric difference of the neighborhoods of $u$ and $v$ in $R$.
    Since $L$ contains no twins in $G$, such a vertex always exists.
    $G'$ has size at most $\log(|L|) $ and can be constructed in time $O(|V(G)|^2)$. 

    %We choose $Q'$ to be the inverse function of $f(x)= \lfloor\sqrt{\log(x)/2}\rfloor$.
    By our choice of \(f\), we observe that $L'$ has size at least $Q(\ell)$.
    By construction, $L'$ contains no twins in $G'$. 
    By the non-algorithmic version of \cref{thm:matching}, we know that $G'$ contains a matching, co-matching, or half-graph of order $\ell$.
    Since $G'$ has at most $\log(|L|)$ vertices, we can perform a brute force search in time $O(|V(G)|^2)$.
    As $G'$ is an induced subgraph of $G$, the computed solution also applies to $G$.
\end{proof}

We can now prove \cref{lem:sampling}, which we restate for convenience.

\lemSampling*

\begin{proof}
    We will inductively compute subsequences $I_0, I_1, \ldots$ of $J$ and subsets $S_0, S_1, \ldots$ of $V(G)$ using the following procedure.
    For the base case we set $I_0 := J$ and $S_0 := \varnothing$.
    In the inductive step we are given $I_i$ and $S_i$ and apply \cref{lem:new-sample-vtx2} on $t$, $A\vert_{I_{i}}$, and $S_i$.
    This yields the subsequence $I_{i+1}$ of size $U_{|S_i|,t}(|I_i|)$ and the insulator $(\gc A\vert_{I_{i}})\vert_{I_{i+1}} = \gc A\vert_{I_{i+1}}$
    such that either
    \begin{enumerate}[leftmargin= 3em, label={(C.$\arabic*$)}]
        \item \label{itm:pattern-case} $G$ contains a prepattern of order $t$ on $\gc A\vert_{I_{i+1}}$, or
        \item \label{itm:sample-case} $S_i$ samples \(G\) on $\gc A\vert_{I_{i+1}}$ with margin \(2\), or
        \item \label{itm:vertex-case} there is a vertex $s_{i+1} \notin \gc A\vert_{I_{i+1}}$, such that for all $s \in S_i$ and every column $C$ in $\gc A\vert_{I_{i+1}}$
        \[
            \atp(s_{i+1}/C) \neq \atp(s/C).   
        \]
    \end{enumerate}
    In the first two cases, we stop the construction and set $I := I_{i+1}$ and $S := S_i$.
    In the third case, we continue the construction with $S_{i+1} := S_i \cup \{s_{i+1}\}$.

    \begin{claim}\label{clm:big-ivl-construction}
        For every $i$ we have
        \begin{enumerate}[leftmargin= 3em, label={\textup{(P.$\arabic*$)}}] 
            \item\label{itm:big-ivl-atp-diff} all vertices of $S_i$ have a pairwise different atomic type over every column of $\gc A\vert_{I_i}$,
            \item\label{itm:big-ivl-s-disjoint} no $s \in S_i$ is contained in $\gc A\vert_{I_i}$,
            \item\label{itm:big-ivl-size-s} $|S_i| = i$, and
            \item\label{itm:big-ivl-size-i} $I_i$ has size $U_{i,t}(J)$.
        \end{enumerate}
    \end{claim}

    \begin{claimproof}
    We prove the properties by induction on $i$.
    The base cases hold trivially.
    The properties \ref{itm:big-ivl-atp-diff} and \ref{itm:big-ivl-s-disjoint} hold on $S_i$ and $\gc A\vert_{I_i}$ by induction, on $S_i$ and $\gc A\vert_{I_{i+1}}$ by \cref{clm:disagree-subseq},
    and finally on $S_{i+1}$ and $\gc A\vert_{I_{i+1}}$ by the choice of $s_{i+1}$ in \ref{itm:vertex-case}.
    By \ref{itm:big-ivl-atp-diff}, all elements of $S_i$ are distinct. As we only add one element per turn, this proves \ref{itm:big-ivl-size-s}. 
    It follows that $I_{i+1}$ has size $U_{|S_i|,t}(|I_i|) = U_{i,t}(|I_i|)$ and by induction 
    $I_{i+1}$ has size $U_{i,t}(J)$, which proves \ref{itm:big-ivl-size-i}.
    \end{claimproof}

    Let $k := Q^{3t}(t)$, where $Q$ is the function given by \cref{thm:matching}.
    We have $k = \const(t)$.

    \begin{claim}\label{clm:find-pattern}
    If the construction runs for $k$ steps, then $G$ contains a prepattern of order $t$ on $\gc A\vert_{I_{k}}$.   
    \end{claim}

    \begin{claimproof}
        By \cref{clm:big-ivl-construction}, the set $S_k$ consists of $k$ vertices which have pairwise different atomic types with respect to every column of the grid $A$ of $\gc A\vert_{I_{k}}$.
        Since no vertex of $S_k$ is contained in $A$, we know that these vertices all have the same type with regard to the equality relation and must therefore have a pairwise different type with regard to the edge relation, that is, in the graph $G$, the vertices of $S_k$ have pairwise different neighborhoods in every column of $A$.
        Therefore, in the semi-induced bipartite graph between any subset $L \subseteq S_k$ and any column $R := A[i,*]$ of $A$, there are no twins in $L$ and the preconditions of \cref{thm:matching} are met.
        We iterate \cref{thm:matching} a total number of $3t$ times between $S_k$ and $3t$ columns of $A$, which we can choose arbitrarily. Finally, we apply the pigeonhole principle.
        This yields a size $t$ subset $S_\star \subseteq S_k$ and columns $C_1, \ldots, C_t$  of $A$ containing subsets $R_1, \ldots, R_t$ such that the semi-induced bipartite graph between $S_\star$ and~$R_i$ 
        \begin{itemize}
            \item is a matching for all $i \in [t]$, or
            \item is a co-matching for all $i \in [t]$, or
            \item is a half-graph for all $i \in [t]$.
        \end{itemize}
        This witnesses a mono-prepattern of order $t$ on $\gc A\vert_{I_k}$.
    \end{claimproof}

    We can now finish the proof.
    By \cref{clm:find-pattern}, if the construction runs for $k$ steps, we set $I := I_k$ and $G$ contains a prepattern of order $t$ on $A\vert_I$.
    By \cref{clm:big-ivl-construction}, $I_k$ has size $U_{k,t}(|J|)$. Since $k = \const(t)$, this is equivalent to $U_{t}(|J|)$, as desired.

    Otherwise, the construction terminates with $I := I_i$ and $S := S_i$ for some $i \leq k$,
    as either \ref{itm:pattern-case} or \ref{itm:sample-case} holds.
    By the same reasoning as before, we have $|I| \geq U_{t}(|J|)$. By \cref{clm:big-ivl-construction}, we have $|S| \leq k$ and $S \subseteq V(G) \setminus \gc A\vert_I$.
    In case \ref{itm:pattern-case}, we have a prepattern of order $t$ on $\gc A\vert_I$ and in case \ref{itm:sample-case}
    $S$ samples \(G\) on $\gc A\vert_{I}$ with margin \(2\).

    As \(k=\const(t)\), the running time of $O_{t}(|V(G)|^2)$ for the construction follows easily from the running times of \cref{lem:new-sample-vtx2} and \cref{thm:matching}.
\end{proof}

\subsection{Sample Sets for Orderless Insulators}

For orderless insulators we want to strengthen \cref{lem:sampling} by 
improving the guarantees given by the sampling set.
For convenience, we restate the definition of a sampling set.

\defSampled*
\defSamples*

We say \emph{$S$ symmetrically samples} $G$ on $\gc A$, if we can choose $s_< = s_>$ in the above definition.
For orderless insulators, we want to decrease the sampling margin to $1$ and make the sampling symmetric.

\begin{lemma}\label{lem:orderless-sample-margin}
    Let $G$ be a graph containing an orderless insulator $\gc A$ indexed by $J$, 
    and let $S \subseteq V(G)$ be a set that samples $G$ on $\gc A$ with margin $2$.
    Let $I \subseteq J$ be obtained by removing every second element of $J$. $S$ samples $G$ on $\gc A\vert_I$ with margin $1$.
\end{lemma}
\begin{proof}
    The exceptional indices of each vertex are successive. By only keeping every second column, we reduce the number of exceptional indices to at most $1$, which corresponds to margin~$1$.
\end{proof}

\begin{lemma}\label{lem:orderless-sampling}
    Fix \(t \in \N\).
    For every graph \(G\) and orderless insulator \(\gc A\) indexed by \(J\) in $G$, we can compute a subsequence \(I \subseteq J\) of size $U_{t}(|J|)$ and a set $S\subseteq V(G) \setminus \gc A\vert_I$ of size $\const(t)$ such that either
    \begin{itemize}
        \item $G$ contains a prepattern of order $t$ on $\gc A\vert_{I}$, or
        \item $\gc A\vert_{I}$ is orderable, or
        \item $S$ symmetrically samples $G$ on $\gc A\vert_{I}$ with margin $1$.
    \end{itemize}
    \effective{Moreover, there is an algorithm that, given $G$ and $\gc A$, computes the sequence $I$ and one of the three outcomes (a prepattern, a witnesses for $\gc A\vert_I$ being orderable, or a set $S$) in time $O_t(|V(G)|^2)$.
    } 
\end{lemma}

\begin{proof}
    We first apply \cref{lem:sampling} to $\gc A$ and $J$, which yields a sequence $I_0$ of length $U_{t}(J)$ and a set $S \subseteq V(G) \setminus \gc A\vert_{I_0}$ of size $\const(t)$ such that either 
    \begin{itemize}
        \item $G$ contains a prepattern of order $t$ on $\gc A\vert_{I_0}$, or
        \item $S$ samples $G$ on $\gc A\vert_{I_0}$ with margin $2$.
    \end{itemize}
    In the first case we are done by setting $I := I_0$, so assume the second case.
    \cref{lem:orderless-sample-margin} yields a sequence $I_1 \subseteq I_0$ of length $U_t(|I_0|)=U_t(|J|)$ such that $S$ samples $G$ on $\gc A\vert_{I_1}$ with margin $1$.
    Let $B$ be the grid of $\gc B := \gc A\vert_{I_1}$.
    %In order to again apply Ramsey's theorem,
    We color every element $i\in \tail(I_1)$ by a color that encodes for all $s_1,s_2 \in S$ the information whether
    \begin{equation}\label{eq:subset-ramsey}
        N(s_1) \cap B[i,*] \subseteq N(s_2) \cap B[i,*].
    \end{equation}
    This requires $|S|^2 = \const(t)$ many colors.
    Inducing \(\tail(I_1)\) on the largest color class
    yields a monochromatic subsequence $I_2 \subseteq \tail(I_1)$ of length $U_t(|I_1|)=U_t(|I_0|)=U_t(|J|)$, where we can interpret monochromaticity as follows.
    
    \begin{fact}\label{fact:subset-ramsey}
        For every $s_1,s_2 \in S$,
        if \(N(s_1) \cap B[i,*] \subseteq N(s_2) \cap B[i,*]\) for one $i \in I_2$, then it holds for every $i \in I_2$.
        In particular, if two vertices $s_1,s_2 \in S$ have the same neighborhood in one column $B[i,*]$ for some $i \in I_2$ then they have the same neighborhood in every column $B[i,*]$ with $i\in I_2$.
    \end{fact}

    By definition, $\gc A \vert_{I_2}$ consists of the columns $\{  B_{i,*} \mid i \in \tail(I_2)\}$.
    By \cref{lem:subgrid-sample}, $S$ also samples $G$ on $\gc A \vert_{I_2}$ with margin $1$.
    Whenever we have two vertices in $S$ that have the same neighborhood in every column of $\gc A \vert_{I_2}$ we can remove one of them from $S$ while still preserving that $S$ samples $G$ on $\gc A \vert_{I_2}$ with margin $1$.
    Thus by \cref{fact:subset-ramsey},
    for distinct vertices $s_1,s_2 \in S$ and for every $i \in I_2$ we have $N(s_1) \cap B[i,*] \neq
    N(s_2) \cap B[i,*]$.
    Together with \cref{fact:subset-ramsey} we obtain the following.
    
    \begin{fact}\label{fact:neq-ramsey}
        For distinct vertices $s_1,s_2 \in S$, either
    \[
        \forall i \in I_2 :
        \bigl(N(s_1) \cap B[i,*]\bigr) 
        \not\subseteq
        \bigl(N(s_2) \cap B[i,*]\bigr)
        \quad
        \text{or}
        \quad
        \forall i \in I_2 : 
        \bigl(N(s_1) \cap B[i,*] \bigr)
        \not\supseteq
        \bigl(N(s_2) \cap B[i,*]\bigr).
    \]
    \end{fact}

    Remember that a vertex $v \in V(G)$ is \emph{$s$-connected} to a set $U \subseteq V(G)$ if \(\atp(s/U) = \atp(v/U)\).
    In the next step we color every $4$-element subsequence $\bar \iota  = (\iota_{1}, \ldots, \iota_{4} )\subseteq I_2$.
    In order to apply Ramsey's theorem,
    we encode for every $4$-tuple $\bar s = s_1 \ldots s_4 \in S^4$ the information whether
    \begin{equation}\label{eq:alternation-ramsey}
        \exists v \bigwedge_{i \in [4]} 
        \text{$v$ is $s_i$-connected to $B[\iota_i,*]$}.
    \end{equation}
    Ramsey's theorem yields a monochromatic subsequence $I_3 \subseteq I_2$ of length $U_t(|J|)$, where we can interpret monochromaticity as follows.

    \begin{fact}\label{fact:alternation-ramsey}
        For every $\bar s \in S^4$, whenever there exists a $4$-element subsequence $\bar \iota$ of $I_3$ which satisfies~\eqref{eq:alternation-ramsey}, then every $4$-element subsequence $\bar \iota$ of $I_3$ satisfies \eqref{eq:alternation-ramsey}.
    \end{fact}

    By \cref{lem:subgrid-sample}, $S$ still samples \(G\) on $\gc A\vert_{I_3}$ with margin $1$.
    Therefore, we can choose for every vertex $v\in V(G)$ two samples $s_<(v), s_>(v) \in S$ and an exceptional index $\ex(v)$ such that 
    \begin{itemize}
        \item $v$ is $s_<(v)$-connected to all $B[i,*]$ with $i \in I_3$ and $i < \ex(v)$, and
        \item $v$ is $s_>(v)$-connected to all $B[i,*]$ with $i \in I_3$ and $i > \ex(v)$.
    \end{itemize}
    We remove the first and last two elements of $I_3$ to obtain $I_4$.
    If for every vertex $v$ with $\mathrm{ex}(v) \in I_4$ we have $s_<(v) = s_>(v)$,
    then $S$ symmetrically samples $G$ on $\gc A\vert_{I_4}$, and we can complete the proof by setting $I := I_4$.
    Assume therefore, there is a vertex $v_\star$ with $\mathrm{ex}(v_\star) \in I_4$ and $s_<(v_\star) \neq s_>(v_\star)$. Let $s_1 := s_<(v_\star)$ and $s_2 := s_>(v_\star)$.

    \begin{claim}\label{clm:distance-order-c}
        For every $i\in I_4$ there exists a vertex $c_i$ such that
        \begin{itemize}
            \item $c_i$ is $s_1$-connected to all $B[j,*]$ with $j \in I_3$ and $j < i$, and
            \item $c_i$ is $s_2$-connected to all $B[i,*]$ with $j \in I_3$ and $j > i$.
        \end{itemize} 
    \end{claim}

    \begin{claimproof}
        Let $\iota_1,\iota_2$ and $\iota_3,\iota_4$ be the two immediate predecessors and successors of $\ex(v_\star)$ in $I_3$. Those are distinct indices, and they exist, since $\ex(v_\star) \in I_4$ and $I_4$ was obtained from $I_3$ by removing the first and last two elements.
        Therefore, $\bar \iota := (\iota_1,\iota_2, \iota_3, \iota_4)$ is a $4$-element subsequence of $I_3$.
        Additionally, $v_\star$ is $s_1$-connected to $B[\iota_1,*], B[\iota_2,*]$ and $s_2$-connected to $B[\iota_3,*], B[\iota_4,*]$

        Pick \(i \in I_4\), and let $\iota'_1,\iota'_2$ and $\iota'_3,\iota'_4$ be the two immediate predecessors and successors of $i$ in $I_3$.
        It follows by \cref{fact:alternation-ramsey} that there is a vertex $c_i$ that is $s_1$-connected to $B[\iota'_1,*], B[\iota'_2,*]$ and $s_2$-connected to $B[\iota'_3,*], B[\iota'_4,*]$.
        Since $S$ samples \(G\) on $\gc A\vert_{I_3}$ with margin $1$, we have that $\ex(c_i) \in \{\iota'_2,i,\iota'_4\}$.
        Since $c_i$ is $s_1$-connected to $B[\iota'_1,*]$, it is $s_1$-connected to all $B[j,*]$ with $j \in I_3$ and $j \leq \iota'_1$. A symmetric statement holds for $s_2$ and $j \ge \iota'_4$.
        This proves that $c_i$ has the desired properties.
    \end{claimproof}

    Consider \cref{fact:neq-ramsey} for \(s_1\) and \(s_2\).
    We can assume 
    \( \forall i \in I_2 : \bigl(N(s_1) \cap B[i,*]\bigr) \not\supseteq \bigl(N(s_2) \cap B[i,*]\bigr) \),
    as the alternative case will follow by a symmetric argument.
    It follows that for every $i \in I_2$ there exists a vertex $b_i \in B[i,*]$ such that $b_i \in N(s_2) \setminus N(s_1)$.
    Now by \cref{clm:distance-order-c}, we have that for all $i,j \in I_4$ 
    \begin{itemize}
        \item $c_i$ is non-adjacent to $b_j$ if $j < i$, and
        \item $c_i$ is adjacent to $b_j$ if $j > i$.
    \end{itemize} 
    By a simple majority argument we find a subsequence $I_5 \subseteq I_4$ of length at least $\frac{1}{2} |I_4|  = U_t(|J|)$ such that either 
    \begin{itemize}
        \item $c_i$ and $b_j$ are adjacent if and only if $j \geq i$ for all $i,j \in I_5$, or
        \item $c_i$ and $b_j$ are adjacent if and only if $i > j$ for all $i,j \in I_5$.
    \end{itemize}
    By possibly dropping the first element from $I_5$ and shifting the indices of the $c_i$ by one, we can always assume the first case applies.
    Now $(b_i)_{i\in I_5}$ and $(c_i)_{i\in I_5}$ witness that $\gc A\vert_{I_5}$ is orderable as desired, so we can set $I:=I_5$.

    \paragraph*{Running Time.}
    Let $n := |V(G)|$.
    Using the same preprocessing as in the run time analysis of \cref{lem:new-sample-vtx2}, we can assume that $|J| \leq n^{1/4}$.
    We first apply \cref{lem:sampling}, which runs in time $O_t(n^2)$.
    Coloring of the elements of $\tail(I_1)$ and building \(I_2\) can be done in time $O_t(n)$.
    We then color the $4$-tuples of the resulting sequence.
    Similarly to the proof of \cref{lem:new-sample-vtx2}, this can be done in time $O_t(n^2)$.
    As we have ensured $|J| \leq n^{1/4}$, applying Ramsey to this coloring takes time $O_t(n)$.
    For the resulting sequence we can search for the $c_i$-vertices of \cref{clm:distance-order-c}.
    We do this by testing for every vertex $v \in V(G)$ the role of which of the $c_i$ it can take.
    For a single vertex this can be done in time $O_t(n)$ by checking its connections to each column of the insulator.
    An exhaustive search over all vertices in $G$ therefore takes time $O_t(n^2)$.
    If we find a suitable candidate for every $i \in I_4$, we can compute $I_5$ in time $O_t(n)$, which yields the desired orderable insulator.
    If we cannot find suitable candidates for all $c_i$, we can conclude that $S$ has the desired sampling property.
    The overall running time is $O_t(n^2)$ as desired.
\end{proof}

\section{Extending Insulators}
\label{sec:growing}

We are now ready to prove the insulator growing lemmas.
We first prove the orderless and then the ordered case.

\subsection{Extending Orderless Insulators}

\lemOrderlessGrowing*

\begin{proof}
    Let $\gc A = (A,\KK,F,F)$.
    Apply \cref{lem:orderless-sampling} to $\gc A$,
    which yields a subsequence \(I \subseteq J\) of size $U_{t}(|J|)$ and a set $S\subseteq V(G) \setminus \gc A\vert_I$ of size $\const(t)$ such that either
    \begin{itemize}
        \item $G$ contains a prepattern of order $t$ on $\gc A\vert_{I}$, or
        \item $\gc A\vert_{I}$ is orderable, or
        \item $S$ symmetrically samples $G$ on $\gc A\vert_{I}$ with margin $1$.
    \end{itemize}
    In the first two cases we are done, so we assume the last case:
    there exist functions $s : V(G) \rightarrow S$ and $\ex : V(G) \rightarrow \tail(I)$ such that every $v \in V(G)$ is $(1,\ex(v),s(v),s(v))$-sampled on $\gc A\vert_I$.
    Let $A$ and $B := A\vert_I$ be the grids of $\gc A$ and $\gc B := \gc A\vert_I$.
    Let $I_\star := \tail(I)$ be the sequence indexing $B$ and let $h$ be the height of $A$ and $B$. 
    
    \paragraph{Defining the Grid.}
    We build a row-extension $C$ of $B$. 
    By definition of a row-extension, we have $C[i,j] := B[i,j]$ for all $i \in I_\star$ and $j \in [h]$.
    It remains to define the row~$C[*,h+1]$.
    For every $i \in I_\star$, we define $C[i,h+1]$ to contain every vertex $v$ that
    \begin{itemize}
        \item is not contained in $B$, and
        \item disagrees with its sample in the cell below, that is,
        $
            \atp(v/C[i,h]) \neq \atp(s(v)/C[i,h]).   
        $
    \end{itemize}
    As every vertex $v$ is sampled with margin $1$ in $\gc B$, $v$ can disagree with $s(v)$ in at most one column of $B$, so no vertex gets assigned into multiple columns.
    Furthermore, we only assign vertices to $C[*,h+1]$ which were not in $B$, so the cells of $C$ are pairwise disjoint and $C$ is a valid grid.
    Thus, \(C\) is a row-extension of \(B\).

    \begin{claim}\label{clm:bg-exception}
        For every $i \in I_\star$ and $v \in C[i,*]$, we have $\ex(v) = i$.
    \end{claim}
    \begin{claimproof}
        As $v$ is $(1,\ex(v),s(v),s(v))$-sampled on $\gc B$, we have
        \[
            \atp(v/B[i,*]) \neq \atp(s(v)/B[i,*]) \Rightarrow \ex(v) = i.
        \]
        If $v \in C[i,h+1]$ then the premise is satisfied by construction.
        If $v \in C[i,\leq h] = B[i,*]$ then, since $s(v) \notin B[i,*]$, the premise is again satisfied as we have 
        \[
            (=,v) \in \atp(v/B[i,*]) \setminus \atp(s(v)/B[i,*]).\qedhere
        \]
    \end{claimproof}

    The rest of the proof will be devoted to constructing $\KK_\star$ and $F_\star$
    such that the row-extension $\gc C := (C,\KK_\star,F_\star,F_\star)$ is an insulator of cost $\const(k,t)$ in \(G\).

    \paragraph{Defining the Insulator.}
    We build $\KK_\star$ as a refinement of $\KK$ by encoding into the color of every vertex $v \in V(G)$ for every color $X \in \KK$ and sample vertex $s \in S$ the information
    \begin{enumerate}[leftmargin= 3em, label={(C.$\arabic*$)}]
        \item\label{itm:bg-col-x} whether $v \in X$,
        \item\label{itm:bg-col-c} whether $v \in C$,
        \item\label{itm:bg-col-c-h} whether $v \in C[*,h]$,
        \item\label{itm:bg-col-c-h1} whether $v \in C[*,h+1]$,
        \item\label{itm:bg-col-ns} whether $v \in N(s)$,
        \item\label{itm:bg-col-s} whether $s(v) = s$.
    \end{enumerate}
    As $|\KK| = k$ and $|S|=\const(t)$, we have $|\KK_\star| = \const(k,t)$.
    By \ref{itm:bg-col-s}, \ref{itm:bg-col-ns}, and \ref{itm:bg-col-x}, we can assign to every color $X \in \KK_\star$ 
    \begin{itemize}
        \item a sample vertex $s(X) \in S$, such that $s(v) = s(X)$ for all $v \in X$,
        \item sample neighbors $S(X)\subset S$, such that $N(v) \cap S = S(X)$ for all $v \in X$, and
        \item a color $\KK(X) \subseteq \KK$, such that $\KK(v) = \KK(X)$ for all $v \in X$.
    \end{itemize}
    In order to show that $\gc C$ is an insulator in \(G\), 
    it remains to define the symmetric relation $F_\star \subseteq \KK_\star^2$ such that property \ref{itm:orderless} is satisfied.
    We define \(F_\star\) via the following four cases.
    Let $X,Y \in \KK_\star$.
    \begin{enumerate}[leftmargin= 3em, label={(F.$\arabic*$)}]
        \item \label{itm:case-top-top}
        If $X \subseteq C[*,h]$ and $Y \subseteq C[*,h]$, then 
        \((X,Y) \in F_\star 
            \Leftrightarrow 
            \big(
            s(Y) \in S(X) 
            \vee
            s(X) \in S(Y) 
            \big) 
            .\)

        \item \label{itm:case-topper-top}
        If $X \not \subseteq C[*,{\leq} h]$ and $Y \subseteq C[*,h]$, then 
        \((X,Y) \in F_\star 
            \Leftrightarrow 
            s(X) \in S(Y) 
            .\)

        \item \label{itm:case-top-topper}
        If $X \subseteq C[*,h]$ and $Y \not \subseteq C[*,{\leq} h]$, then 
        \((X,Y) \in F_\star 
            \Leftrightarrow 
            s(Y) \in S(X) 
            .\)

        \item \label{itm:case-else}
        Otherwise, $(X,Y) \in F_\star \Leftrightarrow \big(\KK(X), \KK(Y)\big) \in F$.
    \end{enumerate}
    By construction, $F_\star$ is symmetric and therefore describes a valid flip.
    Let $G' := G \oplus_{\KK} F$ and $G_\star := G \oplus_{\KK_\star} F_\star$.

    \paragraph*{Proving Properties of Insulator.}
    We have to show \ref{itm:orderless}:
    for all $i \in I_\star$ there exists $a_i \in V(G)$ such that for all $r \in [h+1]$
    \begin{equation}\label{eq:bg-main-condition}
        C[i,1] = N^{G_{\star}}_{0}[a_i] = \{a_i\}  \quad \text{and} \quad C[i,{\leq} r] = N^{G_{\star}}_{r-1}[a_i].
    \end{equation}
    We first show that our flip conserves this property for $r \in [h]$, and handle $r = h +1$ later.

    \begin{claim} \label{clm:bg-preserve-h-balls}
        For all $i \in I_\star$ and $r \in [h]$, we have 
        \[
            C[i,{\leq} r] = B[i,{\leq} r] = N^{G'}_{r-1}[a_i] = N^{G_\star}_{r-1}[a_i] .
        \]
    \end{claim}
        
    \begin{claimproof}
        The first two equalities follow by construction and property \ref{itm:orderless} of $\gc B$.
        It remains to prove $N^{G'}_{r-1}[a_i] = N^{G_\star}_{r-1}[a_i]$.
        We prove the claim by induction on $r$. The base case is trivial.
        For the inductive step, assume the property holds for $r \in [h-1]$ and we want to show it for $r+1$.
        We show that for every vertex $u$ we have $u \in N^{G_\star}_{r}[a_i]$ if and only if $u \in N^{G'}_{r}[a_i]$.
        We can assume $u \notin N^{G_\star}_{r-1}[a_i] = N^{G'}_{r-1}[a_i]$, as we would be done by induction otherwise.
        With these prerequisites, the following are equivalent.
        \begin{enumerate}
            \item $u \in N^{G_\star}_{r}[a_i]$ 
            \item $u$ has a neighbor in $N^{G_{\star}}_{r-1}[a_i]$ in $G_\star$ 
            \hfill (as $u \notin N^{G_\star}_{r-1}[a_i]$)
            \item $u$ has a neighbor in $C[i, {\leq} r]$ in $G_\star$ 
            \hfill (by induction)
        \end{enumerate}
        Let $v$ be a vertex in $C[i, {\leq} r] \subseteq C[i, {<} h]$.
        By \ref{itm:bg-col-c}, \ref{itm:bg-col-c-h}, and \ref{itm:bg-col-c-h1}, we have $\KK_\star(v) \subseteq C[i, {<} h]$ and case \ref{itm:case-else} applies: $v$ has the same neighborhood in $G_\star$ as in $G'$. 
        Hence, the following are equivalent to the above.

        \begin{enumerate}[start=4]
            \item $u$ has a neighbor in $C[i, {\leq} r]$ in $G'$ 
            \item $u$ has a neighbor in $N^{G'}_{r-1}[a_i]$ in $G'$ 
            \hfill (by induction)
            \item $u \in N^{G'}_{r}[a_i]$ 
            \hfill (as $u \notin N^{G'}_{r-1}[a_i]$)
        \end{enumerate}        
    \end{claimproof}

    Having proved \cref{clm:bg-preserve-h-balls}, in order to establish \ref{itm:orderless}, it remains to prove
    \[
        C[i,h+1] = N^{G_\star}_{h}[a_i] \setminus N^{G_\star}_{h - 1}[a_i].
    \]
    We show the equivalence of the two sets by proving containment in both directions separately.

    \begin{claim}\label{clm:bg-roots}
        For all $i \in I_\star$ we have $C[i,h+1] \subseteq N^{G_\star}_{h}[a_i] \setminus N^{G_\star}_{h - 1}[a_i]$.
    \end{claim}

    \begin{claimproof}
        Let $u\in C[i,h+1]$. As $u \notin C[i,\leq h]$, we have by \Cref{clm:bg-preserve-h-balls} that $u \notin N^{G_\star}_{h - 1}[a_i]$. 
        It remains to show $u \in N^{G_\star}_{h}[a_i]$.
        By construction, we have
        \[
            \atp(u/C[i,h]) \neq \atp(s(u)/C[i,h]).
        \]
        As neither $u$ nor $s(u)$ is contained in $B$, the difference in their atomic type must be witnessed by a vertex $v\in C[i,h]$ in the symmetric difference of their neighborhoods.
        We want to argue that $u$ and $v$ are adjacent in $G_\star$.
        Let $X := \KK_\star(u)$ and $Y := \KK_\star(v)$.
        By \ref{itm:bg-col-c-h1} and \ref{itm:bg-col-c-h},
        we have $X \subseteq C[i,h+1]$ and $Y \subseteq C[i,h]$.
        Hence, case \ref{itm:case-topper-top} from the construction of $F_\star$ applies and the following are equivalent.
        \begin{enumerate}
            \item The adjacency between $u$ and $v$ got flipped when going from $G$ to $G_\star$.
            \item $s(X) \in S(Y)$.
            \hfill (by \ref{itm:case-topper-top})
            \item $s(u) \in N(v) \cap S$.
            \hfill (by definition)
            \item $s(u)$ is a neighbor of $v$ in $G$.
            \hfill (by definition)
            \item $u$ is a non-neighbor of $v$ in $G$.
            \hfill ($v$ is in the sym. diff. of $N(s(u))$ and $N(u)$)
        \end{enumerate}
        The equivalence between the first and the last item establishes that $u$ and $v$ are adjacent in $G_\star$.
        By \cref{clm:bg-preserve-h-balls}, we have $v \in N^{G_\star}_{h-1}[a_i]$, so $u \in N^{G_\star}_{h}[a_i]$ and the claim is proven.
    \end{claimproof}

    \begin{claim}\label{clm:bg-disjoint}
        For all $i \in I_\star$ we have $C[i,h+1] \supseteq N^{G_\star}_{h}[a_i] \setminus N^{G_\star}_{h - 1}[a_i]$.
    \end{claim}
    \begin{claimproof}
        Let $u$ be a vertex in $N^{G_\star}_{h}[a_i] \setminus N^{G_\star}_{h - 1}[a_i]$.
        By \cref{clm:bg-preserve-h-balls}, this is witnessed by a vertex 
        \[
            v \in N^{G_\star}_{h - 1}[a_i] = N^{G'}_{h - 1}[a_i] = C[i,h]    
        \]
        that is adjacent to $u$ in $G_\star$.
        We prove that $u \in C[i,h+1]$ by ruling out all other possibilities.
        \begin{itemize}
            \item
            Assume that $u \in C[j,\leq h - 1]$ for some $j \in I_\star$.
            By \cref{clm:bg-preserve-h-balls}, we have
            \[
                u \in N^{G_\star}_{h - 2}[a_j] = N^{G'}_{h - 2}[a_j].
            \]
            As we assumed that $u \notin N^{G_\star}_{h - 1}[a_i]$, we have that
            $j \neq i$.
            Additionally, case \ref{itm:case-else} applies and $u$ has the same neighborhood in $G'$ and $G_\star$. Hence, $u$ is adjacent to $v$ also in $G'$.
            Now $v \in N^{G'}_{h - 1}(a_j) = C[j,{\leq} h]$, but still $v \in C[i,h]$. This is a contradiction to $C[i,h]$ and $C[j,{\leq} h]$ being disjoint.

            \item
            Assume that $u \notin C$ or $u \in C[j,h + 1]$ for some $j \neq i \in I_\star$.
            We first show
            \begin{equation}\label{eq:bg-topper-top}
                \atp(u/C[i,h]) = \atp(s(u)/C[i,h]).
            \end{equation}
            If $u\notin C$, we deduce \eqref{eq:bg-topper-top} from the construction of $C$.
            If $u \in C[j,h + 1]$, we apply \cref{clm:bg-exception} which yields $j = \ex(u) \neq i$. We then deduce \eqref{eq:bg-topper-top} by the sampling property.
            Let $X := \KK_\star(u)$ and $Y := \KK_\star(v)$.
            By our choice of $u$ and $v$, we know their adjacency in $G_\star$ was determined by case \ref{itm:case-topper-top} and the following are equivalent.

            \begin{enumerate}
                \item The adjacency between $u$ and $v$ got flipped when going from $G$ to $G_\star$.
                \item $s(X) \in S(Y)$. 
                \hfill (by \ref{itm:case-topper-top})
                \item $s(u) \in N(v) \cap S$
                \hfill (by definition)
                \item $v$ is adjacent to $s(u)$ in $G$.
                \hfill (by definition)
                \item $v$ is adjacent to $u$ in $G$.
                \hfill (by \eqref{eq:bg-topper-top})
            \end{enumerate}
            The equivalence between the first and the last item establishes that $u$ and $v$ are non-adjacent in $G_\star$, a contradiction.

            \item
            Finally, we assume that $u \in C[j,h]$ for some $j\in I_\star$.
        As we know that $u \notin N^{G_\star}_{h - 1}[a_i]$, \cref{clm:bg-preserve-h-balls} yields $i \neq j$.
        Then \cref{clm:bg-exception} applied to $u \in C[j,h]$ and $v \in C[i,h]$  yields 
        \[
            j = \ex(u) \neq i = \ex(v),    
        \]
        which together with the sampling property gives
        \begin{equation}\label{eq:bg-top-top}
            \atp(u/C[i,h]) = \atp(s(u)/C[i,h]) \quad \text{and} \quad \atp(v/C[j,h]) = \atp(s(v)/C[j,h]).
        \end{equation}
        Let $X := \KK_\star(u)$ and $Y := \KK_\star(v)$.
        By our choice of $u$ and $v$, we know their adjacency in $G_\star$ was determined by case \ref{itm:case-top-top} and the following are equivalent.
        \begin{enumerate}
            \item The adjacency between $u$ and $v$ got flipped when going from $G$ to $G_\star$.
            \item $s(Y) \in S(X)$ or $s(X) \in S(Y)$. 
            \hfill (by \ref{itm:case-top-top})
            \item $s(v) \in N(u) \cap S$ or $s(u) \in N(v) \cap S$. 
            \hfill (by definition)
            \item $u$ is a neighbor of $s(v)$ in $G$ or $v$ is a neighbor of $s(u)$ in $G$.
            \hfill (by definition)
            \item $u$ is a neighbor of $v$ in $G$ or $v$ is a neighbor of $u$ in $G$.
            \hfill (by \eqref{eq:bg-top-top})
            \item $u$ is a neighbor of $v$ in $G$.
        \end{enumerate}
        The equivalence between the first and the last item establishes that $u$ and $v$ are non-adjacent in $G_\star$, a contradiction.

        \end{itemize}
        Having exhausted all other possibilities, we conclude that $u \in C[i,h+1]$, which proves the claim.
    \end{claimproof}
        
    The combination of \cref{clm:bg-preserve-h-balls}, \cref{clm:bg-roots}, and \cref{clm:bg-disjoint} proves 
    property \ref{itm:orderless}.
    Hence, $\gc C := (C,\KK_\star,F_\star,F_\star)$ is an insulator of cost $\const(k,t)$ in \(G\).
    This proves that $\CC$ is the desired row extension. 
    It remains to analyze the running time.

    \paragraph*{Running Time.}
    Let $n := |V(G)|$.
    The application of \cref{lem:orderless-sampling} runs in time $O_t(n^2)$.
    In the case where a sample set $S$ is returned, we can calculate the
    witnessing function $s$ in time $O_t(n^2)$ by comparing each vertex $v \in
    V(G)$ with each of the $\const(t)$ many vertices from $S$ over every column
    of~$\gc A\vert_I$.
    With the function $s$ at hand, we can also build the row-extension $C$ of $B$ in time $O_t(n^2)$.
    The construction of $\KK_\star$ and $F_\star$ runs in time $O_{k,t}(n)$.
    This yields an overall running time of $O_{k,t}(n^2)$.
\end{proof}

\subsection{Extending Ordered Insulators}

\lemOrderedGrowing*

\begin{proof}
    We first apply \cref{lem:sampling} to $\gc A$, which yields a subsequence \(I \subseteq J\) of size $U_{t}(|J|)$ and a set $S\subseteq V(G) \setminus \gc A\vert_I$ of size $\const(t)$ such that either
    \begin{itemize}
        \item $G$ contains a prepattern of order $t$ on $\gc A\vert_{I}$, or
        \item $S$ samples $G$ on $\gc A\vert_I$ with margin $2$.
    \end{itemize}
    In the first case we are done, so assume the second case.
    By possibly taking a subsequence and applying \cref{lem:subgrid-sample}, we can assume the following. 
    \begin{property}\label{prop:drop-borders}
        $I$ does not contain the first and last two elements of $J$.
    \end{property}
    Let $B$ be the grid of $\gc B := \gc A\vert_I$ and let $I_\star := \tail(I)$ be the sequence indexing $B$.
    As $S$ samples $G$ on $\gc B$ with margin $2$, there exist functions
    $\ex : V(G) \rightarrow I_\star$ and $s_<, s_> : V(G) \rightarrow S$, such that every $v \in V(G)$ is $(2,\ex(v),s_<(v),s_>(v))$-sampled in $\gc B$.
    We can assume $\ex$ is chosen maximal in the following sense: for every vertex $v \in V(G)$ and index $i \in I_\star$, if $\atp(v/B[{\leq} i,*]) = \atp(s_<(v)/B[{\leq} i,*])$ then $\ex(v) \geq i$.

    \paragraph{Defining the Grid.}
    We will build a row-extension $C$ of $B$. 
    By definition of a row-extension we have $C[i,r] := B[i,r]$ for all $i \in I_\star$ and $r \in [h]$.
    It remains to define the row~\mbox{$C[*,h+1]$}.
    For every $i \in I_\star$, we set $C[i,h+1]$ to contain every vertex $v$ such that
    \begin{itemize}
        \item $v$ is not contained in $B$, and
        \item $\atp(v/B[i,h]) \neq \atp(s_<(v)/B[i,h])$ and $i$ is the minimal index in $I_\star$ with this property.
    \end{itemize}
    It is easy to see that cells of $C$ are pairwise disjoint and $C$ is a valid grid.
    Furthermore, we have the following property.

    \begin{claim}\label{clm:og-exception}
        For every $i \in I_\star$ and $v \in C[i,*]$, we have $\ex(v) \in \{ i-1,i \}$.
    \end{claim}
    \begin{claimproof}
        If $v \in C[i,\leq h] = B[i,*]$ then, since $B[i,*]$ contains $v$ but neither $s_<(v)$ nor $s_>(v)$,
        the atomic type of $v$ differs from the atomic types of both $s_<(v)$ and $s_>(v)$ over $B[i,*]$. By the sampling property with margin $2$, we have that $i = \ex(v)$ or $i = \ex(v) + 1$.

        If $v \in C[i,h+1]$, then by construction the atomic type of $v$ differs from the atomic type of $s_<(v)$ over $B[i,*]$.
        This yields $\ex(v) \leq i$.
        As we have chosen $i$ minimal, we have
        $\atp(v/B[{\leq}  i-1,*]) = \atp(s_<(v)/B[{\leq} i-1,*])$.
        As we have chosen $\ex$ maximal, we have $\ex(v) \geq i-1$.
    \end{claimproof}

    The rest of the proof will be devoted to constructing $\KK_\star$, $F_\star$, and $R_\star$ such that the row-extension $\gc C := (C,\KK_\star,F_\star,R_\star)$ is an insulator of cost $\const(k,t)$ in \(G\).

    \paragraph*{Defining the Insulator.}
    We build $\KK_\star$ as a refinement of $\KK$ by encoding into the color of every vertex $v \in V(G)$ for every color $X \in \KK$ and sample vertex $s \in S$ the information
    \begin{enumerate}[leftmargin= 3em, label={(C.$\arabic*$)}]
        \item \label{itm:og-col-x} whether $v \in X$,
        \item \label{itm:og-col-c} whether $v \in C$,
        \item \label{itm:og-col-c-h} whether $v \in C[*,h]$,
        \item \label{itm:og-col-c-h1} whether $v \in C[*,h+1]$,
        \item \label{itm:og-col-s} whether $v \in N(s)$,
        \item \label{itm:og-col-intb} whether $v \in \mathrm{int}(B)$,
        \item \label{itm:og-col-left-sample} whether $s_<(v) = s$,
        \item \label{itm:og-col-right-sample}whether $s_>(v) = s$,
        \item \label{itm:og-col-outside} whether $v$ has a neighbor in $X \cap \mathrm{int}(B)$.
        % \item whether $v$ is adjacent to a vertex from $X \cap \mathrm{int}(A)$,
    \end{enumerate}
    %Let $k_\star$ be the number of colors needed to encode the above information.
    As $\KK$ has size $k$ and $S$ has size $\const(t)$, we have $|\KK_\star| = \const(k,t)$.
    By \ref{itm:og-col-left-sample}, \ref{itm:og-col-right-sample}, \ref{itm:og-col-s}, and \ref{itm:og-col-x}, we can define for every color $X \in \KK_\star$
    \begin{itemize}
        \item a left sample $s_<(X) \in S$, such that for all $v \in X$ we have $s_<(v) = s_<(X)$,
        \item a right sample $s_>(X) \in S$, such that for all $v \in X$ we have $s_>(v) = s_>(X)$,
        \item sample neighbors $S(X) \subseteq S$, such that for all $v \in X$ we have $N(v) \cap S = S(X)$, and
        \item a color $\KK(X) \in \KK$, such that for all $v \in X$ we have $\KK(v) = \KK(X)$.
    \end{itemize}

    \medskip\noindent
    We define \(F_\star \subseteq \KK_\star^2\) via the following four cases.
    Let $X,Y \in \KK_\star$.
    \begin{enumerate}[leftmargin= 3em, label={(F.$\arabic*$)}]
        \item \label{itm:og-case-topper-top}
        If $X \subseteq C[*,h + 1]$ and $Y \subseteq C[*,h]$, then
        $
            (X,Y) \in F_\star 
            \Leftrightarrow 
            s_<(X) \in S(Y) 
            .
        $

        \item \label{itm:og-case-top-topper}
        If $X \subseteq C[*,h]$ and $Y \subseteq C[*,h+1]$, then
        $
            (X,Y) \in F_\star 
            \Leftrightarrow 
            s_<(Y) \in S(X) 
            .
        $

        \item \label{itm:og-case-topper-bot}
        If $X \subseteq C[*,h + 1]$ and $Y \subseteq C[*,{<}h]$, or vice-versa, then
        \[
            (X,Y) \in F_\star 
            \Leftrightarrow 
            \text{there is an edge between $X$ and $Y$ in $G$}
            .
        \]
        \item \label{itm:og-case-else}
        Otherwise, $(X,Y) \in F_\star \Leftrightarrow \big(\KK(X), \KK(Y)\big) \in F$.
    \end{enumerate}
    In order for $F_\star$ to define a valid flip, $F_\star$ has to be symmetric.
    This is satisfied by our definition. The cases \ref{itm:og-case-topper-top} and \ref{itm:og-case-top-topper} are dual and for \ref{itm:og-case-topper-bot} and \ref{itm:og-case-else} the symmetry follows from the symmetry of their conditions and the symmetry of the edge relation and $F$.

    \medskip\noindent
    We define \(R_\star \subseteq \KK_\star^2\) via the following three cases.

    \begin{enumerate}[leftmargin= 3em, label={(R.$\arabic*$)}]
        \item \label{itm:og-r-case-topper-bot}
        If $X \subseteq C[*,h + 1]$ and $Y \subseteq C[*,{\leq} h]$, or
        \item \label{itm:og-r-case-top-top}
        if $X \subseteq C[*,h]$ and $Y \subseteq C[*,h]$, then
        \[
            (X,Y) \in R_\star 
            \Leftrightarrow 
            s_>(X) \in S(Y)
            .
        \]

        \item \label{itm:og-r-case-else}
        Otherwise, $(X,Y) \in R_\star \Leftrightarrow \big(\KK(X), \KK(Y)\big) \in R$.
    \end{enumerate}

    \paragraph*{Proving Properties of the Insulator.}
    Let $G' := G \oplus_\KK F$ and $G_\star := G \oplus_{\KK_\star} F_\star$.
    We prove the various insulator properties.

    \begin{claim}[\ref{itm:consistent-rows}]
        Every two vertices in different rows of $C$ have different colors in~$\KK_\star$.
    \end{claim}
    \begin{claimproof}
        If none of the two vertices is in $C[*,h+1]$ then the property holds as $\KK_\star$ is a refinement of $\KK$, which satisfied this property in $B$.
        Otherwise, exactly one of them is contained in $C[*,h+1]$ and we can distinguish them using \ref{itm:og-col-c-h1}.
    \end{claimproof}

    \begin{claim}[\ref{itm:rootedness}]
        Every vertex $u \in C[i,r]$ with $r>1$ has a neighbor in the cell $C[i,r-1]$ in~$G_\star$.
    \end{claim}
    
    \begin{claimproof}
        If $r \leq h$, then by \ref{itm:rootedness} of $\gc B$, there is a vertex $v \in C[i,r-1]$ that is adjacent to $u$ in $G'$.
        By \ref{itm:og-case-else}, $u$ and $v$ are also adjacent in $G_\star$.

        It remains to check the case where $u \in C[i,h+1]$.
        By construction there exists a vertex $v \in C[i,h]$ in the symmetric difference of $N(u)$ and $N(s_<(u))$.
        By \ref{itm:og-col-c-h1} and \ref{itm:og-col-c-h}, we have $\KK_\star(u) \subseteq C[*,h+1]$ and $\KK_\star(v) \subseteq C[*,h]$. 
        Case \ref{itm:og-case-topper-top} applies, and the following are equivalent.
        \begin{enumerate}
            \item The adjacency between $u$ and $v$ got flipped from $G$ to $G_\star$.
            \item $s_<(\KK(u)) \in S(\KK(v))$.
            \hfill{(by \ref{itm:og-case-topper-top})}
            \item $s_<(u) \in N(v) \cap S$.
            \hfill{(by definition)}
            \item $s_<(u)$ is adjacent to $v$ in $G$.
            \hfill{(by definition)}
            \item $u$ is non-adjacent to $v$ in $G$.
            \hfill ($v$ is in the sym. diff. of $N(s_<(u))$ and $N(u)$)
        \end{enumerate}
        The equivalence between the first and the last item establishes that $u$ and $v$ are adjacent in~$G_\star$.
    \end{claimproof}

    \begin{claim}[\ref{itm:outside}]
        For every $v \notin C$ and $X \in \KK_\star$, $v$ is homogeneous to $X \cap \mathrm{int}(C)$ in $G$.
    \end{claim}
    
    \begin{claimproof}
        Let $X_C := X \cap \int(C)$.
        As $C$ is an extension of $B$, we also have $v \notin B$. 
        By construction of $C$ and \ref{itm:og-col-intb} and \ref{itm:og-col-c-h}, we have either $X_C \subseteq \int(B)$ or $X_C \subseteq C[*,h] = B[*,h]$.
        In the first case we conclude by \ref{itm:outside} of $\gc B$.
        In the second case, since $v$ did not get sorted into $C$, we know by construction
        \[
            \atp(v/C[*,h]) = \atp(s_<(v)/C[*,h]).
        \]
        By \ref{itm:og-col-s}, $s_<(v)$ is homogeneous to $X$, and so is $v$. 
    \end{claimproof}
    
    In the following claims we prove \ref{itm:adjacency}.
    For this purpose let $u \in C[i,r]$ for some $i \in I_\star$ and $r < h + 1$, and let $v\in C$.
    Up to renaming, we additionally assume $I_\star = (1,\ldots,n)$.

    \begin{claim}[\ref{itm:adj-different-rows}]
        If $u \in \int(C)$ and $u$ and $v$ are in rows that are not close, then they are non-adjacent in $G_\star$.
    \end{claim}
    \begin{claimproof}
        Assume first $v \in C[*, h+1]$. Then $u \in C[*,{<}h]$, as $u$ and $v$ are in rows that are not close.
        By construction of $C$ this yields $u \in \int(B)$
        and $v \notin B$.
        Using the properties of our coloring, we conclude the following.
        Let $X := \KK_\star(u)$ and $Y := \KK_\star(v)$.
        \begin{itemize}
            \item $X \subseteq \KK(X) \cap \int(B)$.
            \hfill (by \ref{itm:og-col-x} and \ref{itm:og-col-intb})
            \item No vertex of $Y$ is contained in $B$. 
            \hfill (by \ref{itm:og-col-c})
            \item Every or no vertex in $Y$ has a neighbor in $\KK(X) \cap \int(B)$.
            \hfill (by \ref{itm:og-col-outside})
        \end{itemize}
        Combining the above with \ref{itm:outside} of $\gc B$, we know that the connection between $X$ and $Y$ is homogeneous in $G$.
        Additionally, $X\subseteq C[*,{<}h]$ and $Y \subseteq C[*,h+1]$ so 
        case \ref{itm:og-case-topper-bot} applies and the following are equivalent.

        \begin{enumerate}
            \item The adjacency between $u$ and $v$ got flipped from $G$ to $G_\star$.
            \item There is an edge between $X$ and $Y$ in $G$. \hfill (by \ref{itm:og-case-topper-bot})
            \item There is an edge between $u$ and $v$ in $G$. \hfill (as $X \ni u$ and $Y\ni v$ are homogeneous)
        \end{enumerate}
        The equivalence between the first and the last item establishes that $u$ and $v$ are non-adjacent in~$G_\star$.

        If $v \notin C[*, h+1]$, then since also $u \notin C[*, h+1]$, case \ref{itm:og-case-else} applies and $u$ and $v$ have the same adjacency in $G_\star$ as in $G'$.
        Note that in this case, by construction and \cref{obs:mon-and-cov}, $u$ and $v$ 
        are both contained in both grids $A$ and $B$ and in both grids, they are in rows that are not close.
        If $v \in C[*, h]$, then again $u\in \int(B)$ and $u$ and $v$ non-adjacent in $G'$ by \ref{itm:adj-different-rows} of $\gc B$.
        Otherwise, $v \in C[*,{<}h]$
        and by \cref{prop:drop-borders} and \cref{obs:mon-and-cov} we have $v\in \int(A)$.
        Now $u$ and $v$ non-adjacent in $G'$ by \ref{itm:adj-different-rows} of $\gc A$.

    \end{claimproof}

    \begin{claim}[\ref{itm:adj-bot-left}]
        If $v \in C[{<}i,r-1] \cup C[{>}i,r+1]$, then $u$ and $v$ are non-adjacent in $G_\star$.
    \end{claim}
    \begin{claimproof}
        If at least one of $u$ and $v$ is contained in $C[*,{<}h] = B[*,{<}h]$, then case \ref{itm:og-case-else} applies and $u$ and $v$ have the same adjacency in $G_\star$ as in $G'$.
        Possibly exchanging the roles of $u$ and $v$, we can apply \ref{itm:adj-bot-left} of $\gc B$ to deduce that $u$ and $v$ are non-adjacent in $G'$.

        In the remaining case we have $u\in C[i,h] = B[i,h]$ and $v \in C[{>}i,h+1]$ and
        by construction of~$C$:
        \[
            \atp(v/B[i,h]) = \atp(s_<(v)/B[i,h]).
        \]
        Also, $\KK_\star(u) \subseteq C[*,h]$ and $\KK_\star(v) \subseteq C[*,h+1]$ by \ref{itm:og-col-c-h} and \ref{itm:og-col-c-h1}.
        Hence, case \ref{itm:og-case-top-topper} applies, and the following are equivalent.

        \begin{enumerate}
            \item The adjacency between $u$ and $v$ got flipped from $G$ to $G_\star$.
            \item $s_<(\KK_\star(v)) \in S(\KK_\star(u))$. 
            \hfill (by \ref{itm:og-case-top-topper})
            \item $s_<(v) \in N(u) \cap S$. 
            \hfill (by definition)
            \item $u$ and $s_<(v)$ are adjacent in $G$.
            \hfill (by definition)
            \item $u$ and $v$ are adjacent in $G$.
            \hfill (by $\atp(v/B[i,h]) = \atp(s_<(v)/B[i,h])$)
        \end{enumerate}
        The equivalence of the first and the last item establishes that $u$ and $v$ are non-adjacent in~$G_\star$.

    \end{claimproof}

    \begin{claim}[\ref{itm:adj-left}]
        If $v \in C[{>}i+1, \{r,r-1\}]$, then $G \models E(u,v) \Leftrightarrow (\KK_\star(u),\KK_\star(v))\in R_\star$.
    \end{claim}
    
    \begin{claimproof}
        If either $u \in C[*,<h]$ or $v \in C[*,<h]$, then case \ref{itm:og-r-case-else} applies, 
        we have 
        \[
            (\KK_\star(u),\KK_\star(v)) \in R_\star 
            \Leftrightarrow
            (\KK(u),\KK(v)) \in R,
        \]
        and it remains to establish 
        \[
            G \models E(u,v)
            \Leftrightarrow
            (\KK(u),\KK(v)) \in R.
        \]
        If $u \in C[*,{<}h]$, we argue using \ref{itm:adj-left} of $\gc B$.
        Otherwise, we have $v \in C[*,{<}h]$ and argue using \ref{itm:adj-right} of $\gc B$, where we exchange the roles of $u$ and $v$.

        We can now assume $u,v \in C[*,h]$.
        By \ref{itm:og-col-c-h}, also $\KK_\star(u), \KK_\star(v) \subseteq C[*,h]$, and case \ref{itm:og-r-case-top-top} applies.
        By assumption, we have $v \in C[i',h]$ for $i + 1 <i'$.
        By \cref{clm:og-exception}, we have that $\ex(u) + 1 < i'$, and the following are equivalent.
        \begin{enumerate}
            \item $u$ and $v$ are adjacent in $G$.
            \item $s_>(u)$ and $v$ are adjacent in $G$.
            \hfill (by the sampling property)
            \item $s_>(u) \in N(v) \cap S$.
            \hfill (by definition)
            \item $s_>(\KK_\star(u)) \in S(\KK_\star(v))$.
            \hfill (by definition)
            \item $(\KK_\star(u),\KK_\star(v)) \in R_\star$.
            \hfill (by \ref{itm:og-r-case-top-top})
        \end{enumerate}
    \end{claimproof}

    \begin{claim}[\ref{itm:adj-right}]
        If $v \in C[{<}i-1, \{r,r+1\}]$, then $G \models E(u,v) \Leftrightarrow (\KK_\star(v),\KK_\star(u))\in R_\star$.
    \end{claim}
    \begin{claimproof}
        If either 
        \begin{itemize}
            \item one of $u$ and $v$ is contained in $C[*,{<}h]$, or
            \item both $u$ and $v$ are contained in $C[*,h]$,
        \end{itemize}
        then we can exchange $u$ and $v$ and the property follows from the already established property \ref{itm:adj-left}.
        It remains to prove the case where $u \in C[i,h]$ and $v \in C[{<}i-1, h+1]$.
        We have $\KK_\star(u) \subseteq C[*,h]$ and $\KK_\star(v) \subseteq C[*,h+1]$, and case \ref{itm:og-r-case-topper-bot} applies for $X = \KK_\star(v)$ and $Y = \KK_\star(v)$.
        By assumption, we have $v \in C[i',h + 1]$ for $i' + 1 < i$.
        By \cref{clm:og-exception}, we have that $\ex(v) + 1 < i$, and the following are equivalent.
        \begin{enumerate}
            \item $u$ and $v$ are adjacent in $G$.
            \item $u$ and $s_>(v)$ are adjacent in $G$.
            \hfill (by the sampling property)
            \item $s_>(v) \in N(u) \cap S$.
            \hfill (by definition)
            \item $s_>(\KK_\star(v)) \in S(\KK_\star(u))$.
            \hfill (by definition)
            \item $(\KK_\star(v),\KK_\star(u))\in R_\star$.
            \hfill (by \ref{itm:og-r-case-topper-bot})
        \end{enumerate}
    \end{claimproof}
    This proves property \ref{itm:adjacency}.
    Finally, property \ref{itm:ordered} only concerns the first row $C[*,1] = B[*,1]$, so its truth carries over from~$\gc B$.
    Having proven all properties, it follows that $\gc C$ is an insulator.
    Its cost is \(|\KK_\star|=\const(k,t)\).
    This proves that $\CC$ is the desired row extension. 
    It remains to analyze the running time.

    \paragraph*{Running Time.}
    Let $n:=|V(G)|$.
    \cref{lem:sampling} runs in time $O_t(n^2)$.
    Similarly as in the proof of \cref{thm:orderless-grid-growing},
    we can build the row-extension $C$ of $B$ in time $O_t(n^2)$.
    The construction of $\KK_\star$, $F_\star$, and $R_\star$ runs in time $O_{k,t}(n)$.
    This yields an overall running time of $O_{k,t}(n^2)$.
\end{proof}

Having proven the insulator growing lemmas, this concludes \cref{part1}.

\newpage
\part{Non-Structure}\label{part:nonstructure}
\label{part2}
In  \Cref{part1}, we have shown that for any graph class the following implications hold:
\[
    \text{prepattern-free}
    \spacedRightarrow
    \text{insulation-property}
    \spacedRightarrow
    \text{flip-breakable}
    \spacedRightarrow
    \text{mon.\ dependent}
\]
In \Cref{part2}, we close the circle of implications by showing:
\[
    \text{not prepattern-free}
    \spacedRightarrow
    \text{large flipped crossings/comparability grids}
    \spacedRightarrow
    \text{mon.\ independent}
\]

\newcommand{\trans}{{\mathsf{T}}}
\newcommand{\mesh}{M}
\DeclareRobustCommand{\grid}[1]{
    \begin{tikzpicture}[baseline=-0.87mm]
        \filldraw (0,0) circle (0.2mm);
        \foreach \i in {#1} {
            \draw (0,0) -- (90-45*\i:1mm);
            \draw (0,0) -- (270-45*\i:1mm);
        }
        \foreach \i in {0,1} {%invisible cross, so that all figures have the same width and height
            \path (0,0) -- (90*\i:1mm);
            \path (0,0) -- (180+90*\i:1mm);
        }
    \end{tikzpicture}
}
%scaled up version
\newcommand{\lgrid}[1]{
    \hspace{-.5cm}
    \scalebox{2}{
        \grid{#1}
    }
    \hspace{-.5cm}
}
%grayed out version
\DeclareRobustCommand{\ggrid}[1]{
    \begin{tikzpicture}[baseline=-0.87mm]
        \filldraw[color=gray!40] (0,0) circle (0.2mm);
        \foreach \i in {#1} {
            \draw[color=gray!40] (0,0) -- (90-45*\i:1mm);
            \draw[color=gray!40] (0,0) -- (270-45*\i:1mm);
        }
        \foreach \i in {0,1} {%invisible cross, so that all figures have the same width and height
            \path (0,0) -- (90*\i:1mm);
            \path (0,0) -- (180+90*\i:1mm);
        }
    \end{tikzpicture}
}
\newcommand{\lggrid}[1]{
    \hspace{-.5cm}
    \scalebox{2}{
        \ggrid{#1}
    }
    \hspace{-.5cm}
}

\newcommand{\plusplus}{%
  \mathbin{{+}\mspace{-8mu}{+}}%
}

\newcommand{\cl}{\mathbf{C}}
\newcommand{\hg}{\mathbf{H}}
\newcommand{\st}{\mathbf{S}}

The main goal of \Cref{part2} is to prove the following.

\begin{restatable}{proposition}{thmPatterns}
\label{prop:patterns-main}
    Let $G$ be a graph containing a prepattern of order $n$ on an insulator  of height $h$ and cost $k$.
    Then $G$ contains as an induced subgraph either
    \begin{itemize}
        \item a flipped star $r$-crossing of order $m$, or
        \item a flipped clique $r$-crossing of order $m$, or
        \item a flipped half-graph $r$-crossing of order $m$, or
        \item the comparability grid of order $m$,
    \end{itemize}
    for some $m\ge U_{h,k}(n)$ and $1 \le r < 8h$.
\end{restatable}

This is achieved in \Cref{sec:transformers} and \Cref{sec:conv-and-cross}. Additionally, in \Cref{sec:to-independence} we prove that classes 
which exhibit large patterns as listed in \cref{prop:patterns-main} are monadically independent.
Finally, in \Cref{sec:wrapup}, we wrap up the proof of the main results of the paper.

\paragraph{Convention:}
We often write that some property or function \(P(\bar x)\) \emph{``is the same for all tuples \(\bar x\) from a given domain.''}
This means that for all \(\bar x,\bar y\) from a specified domain,  \(P(\bar x)=P(\bar y)\). 

\section{Transformers}\label{sec:transformers}
As a first step, in Section~\ref{sec:transformers}, starting from large prepatterns, we will extract well-structured, but rather abstract objects called \emph{transformers}.
After that, in Section~\ref{sec:conv-and-cross}, transformers will be analyzed at a low level, and crossings will be extracted from them.

\subsection{Meshes and Transformers}\label{sec:meshesandtransformers}

\begin{definition}
    A \emph{mesh} in a graph $G$
    is an injective function $\mesh\from I\times J\to V(G)$,
    where  $I$ and $J$ are two non-empty indexing sequences of the same length.
    We denote $V(\mesh):=\setof{\mesh(i,j)}{i\in I,j\in J}$.
    For a mesh $\mesh$ as above, by $\mesh^\trans$ denote the mesh $\mesh^\trans \from J\times I\to V(G)$ such that $\mesh^\trans(i,j)=\mesh(j,i)$ for all $i\in I,j\in J$.
    A mesh $\mesh$ has \emph{order} $m$ if $|I|=|J|=m$.

    For $I'\subset I$ and $J'\subset J$,
    by $\mesh|_{I'\times J'}$ we denote the mesh
    obtained by restricting $\mesh$ to $I'\times J'$.
    We call $\mesh|_{I'\times J'}$ a \emph{submesh} of $\mesh$.
\end{definition}

\begin{definition}
    Let $\mesh\from I\times J\to V(G)$ be a mesh in a (possibly colored) graph $G$.
    Then $\mesh$ is \emph{vertical} in $G$ if \(|I|=|J|\le 3\), or if there is a function $a\from I\to V(G)$
    such that
    \begin{itemize}
        \item     $\atp_G(\mesh(i,j),a(i'))$ depends only on $\otp(i,i')$ for all $i,i'\in I$ and $j\in J$, and 

        \item $\atp_G(\mesh(i,j),a(i'))$ is not the same for all $i,i'\in I$ and $j\in J$.
    \end{itemize} 
    We say that $\mesh$ is \emph{horizontal} in $G$ if $\mesh^{\trans}$ is vertical in $G$.
Note that a mesh can be both horizontal and vertical.
\end{definition}

\begin{definition}
    Let $\mesh,\mesh'\from I\times J\to V(G)$ be meshes in a (possibly colored) graph $G$. We say that in $G$ the pair $(\mesh,\mesh')$ is
    \begin{itemize}
        \item \emph{regular} if 
        $\atp_G(\mesh(i,j),\mesh'(i',j'))$ 
        depends only on $\otp(i,i')$ and $\otp(j,j')$ \\for all $i,i'\in I$ and $j,j'\in J$,

        \item \emph{homogeneous}
        if $\atp_G(\mesh(i,j),\mesh'(i',j'))$ is the same \\for all $i,i'\in I$ and $j,j'\in J$,

        \item \emph{conducting} if either $|I| = |J| \leq 3$ or $(\mesh,\mesh')$ is regular but not homogeneous in $G$.

    \end{itemize}
\end{definition}

\begin{definition}
    A \emph{conductor} of order $n$ and length $h$ is a sequence 
    $\mesh_1,\ldots,\mesh_h\from I\times J\to V(G)$ of meshes of order $n$,
    such that each pair $(\mesh_s,\mesh_{s+1})$ is conducting for $s=1,\ldots,h-1$.
\end{definition}

We now define the central notion of \Cref{sec:transformers}.
\begin{definition}
    A \emph{transformer} of order $n$ and length $h$ is a conductor 
    $\mesh_1,\ldots,\mesh_h$ of order $n$ and length $h$,
    such that $\mesh_1$ is vertical, and $\mesh_h$ is horizontal.
\end{definition}

The main result of \Cref{sec:transformers} is the following.

\begin{restatable}{proposition}{PropTransformer}\label{prop:transformer}
    Let $G$ be a graph containing a prepattern of order \(n\) on an insulator of height $h$ and cost $k$.
    Then $G$ contains a transformer of order $U_{h,k}(n)$ and length at most $4h-1$.
\end{restatable}

We start with some simple observations.

\begin{observation}\label{obs:submesh}
    Let $\mesh$ be a mesh in a graph $G$.
    If $\mesh$ is horizontal/vertical in $G$ then this also holds for every submesh of $\mesh$.
\end{observation}

\begin{observation}\label{obs:ramsey-subconductors}
    Let $\mesh_1,\mesh_2 \from I \times J \to V(G)$ be meshes in a graph $G$ such that the pair $(\mesh_1,\mesh_2)$ is conducting in $G$.
    For all $I' \subseteq I$ and $J' \subseteq J$ with $|I'| = |J'|$ the pair $(\mesh_1|_{I'\times J'},\mesh_2|_{I'\times J'})$ is also conducting in $G$.
\end{observation}

\begin{lemma}\label{lem:ramsey-forget-colors-conducting}
    Let $G$ be a graph and $G^+$ be a coloring of $G$. 
    If a pair of meshes $(\mesh,\mesh')$ is conducting in $G^+$ then it is also conducting in $G$.
\end{lemma}
\begin{proof}
    Obviously the pair $(\mesh,\mesh')$ is also regular in $G$.
    Via contrapositive, it remains to assume that
    $(\mesh,\mesh')$ is homogeneous in \(G\),
    and show that it also is homogeneous in \(G^+\).
    To this end, observe that the atomic type $\atp_{G^+}(\mesh(i,j),\mesh'(i',j'))$
    depends only on $\atp_{G}(\mesh(i,j),\mesh'(i',j'))$
    as well as the colors of \(\mesh(i,j)\) and \(\mesh'(i',j')\).
    The former has the desired properties by homogeneity in \(G\),
    and the latter by regularity in \(G^+\).
\end{proof}

\begin{lemma}\label{lem:ramsey-forget-colors-vertical}
    Let $G$ be a graph and $G^+$ be a coloring of $G$. If  a  mesh $\mesh$ is vertical in~$G^+$ then it is also vertical in $G$.
\end{lemma}
\begin{proof}
    Let $a(\cdot)$ be the function witnessing that $\mesh$ is vertical in $G^+$.
    We argue as in the proof of \cref{lem:ramsey-forget-colors-conducting}.
    Here, instead of regularity, we use verticality to argue that all vertices in the range of $a$ (respectively all vertices in the range of $\mesh$) have the same atomic type in $G^+$.
\end{proof}

\subsection{From Prepatterns to Transformers}\label{sec:pre-to-transformers}

In this section, we prove \cref{prop:transformer},
 extracting transformers from prepatterns.
Let us first give an overview of the proof in the case where $G$ contains a bi-prepattern.
The mono-prepattern case later falls out as a subcase.
The bi-prepattern in $G$ is witnessed by an insulator $\gc A$ whose columns are indexed by a sequence $K$ containing two subsequences $I$ and~$J$. Every pair of columns $(i,j) \in I \times J$ is ``matched up'' by an element $c_{i,j}$ through quantifier-free formulas (cf.\ \cref{def:bi-pattern}).

\cref{fig:transformer-in-insulator} is a schematic depiction of how the transformer will embed into the insulator $\gc A$. 
The transformer will be assembled from two conductors $C = \mesh_0, \ldots, \mesh_t$ and $C' = \mesh'_0, \ldots, \mesh'_t$.
The columns of the meshes of $C$ and $C'$ are contained in the columns of the insulator $\gc A$ indexed by $I$ and $J$, respectively.
The structure of $\gc A$ imposes that $\mesh_t$ and $\mesh'_t$ are both vertical.
The meshes $\mesh_0$ and $\mesh'_0$ are chosen from the vertices $c_{i,j}$ such that $\mesh_0 = {\mesh'_0}^\trans$.
It follows that we can glue the conductor $C$ to the transposed meshes from ${C'}$ yielding the desired transformer
\[
    \mesh_t,\mesh_{t-1},\ldots,\mesh_1,\mesh_0 = \mesh_0'^\trans, \mesh_1'^\trans,\ldots, \mesh_{t-1}'^\trans,\mesh_t'^\trans,
\]
where $\mesh_t$ is vertical and $\mesh_t'^\trans$ is horizontal.

The construction of $C$ and $C'$ is implemented by \cref{lem:down-down} or \cref{lem:down-down-negative} ($\clubsuit_1$ or $\clubsuit_2$ in the picture) depending on the choice of $\sim_1$ and $\sim_2$ in the definition of a bi-prepattern.
\cref{lem:down-down} iteratively extends the conductor by ``stepping down'' the insulator.
A single step is performed using \cref{lem:one-step-down} ($\blacklozenge$), which creates meshes \emph{embedded into} descending rows of the insulator (cf.\ \cref{def:ramsey-row-embedding}).
The structure of the insulator is used to establish conductivity between successive meshes/rows.
This process continues until either a vertical mesh is produced or we reach the bottom row of the insulator.
If we reach the bottom row, we use \cref{lem:ramsey-row1} ($\spadesuit$) to further extend the conductor to reach a vertical mesh.
Here we use the special structure of the bottom row: In the unordered case, the cells how the bottom row contain only a single vertex each, which implies that the mesh in the bottom row is already vertical. 
In the ordered case the mesh in the bottom row is 
connected through short parts to a vertical mesh and we can bridge those short paths by a conductor.
This finishes the sketch of \cref{lem:down-down}~($\clubsuit_1$).
\cref{lem:down-down-negative}~($\clubsuit_2$) is a rather technical case distinction which reduces the construction of the conductor to some previously established subcase.

\begin{figure}[htbp]
    \centering
    \includegraphics{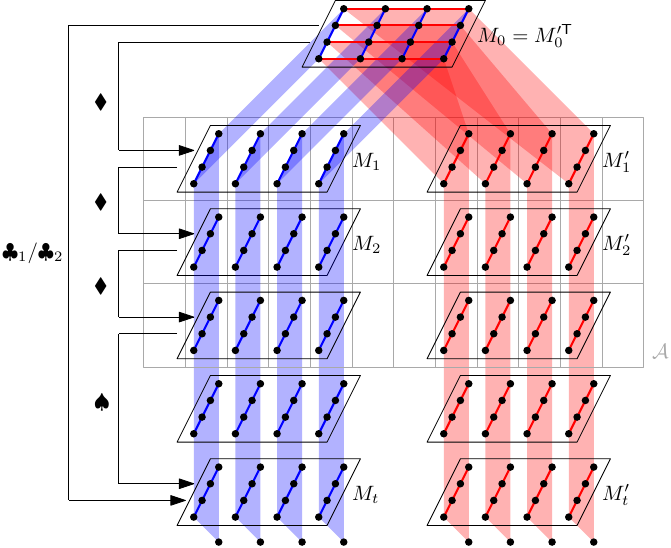}
    \caption{Constructing a transformer from a bi-prepattern.}
    \label{fig:transformer-in-insulator}
\end{figure}

This concludes the overview of the proof of \cref{prop:transformer}. We now give a detailed proof.

\begin{definition}\label{def:ramsey-row-embedding}
    Let $G$ be a graph containing an insulator $\gc A$ of height $h$ with grid $A$ and indexed by a sequence \(K\).
    Let $\mesh\from I\times J\to V(G)$ be a mesh of order \(n\) in $G$ with $I \subseteq K$.
    Let $r \in [h]$.
    We say that $\mesh$ \emph{embeds into row $r$ of $\gc A$} if for all $i \in I$ and $j\in J$.
    \[
        \mesh(i,j) \in A[i,r].
    \]
\end{definition}

\begin{lemma}[$\blacklozenge$ in \cref{fig:transformer-in-insulator}]\label{lem:one-step-down}
    Let $G$ be a graph containing an insulator $\gc A$ of height $h$ with grid $A$ and indexed by a sequence \(K\).
    Let $\mesh\from I\times J\to V(G)$ be a mesh of order \(n\) in $G$ with $I \subseteq K$.
    Let $\alpha(x,y)$ be a quantifier-free formula in a $k$-coloring $G^+$ of $G$, 
    and let $r\in [h]$ be such that for all $i\in I,j\in J$
    \begin{align}\label{eq:first-nonempty}
        i=\min\setof{i'\in I}{\exists v\in A[i',{\le} r]\ :\ G^+\models \alpha(v,\mesh(i,j))}.    
    \end{align}
    Then there are meshes $\mesh_1,\mesh_2$ of order $U_{k,h}(n)$
    such that
    \begin{itemize}
        \item $\mesh_1$ is a submesh of $\mesh$, and
        \item the pair $(\mesh_1,\mesh_2)$ is conducting in $G$, and
        \item $\mesh_2$ is vertical in $G$ or
        $\mesh_2$ embeds into some row $r' \leq r$ of $\gc A$.
    \end{itemize}

\end{lemma}
\begin{proof}
Note that by \cref{lem:ramsey-forget-colors-conducting} and \cref{lem:ramsey-forget-colors-vertical}, it does not matter whether we show conductivity and verticality in $G$ or $G^+$.
For every $i\in I$ and $j\in J$, pick a vertex $\mesh'(i,j)\in A[i,\le r]$ 
such that $G^+ \models \alpha(\mesh'(i,j),\mesh(i,j))$.
Note that \(M\) is not necessarily a mesh, as it may not be injective.

\begin{claim}\label{clm:ramsey-ramsey-application}
    There are sequences $I'\subset I$ and $J'\subset J$ 
    with $|I'|=|J'|\ge U_{k,h}(|I|) = U_{k,h}(n)$, and
    a row number $r'\in [r]$
    such that
    \begin{enumerate}[leftmargin= 3em, label={\textup{(R.$\arabic*$)}}]
        \item \label{itm:ramsey-tf1-injective-1}
        $\atp_{G^+}(\mesh(i,j),\mesh(i',j'),\mesh'(i,j),\mesh'(i',j'))$
         depends 
        only on $\otp(i,i')$ and $\otp(j,j')$\\ 
        for $i,i'\in I'$ and $j,j'\in J'$, and
         
        \item\label{itm:ramsey-row-embed} $\mesh'(i,j)\in A[i,r']$ for all $i\in I'$.
    \end{enumerate}
\end{claim}
The proof of the claim is a straightforward application of Bipartite Ramsey (\Cref{gridramsey}).
Readers experienced in Ramsey theory are invited to skip it.
However, due to the importance of Ramsey type arguments for this section, we include the details of an exemplary application here.
\begin{claimproof}[Proof of \cref{clm:ramsey-ramsey-application}]    
Up to renaming we can assume $I = J = [n]$.
Let $\Pi$ be the set of possible atomic types of four tuples in $k$-colored graphs. 
There exists a constant $k^* \leq \const(k,h)$ such that there is a bijection $b \from \Pi \times [r] \to [k^*]$.
We now define the coloring $c \from [n]^2 \times [n]^2 \to [k^*]$ as 
\[
    c\bigl((i,i'),(j,j')\bigr) := b\Bigl(\atp_{G^+}\bigl(\mesh(i,j),\mesh(i',j'),\mesh'(i,j),\mesh'(i',j')\bigr), \ r'\Bigr)
\]
where $r'$ is the unique row such that $\mesh'(i,j)\in A[i,r']$.
Applying Bipartite Ramsey (\Cref{gridramsey}) to the defined coloring yields sequences $I'\subset I$ and $J'\subset J$ 
with $|I'|=|J'|\ge U_{k,h}(|I|) = U_{k,h}(n)$
such that $c\bigl((i,i'),(j,j')\bigr)$ depends only on $\otp(i,i')$ and $\otp(j,j')$
for $i,i'\in I'$ and $j,j'\in J'$.
By the construction of the coloring, this proves \ref{itm:ramsey-tf1-injective-1}.
To see that the row containing $\mesh'(i,j)$ is the same for all $i\in I$ and $j\in J$, notice that 
\[
    c\bigl((i,i),(j,j)\bigr) = c\bigl((i',i'),(j',j')\bigr) \quad
    \text{for all $i,i'\in I'$ and $j,j'\in J'$.}
\]
Therefore, $\mesh(i,j)$ and $\mesh(i',j')$ are in the same row,
which proves \ref{itm:ramsey-row-embed}.
\end{claimproof}

Let $I'$, $J'$, and $r'$ be as from the previous claim. We set $\mesh_1 := \mesh|_{I'\times J'}$ and $\mesh_2 := \mesh'|_{I'\times J'}$.

\begin{claim}
    Either $\mesh_1$ is vertical, or
    $\mesh_2$ is a mesh (that is, $\mesh_2$ is injective).
\end{claim} 
\begin{claimproof}
    Assume $|I'| > 3$, as otherwise $\mesh_1$ is vertical by definition.
    Assume that $\mesh_2$ is not injective.
    Then there exist distinct pairs $(i,j), (i',j') \in I' \times J'$ such that $\mesh_2(i,j) = \mesh_2(i',j')$.
    Since $\mesh_2(i,j)$ (equivalently $\mesh_2(i',j')$) is from the $i$th (equivalently $i'$th) column of~$A$, we have that $i = i'$. Hence, $j \neq j'$.
    
    By \ref{itm:ramsey-tf1-injective-1} we have $a(i) := \mesh_2(i,j) = \mesh_2(i,j')$ for all $i\in I'$ and $j,j' \in J'$.
    Now also by \ref{itm:ramsey-tf1-injective-1},
    $\atp_{G^+}(\mesh_1(i,j),a(i'))$ depends only on $\otp(i,i')$ for all $i,i'\in I$ and~$j\in J'$.
    Finally, by \eqref{eq:first-nonempty} and $|I'|>3$, $\atp_{G^+}(\mesh_1(i,j),a(i'))$ is not the same
    for all $i,i'\in I$ and~$j\in J'$.
    In summary: $a(\cdot)$ witnesses that $\mesh_1$ is vertical.
\end{claimproof}
Assume now the mesh $\mesh_1$ is vertical.
By \ref{itm:ramsey-tf1-injective-1}, the pair $(\mesh_1,\mesh_1)$ is regular.
As witnessed by the equality type, the pair is not homogeneous and therefore conducting.
We can return $(\mesh_1,\mesh_1)$.

Otherwise, $\mesh_2$ is a mesh and
embeds into the row $r'$ of $\gc A$ by \ref{itm:ramsey-row-embed}.
Furthermore, the pair $(\mesh_1,\mesh_2)$ is conducting: regularity follows from \ref{itm:ramsey-tf1-injective-1} and non-homogeneity from \eqref{eq:first-nonempty}.
\end{proof}

\begin{lemma}[$\spadesuit$ in \cref{fig:transformer-in-insulator}]\label{lem:ramsey-row1}
    Let $G$ be a graph containing an insulator
    $\gc A$ of cost $k$ and height $h$ and let $\mesh$ be a mesh of order $n$ that embeds into row $1$ of $\gc A$.
    There exists a conductor
    $\mesh_1,\ldots,\mesh_t$ for \(t \le h\) of order $U_{k,h}(n)$ such that $\mesh_1$ is a submesh of $\mesh$ and $\mesh_t$ is vertical.
\end{lemma}
\begin{proof}
    Let $I,J$ be the sequences indexing the mesh $\mesh$.
    Let $A$ be the grid of $\gc A$.
    Assume first that $\gc A$ is orderless.
    Then by the insulator property~\ref{itm:orderless}, each cell in row $1$ of $\gc A$ contains only a single vertex. As $\mesh$ embeds into this row, it has order $1$. Then $\mesh$ is vertical and the conductor consisting only of $\mesh$ satisfies the conditions of the lemma. 

    Assume now that $\gc A$ is ordered.
    By \ref{itm:ordered}, there exists a $k$-flip $H$ of $G$ and some radius $r<h$ such that the radius-$r$ balls around the vertices in $A[*,1]$ are pairwise disjoint.
    Moreover, there are vertices   
        $\{b(v) \in N_r^{H}[v] : v\in A[*,1]\}$ and 
        $\{c_i \in V(G): i\in I\}$
        and a symbol ${\sim} \in \{{\leq},{\geq}\}$ such that for all $i,j \in I$ and $v \in A[j,1]$
        \begin{center}
            $G \models E(c_i, b(v))$
            if and only if
            $i \sim j$. 
        \end{center}
        Let $G^+$ be the $k$-coloring of $G$ in which the edge relation of the flipped graph $H$ can be expressed by a quantifier-free formula.
        For each $i \in I$ and $j \in J$ let $\pi(i,j)$ be the tuple of vertices forming a shortest path from $\mesh(i,j)$ to $b(\mesh(i,j))$ in $H$.
        By \ref{itm:ordered}, these paths have equal length for all \(i,j\) and consist of at most $h$ vertices.
        By Bipartite Ramsey (\Cref{gridramsey}),
        there exist sequences $I' \subseteq I$ and $J' \subseteq I$ of length at least $U_{k,h}(n)$ such that 

        \begin{center}
            $\atp_{G^+}(\pi(i,j),c(i),\pi(i',j'),c(i'))$
            depends 
            only on $\otp(i,i')$ and $\otp(j,j')$

            for all $i,i'\in I'$ and $j,j'\in J'$.
        \end{center}
        For distinct pairs $(i,j)$ and $(i',j')$ in $I' \times J'$, $\pi(i,j)$ and $\pi(i',j')$ have no vertex in common, as they stem from two disjoint balls in $H$.
        Therefore, each of the functions $\mesh_1, \ldots, \mesh_h \from I' \times J' \to V(G)$, where $\mesh_t(i,j)$ is defined as the $t$th component of $\pi(i,j)$, is injective and forms a mesh.
        By construction, $\mesh_1$ is a submesh of $\mesh$.
        By Ramsey and \ref{itm:ordered}, $\mesh_h$ is vertical.
        We next show that $\mesh_1, \ldots, \mesh_h$ is a conductor in $G^+$.
        By Ramsey, any pair of successive meshes $(\mesh_t,\mesh_{t+1})$ in the sequence is regular in $G^+$. It remains to show that $(\mesh_t,\mesh_{t+1})$ is not homogeneous in $G^+$.
        Consider two distinct pairs $(i,j), (i',j') \in I' \times J'$.
        As $\pi(i,j)$ and $\pi(i',j')$ are constructed via paths through disjoint balls in $H$, we have that 
        $\mesh_t(i,j)$ is adjacent to $\mesh_{t+1}(i,j)$ and non-adjacent to $\mesh_{t+1}(i',j')$ in $H$.
        It follows that $(\mesh_t,\mesh_{t+1})$ is not homogeneous in~$G^+$.
        Therefore, $(\mesh_t,\mesh_{t+1})$ is conducting in~$G^+$.
        Hence, $\mesh_1, \ldots, \mesh_h$ is a conductor in $G^+$, and by \cref{lem:ramsey-forget-colors-conducting} and \cref{lem:ramsey-forget-colors-vertical} also in~$G$.
\end{proof}

%\janin{in the following lemma and in \Cref{lem:down-down-negative}: do we need \(A[i',\le r]\) or is \(A[i',*]\) fine?}
\begin{lemma}[$\clubsuit_1$ in \cref{fig:transformer-in-insulator}]\label{lem:down-down}
    Let $G$ be a graph containing an insulator $\gc A$ of cost $k$ and height $h$ with grid $A$ and indexed by a sequence \(K\).
    Let $\mesh\from I\times J\to V(G)$ be a mesh of order \(n\) in $G$ with $I \subseteq K$.
    Let $\alpha(x,y)$ be a quantifier-free formula in a $k$-coloring $G^+$ of $G$ such that for all $i\in I,j\in J$

    $$i=\min\setof{i'\in I}{\exists v\in A[i',*]\ :\ G^+ \models \alpha(v,\mesh(i,j))}.$$
Then there is a conductor $\mesh_1,\ldots,\mesh_t$ of order $U_{k,h}(n)$ and length at most $2h$ in $G$ such that $\mesh_1$ is a submesh of $\mesh$ and $\mesh_t$ is vertical in $G$.
\end{lemma}
\begin{proof}
    Note that by \cref{lem:ramsey-forget-colors-conducting} and \cref{lem:ramsey-forget-colors-vertical}, it does not matter whether we show conductivity and verticality in $G$ or $G^+$.
    Denote $\mesh_0:=\mesh$ and $r_0 := h + 1$.
    We inductively construct a sequence of meshes $\mesh_1, \mesh_2,\ldots,\mesh_t$ 
        where for each $s=1, 2,\ldots,t$, the mesh $\mesh_s$ satisfies the following conditions:
        \begin{itemize}
            \item $\mesh_s$ has order $U_{s,k,h}(n)$,
            
            \item there is a row $r_s\in [r_{s-1} - 1]$
            such that 
            $\mesh_s$ embeds into row $r_s$ of $\gc A$,

            \item there is a submesh $\mesh_{s-1}'$ of $\mesh_{s-1}$ such that the pair $(\mesh_{s-1}', \mesh_s)$ is conducting.
        \end{itemize}
        To construct $\mesh_1$, we apply \cref{lem:one-step-down} to the mesh $\mesh_0$ and the formula $\alpha(x,y)$ to obtain
        a conducting pair of meshes $(\mesh_0',\mesh_1)$ of order $U_{k,h}(n)$
        where $\mesh_0'$ is a submesh of $\mesh_0$.
        If $\mesh_1$ is vertical then $\mesh_0', \mesh_1$ is the desired conductor, and we conclude the proof of the lemma.
        Otherwise, $\mesh_1$ embeds into some row $r_1 \in [h]$ of $\gc A$ and satisfies the induction hypothesis.

    Suppose the sequence $\mesh_1,\ldots,\mesh_s$ has been constructed for some $s\ge 1$.
    Below we either extend the sequence by one, or terminate the process.

    Assume first $r_s > 1$.
    As the sequence $r_1,r_2,\ldots,r_s$ is strictly decreasing we have $s < h$.
    Let~$\beta(x,y)$ be the formula expressing adjacency in the flip $G'$ associated to~$\gc A$. As $G'$ is a $k$-flip of $G$, $\beta$ is expressible in a $k$-coloring of $G$.
    Let $I_s$ and $J_s$ be the sequences indexing $\mesh_s$.
    For every $i_0 < i \in I_s$ and $j \in J_s$, we have that $\mesh_s(i,j)$ has a $\beta$-neighbor in $A[i,r_s-1]$ but no $\beta$-neighbor in $A[i_0,r_s-1]$.
    This is by the insulator property \ref{itm:orderless} if $\gc A$ is orderless and by \ref{itm:rootedness} and \ref{itm:adj-bot-left} if $\gc A$ is ordered.
    Hence, we can apply \cref{lem:one-step-down} to
    the mesh $\mesh_s$, the row number $r_{s}-1$, and 
    the formula $\beta(x,y)$.
    We obtain
    a conducting pair of meshes $(\mesh_s',\mesh_{s+1})$ of order $U_{s,k,h}(n)$
    where $\mesh_s'$ is a submesh of $\mesh_s$.
    If $\mesh_{s+1}$ is vertical then we conclude the proof of the lemma returning the conductor
    \[
        \mesh_0|_{I'\times J'},
        \mesh_1|_{I'\times J'},
        \ldots,
        \mesh_{s}|_{I'\times J'},
        \mesh_{s+1}
    \]
    of length at most $h + 1$, where $I' \subseteq I$ and $J' \subseteq J$ are the indexing sequences of $\mesh_{s+1}$.
    Otherwise, $\mesh_{s+1}$ embeds into some row $r_{s+1} \in [r_s - 1]$ of $\gc A$ and satisfies the induction hypothesis. We continue the process. 

    Assume now $r_s = 1$. Note that $s \leq h$.
    Then $\mesh_s$ embeds into the first row of $\gc A$, and we can apply \cref{lem:ramsey-row1}.
    This yields a conductor $\mesh_1^\star,\ldots,\mesh_{t^\star}^\star$ of length $t^\star \leq h$
    indexed by sequences $I' \subseteq I$ and $J' \subseteq J$ of length $U_{s,k,h}(n)$ such that $\mesh_1^\star$ is a submesh of $\mesh_s$, and $\mesh_{t^\star}^\star$ is vertical.
    Now we conclude the proof of the lemma returning the following conductor of length at most $2h$:
    \[
        \mesh_0|_{I'\times J'},
        \ldots,
        \mesh_{s-1}|_{I'\times J'},
        \mesh^\star_{1},
        \ldots,
        \mesh^\star_{t^\star}.
        \qedhere
    \]
\end{proof}

\begin{lemma}\label{lem:mono-pattern-to-transformer}
    Assume $G$ contains a mono-prepattern of order $n$ on an insulator $\gc A$ of cost $k$ and height $h$.
    Then $G$ contains a transformer of order $U_{k,h}(n)$ and length at most $2h$.
\end{lemma}
\begin{proof}
    Let $(c_j:j\in J)$ and $(b_{i,j} :i\in I,j\in J)$ 
    form a mono-prepattern of order $n=|I|=|J|$ on the insulator $\gc A$.
    Define the function $\mesh\from I\times J\to V(G)$ where $\mesh(i,j)=b_{i,j}$.
    By definition of a mono-prepattern, all the $b_{i,j}$ are distinct, so $\mesh$ is a mesh.
    By Bipartite Ramsey (\Cref{gridramsey}) and \Cref{def:mono-pattern},
    there exist sequences $I' \subseteq I$ and $J' \subseteq I$ of length at least $U(n)$ such that 
    $\atp_G(\mesh(i,j),c_{j'})$ depends only on $\otp(j,j')$ for all $i \in I'$ and $j,j'\in J'$.
    Since $\atp_G(\mesh(i,j),c_{j'})$ is not the same for all $i\in I'$ and $j,j'\in J'$, the submesh $\mesh|_{I' \times J'}$ is horizontal.

    We apply \cref{lem:down-down} to the mesh $\mesh|_{I' \times J'}$ and formula $\alpha(x,y) :=(x=y)$. 
    This yields a conductor
    $\mesh_1,\ldots,\mesh_t$ of length at most $2h$ and order $U_{h,k}(n)$ such that $\mesh_1$ is a submesh of $\mesh|_{I'\times J'}$, and $\mesh_t$ is vertical.
    As $\mesh|_{I'\times J'}$  is horizontal, the same holds for $\mesh_1$, and thus $\mesh_t,\ldots,\mesh_1$ is the desired transformer.
\end{proof}

\Cref{lem:down-down} is complemented by the following more technical \Cref{lem:down-down-negative}.
Together, these two lemmas accommodate the two possible choices for the symbols $\sim_1,\sim_2\in\set{=,\neq}$ in the definition of a bi-prepattern (\Cref{def:bi-pattern}).

\begin{lemma}[$\clubsuit_2$ in \cref{fig:transformer-in-insulator}]\label{lem:down-down-negative}
    Let $\gc A$ be an insulator of cost $k$ and height $h$ indexed by a sequence \(K\), let \(I,J\) be indexing sequences with \(I \subseteq K\), 
    let $\mesh\from I\times J\to V(G)$ be a mesh of order \(n\), and let $\alpha$ be a quantifier-free formula in a $k$-coloring $G^+$ of $G$, 
    such that for all $i\in I,j\in J$
    \[
        i=\min\setof{i'\in I}{\neg\exists v\in A[i',*]\ :\  G^+\models \alpha(v,\mesh(i,j))}.  
    \]
    Then there is a conductor $\mesh_1,\ldots,\mesh_t$ of order $U_{k,h}(n)$ and length at most $2h$ such that 
    either
    \begin{itemize}
        \item $\mesh_1$ is a submesh of $\mesh$ and $\mesh_t$ is vertical in \(G\), or
        \item $\mesh_1,\ldots,\mesh_t$ is a transformer in \(G\).
    \end{itemize}
    
\end{lemma}

\begin{proof}
    For all $i^*,i\in I$ with $i^*<i$ and $j\in J$,
    fix a vertex $F(i^*,i,j)\in A[i^*,*]$ such that
    \[
        G^+ \models \alpha(F(i^*,i,j), \mesh(i,j))
    \]
    By Bipartite Ramsey (\Cref{gridramsey}), there are sets $I'\subset I$ and $J'\subset J$ with $|I'|=|J'|\ge U(|I|)$,
    such that for all $i^*,i,i'\in I'$ with $i^*<i,i'$, and for all $j,j'\in J'$, 
    \begin{equation}\label{eq:ramsey-negative}
        \atp_{G^+}(F(i^*,i',j'), \mesh(i,j))
        \quad
        \text{depends only on $\otp(i,i')$ and $\otp(j,j')$}.
        \tag{$*$}
    \end{equation}
    In particular this holds for the fact whether $G^+ \models \alpha(F(i^*,i',j'), \mesh(i,j))$.
    For an exemplary application of Bipartite Ramsey that illustrates how to obtain \eqref{eq:ramsey-negative}, see the proof of \cref{clm:ramsey-ramsey-application}.
    Denote 
    \begin{multicols}{2}
    \begin{itemize}
        \item $i_{\min}:=\min(I')$,
        \item $i_{\max}:=\max(I')$,
        \item $i'_{\max}:=\max(I'\setminus i_{\max})$,
        \item $j_{\min}:=\min(J')$,
        \item $j_{\max}:=\max(J')$,
        \item $j'_{\max}:=\max(J'\setminus j_{\max})$.
    \end{itemize}
    \end{multicols}
    For convenience, we redefine \(I := I' \setminus \{i_{\min},i_{\max},i'_{\max}\}\),
    and \(J := J' \setminus \{i_{\min},i_{\max},i'_{\max}\}\).
    We do so to ensure that all elements in \(I\) and \(J\) have the same order type
    with respect to the previously chosen extremal elements.
    
    \begin{figure}[htbp]
        \centering
        \includegraphics{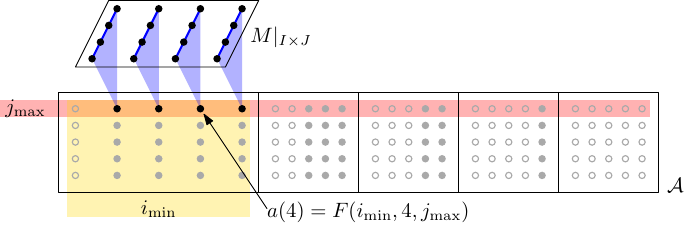}
        \caption{A depiction of \emph{Case 1}. On top: the desired vertical submesh of $\mesh$. On the bottom: the insulator $\gc A$, whose columns contain the $F(i^*,i,j)$ vertices (depicted as \textcolor{darkgray}{$\bullet$}).
        Here, the outermost columns correspond to \(i^*\)-coordinates of \(F(i^*,i,j)\), and
        within each outermost column, the inner columns and rows correspond to
        \(i\)- and \(j\)-coordinates, respectively.
        The $F(i^*,i,j)$ vertices are only defined for indices $i^* < i$, 
        thus going rightwards the columns are filled up with placeholders (depicted as \textcolor{darkgray}{$\circ$}).}
        \label{fig:ramsey-neg-case1}
    \end{figure}

    \paragraph{Case 1:}
    Assume 
    $\atp_{G^+}(F(i^*,i',j'), \mesh(i,j))$ is not the same for all 
    $i^*,i,i' \in I$ and $j,j'\in J$ with $i^*<i,i'$ and $j<j'$ (here $j<j'$ is the crucial assumption).
    In this case, we work towards proving that $M$ contains a large vertical submesh.
    The situation is depicted in \cref{fig:ramsey-neg-case1}. 
    We set $a({i'}):=F(i_{\min},i',j_{\max})$ for all $i'\in I$.
    Let us now argue that $\atp_{G^+}(a(i'),\mesh(i,j))$ is not the same for all $i,i' \in I$ and $j \in J$.

    By assumption, 
    there exist indices 
    \begin{itemize}
        \item $i_{1}^*,i_1,i_1' \in I$ and $j_1,j_1' \in J$ with $i_{1}^* < i_1, i_1'$ and $j_1 < j_1'$, and

        \item $i_{2}^*,i_2,i_2' \in I$ and $j_2,j_2' \in J$ with $i_{2}^* < i_2, i_2'$ and $j_2 < j_2'$
    \end{itemize}
    such that
    \[
        \atp_{G^+}(F(i_1^*,i'_1,j'_1), \mesh(i_1,j_1))
        \neq
        \atp_{G^+}(F(i_2^*,i'_2,j'_2), \mesh(i_2,j_2)).
    \]
    As we removed the extremal elements from \(I\) and \(J\), we have
    \begin{itemize}
        \item $\otp(i_{1}^*,i_1,i_1')= \otp(i_{\min},i_1,i_1')$ and $\otp(j_1,j_1') = \otp(j_1,j_{\max})$, as well as
        \item $\otp(i_{2}^*,i_2,i_2')= \otp(i_{\min},i_2,i_2')$ and $\otp(j_2,j_2') = \otp(j_2,j_{\max})$.
    \end{itemize}
    Now applying \eqref{eq:ramsey-negative}, we get
    \[
        \atp_{G^+}(\underbrace{F(i_{\min},i'_1,j_{\max}}_{a(i'_1)}), \mesh(i_1,j_1))
        \neq
        \atp_{G^+}(\underbrace{F(i_{\min},i'_2,j_{\max}}_{a(i'_2)}), \mesh(i_2,j_2)).
    \]
    Thus, $\atp_{G^+}(a(i'),\mesh(i,j))$ is not the same for all $i,i' \in I$ and $j \in J$.
    Additionally, as $j < j_{\max}$ for all $j \in J$, we have that $\atp_{G^+}(a(i'), \mesh(i,j))$ only depends on $\otp(i,i')$ and no longer on $j$, for all $i,i' \in I$ and $j \in J$.
    It follows that $\mesh_{I \times J}$ is vertical and forms the desired conductor (in this case of length one).

    \paragraph{Case 2:}
    Assume 
    $\atp_{G^+}(F(i^*,i',j'), \mesh(i,j))$ is not the same for all 
    $i^*,i,i'\in I$ and $j,j'\in J$ with $i^*<i,i'$ and $j>j'$  (here $j>j'$ is the crucial assumption).
    We proceed as in the previous case, but with $j_{\min}$ instead of \(j_{\max}\).

    \begin{figure}[htbp]
        \centering
        \includegraphics{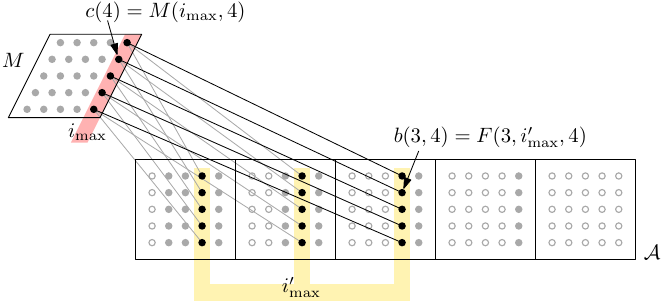}
        \caption{A depiction of the mono-prepattern we discover in \emph{Case 3}. In this example, we have that $G\models E(b(i^*,j'),c(j)) \Leftrightarrow j = j'$.}
        \label{fig:ramsey-neg-case3}
    \end{figure}

    \paragraph{Case 3:}
    Assume 
    $\atp_{G^+}(F(i^*,i',j'), \mesh(i,j))$ is not the same for all
    $i^*,i,i' \in I$ and $j,j'\in J$ with $i^* < i' < i$ (here $i'<i$ is the crucial assumption).
    Now let 
    \begin{center}
        $c(j) := \mesh(i_{\max},j)\quad$ and $\quad b(i^*,j') := F(i^*,i_{\max}',j') \in A[i^*,*]\quad$ for all $i^* \in I$ and $j \in J$.    
    \end{center}
    Our goal is to show that the ranges of $b(\cdot,\cdot)$ and $c(\cdot)$ form a mono-prepattern on $\gc A$.
    The situation is depicted in \cref{fig:ramsey-neg-case3}.
    As in \emph{Case 1}, we argue via \eqref{eq:ramsey-negative} that for all $i^* \in I$ and $j,j'\in J$ 
    \begin{enumerate}[leftmargin= 3em, label={(A.$\arabic*$)}]
        \item\label{itm:ramsey-atp1} $\atp_{G^+}(b(i^*,j'),c(j))$ depends only on $\otp(j,j')$, and
        \item\label{itm:ramsey-atp2} $\atp_{G^+}(b(i^*,j'),c(j))$ is not the same.
    \end{enumerate}
    We claim that for all $i^* \in I$ and $j,j'\in J$ then also
    \begin{enumerate}[leftmargin= 3em, label={(E.$\arabic*$)}]
        \item\label{itm:ramsey-adj1} $G \models E(b(i^*,j'),c(j))$ depends only on $\otp(j,j')$, and
        \item\label{itm:ramsey-adj2} $G \models E(b(i^*,j'),c(j))$ is not the same.
    \end{enumerate}
    By definition of an atomic type, \ref{itm:ramsey-atp1} implies \ref{itm:ramsey-adj1}.
    To show \ref{itm:ramsey-adj2} we argue similarly to the proof of \cref{lem:ramsey-forget-colors-conducting}:
    We know that there are indices 
    $i^*_1,i^*_2 \in I$ and $j_1,j'_1,j_2,j'_2 \in J$ and atomic types $\tau_1$ and $\tau_2$ such that 
    \[
        \tau_1 = \atp_{G^+}(\underbrace{b(i^*_1,j'_1),c(j_1)}_{ \bar u}) \neq 
        \atp_{G^+}(\underbrace{b(i^*_2,j'_2),c(j_2)}_{ \bar v}) = \tau_2.
    \]
    By \eqref{eq:ramsey-negative}, all $b(\cdot,\cdot)$ elements have the same atomic type in $G^+$ and all $c(\cdot)$ elements have the same atomic type in $G^+$.
    This means that the difference in the atomic types of $\bar v$ and $\bar v'$ must be caused by either a difference in their equality or adjacency type.
    If the difference is witnessed in the equality type, then by symmetry we can assume that 
    \[
        b(i^*_1,j'_1) = c(j_1)
        \quad
        \text{and}
        \quad
        b(i^*_2,j'_2) \neq c(j_2).
    \]
    It follows that $b(i^*,j'_1) = c(j_1)$ for all $i^* \in I$, since all of these pairs have the same order type $\otp(j'_1,j_1)$.
    This is a contradiction to the columns of $\gc A$ being disjoint.
    Therefore, the difference in the types of $\bar u$ and $\bar v$ must be witnessed by their adjacency type and we have
    \[
        G\models E(b(i^*_1,j'_1), c(j_1))
        \quad
        \text{if and only if}
        \quad
        G \not \models E(b(i^*_2,j'_2), c(j_2)),
    \]
    which proves \ref{itm:ramsey-adj2}.

    Having proven \ref{itm:ramsey-adj1} and \ref{itm:ramsey-adj2} it is easily verified that the ranges of $b(\cdot,\cdot)$ and $c(\cdot)$ form a mono-prepattern on $\gc A$.
    Applying \cref{lem:mono-pattern-to-transformer} to this mono-prepattern yields the desired transformer.

    \paragraph{Case 4:}
    Assume 
    $\atp_{G^+}(F(i^*,i',j'), \mesh(i,j))$ is not the same for all
    $i^*,i,i' \in I$ and $j,j'\in J$ with $i^* < i < i'$ (here $i<i'$ is the crucial assumption).
    We argue as in the previous case, exchanging the role of $i_{\max}$ and $i_{\max}'$.
    This means we define
    \begin{center}
        $c(j) := \mesh(i_{\max}',j)\quad$ and $\quad b(i^*,j') := F(i^*,i_{\max},j') \in A[i^*,*]\quad$ for all $i^* \in I$ and $j \in J$.
    \end{center}
    As in \emph{Case 3}, the ranges of $b(\cdot,\cdot)$ and $c(\cdot)$ form a mono-prepattern on $\gc A$ and
    we conclude by \cref{lem:mono-pattern-to-transformer}.

    \paragraph{Case 5:}
    If none of the previous cases hold, then 
    $\atp_{G^+}(F(i^*,i',j'), \mesh(i,j))$ and in particular $G^+ \models \alpha(F(i^*,i',j'), \mesh(i,j))$
    is the same for all 
    $i^*,i,i' \in I$ and $j,j'\in J$ with $(i,j)\neq (i',j')$
    and $i^*<i,i'$.
    By \eqref{eq:ramsey-negative}, this also holds for the extremal elements that are in \(I'\) and \(J'\), but not in \(I\) and \(J\).
    Let hence $\gamma\in \set{\mathit{true},\mathit{false}}$ 
    be such that 
    \[
    G^+ \models \alpha(F(i^*,i',j'), \mesh(i,j))\iff \gamma
    \]
    for all 
    $i^*,i,i' \in I'$ and $j,j'\in J'$ with $(i,j)\neq (i',j')$
    and $i^*<i,i'$.

    \begin{figure}[htbp]
        \centering
        \includegraphics{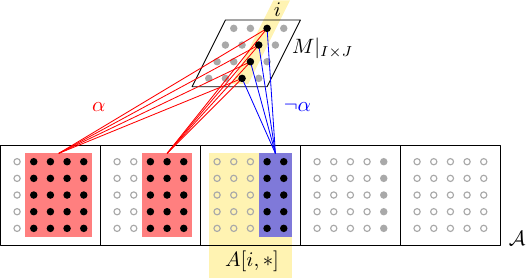}
        \caption{Illustration for \emph{Case 5.1}. 
        Depicted is the insulator $\gc A$, whose columns contain the $F(i^*,i,j)$ vertices.
        The $F(i^*,i,j)$ vertices are only defined for indices satisfying $i^* < i$, so going rightwards, the columns are filled up with placeholders (depicted as \textcolor{darkgray}{$\circ$}).
        The vertices in the $i$th column of the submesh $M\vert_{I\times J}$ are $\alpha$-connected to all the $F(i^*,i,j)$ vertices in $A[{<}i,*]$ (marked in red), but $\neg\alpha$-connected to all the $F(i^*,i,j)$ vertices in $A[i,*]$ (marked in blue).
        }
        \label{fig:ramse-neg-case-5-1}
    \end{figure}

    \paragraph{Case 5.1:} Assume $\gamma=\mathit{true}$.
    At the beginning of the proof, we chose \(F\) such that
    $G^+ \models \alpha(F(i^*,i,j), \mesh(i,j))$.
    Therefore,
    \begin{equation}\label{eq:ramsey-alpha-all-left}
        G^+ \models \alpha(F(i^*,i',j'), \mesh(i,j))\quad
        \text{ for all } 
        i^*,i,i'\in I'
        \text{ with }
        i^*<i,i'
        \text{ and }
        j,j'\in J'.
    \end{equation}
    Our goal is to reduce to \cref{lem:down-down} (\(\clubsuit_1\)).
    The situation is depicted in \cref{fig:ramse-neg-case-5-1}.
    Let \(P = \{ F(i^*,i,j) \mid i^*, i \in I', j \in J' \text{ with } i^* < i \}\) and let $G^{\plusplus}$ be the coloring of $G^+$ where the vertices of $P$ are marked with an additional fresh color predicate.
    (By the definition of colored graphs given in the preliminaries,
    we are technically required to give every vertex \emph{exactly} one color,
    and it would of course be trivial, though more cumbersome, to take this into account.)
    Define the quantifier-free formula $\beta(x,y) := x \in P \wedge \neg \alpha(x,y)$ in the signature of $G^{\plusplus}$.
    Let us now verify that 
    \begin{equation}\label{eq:ramsey-beta-first}
        i=\min\setof{i'\in I}{\exists v\in A[i',*]\ :\ G^{\plusplus} \models \beta(v,\mesh(i,j))}
    \end{equation}
    for all $i\in I,j\in J$.
    As all elements in \(I\) are smaller than \(i_{\max}\),
     for every $i\in I$, we have that the column $A[i,*]$ contains at least one element, say, $F(i,i_{\max},j_{\max}) \in P$.
    By the assumption of the lemma, $\mesh(i,j)$ is not $\alpha$-connected to any element in $A[i,*]$ and therefore $\beta$-connected to $F(i,i_{\max},j_{\max})$.
    By \eqref{eq:ramsey-alpha-all-left} and definition of $P$, $\mesh(i,j)$ is $\alpha$-connected (and therefore not $\beta$-connected) to all elements from $P$ in columns to the left of $I$, as desired.
    Having verified \eqref{eq:ramsey-beta-first}, we conclude by \cref{lem:down-down}.

    \paragraph{Case 5.2:} Assume $\gamma=\mathit{false}$.
    For all 
    $i,i',i^* \in I'$ with $i^*<i,i'$ and $j,j'\in J'$,
    as $G^+ \models \alpha(F(i^*,i,j), \mesh(i,j))$,   
    it follows that
    \[
        G^+ \models \alpha(F(i^*,i',j'), \mesh(i,j))
        \quad
        \text{ if and only if }
        \quad
        (i,j)=(i',j').
    \]
    Let $c(j) := \mesh(i_{\max},j)$ and $b(i^*,j') := F(i^*,i_{\max},j')$ for all $i^* \in I$ and $j,j' \in J$. 
    Then by \eqref{eq:ramsey-negative}, and as 
    $G^+ \models \alpha(b(i^*,j'),c(j))$ if and only if $j = j'$, we have that
    \begin{itemize}
        \item $\atp_{G^+}(b(i^*,j'),c(j))$ depends only on $\otp(j,j')$, and
        \item $\atp_{G^+}(b(i^*,j'),c(j))$ is not the same
    \end{itemize}
    for all $i^* \in I$ and $j,j' \in J$.
    As in \emph{Case 3}, the ranges of $b(\cdot,\cdot)$ and $c(\cdot)$ form a mono-prepattern on $\gc A$ and
    we conclude by \cref{lem:mono-pattern-to-transformer}.

    \medskip\noindent
    Having exhausted all cases, this proves the lemma.
\end{proof}

We can now prove Proposition~\ref{prop:transformer}, which we restate for convenience.

\PropTransformer*
\begin{proof}
    If the prepattern is a mono-prepattern, we conclude by Lemma~\ref{lem:mono-pattern-to-transformer}.
    Therefore, suppose that the prepattern is a bi-prepattern.
    Let the sequences \(I,J\), vertices $(c_{i,j}:i\in I,j\in J)$,
    vertices $s_1,s_2$, quantifier-free formulas $\alpha_1(x;y,z_1),\alpha_2(x;y,z_1)$,
    and symbols ${\sim_1},{\sim_2}\in\set{=,\neq}$
    be as in the definition of a bi-prepattern (see \Cref{def:bi-pattern}).

    Let $G^+$ be the expansion of $G$ by 
    four unary predicates, representing the sets $\set{s_1},\set{s_2},N_G(s_1)$, and $N_G(s_2)$.
    Then there are quantifier-free formulas $\beta_1(x,y)$ and $\beta_2(x,y)$ 
    such that for all $u,v\in V(G)$ and $i\in\set{1,2}$, we have \(G^+\models \alpha_i(u,v,s_i)\iff \beta_i(u,v)\).

    Define the mesh $\mesh\from I\times J\to V(G)$
    with $\mesh(i,j)=c_{i,j}$ for $i\in I,j\in J$.
    Then $\mesh$ with the formula $\beta_1$ satisfies
    the assumptions of \cref{lem:down-down} if 
    $\sim_1$ is~$\neq$, 
    and the assumptions of \cref{lem:down-down-negative} if $\sim_1$ is~$=$.
    We apply the appropriate lemma.
    In case of \cref{lem:down-down-negative}, this might yield a transformer and we are done.
    Otherwise, we obtain a conductor $\mesh_1,\ldots,\mesh_t\from I'\times J'\to V(G)$
    with \(|I'|=|J'| = U_{h,k}(n)\), $t\le 2h$, 
    where 
    $\mesh_t$ is vertical, and $\mesh_1=\mesh|_{I'\times J'}$.

    Denote $\mesh':=\mesh_1^\trans\from J'\times I'\to V(G)$.
    Then $\mesh'$ with the formula $\beta_2$ 
    satisfy the assumptions of Lemma~\ref{lem:down-down} 
    (with the roles of $I'$ and $J'$ swapped) if 
    $\sim_2$ is $\neq$, 
    and the assumptions of Lemma~\ref{lem:down-down-negative} if $\sim_2$ is $=$.
    We again apply the appropriate lemma.
    Either this yields a transformer and we are done, or
    we obtain a conductor $\mesh_1',\ldots,\mesh_u'\from J''\times I''\to V(G)$
    with \(|J''|=|I''| = U_{h,k}(|I'|)\), $u\le 2h$, 
    where $\mesh_u'$ is vertical and $\mesh_1'=\mesh'|_{J''\times I''}$.
    Notice that
    \[
        \mesh'_u|_{J'' \times I''},
        \dots,
        \mesh'_{1}|_{J'' \times I''}
        =
        \mesh^\trans_{1}|_{J'' \times I''}, 
        \dots, 
        \mesh^\trans_t|_{J'' \times I''}
    \]
    is a conductor of order \(U_{h,k}(n)\) and length $u+t-1\le 4h-1$.
    Observe that \(\mesh'_u\) and \(\mesh_{t}\) are both vertical, and that \(\mesh^\trans_{t}\) is horizontal.
    Thus, the sequence forms a transformer.
\end{proof}

\subsection{Regular and Minimal Transformers}

In this section, we normalize the transformers derived in \Cref{sec:pre-to-transformers}.

\begin{definition}\label{def:minimaltransformer}
    A transformer $T= (\mesh_1,\ldots,\mesh_h)$ in a graph $G$   is \emph{regular}
    if for all $s,t\in [h]$, the pair $(\mesh_s,\mesh_t)$ is a regular pair of meshes (in particular also for \(s=t\)).
    We say that $T$ is \emph{minimal} if it is regular
    and for all $s,t\in [h]$ the following conditions hold:
    \begin{itemize}
        \item $\mesh_s$ is vertical if and only if $s=1$,
        \item $\mesh_s$ is horizontal if and only if $s=h$,
        \item $\mesh_s\neq\mesh_t$ if $s\neq t$ (that is, no two meshes are identical),
        \item     the pair $(\mesh_s,\mesh_t)$ is conducting if $|s-t|=1$,
        \item     the pair $(\mesh_s,\mesh_t)$ is homogeneous if $|s-t|>1$.
    \end{itemize} 
\end{definition}    

In particular, in a minimal transformer $\mesh_1,\ldots,\mesh_h$,
we have that either $h=1$ and $\mesh_1=\mesh_h$ is both horizontal and vertical, 
or $h>1$ and $\mesh_1$ is vertical and not horizontal,
$\mesh_h$ is horizontal and not vertical, and $\mesh_2,\ldots,\mesh_{h-1}$ are neither horizontal nor vertical.

\begin{lemma}\label{lem:regularize}
    If $G$ contains a transformer of length $h$ and order $n$ then $G$ contains a regular transformer of length $h$ and order $U_h(n)$.
\end{lemma}
\begin{proof}
Let $\mesh_1,\ldots,\mesh_h$ be a transformer of length $h$ in $G$.
We can assume \(I,J=[n]\).
    For all $i,j\in[n]$, let $\pi_{i,j}\in V(\mesh)^h$
    be the $h$-tuple 
    $$\pi_{i,j}:=(\mesh_1(i,j),\ldots,\mesh_h(i,j)).$$
    By Bipartite Ramsey (\Cref{gridramsey}), there are sets $I',J'\subset [n]$ 
    with $|I'|=|J'|\ge U_h(n)$
    such that $\atp(\pi_{i,j},\pi_{i',j'})$
    depends only on $\otp(i,i')$ and $\otp(j,j')$,
    for all $i,i'\in I'$ and $j,j'\in J'$.
    It follows that for all $s,t\in [h]$,
    the meshes $\mesh_s|_{I'\times J'}$
    and $\mesh_t|_{I'\times J'}$ form a regular mesh pair.
    Thus, the sequence $\mesh_1|_{I'\times J'},\ldots,\mesh_h|_{I'\times J'}$ is a regular transformer.
\end{proof}

\begin{lemma}\label{lem:to-minimal}
    If $G$ contains a transformer of length $h$ and order $n$
    then $G$ contains a minimal transformer of length at most $h$ and order $U_h(n)$.
\end{lemma}
\begin{proof}
    By Lemma~\ref{lem:regularize}, there is 
    a regular transformer $\mesh_1,\ldots,\mesh_h$ of length $h$ and order $U_h(n)$ in $G$.
Consider the graph $\cal G$ whose vertices are the meshes $\mesh_1,\ldots,\mesh_h$,
and edges are pairs $\mesh_i\mesh_j$ such that the pair $(\mesh_i,\mesh_j)$ is conducting. 
Clearly, $\cal G$ contains a path of length $h$ which starts in a vertical mesh and ends in a horizontal mesh.
Let $\pi:=\mesh_{i_1},\ldots,\mesh_{i_p}$ be a shortest path in $\cal G$
which starts in a vertical mesh and ends in a horizontal mesh.
Then the path $\pi$ is an induced path of length at most $h$ in $\cal G$,
which means that a pair $\mesh_{i_s},\mesh_{i_t}$, for distinct $s,t\in [p]$ 
is conducting if and only if $|s-t|=1$.
As  every pair $\mesh_{i_s},\mesh_{i_t}$ is regular, 
it follows that $\mesh_{i_1},\ldots,\mesh_{i_p}$ is a minimal transformer of length $p\le h$ and order $U_h(n)$.
\end{proof}

\section{Converters and Crossings}\label{sec:conv-and-cross}
Our next goal is to analyze the structure of minimal transformers in graphs.
We will arrive at a notion of a converter, which is similar to a crossing. 
Finally, from converters, we will obtain crossings.

\subsection{Regular Pairs of Meshes}
We study the structure of regular pairs of meshes in graphs.
We introduce the following notions.

\begin{definition}
    Let $\mesh,\mesh'\from I\times J\to V(G)$ be two meshes in a graph $G$. We say that the pair $(\mesh,\mesh')$ is
    \begin{itemize}
        \item \emph{disjoint} if $V(\mesh)\cap V(\mesh')=\emptyset$,
        \item \emph{matched} if 
        for all $i,i'\in I$ and $j,j'\in J$,  $$G\models E(\mesh(i,j),\mesh'(i',j'))
        \quad\text{if and only if}\quad 
       (i,j)=(i',j').$$

       \item \emph{co-matched} if the pair $(\mesh,\mesh')$ is matched in the complement graph $\bar G$,
        \item \emph{non-adjacent} if  $V(\mesh)$ and $V(\mesh')$ are non-adjacent in \(G\),
        \item \emph{fully adjacent} if $uv\in E(G)$ for all $u\in V(\mesh)$ and $v\in V(\mesh')$.
    \end{itemize}
\end{definition}

We prove some preliminary observations regarding regular pairs of meshes.

\begin{lemma}\label{lem:identical-meshes}
    Let $\mesh,\mesh'\from I\times J\to V(G)$ be a regular pair of meshes of order $n>2$ in a graph $G$.
    Then $\mesh$ and $\mesh'$ are either identical, or disjoint.    
\end{lemma}
\begin{proof}
    We show that if $\mesh(i,j)=\mesh'(i',j')$ for some $i,i'\in I$ and $j,j'\in J$,
    then $(i,j)=(i',j')$.
    By regularity of the pair $(\mesh,\mesh')$, this implies that $\mesh(i,j)=\mesh'(i,j)$ for all $i,j\in [n]$, so $\mesh$ and $\mesh'$ are identical.
    So suppose that $\mesh(i,j)=\mesh'(i',j')$ for some $i,i'\in I$ and $j,j'\in J$
    with $(i,j)\neq(i',j')$.
    Up to reversing the order of $I$ and up to exchanging the role of $I$ and $J$, we can assume that $i<i'$. Pick $i_1<i_2<i_3\in I$.
    Then by regularity we have that $\mesh(i_1,j)=\mesh'(i_2,j')$ 
    and $\mesh(i_1,j)=\mesh'(i_3,j')$.
    Thus, $\mesh'(i_2,j')=\mesh'(i_3,j')$,
    contradicting injectivity of $\mesh'$.
\end{proof}

\begin{lemma}\label{lem:vertical}
    Let $\mesh,\mesh'\from I\times J\to V(G)$ be a regular pair of meshes of order $n$ in a graph $G$. 
    Suppose that $G \models E(\mesh(i,j),\mesh'(i',j'))$ is not the same for all $i,i'\in I$, $j,j'\in J$ with $j< j'$,
    or is not the same for all $i,i'\in I$, $j,j'\in J$ with $j> j'$.
    Then there exist subsequences $I' \subseteq I$ and $J' \subseteq J$ of order $U(n)$ such that the submeshes $\mesh|_{I'\times J'}$ and $\mesh'|_{I' \times J'}$ are both vertical.
\end{lemma}

\begin{proof}
    Let $j_{\min}:= \min(J)$, $j_{\max}:=\max(J)$, $J':=J-\set{j_{\min},j_{\max}}$, $I':= I -\set{\min(I), \max(I)}$.
    We show the argument for $\mesh$, while the case for $\mesh'$ follows by symmetry.
Suppose the first case holds, 
that is, $G \models E(\mesh(i,j),\mesh'(i',j'))$ is not the same for all $i,i'\in I$, $j,j'\in J$ with $j< j'$.
The other case proceeds by the same argument, exchanging the roles of $j_{\max}$ and $j_{\min}$.

Let $a(i')=\mesh'(i',j_{\max})$ for $i'\in I'$.
Then by regularity, 
$\atp(\mesh(i,j),a(i'))$ depends only on $\otp(i,i')$, for $i,i'\in I'$ and $j\in J'$.
Furthermore, $G \models E(\mesh(i,j),a(i'))$ is not the same for all $i,i'\in I'$ and $j\in J'$, by the assumption and by regularity. 
Hence, $\mesh|_{I'\times J'}$ is vertical.
\end{proof}

\begin{lemma}\label{lem:two-non-trivial}
    Let $\mesh,\mesh'\from I\times J\to V(G)$ be a regular pair of meshes of order $n$ in a graph $G$. Suppose that $G \models E(\mesh(i,j),\mesh'(i',j'))$ is not the same for all $i,i'\in I,j,j'\in J$ with 
    $(i,j)\neq (i',j')$. 
    Then there exist subsequences $I' \subseteq I$ and $J' \subseteq J$ of order $U(n)$ such that the submeshes $\mesh|_{I'\times J'}$ and $\mesh'|_{I' \times J'}$ are either both vertical, or both horizontal.
\end{lemma}
\begin{proof}
    It follows from the assumption 
    that one of the following cases holds:
\begin{itemize}
    \item $G \models E(\mesh(i,j),\mesh'(i',j'))$ is not the same for all $i,i'\in I$, $j,j'\in J$ with $j<j'$,
    \item $G \models E(\mesh(i,j),\mesh'(i',j'))$ is not the same for all $i,i'\in I$, $j,j'\in J$ with $j>j'$,    
    \item $G \models E(\mesh(i,j),\mesh'(i',j'))$ is not the same for all $i,i'\in I$, $j,j'\in J$ with $i<i'$,
    \item $G \models E(\mesh(i,j),\mesh'(i',j'))$ is not the same for all $i,i'\in I$, $j,j'\in J$ with $i>i'$.
\end{itemize}
In the first two cases we conclude by \cref{lem:vertical}. In the last two cases we conclude by applying the same lemma to $\mesh^\trans$ and $\mesh'^\trans$.
\end{proof}

\begin{lemma}\label{lem:regular-pairs}
    Let $\mesh,\mesh'\from I\times J\to V(G)$ be a conducting pair of disjoint meshes of order $n$ in~$G$.
    Then there exist subsequences $I' \subseteq I$ and $J' \subseteq J$ of order $U(n)$ such that $\mesh|_{I'\times J'}$ and $\mesh'|_{I' \times J'}$ are 
    either
    \begin{enumerate*}
        \item both vertical, 
        \item both horizontal, 
        \item matched, or
        \item co-matched.
    \end{enumerate*}
\end{lemma}
\begin{proof}
    Since $\mesh$ and $\mesh'$ are conducting, the pair is regular but not homogeneous.
    For all $i,i'\in I$ and $j,j'\in J$ we have that
    \begin{enumerate}[label={($\arabic*$)}]
        \item \label{itm:ramsey-regpair1} $G \models E(\mesh(i,j),\mesh'(i',j'))$ depends only on 
        $\otp(i,i')$ and $\otp(j,j')$, \hfill(by regularity)
        \item \label{itm:ramsey-regpair2} $\atp_G(\mesh(i,j),\mesh'(i',j'))$ is not always the same, \hfill(by non-homogeneity)
        \item \label{itm:ramsey-regpair3} the equality type of $\mesh(i,j)$ and $\mesh'(i',j')$ is always $(\neq)$, and \hfill(by disjointness)
        \item \label{itm:ramsey-regpair4} the adjacency between $\mesh(i,j)$ and $\mesh'(i',j')$ is not always the same. \hfill(by \ref{itm:ramsey-regpair2} and \ref{itm:ramsey-regpair3})
    \end{enumerate}

    Assume $G \models E(\mesh(i,j),\mesh'(i',j'))$ is the same for all $i,i'\in I$, $j,j'\in J$ with $(i,j)\neq (i',j')$.
    Up to replacing $G$ with $\bar G$,
     by \ref{itm:ramsey-regpair4} we have that for all $i,i'\in I$, $j,j'\in J$,
    $$G\models E(\mesh(i,j),\mesh'(i',j'))\quad\text{if and only if }(i,j)=(i',j').$$
    Thus, $\mesh$ and $\mesh'$ are matched or co-matched.
    
    Otherwise, $G \models E(\mesh(i,j),\mesh'(i',j'))$ is not the same for all $i,i'\in I$, $j,j'\in J$ with $(i,j)\neq (i',j')$, and we conclude that $\mesh$ and $\mesh'$ are both vertical or both horizontal by \cref{lem:two-non-trivial}.
\end{proof}

\begin{lemma}\label{lem:minimaltransformer}
    Let $G$ contain a minimal transformer of order $n$ and length $h$. Then 
    $G$ contains a minimal transformer $\mesh_1,\ldots,\mesh_{h'}$ of order $U_h(n)$ and length $h'\le h$ such that the following conditions are satisfied for $s,t\in [h']$:
        \begin{enumerate}
            \item\label{min:disj} If $s \neq t$, then $\mesh_s$ and $\mesh_t$ are disjoint.
            \item\label{min:homo} If $|s-t|>1$ then $\mesh_s$ and $\mesh_t$ are either non-adjacent or fully adjacent.
            \item\label{min:cons} If $|s - t| = 1$ then $\mesh_s$ and $\mesh_t$ are either matched or co-matched.
        \end{enumerate}
\end{lemma}
\begin{proof}
If \(n \le 3\), as the minimal transformer we take any transformer of order 1 and length 1, so assume \(n > 3\).
The first two properties hold in any minimal transformer $T$ of order $n>3$.
Indeed,  as all meshes of $T$ are pairwise distinct and regular,
by Lemma~\ref{lem:identical-meshes} they are pairwise disjoint.
Also, any pair of non-consecutive meshes is regular and not conducting, hence (as $n>3$) homogeneous.

We argue that we can find a minimal transformer satisfying additionally the last property. Let 
    $T=(\mesh_1,\ldots,\mesh_{h})$
    be a minimal transformer of order 
    $n$.

Applying Lemma~\ref{lem:regular-pairs} to every pair 
of consecutive meshes in $T$ 
(and each time reducing the order of the transformer to $U_h(n)$)
we may assume that for every pair $(\mesh_{s},\mesh_{s+1})$ of consecutive meshes in $T$,
the pair is either matched, or co-matched, or both meshes are vertical, or both are horizontal. Let $\mesh_s,\ldots,\mesh_{t}$ be a 
 subsequence of $\mesh_1,\ldots,\mesh_{h}$ of shortest length
such that $\mesh_s$ is vertical and $\mesh_t$ is horizontal.
It follows that every two consecutive meshes in $T'$ are matched or co-matched,
and that $T'=(\mesh_s,\ldots,\mesh_t)$ is a minimal transformer. Thus, the last property in the statement is satisfied.
\end{proof}

\subsection{Regular Meshes}

We now analyze the structure of single  meshes, depending on whether they are horizontal and/or vertical. 
We introduce some notation.

\begin{definition}
    A single mesh $\mesh \from I \times J \to V(G)$ is \emph{regular} in a graph $G$, if the pair $(\mesh,\mesh)$ is.
    That is,
    $\atp_G(\mesh(i,j),\mesh'(i',j'))$ 
    depends only on $\otp(i,i')$ and $\otp(j,j')$ for all $i,i'\in I$ and $j,j'\in J$.
\end{definition}
Note that every mesh in a regular/minimal transformer is regular.

\begin{definition}
    A \emph{mesh pattern}
    is a subset $P$ of the four lines in the following diagram:
    $$
    \lgrid{0,1,2,3}
    $$
    All  $2^4$ possible  mesh patterns are depicted below (including the empty pattern $\grid{}$).
    \begin{align*}%\label{eq:grid}
        \lgrid{0,1,2,3}\quad\lgrid{0,1,2}\quad\lgrid{0,1,3}\quad\lgrid{0,2,3}\quad\lgrid{1,2,3}\quad\lgrid{0,1}\quad\lgrid{0,2}\quad\lgrid{0,3}\quad
        \lgrid{1,2}\quad\lgrid{1,3}\quad\lgrid{2,3}\quad\lgrid{0}\quad\lgrid{1}\quad\lgrid{2}\quad\lgrid{3}\quad\lgrid{}
    \end{align*}
    Let $P$ be a mesh pattern.
    A regular mesh $\mesh\from I\times J\to V(G)$ with $|I|,|J|\ge 2$ is a \emph{$P$-mesh} in a graph $G$ if for all $(i,j),(i',j')\in I\times J$ with $i\le i'$, the vertices $\mesh(i,j)$ and $\mesh(i',j')$ are adjacent in $G$ if and only if one of the following conditions holds:
    \begin{itemize}
        \item $i=i'$ and $j\neq j'$ and $\grid{0}\in P$,
        \item $i<i'$ and $j<j'$ and $\grid{1}\in P$,
        \item $i<i'$ and $j=j'$ and $\grid{2}\in P$, 
        \item $i<i'$ and $j>j'$ and $\grid{3}\in P$.
    \end{itemize}
\end{definition}
    
    For example, $\mesh$ is a $\grid{}$-mesh if and only if $V(\mesh)$ induces an independent set, 
    and if $\mesh$ is a $\grid{0,2}\,$-mesh, then 
    $V(\mesh)$ induces a rook graph in $G$,
    and if $\mesh$ is a $\grid{0,1,2}\,$-mesh, then $V(\mesh)$ induces a comparability grid in $G$.
    (Recall that the \emph{comparability grid} of order \(n\) consists of vertices
\(\{ a_{i,j} \mid i,j \in [n] \}\) and edges between vertices \(a_{i,j}\) and \(a_{i',j'}\) if and only if either $i=i'$, or $j=j'$,
or $i<i'\Leftrightarrow j<j'$.)

    \begin{definition}
    A \emph{generalized grid} in a graph \(G\) is a regular mesh $\mesh\from I\times J\to V(G)$ 
    satisfying the following conditions:
        \begin{itemize}
            \item $G \models E(\mesh(i,j),\mesh(i',j'))$ does not depend only on $\otp(i,i')$, for $i,i'\in I$ and $j,j'\in J$ with $(i,j)\neq (i',j')$, and 
            \item $G \models E(\mesh(i,j),\mesh(i',j'))$ does not depend only on $\otp(j,j')$, for $i,i'\in I$ and $j,j'\in J$ with $(i,j)\neq (i',j')$.
        \end{itemize}
    \end{definition}
    Observe that a $P$-mesh $\mesh$ is a generalized grid if and only if $P$ is \emph{not}
    among $\set{\grid{},\grid{0},\grid{2}}$ 
    or their complements $\set{\grid{0,1,2,3},\grid{1,2,3},\grid{0,1,3}}$.

\begin{lemma}\label{lem:non-trivial-mesh}
    Let $\mesh\from I\times J\to V(G)$ be a regular mesh in a graph $G$.
    Then  $\mesh$ is vertical, or is horizontal, or is a $\grid{}$-mesh in $G$ or in $\bar G$.
\end{lemma}
\begin{proof}
Suppose that 
  $G \models E(\mesh(i,j),\mesh(i',j'))$ is the same for all $i,i'\in I$, $j,j'\in J$ with $(i,j)\neq (i',j')$.
  Then $V(\mesh)$ forms an independent set or a clique in $G$.
Otherwise, the statement follows by Lemma~\ref{lem:two-non-trivial}, applied to $\mesh=\mesh'$.
\end{proof}

\begin{lemma}\label{lem:same-for-i<i'}
    Let $\mesh\from I\times J\to V(G)$ be a regular mesh in a graph $G$,
    and assume $G \models E(\mesh(i,j),\mesh(i',j'))$ is the same for all $i,i'\in I$ and $j,j'\in J$ with $i<i'$.
    Then either in $G$ or in $\bar G$, $\mesh$ is a $\grid{0}$-mesh or a $\grid{}$-mesh.
\end{lemma}
\begin{proof}
 Replacing $G$ with $\bar G$ if needed,
 we may assume that 
 $G\models \neg E(\mesh(i,j),\mesh(i',j'))$ for all $i,i'\in I$ and $j,j'\in J$ with $i<i'$.
 By symmetry of the edge relation, we have that $G\models\neg E(\mesh(i,j),\mesh(i',j'))$ holds for all $i,i'\in I$ and $j,j'\in J$ with $i\neq i'$.
 If $G\models E(\mesh(i,j),\mesh(i,j'))$ for some $i\in I$ and  distinct $j,j'\in J$, then $\mesh$ is a $\grid{0}$-mesh.
 Otherwise, $\mesh$ is a $\grid{}$-mesh.
\end{proof}

\begin{lemma}\label{lem:not-horizontal}
    Let $\mesh$ be a regular mesh in a graph $G$
    which is not horizontal. 
    Then either in $G$ or in $\bar G$, $\mesh$ is a $\grid{0}$-mesh or a $\grid{}$-mesh.
\end{lemma}
\begin{proof}
    As \(\mesh^\trans\) is not vertical,
    we can apply the contrapositive of Lemma~\ref{lem:vertical} to $\mesh^\trans$ and \(\mesh^\trans\). 
    We conclude that 
 $G \models E(\mesh(i,j),\mesh(i',j'))$ is the same for all $i,i'\in I$ and $j,j'\in J$ with $i<i'$. The conclusion follows from Lemma~\ref{lem:same-for-i<i'}.
\end{proof}

\begin{definition}\label{def:capped}
    A mesh $\mesh\from I\times J\to V(G)$ is \emph{capped} if there is a function
    $a\from I'\to V(G)$, where $I'=I-\set{\min(I),\max(I)}$, such that one of the following conditions holds in $G$ or in $\bar G$:
    \begin{description}
        \item[($=$)]
        for all $i,i'\in I'$ and $j\in J$, $\mesh(i,j)$ is adjacent to $a(i')$ if and only if $i=i'$, or
        \item[($\le$)]
        for all $i,i'\in I'$ and $j\in J$, $\mesh(i,j)$ is adjacent to $a(i')$ if and only if $i\le i'$.
    \end{description}
    More precisely, a capped mesh $\mesh$ is $\alpha$-capped,
    for $\alpha\in\set{=,\le}$, if the above condition \(\alpha\) holds.
\end{definition}

\begin{lemma}\label{lem:vertical2}
    Let $\mesh\from I\times J\to V(G)$ be a vertical mesh in a graph $G$ with $|J|>1$.
    Then $\mesh$ is capped.
\end{lemma}
\begin{proof}
    Since $\mesh$ is vertical, there is some $a\from I\to V(G)$ 
    such that, for $J'=J-\set{\min(J),\max(J)}$,
    \begin{itemize}
        \item     $\atp(\mesh(i,j),a(i'))$, depends only on $\otp(i,i')$, for all $i,i'\in I$ and $j\in J'$, and 
        \item $\atp(\mesh(i,j),a(i'))$ is not the same for all $i,i'\in I$ and $j\in J'$.
    \end{itemize} 

First observe that the ranges of the functions $a$ and $\mesh$ are disjoint. Assume otherwise, that is, that $a(i')=\mesh(i,j)$ for some $i,i'\in I$ and $j\in J$.
Pick $j'\in J$ distinct from $j$, which exists since we assume that $|J|>1$.
Then we have that $a(i')=\mesh(i,j')$, by the first defining condition of the function $a$. This contradicts the fact that $\mesh$ is an injective function.

Let $p_<,p_=,p_>\in\set{0,1}$ 
be such that for each $R\in\set{<,=,>}$,
\[
    G\models E(\mesh(i,j),a(i'))\quad \iff\quad  p_R=1\quad\textit{for all $j\in J'$, $i,i'\in I$ with $i\,R\,i'$}.
\]

By the assumption on $a$, 
the values $p_<,p_=,p_>$ are not all equal.
Replacing $G$ with $\bar G$ if needed, we can assume that 
$p_>=0$. Thus, one of three cases occurs:
\begin{enumerate}
    \item $(p_<,p_=,p_>)=(0,1,0)$,
    \item $(p_<,p_=,p_>)=(1,1,0)$,
    \item $(p_<,p_=,p_>)=(1,0,0)$.
\end{enumerate}

Let $I'= I-\set{\min(I),\max(I)}$ and let $a|_{I'}$ be the restriction of $a$ to the domain $I'$.
In the first case, $(\mesh,a|_{I'})$ is a $=$-capped mesh.
In the second case, $(\mesh,a|_{I'})$ is a $\le$-capped mesh.
Suppose the third case occurs,
and let $b\from I'\to V(G)$ be such that $b(i)=a(i_+)$ where 
$i_+$ is the successor of $i$ in~$I$.
Then $(\mesh,b)$ is a $\le$-capped mesh.
\end{proof}

\begin{lemma}\label{lem:horizontal-vertical-cases}
    Let $\mesh\from I\times J\to V(G)$ be a regular mesh of order $n>1$ in a graph $G$. Then the following hold.
    \begin{enumerate}
        \item If $\mesh$ is not horizontal and not vertical, then $\mesh$ forms a $\grid{}$-mesh in $G$ or in $\bar G$.
        \item If $\mesh$ is vertical and not horizontal, then one of two cases occurs in $G$ or in $\bar G$: 
        \begin{enumerate}[label=(\alph*)]
            \item $\mesh$ is a $\grid{0}$-mesh, or 
            \item  $\mesh$
            is a $\grid{}$-mesh and capped.
        \end{enumerate}
        
         \item If $\mesh$ is both horizontal and vertical, then one of four cases occurs in $G$ or in $\bar G$:
         \begin{enumerate}[label=(\alph*)]
         \item\label{it:gen} $\mesh$ is a generalized grid,
         \item\label{it:0-alpha} $\mesh$ is a $\grid{0}$-mesh and $\mesh^\trans$ is capped,
          
          \item   $\mesh$ is a $\grid{2}\,$-mesh and $\mesh$ is capped, or
          \item\label{it:0-alpha-beta} $\mesh$ is a $\grid{}$-mesh and both \(\mesh\) and \(\mesh^\trans\) are capped.
         \end{enumerate}
    \end{enumerate}
\end{lemma}
\begin{proof}
    The first item is by Lemma~\ref{lem:non-trivial-mesh}.

\medskip
    We prove the second item. By Lemma~\ref{lem:not-horizontal}, either in $G$ or in $\bar G$, $\mesh$ is a $\grid{0}$-mesh or a \(\grid{}\)-mesh.
    In the first case we are done.
    In the second case,  Lemma~\ref{lem:vertical2}
yields the conclusion.

\medskip
Finally, we prove the third item. 
Assume that $\mesh$ is both vertical and horizontal.
Then, by Lemma~\ref{lem:vertical2}, both $\mesh$ and $\mesh^\trans$ are capped.
Suppose $\mesh$ is not a generalized grid, as otherwise condition \ref{it:gen} holds, and we are done.
Then either $G \models E(\mesh(i,j),\mesh(i',j'))$ depends only on $\otp(i,i')$, or it depends only on $\otp(j,j')$,
for all distinct $(i,j), (i',j')\in I\times J$.
Suppose it depends only on $\otp(i,i')$,
while the other case follows by replacing $\mesh$ with $\mesh^\trans$.
By Lemma~\ref{lem:same-for-i<i'} we have that either in $G$ or in $\bar G$,
$\mesh$ is a $\grid{0}$-mesh or a $\grid{}$-mesh.
If $\mesh$ is a $\grid{0}$-mesh, condition \ref{it:0-alpha} holds.
If $\mesh$ is a $\grid{}$-mesh, condition \ref{it:0-alpha-beta} holds.

\medskip
This concludes the lemma.
\end{proof}

\begin{lemma}\label{lem:generalized-grids}
    Let $G$ be a graph and $\mesh$ be a regular mesh of order $n$ in $G$ which 
    is a generalized grid. 
    Then \(\mesh\) is a $\grid{0,2}\,$-mesh or $\grid{1,3}$-mesh, 
    or \(G\) contains a comparability grid of order $\lfloor\sqrt{n}\rfloor$ as induced subgraph.
\end{lemma}

\begin{proof}
Let $\mesh$ be a regular mesh of order $m$ in $G$,
and let $P$ be the pattern of $\mesh$.
Note that if $P$ is among $\set{\grid{},\grid{0},\grid{2}}$ or their complements $\set{\grid{0,1,2,3},\grid{1,2,3},\grid{0,1,3}}$, then $\mesh$ is not a generalized grid.
So assume that $P$ is not among those patterns.
We can exclude $P\in \set{\grid{0,2},\grid{1,3}}$,
so assume that $P$ is among the remaining patterns,
that is, $P\in\set{\grid{1},\grid{3},\grid{0,1},\grid{1,2},\grid{0,3},\grid{2,3},\grid{0,1,2},\grid{0,2,3}}$.

Up to symmetries that swap the two coordinates and invert their orders, it is enough to consider the cases  $P\in\set{\grid{1},\grid{0,1},\grid{0,1,2}}$.
If $P=\grid{0,1,2}$, then $V(\mesh)$ induces a comparability grid of order $m$ in $G$.
It remains to  show that if $G$ has a $P$-mesh of order $m=n^2$, for some $P\in\set{\grid{1},\grid{0,1}}$,
then $G$ contains a comparability grid of order $n$.

\newcommand{\lex}{_{\mathrm{lex}}}
Let $I=[n]\times [n]$,
and let $\le\lex$ denote the lexicographic order on $[n]\times [n]$. 
To declutter notation, below we write $ij$ for a pair $(i,j)\in I$.

Assume $G$ has a $P$-mesh of order $n^2$.
By reindexing $([n^2],\le)$ as $(I,\le\lex)$,
we can view it as a \(P\)-mesh
$\mesh\from I\times I\to V(G)$ in \(G\).
Then for all $i_1,i_2,i_1',i_2' \in I$ and $j_1,j_2,j_1',j_2' \in I$ with $(i_1 i_2,j_1 j_2) \neq (i_1' i_2',j_1' j_2')$ we have that
\begin{itemize}
    \item if $P=\grid{1}$, then $\mesh(i_1 i_2,j_1 j_2)$ and $\mesh(i_1' i_2',j_1' j_2')$
    are adjacent if and only if
    
    \begin{equation}\label{eq:ramsey-pdiag}
        \underbrace{
        (i_1 i_2 <\lex  i_1' i_2' 
        \text{ and }
         j_1 j_2  <\lex j_1' j_2') 
        }_\text{\tiny edges from $\mesh(i_1 i_2,j_1 j_2)$ to the \color{blue}{top right}}   
        \text{ or }
        \underbrace{(i_1' i_2' <\lex i_1 i_2
        \text{ and }
        j_1' j_2' <\lex j_1 j_2),}
        _\text{\tiny edges from $\mesh(i_1 i_2,j_1 j_2)$ to the \color{red}{bottom left}}  
    \end{equation}

    \item if $P=\grid{0,1}$, then $\mesh(i_1 i_2,j_1 j_2)$ and $\mesh(i_1' i_2',j_1' j_2')$
    are adjacent if and only if
    \[
        \underbrace{(i_1 i_2 = i_1' i_2')}
        _{\mathclap{\text{\tiny edges in the \color{magenta}{{same column}}}}}
        \text{ or }
        \underbrace{
        (i_1 i_2 <\lex  i_1' i_2' 
        \text{ and }
         j_1 j_2  <\lex j_1' j_2') 
        }_\text{\tiny  edges from $\mesh(i_1 i_2,j_1 j_2)$ to the \color{blue}{top right}}   
        \text{ or }
        \underbrace{(i_1' i_2' <\lex i_1 i_2
        \text{ and }
        j_1' j_2' <\lex j_1 j_2).}
        _\text{\tiny edges from $\mesh(i_1 i_2,j_1 j_2)$ to the \color{red}{bottom left}}  
    \]
\end{itemize}
The adjacencies are depicted in \cref{fig:comp-grid2}.
Note that the cases  $P=\grid{1}$ and $P=\grid{0,1}$ differ only if $(i_1 i_2 = i_1' i_2')$.

\begin{figure}[htbp]
    \centering
    \includegraphics{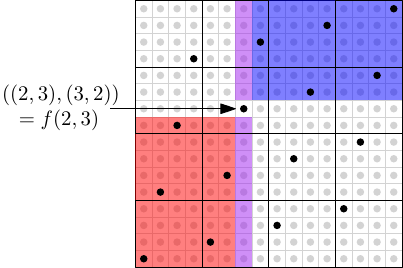}
    \caption{A $P$-mesh indexed by $([4]^2, <\lex) \times ([4]^2, <\lex)$.
    The neighbors of $((2,3),(3,2))$ are (red $\cup$ blue) if $P=\grid{1}$ and (red $\cup$ blue $\cup$ purple) if $P=\grid{0,1}$.
    In solid black: the vertices in the range of $f$ which form the comparability grid of order $4$.
    }
    \label{fig:comp-grid2}
\end{figure}

Consider the function $f\from[n]\times[n]\to I\times I$,
such that $$f(i,j)=(ij,ji)\qquad\text{for $i,j\in [n]$}.$$
The range of this function is depicted in \Cref{fig:comp-grid2}.
We verify that $\mesh':=\mesh\circ f\from [n]\times [n]\to V(G)$ 
is a $\grid{0,1,2}\,$-mesh.
Let $(i,j),(i',j')\in [n]\times[n]$ be distinct,
with $i\le i'$.
We need to show that 
$\mesh'(i,j)=\mesh(ij,ji)$ and $\mesh'(i',j')=\mesh(i'j',j'i')$ are adjacent in $G$ if and only if $i=i'$ or $j\leq j'$.
Note that $(ij)\neq (i'j')$, as $(i,j)\neq (i',j')$ are distinct, so the distinction between $P=\grid{1}$ and $P=\grid{0,1}$ is irrelevant. 
We can therefore assume $P=\grid{1}$, and we will argue using \eqref{eq:ramsey-pdiag} for 
\begin{center}
    $i_1 i_2 := ij$, $i_1' i_2' := i'j'$, $j_1 j_2 := ji$, and $j_1' j_2' := j'i'$.    
\end{center}
Assume first that $i=i'$.
In this case we want to show that $\mesh'(i,j)$ and $\mesh'(i',j')$ are adjacent.
As $ij$ and $i'j'$ are distinct, we either have $j<j'$ or $j' < j$.
If $j<j'$
then $ij <\lex i'j'$ and $ji <\lex j'i'$, and we conclude using the first disjunct of \eqref{eq:ramsey-pdiag}.
If $j' < j$ 
then $i'j' <\lex ij$ and $j'i' <\lex ji$, and we conclude using the second disjunct of \eqref{eq:ramsey-pdiag}.

Assume now that $i \neq i'$, so $i < i'$.
In this case we want to show that $\mesh'(i,j)$ and $\mesh'(i',j')$ are adjacent if and only if $j \leq j'$.
As $ij <\lex i'j'$, by \eqref{eq:ramsey-pdiag} we have that $\mesh'(i,j)$ and $\mesh'(i',j')$ are adjacent if and only if $ji <\lex j'i'$, which is the case if and only if $j \leq j'$, as desired.

Having exhausted all cases, we conclude that $\mesh'$ is a $\grid{0,1,2}\,$-mesh of order $n$,
and therefore, $G$ contains an induced comparability grid of order $n$.
\end{proof}

\subsection{Converters}\label{sec:conf}

The following notion is aimed at providing a low-level description of minimal transformers.

\begin{definition}Fix $h,n\ge 1$ and a graph $G$.
    A \emph{converter} of length $h$ and order $n$ in $G$ consists of
    \begin{itemize}
        \item  meshes $\mesh_1,\ldots,\mesh_h\from [n]\times [n]\to V(G)$ in $G$,
        \item two functions $a,b\from [n]\to V(G)$; we denote the range of these functions by \(V(a),V(b)\), 
    \end{itemize}     
     such that the following conditions hold:
    \begin{enumerate}[leftmargin= 4em, label={(G.$\arabic*$)}]
        \item\label{itm:Gone} the sets $V(\mesh_1),\ldots,V(\mesh_h)$ are pairwise disjoint,
        \item for $s,t\in \set{1,\ldots,h}$, if $|s-t|=1$, then
        $\mesh_s$ is matched or co-matched with $\mesh_{t}$, 
        
        \item for $s,t\in\set{1,\ldots,h}$, if $|s-t|>1$, then 
         $V(\mesh_s)$ is fully adjacent or non-adjacent to $V(\mesh_{t})$,
         \item for $s\in \set{2,\ldots,h-1}$, $V(\mesh_s)$ is an independent set or a clique,
        \item\label{conf:ends} Let $(\mesh,f)\in\set{(\mesh_1,a),(\mesh_h^\trans,b)}$.
         Then one of the following three conditions holds in $G$ or in $\bar G$:
        \begin{description}
            \item[(C)] $\mesh$ is a $\grid{0}$-mesh, or $h=1$ and $\mesh$ is a $\grid{0,2}\,$-mesh.
             Moreover,
             $\mesh(i,j)$ is adjacent to $f(i')$  if and only if $i=i'$, for all $i,i',j\in[n]$,  
            \item[(S)] $\mesh$ is a $\grid{}$-mesh, or $h=1$ and $\mesh$ is a $\grid{2}\,$-mesh. Moreover,
            $\mesh(i,j)$ is adjacent to $f(i')$  if and only if $i=i'$, for all $i,i',j\in[n]$,  
            \item[(H)] $\mesh$ is a $\grid{}$-mesh, or $h=1$ and $\mesh$ is a $\grid{2}\,$-mesh. Moreover,  $\mesh(i,j)$ is adjacent to $f(i')$  if and only if $i'\le i$, for all $i,i',j\in[n]$.
        \end{description}          
    \end{enumerate}
    Say that $(\mesh,f)$ as above 
    has \emph{kind} $\cl,\st$, or $\hg$ respectively, 
    if it satisfies the corresponding condition above
    (which stand for \emph{clique}, \emph{star}, and \emph{half-graph}, respectively).
    A converter has \emph{kind} $(\alpha,\beta)$, where $\alpha,\beta\in\set{\cl,\st,\hg}$, if $(\mesh_1,a)$ has kind $\alpha$ and $(\mesh_h^\trans,b)$ has kind~$\beta$.

    A converter is \emph{proper} if the following hold:
    \begin{itemize}
        \item the sets $V(a),V(\mesh_1),\ldots,V(\mesh_h),V(b)$ are pairwise disjoint,
        \item the sets $V(a)$ and $V(\mesh_{2}) \cup \dots \cup V(\mesh_h)$ are homogeneously connected, 
        \item the sets $V(b)$ and $V(\mesh_{1}) \cup \dots \cup V(\mesh_{h-1})$ are homogeneously connected, 
        \item the sets $V(a)$ and $V(b)$ are homogeneously connected, 
        \item each of the sets $V(a)$ and $V(b)$ induces an independent set or a clique. 
    \end{itemize}
\end{definition}

\begin{lemma}\label{lem:proper}
    Let $G$ contain a converter of length  $h$ and order $n$. Then, for some number $m\ge U_h(n)$,
    either $G$ contains a proper converter of length at most $h$ and order $m$, or $m=1$.
\end{lemma}
\begin{proof}
Let $\mesh_1,\ldots,\mesh_h\from [n]\times[n]\to V(G)$ and $a,b\from [n]\to V(G)$ form a converter.
For every pair $i,j\in [n]$ let $\pi_{i,j}$
denote the sequence $(a(i),\mesh_1(i,j),\ldots,\mesh_h(i,j),b(j))$. 
Apply Bipartite Ramsey (\Cref{gridramsey}) to get sets
 $I\subset [n]$ and $J\subset [n]$ of size $m=U_h(n)$,
 so that 
$\atp(\pi_{i,j}\pi_{i',j'})$ only depends on $\otp(i,i')$ and $\otp(j,j')$ for $i\in I$ and $j\in J$.
Observe that if $a(i)=\mesh_s(i',j)$ for some $i,i'\in I$, $j\in J$, and $s\in[h]$, then also $a(i)=\mesh_s(i',j')$ for some other $j'$, which implies $\mesh_s(i',j)=\mesh_s(i',j')$ and contradicts injectivity of $\mesh_s$, unless $|J|=m=1$.
Similarly, if $a(i)=b(j)$ for some $i\in I$ and $j\in J$, then we conclude that $a(i)=a(i')$ for all $i,i'\in I$, which implies that $|I|=m=1$, since 
the conditions $\cl,\st,\hg$ imply that $a$ is injective.
A similar argument can be made for \(b\).

The tuple $\mesh_1':=\mesh_1|_{I\times J},\ldots,\mesh_h':=\mesh_h|_{I\times J},a':=a|_I,b':=b|_J$, with $I$ and $J$ both reindexed to $[m]$,
forms a converter of order $m$.
As discussed above and by \ref{itm:Gone}, we observe the sets $V(a'),V(b'),V(\mesh'_1),\ldots,V(\mesh'_h)$ to be pairwise disjoint.
Further note that $V(a')$ and $V(b')$ each induce an independent set or a clique by construction.
Lastly, to have a \emph{proper} converter,
we have to ensure that $V(\mesh_t')$ and $V(b')$ are homogeneous for all $1\le t<h$,
and that $V(\mesh_t')$ and $V(a')$ are homogeneous for all $1<t\le h$. 

We only sketch this last argument.
Suppose for example that $V(b')$ and $V(\mesh_t')$ 
are not homogeneous, for some $1\le t<h$.
Then $\mesh_t'$ is either horizontal, or is vertical, as witnessed by $b'$. We can therefore obtain a converter of the same order and smaller length.
Thus, by replacing the converter $a',\mesh_1',\ldots,\mesh_h',b'$ by a shorter one if needed, we arrive at a proper converter of order~$m$.
\end{proof}

\begin{lemma}\label{lem:to-converter}
    Let $G$ contain a transformer of length $h$ and order $n$. Then there is a number $m\ge U_h(n)$ such that,
    $G$ contains a converter of length at most $h$ and order $m$,
    or contains a comparability grid of order $m$ as an induced subgraph.
\end{lemma}
\begin{proof}
    By Lemma~\ref{lem:to-minimal} and Lemma~\ref{lem:minimaltransformer}, $G$ contains a minimal transformer $\mesh_1,\ldots,\mesh_{h'}\from I\times J\to V(G)$ of order $U_h(n)$ and length $h'\le h$,
such that:
\begin{itemize}
    \item    the sets $V(\mesh_1),\ldots,V(\mesh_h')$ are pairwise disjoint,
    \item for $s,t\in [h']$ with $|s-t|=1$, $\mesh_s$ and $\mesh_t$ are matched or co-matched,
    \item for  $s,t\in [h']$ with $|s-t|>1$, $\mesh_s$ and $\mesh_t$ are fully adjacent or fully non-adjacent.
\end{itemize}

    Let $I'=I-\set{\min(I),\max(I)}$ and $J'=J-\set{\min(J),\max(J)}$.

\medskip
Suppose first that $h'=1$.
Then $\mesh_1$ is both vertical and horizontal.
By Lemma~\ref{lem:horizontal-vertical-cases},
    one of the following holds for $\mesh:=\mesh_1$ in either $G$ or $\bar G$:
    \begin{enumerate}[label=(\alph*)]
        \item\label{it:gen'} $\mesh$ is a generalized grid,
        \item\label{it:0-alpha'} $\mesh$ is a $\grid{0}$-mesh and $\mesh^\trans$ is $\alpha$-capped for some $\alpha\in\set{=,\le}$, 
         
         \item \label{it:0-beta'}  $\mesh$ is a $\grid{2}\,$-mesh and $\mesh$ is $\alpha$-capped for some $\alpha\in\set{=,\le}$,  or 
         \item\label{it:0-alpha-beta'} $\mesh$ is a $\grid{}$-mesh and 
        there are some $\alpha_1,\alpha_2\in \set{=,\le}$ such that $\mesh$ is $\alpha_1$-capped and $\mesh^\trans$ is $\alpha_2$-capped.
        \end{enumerate}
    In case \ref{it:gen'}, \(\mesh\) is by definition a generalized grid in both \(\bar G\) and \(G\). 
    By Lemma~\ref{lem:generalized-grids}, \(\mesh\) is a \(P\)-mesh with $P\in \set{\grid{0,2},\grid{1,3}}$ in \(G\),
    or \(G\) contains or a comparability grid of order $U_h(n)$ as an induced graph.
    If \(\mesh\) is a \(P\)-mesh, then $\mesh|_{I'\times J'}$ together with $a,b$ has kind $(\cl,\cl)$,
    where $a\from I'\to V(G)$ is defined by $a(i)=\mesh(i,\min(J))$ for $i\in I'$, and
    where $b\from J'\to V(G)$ is defined by $b(j)=\mesh(\min(I),j)$ for $j\in J'$.
    In either case, the statement holds.

    Suppose we are in case \ref{it:0-alpha'}.
     Let $b\from J'\to V(G)$
    witness that $\mesh^\trans$ is $\alpha$-capped.
    Let $a\from I'\to V(G)$ 
    be defined by $a(i)=\mesh(i,\min(J))$.
    Then $\mesh|_{I'\times J'}$, $a$ and $b$ form a converter of kind $(\cl,\st)$
    if $\alpha$~is~$=$, and a converter of kind $(\cl,\hg)$ if $\alpha$ is $\le$.
    The case \ref{it:0-beta'} is symmetric.

    Finally, suppose we are in case \ref{it:0-alpha-beta'}. 
    Let $a\from I'\to V(G)$
    witness that $\mesh$ is $\alpha_1$-capped, 
    and $b\from J'\to V(G)$
    witness that $\mesh^\trans$ is $\alpha_2$-capped.
    Then $\mesh|_{I'\times J'}$, $a$, and $b$ 
    form a converter of kind $(\tau_1,\tau_2)$, where 
    $\tau_i=\st$ if $\alpha_i$ is $=$,
    and $\tau_i=\hg$ if $\alpha_i$ is $\le$, for $i=1,2$.

\medskip
This concludes the case of length $h'=1$.
Suppose now that $h'>1$.
Then, by the minimality assumption (\Cref{def:minimaltransformer}),
\begin{itemize}
    \item $\mesh_1$ is vertical and not horizontal,
    \item $\mesh_{h'}$ is horizontal and not vertical, and 
    \item $\mesh_2,\ldots,\mesh_{h'-1}$ are neither vertical nor horizontal.
\end{itemize}

By Lemma~\ref{lem:horizontal-vertical-cases}, 
we conclude that each of $\mesh_2,\ldots,\mesh_{h'-1}$ is a $\grid{}$-mesh in either $G$ or $\bar G$.
By Lemma~\ref{lem:horizontal-vertical-cases} applied to $\mesh_1$, one of two cases holds in either \(G\) or \(\bar G\):
\begin{enumerate}[label=(\alph*)]
    \item $\mesh_1$ is a $\grid{0}$-mesh, or 
    \item $\mesh_1$ is a $\grid{}$-mesh and is $\alpha$-capped for some $\alpha\in\set{=,\le}$.
\end{enumerate}
In the first case, $\mesh_1|_{I'\times J'}$ together with $a$ has kind $\cl$,
where $a\from I'\to V(G)$ is defined by $a(i)=\mesh(i,\min(J))$ for $i\in I'$.
In the second case, let $a\from I'\to V(G)$
witness that $\mesh_1$ is $\alpha$-capped.
Then 
$\mesh_1|_{I'\times J'}$ together with $a$ has kind $\st$ if 
$\alpha$ is $=$, and has kind  $\hg$ if $\alpha$ is $\le$.

Similarly, by Lemma~\ref{lem:horizontal-vertical-cases} applied to $\mesh_{h'}^\trans$, we conclude that
 $\mesh_{h'}^\trans|_{J'\times I'}$ together with some $b\from J'\to V(G)$ has kind $\cl,\st$ or $\hg$.
We thus conclude that $\mesh_1,\ldots,\mesh_{h'}$ induced on $I'\times J'$, together with $a\from I'\to V(G)$ and $b\from J'\to V(G)$,
form a converter of length $h'$.
\end{proof}

\subsection{Crossings}\label{sec:symmetric-patterns}
The last step is to go from converters to crossings, whose definition we recall for convenience.
For \(r \ge 1\), the \emph{star \(r\)-crossing} of order \(n\) is the \(r\)-subdivision of \(K_{n,n}\).
More precisely, it consists of \emph{roots} \(a_1,\dots,a_n\) and \(b_1,\dots,b_n\)
together with \(r\)-vertex paths \(\{ \pi_{i,j} \mid i,j \in [n] \}\) that are pairwise vertex-disjoint (see \Cref{fig:pattern}).
We denote the two endpoints of a path \(\pi_{i,j}\) by \(\Start(\pi_{i,j})\) and \(\End(\pi_{i,j})\).
We require that roots appear on no path, that each root \(a_i\) is adjacent to \(\{ \Start(\pi_{i,j}) \mid j \in [n] \}\),
and that each root \(b_j\) is adjacent to \(\{ \End(\pi_{i,j}) \mid i \in [n] \}\).
The \emph{clique \(r\)-crossing} of order $n$ is the graph obtained from the star \(r\)-crossing of order $n$
by turning the neighborhood of each root into a clique.
Moreover, we define the \emph{half-graph \(r\)-crossing} of order $n$ similarly to the star \(r\)-crossing of order $n$,
where each root \(a_i\) is instead adjacent to \(\{ \Start(\pi_{i',j}) \mid i',j \in [n], i \le i' \}\),
and each root \(b_j\) is instead adjacent to \(\{ \End(\pi_{i,j'}) \mid i,j' \in [n], j \le j' \}\).
Each of the three $r$-crossings contains no edges other than the ones described.

\begin{figure}[h]
    \begin{center}
    \includegraphics[width=\textwidth]{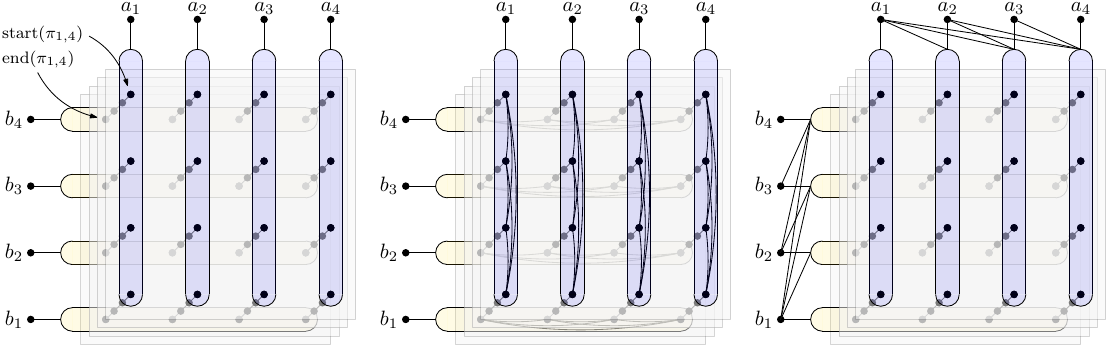}%
    \end{center}%
    \vspace{-0.5cm}
    \caption{
    (\emph{i}) star 4-crossing of order 4. 
    (\emph{ii}) clique 4-crossing of order 4. 
    (\emph{iii}) half-graph 4-crossing of order 4. 
    In (\emph{i}), (\emph{ii}), (\emph{iii}), the roots are adjacent to all vertices in their respective colorful strip.
    }\label{fig:crossings}
\end{figure}

We partition the vertex sets of the \(r\)-crossings into \emph{layers}:
The 0th layer consists of the vertices \(\{a_1,\dots,a_n\}\).
The \(l\)th layer, for \(l \in [r]\), consists of the \(l\)th vertices of the paths \(\{ \pi_{i,j} \mid i,j \in [n] \} \)
(that is, the 1st and \(r\)th layer, respectively, are
\(\{ \Start(\pi_{i,j}) \mid i,j \in [n] \}\) and
\(\{ \End(\pi_{i,j}) \mid i,j \in [n] \}\)).
Finally, the \((r+1)\)th layer consists of the vertices \(\{b_1,\dots,b_n\}\). 
A \emph{flipped} star/clique/half-graph \(r\)-crossing
is a graph obtained from a 
star/clique/half-graph \(r\)-crossing
by performing a flip where the parts of the specifying partition are the layers of the $r$-crossing.

We observe that a flipped $r$-crossing is the same as a proper converter of kind $(\alpha,\alpha)$ for some $\alpha\in\set{\cl,\st,\hg}$.
So the goal is to show that 
from a converter of kind $(\alpha,\beta)$ we can extract a converter 
with $\alpha=\beta$.
This is achieved in the next lemma. 

\begin{lemma}\label{lem:crossings}
    Let $G$ be a graph
    containing proper converter $C$ of length $h$ and order $n$.
    Then $G$ contains as an induced subgraph either
    \begin{itemize}
        \item a flipped star $r$-crossing, or
        \item a flipped clique $r$-crossing, or
        \item a flipped half-graph $r$-crossing,
    \end{itemize}
    of order $\lfloor\sqrt n\rfloor$, for some $1 \le r\le 2h+1$.
\end{lemma}

\begin{proof}
    Suppose the meshes $\mesh_1,\ldots,\mesh_h\from [n]\times[n]\to V(G)$ and functions $a,b\from [n]\to V(G)$ form a proper converter of kind $(\alpha,\beta)$ and order $n$ in $G$. 
    If \((\alpha , \beta)\neq (\hg,\hg)\),
    then either $\alpha\in\set{\cl,\st}$ or $\beta\in\set{\cl,\st}$.
    By replacing $\mesh_1,\ldots,\mesh_h$ with the converter $\mesh_h^\trans,\ldots,\mesh_1^\trans$ of kind $(\beta,\alpha)$ if needed,
    we may assume that either 
    $\beta\in\set{\cl,\st}$,
    or $\alpha=\beta=\hg$.

    Let $\cal P=\set{V(a),V(\mesh_1),\ldots,V(\mesh_h),V(b)}$.
    By taking an induced subgraph if needed,
    we may assume that $V(G)=\bigcup \cal P$. 
    It follows from the definition of a proper converter that there is a  unique $\cal P$-flip $G'$ of $G$ 
    with the following properties:
    \begin{itemize}
        \item the meshes $\mesh_s$ and $\mesh_t$ are non-adjacent if $|t-s|>1$ and are matched if $|t-s|=1$,
        \item for $1< s\le h$, $V(\mesh_s)$ is non-adjacent to $V(a)$,
        \item for $1\le s<h$, $V(\mesh_s)$ is non-adjacent to $V(b)$,
        \item $V(a)$ and $V(b)$ are non-adjacent,
    \end{itemize}

\begin{figure}[h]
    \begin{center}
    \includegraphics[scale=1]{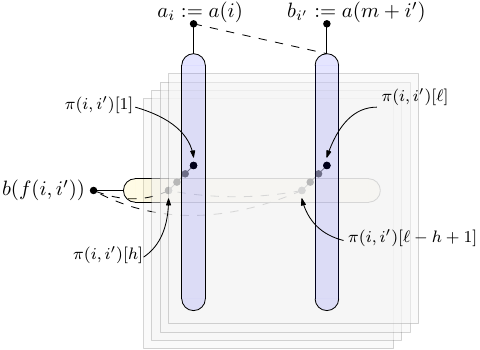}%
    \end{center}%
    \vspace{-0.5cm}
    \caption{Construction of an \(\ell\)-crossing from a proper converter.
        Edges whose presence depends on the values of \(\alpha\) and \(\beta\) are dashed.
    }\label{fig:patternlaststep}
\end{figure}

In the case $\alpha=\beta=\hg$ 
we construct a flipped half-graph $h$-crossing of order $n$ in $G$.
For this, we choose the roots \(a_1 := a(1),\dots,a_n := a(n)\) on one side, and \(b_1 := b(1),\dots,b_n := b(n)\) on the other side,
and connect them via the paths \(\pi_{i,j}\) 
with $\pi_{i,j}[s]=\mesh_1(i,j)$ for $s=1,\ldots,h$.
It follows that $G$ contains a flipped half-graph $h$-crossing of order \(n\) as an induced subgraph.

\medskip
Consider now the case where $\beta\in\set{\cl,\st}$.
We observe that 
\begin{itemize}
        \item 
        for all $i,j,j'\in[n]$,  $\mesh_h(i,j)$ is adjacent to $b(j')$ if and only if \(j=j'\).
\end{itemize}
Set $\ell:=2h$ if $\beta=\cl$ and $\ell:=2h+1$ if $\beta=\st$.
We construct an \(\ell\)-crossing.
For all $i,i',j\in [n]$ with $i\neq i'$ consider the unique
\(\ell\)-vertex path 
$\sigma_{i,i',j}$ in $G'$
from \(\mesh_1(i,j)\) to \(\mesh_1(i',j)\)
with the following properties:
\begin{itemize}
    \item $\sigma_{i,i',j}$ has length $\ell$, 
    \item $\sigma_{i,i',j}[d]=\mesh_d(i,j)$ for $d=1,\ldots,h$,
    \item $\sigma_{i,i',j}[\ell+1-d]=\mesh_d(i',j)$ for $d=1,\ldots,h$,
    \item if $\beta=\st$ then $\sigma_{i,i',j}[h+1]=b(j)$.
\end{itemize}
Note that  $\sigma_{i,j,j'}$ is an induced path in $G'$.
Let $m = \lfloor\sqrt{n}\rfloor$ and pick any injective function  $f\from [m]\times [m]\to [n]$.
For $i,i'\in[m]$, let 
$\pi_{i,i'}:=\sigma_{i,i'+m,f(i,i')}$.
Observe that there are no edges in \(G\) between any pair of paths \(\pi_{i,i'}\), \(i,i' \in [m]\).
We construct a flipped star/clique/half-graph $\ell$-crossing in $G$, corresponding to the cases $\alpha=\st,\cl,\hg$, respectively.
For this, we choose the roots \(a_1 := a(1),\dots,a_m := a(m)\) on one side, and \(b_1 := a(m+1),\dots,b_m := a(2m)\) on the other side
and connect them via the paths \(\pi_{i,i'}\) for \(i,i' \in [m]\).
Note that if \(\alpha=\hg\), then \(\{a_i \mid i \in [m]\}\) is fully connected to \(\{\End(\pi_{i,i'}) \mid i,i' \in [m]\}\).
The construction is illustrated in \Cref{fig:patternlaststep}.
It follows that $G$ contains a flipped $\ell$-crossing of order \(m\) as an induced subgraph.
\end{proof}

% Proof of \Cref{prop:patterns-main}
\subsection{From Prepatterns to
Flipped Crossings and Comparability Grids}
We finally prove the main non-structure result.

\thmPatterns*

\begin{proof}By Proposition~\ref{prop:transformer},
    $G$ contains a transformer of order $U_{k,h}(n)$ and length at most $4h-1$. By Lemma~\ref{lem:to-converter},
 $G$ contains a comparability grid of order  $U_{k,h}(n)$ as an induced subgraph, 
 or a converter of length at most $4h-1$ and order $U_{k,h}(n)$. In the latter case, we can make the converter proper 
 by Lemma~\ref{lem:proper}. By Lemma~\ref{lem:crossings}, we obtain the required flipped $r$-crossing of order $U_{k,h}(n)$
 for some $r\le 2 \cdot (4h-1)+1 = 8h-1$.
\end{proof}

As a corollary, we obtain the following.

\begin{proposition}\label{prop:prepatterns-to-patterns}
    Let $\CC$ be a graph class that is not prepattern-free.
    Then there exists $r\in \N$ such that for every $k\in \N$, $\CC$ contains as an induced subgraph either
    \begin{itemize}
        \item a flipped star $r$-crossing of order $k$, or
        \item a flipped clique $r$-crossing of order $k$, or
        \item a flipped half-graph $r$-crossing of order $k$, or
        \item a comparability grid of order $k$.
    \end{itemize}
\end{proposition}

\subsection{From
Flipped Crossings and Comparability Grids to Independence}
\label{sec:to-independence}

The next result states that the patterns are a witness to monadic independence.
\begin{proposition}\label{lem:transduceAll}
    Let \(\mathcal{C}\) be a graph class and $r\geq 1$, such that for every $k\in \N$,
    \(\mathcal{C}\) contains as an induced subgraph
    \begin{itemize}
        \item a flipped star $r$-crossing of order $k$, or
        \item a flipped clique $r$-crossing of order $k$, or
        \item a flipped half-graph $r$-crossing of order $k$, or
        \item a comparability grid of order \(k\). 
    \end{itemize}
    Then $\CC$ is monadically independent.
\end{proposition}

\begin{proof}[Proof sketch]
    We will be brief, as a more rigorous proof of a stronger statement is later given in \cref{prop:patternsImplyNotNIP}, where we show that the hereditary closure of $\CC$ \emph{interprets} the class of all graphs.
    For this sketch, we use the fact that transductions are transitive: if $\CC$ transduces $\mathcal{D}$ and $\mathcal{D}$ transduces $\mathcal{E}$, then already $\mathcal{C}$ transduces $\mathcal{E}$.
    We transduce the class of all graphs from $\CC$ by concatenating multiple simpler transductions,
    each of which depends only on \(r\) (and not on \(\mathcal{C}\) or~\(k\)).

    As there is a transduction which produces from $\CC$ all its induced subgraphs, we can assume that $\CC$ is hereditary.
    Additionally, there is a fixed transduction that turns flipped $r$-crossings into their non-flipped versions (or more generally: a transduction that maps $\CC$ to all of its $(r+2)$-flips).
    Now by the pigeonhole principle, we can assume that $\CC$ is either the class of all (non-flipped) star/clique/half-graph $r$-crossings or the class of all comparability grids.
    
    We start with describing a transduction
    which, when given a star \(r\)-crossing with roots \(A\) and \(B\),
    can create an arbitrary bipartite graph \((A,B, E)\).
    The transduction colors the roots in \(A\) and  \(B\) with colors \(C_A\) and \(C_B\) respectively, 
    and the vertices on the paths between \(a \in A\) and \(b \in B\) 
    with color \(C_+\) if \(\{a,b\} \in E\) and color \(C_-\) otherwise.
    It is then trivial to connect \(a \in A\) with \(b \in B\)
    by a first-order formula checking if there is a path of color \(C_+\) between them.

    Let us further argue that there is a transduction that takes as input a clique or half-graph \(r\)-crossing
    with root sets \(A\) and \(B\), and creates a star \(r\)-crossing with the same roots.
    For clique \(r\)-crossings this is easy to see, as it suffices to turn the neighborhood of each root into a star.
    For half-graph \(r\)-crossings, we first recover the linear order of the roots.
    Focusing on the side \(A\) first, we observe for \(a,a' \in A\) that \(a \le a'\) if and only if \(N(a) \subseteq N(a')\).
    The latter condition is expressible in first-order logic, therefore a transduction can remove all edges from a vertex \(a \in A\) to \(\bigcup_{a' > a} N(a')\).
    By also proceeding similarly for the roots in \(B\), this turns a half-graph \(r\)-crossing into a star \(r\)-crossing.

    The case of comparability grids proceeds similarly as the case of half-graph $1$-crossings above. Namely, 
    for a bipartite graph $G=(A,B,E)$ with $A=\set{a_1,\ldots,a_k}$ 
    and $B=\set{b_1,\ldots,b_\ell}$,
we consider a comparability grid with vertex set $\set{0,\ldots,k}\times\set{0,\ldots,\ell}$,
    and identify $B$ with the vertices $[k]\times\set{0}$,
    and $A$ with the vertices  $\set{0}\times [\ell]$,
    and then proceed as in the case of half-graph $1$-crossings
    (see also the proof of Lemma~\ref{lem:fw-comp-grid} where this construction is formalized).

    In summary, by chaining transductions, we can take a flipped \(r\)-crossing (or comparability grid),
    undo the flips, turn it into a star \(r\)-crossing,
    and then into any bipartite graph on the same root sets.
    From the class of bipartite graphs one can trivially transduce the class of all graphs.
    Hence, \(\mathcal{C}\) is monadically independent.
\end{proof}

We formulate a lemma which follows from the construction 
presented in the proof sketch above, and will be used  in the proofs of \cref{thm:quant} and \cref{thm:small}.

We recall the notion of radius-$r$ encodings of a bipartite graph $G$.
Fix an integer $r\ge 1$.
Let $G=(A,B,E)$ be a bipartite graph with $|A|=|B|=n$ for some $n$,
and let $A=\set{a_1,\ldots,a_n}$ and $B=\set{b_1,\ldots,b_n}$.
Consider a graph $H_0$ which is either a star $r$-crossing, or a clique $r$-crossing, or a half-graph $r$-crossing
with roots $a_1,\ldots,a_n$ and $b_1,\ldots,b_n$.
Recall that $V(H_0)$ can be partitioned into $r+2$ \emph{layers}, and there are $n^2$ distinguished $r$-vertex \emph{paths} $\pi_{i,j}$.
Let $H$ be a graph  obtained from $H_0$ by:
\begin{enumerate}
    \item adding arbitrary edges within each layer of $H_0$,
    \item removing all vertices of the paths $\pi_{i,j}$ for $i,j\in[n]$ such that $\set{a_i,b_j}\notin E(G)$,
    \item flipping pairs of layers arbitrarily.
\end{enumerate}
We call $H$ a \emph{radius-$r$ encoding of $G$}.

\begin{lemma}\label{lem:encode}
    For every fixed $r$ there is a number $k$ and a formula $\phi(x,y)$ 
    in the signature of $k$-colored graphs, with the following property.
    For every bipartite graph $G$
    and radius-$r$ encoding $H$ of $G$,
    there is a $k$-coloring $H^+$ of $H$
such that the graph 
$\phi(H^+)$ with vertices 
$V(H)$ 
and edges 
$$\setof{\set{u,v}}{u,v\in V(H), H^+\models \phi(u,v)}$$
contains the $1$-subdivision of $G$ as an induced subgraph.
\end{lemma}
\begin{proof}[Proof sketch]
    We modify slightly the construction described in the proof of \cref{lem:transduceAll}. 
    There, a transduction was described which, given a star/clique/half-graph $r$-crossing $H_0$ with roots $A$ and $B$, or a comparability grid, can output an arbitrary bipartite graph $G=(A,B,E)$.
    Essentially, we observe that the vertices marked $C_-$ in the proof of \cref{lem:transduceAll} can be removed from $H_0$  without breaking the construction (with one detail in the case of half-graph $r$-crossings, commented below).
    Note that, differing from $r$-crossings, radius-$r$ encodings 
    may contain additional arbitrary edges between vertices in the same layer.
    However, by coloring each of the at most $r+2$ layers in a different color palette (only increasing the number of colors by a factor of $r+2$), we can ignore these arbitrary edges.

    To prove the lemma, we therefore proceed the same way as in that construction,
    but instead of representing $G$ in the appropriate $r$-crossing with roots $A$ and $B$,
    we represent it in a radius-$r$ encoding $H$ of $G$.
    
    % which corresponds to removing 
    %  removing all vertices on paths connecting roots $a\in A$ and $b\in B$, such that $\set{a,b}\notin E$
    % (that is, removing all vertices that are colored $C_-$ in the construction from \cref{lem:transduceAll}).

    In the case of $r$-half-graph crossings,  we additionally mark  (doubling the number of colors) the vertices of $A$ and $B$ that are 
    isolated in $G$. The key observation is that for two non-isolated roots $a,a'\in A$ 
    we still have that $a\le a'$ if and only if $N(a)\subset N(A')$,
    and similarly for non-isolated roots in $B$.
    This is enough to recover the edges of $G$ using a fixed first-order formula, just as 
    in \cref{lem:transduceAll}.

    Technically, this way we produce three formulas $\phi_\st^r(x,y),\phi_\cl^r(x,y),\phi_\hg^r(x,y)$,
    rather than a single formula $\phi(x,y)$, corresponding to the cases where the host crossing is a star $r$-crossing, a clique $r$-crossing, or a half-graph $r$-crossing.
    To obtain a single formula $\phi(x,y)$,
    we can triple the number of colors used in the construction to give each of the three cases its own color palette.
    The formula $\phi(x,y)$ can then be written as a boolean combination of the formulas $\phi_\st^r(x,y),\phi_\cl^r(x,y),\phi_\hg^r(x,y)$, which tests which color palette was used.
\end{proof}

\section{Wrapping Up}\label{sec:wrapup}
We are now ready to prove the main results of this paper,
\Cref{thm:mainflipbreakable,thm:mainforbiddenpatterns}.
They are implied by the following.

\begin{theorem}\label{thm:main-circle}
    Let \(\mathcal{C}\) be a graph class.
    Then the following are equivalent.
    \begin{enumerate}[label=(\roman*)]
        \item\label{itm:circle1} \(\mathcal{C}\) is monadically dependent;
        
        \item\label{itm:circle2} 
        For every \(r \geq 1\) there exists \(k \in \N\)
        such that \(\mathcal{C}\) excludes as induced subgraphs
        \begin{itemize}
            \item all flipped star $r$-crossings of order $k$, and
            \item all flipped clique $r$-crossings of order $k$, and
            \item all flipped half-graph $r$-crossings of order $k$, and
            \item the comparability grid of order \(k\);
        \end{itemize}
        
        \item\label{itm:circle3} $\CC$ is prepattern-free;

        \item\label{itm:circle4} \(\mathcal{C}\) has the insulation property;
        
        \item\label{itm:circle5} \(\mathcal{C}\) is flip-breakable.
    \end{enumerate}
\end{theorem}
    
\begin{proof}
    We have already proven all the necessary implications:
    \begin{itemize}
        \item $\ref{itm:circle1} \Rightarrow \ref{itm:circle2}$: \cref{lem:transduceAll}
        \item $\ref{itm:circle2} \Rightarrow \ref{itm:circle3}$: \cref{prop:prepatterns-to-patterns}
        \item $\ref{itm:circle3} \Rightarrow \ref{itm:circle4}$: \cref{prop:prepatternImpliesInsulation}
        \item $\ref{itm:circle4} \Rightarrow \ref{itm:circle5}$: \cref{prop:ge-implies-fb}
        \item $\ref{itm:circle5} \Rightarrow \ref{itm:circle1}$: \cref{prop:fb-implies-mnip}\qedhere
    \end{itemize}
\end{proof}

We obtain the following algorithmic version of flip-breakability by combining \cref{thm:main-circle} and \cref{prop:ge-implies-fb}.

\begin{theorem}
\label{thm:alg-flip-breakability}
For every monadically dependent class $\CC$ and radius \(r \in \N\), there exists an unbounded function
\(f_r : \N \to \N\), a constant \(k_r \in \N\), and an algorithm that, given a graph \(G \in \CC\) and \(W \subseteq V(G)\), computes in time $O_{\CC,r}(|V(G)|^2)$ two subsets $A,B\subset W$ with \(|A|,|B| \ge f_r(|W|)\)
and a \(k_r\)-flip \(H\) of \(G\) such that:
\[
    \dist_{H}(A,B) > r.
\]
\end{theorem}

\thmQuant*

\begin{proof}
    The implication \ref{it:quant-eps}$\rightarrow$\ref{it:quant-eps-1/2} is immediate.
    We prove the implications \ref{it:quant-mdep}$\rightarrow$\ref{it:quant-eps} and \ref{it:quant-eps-1/2}$\rightarrow$\ref{it:quant-mdep}.

    \ref{it:quant-mdep}$\rightarrow$\ref{it:quant-eps}.
    Fix $\eps>0$, $r\ge 1$ and a graph class $\CC$. Let $\cal B_r$ denote the class of bipartite graphs $G$ such that  some radius-$r$ encoding of $G$ belongs to $\CC$,
    and let $\cal B_r^{(1)}$ denote the class of $1$-subdivisions of the bipartite graphs in $\cal B_r$. 
    By Lemma~\ref{lem:encode}, the class $\CC$ transduces the class $\cal B_r^{(1)}$.
    As $\CC$ is monadically dependent, $\cal B_r^{(1)}$ is as well. As $\cal B_r^{(1)}$ is \emph{weakly sparse}\footnote{A graph class $\CC$ is \emph{weakly sparse} if there is a bound $t$ such that $\CC$ excludes the biclique $K_{t,t}$ of order $t$ as a subgraph. In our case $\CC$ consists of subdivided graphs and excludes $K_{2,2}$ as a subgraph.} and monadically dependent, it is nowhere dense 
    (this follows from \cite{dvovrak2018induced}, see \cite{nevsetvril2021rankwidth}). As the class $\cal B_r^{(1)}$ of $1$-subdivisions of graphs from $\cal B_r$ is nowhere dense, 
    it follows that $\cal B_r$ is nowhere dense
    (indeed, if $\cal B_r$ contains some $k$-subdivided clique $K_n$ as a subgraph, for some $k,n\in\N $, then $\cal B_r^{(1)}$ contains the $(2k+1)$-subdivided clique $K_n$ as a subgraph).
    Since $\cal B_r$ is nowhere dense, it follows from \cite{dvorak-thesis}  (see also \cite{nevsetvril2011nowhere}) that 
    $|E(G)|\le O_{\cal B_r,\eps}(|V(G)|^{1+\eps})\le O_{\cal C,r,\eps}(|V(G)|^{1+\eps})$. The conclusion follows.

    \ref{it:quant-eps-1/2}$\rightarrow$\ref{it:quant-mdep}.
    We proceed by contrapositive. Let $\CC$ be a hereditary, monadically independent graph class. By \Cref{thm:main-circle},
    there is some $r\ge 1$ such that for every $n$,
    $\CC$ contains some flipped $r$-crossing of order $n$,
    or the comparability grid of order $n+1$.
As mentioned earlier, those contain as induced subgraphs radius-$r$ encodings ($r=1$ in the case of comparability grids) of the complete bipartite graph $K_{n,n}$ of order $n$,
which has $|E(K_{n,n})|=n^2$ and $|V(K_{n,n})|=2n$,
and therefore $|E(K_{n,n})|\ge |V(K_{n,n})|^2/4$.
As $n$ is arbitrarily large, this proves the negation of condition \ref{it:quant-eps-1/2}, and finishes the proof of the implication.
\end{proof}

\newpage
\part{Lower Bounds}
In this part we present several algorithmic and combinatorial lower bounds for hereditary, monadically independent graph classes.

\section{Hardness of Model Checking}\label{sec:hardness}

This section is devoted to proving the following theorem.

\thmHardnessMain*

We show this by reducing from the first-order model checking problem on the class of all graphs, which is $\mathrm{AW}[*]$-complete \cite{Downey1996ThePC}. 
Our main tool is that of an \emph{interpretation}.

\begin{definition}\label{def:interpretation}
    Let $\delta(x)$ and $\phi(x,y)$ be formulas, where $\phi$ is symmetric and irreflexive.
    The \emph{interpretation} $I_{\delta,\phi}$ is defined as the operation that maps a given input graph $G$ to the output graph $I_{\delta,\phi}(G):=H$ where $H$ has vertex set $V(H) := \{v \in V(G) : G \models \delta(v)\}$ and edge set $\{(u,v) \in V(H)^2: G \models \phi(u,v) \}.$ 
    A class of graphs $\CC$ \emph{efficiently interprets} a class $\DD$ if there exists 
    a polynomial-time computable 
    function $f\from \DD\to\CC$ such that 
     $I(f(H)) = H$ for all $H\in \DD$.
\end{definition}
Note that since the graph $f(H)$ is computed in polynomial time given $H$, in particular,
  $|V(f(H))|\le p(|V(H)|)$, for some fixed polynomial $p\from\N\to\N$.

The following property of interpretations is well known and follows by a simple formula rewriting procedure (see, e.g., \cite[Thm. 4.3.1]{hodges-shorter}).

\begin{lemma}
    For every interpretation $I$ and formula $\psi(\bar x)$, we can compute a formula $\psi_\star(\bar x)$, such that for all graphs $G$ and $H$ satisfying $G = I(H)$ and every tuple $\bar a \in V(G)^{|\bar x|}$,
    \[
        G \models \psi(\bar a) \quad\Leftrightarrow \quad H \models \psi_\star(\bar a).
    \]
\end{lemma}

\begin{corollary}[Transitivity]\label{lem:hardness-transitivity}
    Let $\CC, \DD, \EE$ be classes of graphs such that $\CC$ efficiently interprets $\DD$ and $\DD$ efficiently interprets $\EE$.
    Then also $\CC$ efficiently interprets $\EE$.
\end{corollary}

\begin{corollary}[Reduction]\label{lem:hardness-reductions}
    Let $\CC$ be a class of graphs that efficiently interprets the class of all graphs. Then the first-order model checking problem is $\mathrm{AW[*]}$-hard on $\CC$.
\end{corollary}

We spend the rest of this section proving the following proposition.

\begin{restatable}{proposition}{propPatternsImplyNotNIP}\label{prop:patternsImplyNotNIP}
    Let \(\mathcal{C}\) be a hereditary graph class and $r\geq 1$, such that for all $k\in \N$,
    \(\mathcal{C}\) contains
    \begin{itemize}
        \item a comparability grid of order \(k\), or 
        \item a flipped star $r$-crossing of order $k$, or
        \item a flipped clique $r$-crossing of order $k$, or
        \item a flipped half-graph $r$-crossing of order $k$.
    \end{itemize}
    Then $\CC$ efficiently interprets the class of all graphs.
\end{restatable}

Using \Cref{thm:mainforbiddenpatterns}, this is equivalent to the \cref{thm:hardness-interprets}, which we restate for convenience.

\hardnessInterprets*

By \cref{lem:hardness-reductions}, this then immediately implies \cref{thm:hardness-main},
stating $\mathrm{AW}[*]$-hardness of the model checking problem on every hereditary, monadically independent graph class.
Moreover, \Cref{thm:hardness-interprets} implies 
that a hereditary graph class is dependent if and only if it is monadically dependent.
This equivalence was previously proven by Braunfeld and Laskowski for the more general setting
of hereditary classes of relational structures~\cite{braunfeld2022existential}.

Note that 
it is rather straightforward to prove 
that, under the assumptions of \Cref{prop:patternsImplyNotNIP},
some class $\CC^+$ of \emph{colored} graphs 
from $\CC$ efficiently interprets the class of all graphs.  Indeed, using colors we can mark the 
$r+2$ layers of (an induced subgraph of) a flipped $r$-crossing representing the input graph, and then use the colors to ``undo'' the flip, thus obtaining 
an unflipped $r$-crossing,
with which we can proceed similarly as in the sketched proof of \cref{lem:transduceAll}.
This, however, would yield a significantly weaker statement than \Cref{thm:hardness-main}:
 that the first-order model checking problem
 on \emph{colored} graphs from $\CC$
 is AW[$*$]-hard, for every hereditary,  monadically independent class $\CC$.

\subsection{Crossings and Comparability Grids}
Since we explicitly refer to the individual vertices of our \(r\)-crossing patterns in this section,
let us restate their definition in greater detail and explicitly name their vertex sets.

\begin{definition}[\(r\)-crossings]\label{def:rcrossings}
For every radius $r \geq 1$ we define the \emph{star, clique, and half-graph \(r\)-crossing} of order $n$ as the graph whose vertex set 
\[
    \{a_i : i \in [n]\} \cup \{b_i : i \in [n]\} \cup \{p_{i,j,t} : i,j \in [n], t\in [r]\}    
\]
is partitioned into $l := r + 2$ \emph{layers} $\LL := \{ L_0, \ldots, L_{r+1} \}$ with
\begin{itemize}
    \item $L_0 := \{a_i : i \in [n]\}$,
    \item $L_t := \{p_{i,j,t} : i,j \in [n]\}$ for all $t \in [r]$,
    \item $L_{r+1} := \{b_j : j \in [n]\}$,
\end{itemize}
and whose edges are defined as follows.
The vertices $(a_i, p_{i,j,1}, \ldots, p_{i,j,r}, b_j)$ form a path for all $i,j \in [n]$.
Each of the three types enforces separate additional edges.
\begin{itemize}
    \item The star $r$-crossing contains no additional edges.
    \item For the clique $r$-crossing,
    \begin{itemize}
        \item $p_{i,j,1}$ and $p_{i,j',1}$ are adjacent for all $j \neq j' \in [n]$, and
        \item $p_{i,j,r}$ and $p_{i',j,r}$ are adjacent for all $i \neq i' \in [n]$.
    \end{itemize}
    \item For the half-graph \(r\)-crossing,
    \begin{itemize}
        \item $a_i$ is adjacent to $p_{i',j,1}$ for all $i \leq i' \leq n$ and for all $j \in [n]$, and
        \item $b_j$ is adjacent to $p_{i,j',r}$ for all $i \in [n]$ and for all $j \leq j' \leq n$.
    \end{itemize}
\end{itemize}
For every $r \geq 1$, let $\Bs_r$, $\Bc_r$, $\Bo_r$ be the hereditary closure of the class of all star, clique, and half-graph \(r\)-crossings, respectively.
\end{definition}

\medskip \noindent
Lastly, recall that the \emph{comparability grid} of order $n$ is the graph with vertex set $\{ a_{i,j} : i,j \in [n] \}$ where for all $i \leq i'$, $a_{i,j}$ and $a_{i',j'}$ are adjacent if and only if $i=i'$ or $j \leq j'$.
Denote by $\Bg$ the hereditary closure of the class of all comparability grids.

\subsection{Twins}

Two vertices $u$ and $v$ are \emph{twins} in a graph $G$, if \(N_G(u)\setminus \{ u,v \} = N_G(v)\setminus \{ u,v \}\).
This relation is transitive and definable in first-order logic:
\[
    \twins(x,y) := \forall z : (z \neq x \wedge z \neq y) \rightarrow (E(z,x) \leftrightarrow E(z,y)).
\]
The \emph{twin classes} of a graph $G$ are the equivalence classes of the twin-relation of $G$.
For every fixed $k\in \N$, 
the formulas 
\begin{gather*}
    \counttwins_{\geq k}(x) := \exists z_1, \ldots, z_k:
    \bigwedge_{(i,j) \in \binom{k}{2}} z_i \neq z_j \wedge z_i \neq x \wedge \twins(z_i,x),\\
    \counttwins_{= k}(x) := \counttwins_{\geq k}(x) \wedge \neg (\counttwins_{\geq k + 1}(x))
\end{gather*}
express that $x$ has at least or exactly $k$ twins, respectively  
(equivalently, the twin class containing $x$ has at least or exactly $k + 1$ elements).

\subsection{Reversing Flips}
Combining hereditariness and the pigeonhole principle, we observe the following.

\begin{observation}\label{obs:pigeonhole-patterns}
    Let \(\mathcal{C}\) be a hereditary graph class and $r\geq 1$, such that for all $k\in \N$,
    \(\mathcal{C}\) contains
    \begin{itemize}
        \item a comparability grid of order \(k\), or 
        \item a flipped star $r$-crossing of order $k$, or
        \item a flipped clique $r$-crossing of order $k$, or
        \item a flipped half-graph $r$-crossing of order $k$.
    \end{itemize}
    Then $\CC$ contains either $\Bg$
    or there is $\BB \in \{\Bs_r, \Bc_r, \Bo_r\}$, such that $\CC$ contains 
    a \emph{layer-wise} flip of each graph in $\BB$.
\end{observation}

Recall that $\Bs_r$/$\Bc_r$/$\Bo_r$ contains all the \emph{induced subgraphs} of all star/clique/half-graph $r$-crossings.
A \emph{layer-wise} flip of a graph $G$ in $\BB \in \{\Bs_r, \Bc_r, \Bo_r\}$ 
is an \(\{L_0,\dots,L_{r+1}\}\)-flip of $G$: the flip respects the layered structure of the class $\BB$.
As the vertices of the graphs in $\BB$ are named (cf. \cref{def:rcrossings}), it is clear in which layer each vertex of $G$ is located.

In this subsection we use interpretations to undo the flips and recover the graphs from $\BB$.
However, we only recover graphs without twins and without isolated vertices.
Let $\TT$ be the class of all graphs containing twins.
Let $\II$ be the class of all graphs containing isolated vertices.

\begin{lemma}\label{lem:hardness-undo-flips-sco}
    Fix $r \geq 1$. Let $\BB \in \{\Bs_r, \Bc_r, \Bo_r\}$ 
    and let $\CC$ be a hereditary class containing a layer-wise flip of each graph from $\BB$. 
    \begin{center}
    $\CC$ efficiently interprets $\BB \setminus (\TT \cup \II)$.
    \end{center}
\end{lemma}

\begin{proof}
Denote by $B_n$ the star/clique/half-graph \(r\)-crossing of order $n$.
Since \(\mathcal{C}\) is hereditary
and by the pigeonhole principle,
we can assume that there exist an integer $k \leq r + 2$, a mapping $\lc : \{0,\dots,r+1\} \rightarrow [k]$ (for \textbf{l}ayer \textbf{c}olor), and a symmetric relation $R \subseteq [k]^2$ with the following property.
\begin{itemize}
    \item 
    For every $n \in \N$,
    let $\KK_n$ be the $k$-coloring of $V(B_n)$ in which each layer \(L_0,\dots,L_{r+1}\) of \(B_n\)
    is monochromatically colored by the color \(\lc(0),\dots,\lc(r+1)\), respectively.
    Note that two different layers might be assigned the same color.

    The graph $\flip(B_n) := B_n \oplus_{\KK_n} R$ is contained in $\CC$.
\end{itemize}
Without loss of generality, we can assume $k$ to be minimal in the following sense.
\begin{itemize}
    \item Every color is used: the map $\mathrm{lc}$ is surjective.
    \item No two colors can be merged: for all $i \neq j \in [k]$ there exists $d \in [k]$ such that 
        \[
            (i,d) \in R \Leftrightarrow (j,d) \not\in R.
        \]
\end{itemize}

Having fixed $\lc$,  from now on we assume every $B_n$ to be implicitly $k$-colored,
and that \(\BB\) is the hereditary closure of these \(k\)-colored graphs.
We extend our notion $\flip(G) := G \oplus_{\KK} R$ to all graphs \(G \in \BB\),
where \(\KK\) is the aforementioned \(k\)-coloring associated with \(G\).
Note that $\flip(G) \in \CC$ for all \(G \in \BB\). 
Our notion of layers carries over to all graphs $G$, and $\flip(G)$ in the obvious way.
We use twin classes to uniquely quantify representative vertices of each color class.
The representatives are added as follows.
For any $k$-colored graph \(G\), let $\prep(G)$ be the graph
obtained from \(G\) by adding, for each color $i\in[k]$,
$(i+1)$ many isolated vertices $s_{i,1}, \ldots, s_{i,(i+1)}$ of color $i$ to $G$.

\begin{claim}
    If $G \in \BB$, then the same holds for $\prep(G)$.
\end{claim}

\begin{claimproof}
    By assumption, there exists an embedding $f$ from $G$ to $B_{n}$, for some $n \in \N$.
    Let $c := r \cdot (l+1)$ and $m := n + 2c$.
    We show how to embed $\prep(G)$ into $B_m$.
    First note that the function
    \[
        g(\cdot):=\bigcup_{\mathclap{i,j\in[n], t \in [r] }}\,
        \{ 
        a_i \mapsto a_{i+c},\quad
        b_j \mapsto b_{j+c}, \quad
        p_{i,j,t} \mapsto p_{i+c,j+c,t}
        \}
    \]
    is an embedding from $B_{n}$ to $B_m$.
    It follows that $h(\cdot) := g \circ f$ is an embedding from $G$ to $B_m$.
    Let $h(G) := \{h(v) : v \in V(G)\}$.
    Importantly, every vertex $a_i$, $b_j$, or $p_{i,j,t}$ in $h(G)$ satisfies
    \[
        i,j \in \{c+1, \ldots, c + n\}
        \quad
        \text{and} 
        \quad
        t \in [k].
    \]
    We next choose a set $S(L)$ of $(l+1)$ vertices from each layer $L \in \{L_0,\dots,L_{r+1}\}$ as follows.
    \begin{itemize}
        \item $S(L_0) := \{a_{c+n+i} : i \in [l+1]\}$,
        \item $S(L_t) := \{p_{i'+i,i'+i,t} : i \in [l+1]\}$ for every $t \in [r]$ and $i' := (t-1)\cdot(l+1)$,
        \item $S(L_{r+1}) := \{b_{c+n+j} : j \in [l+1]\}$.
    \end{itemize}
    See \Cref{fig:hardness-undo-flips} for a visualization.

    \begin{figure}[htbp]
        \centering
        \includegraphics[scale = 1]{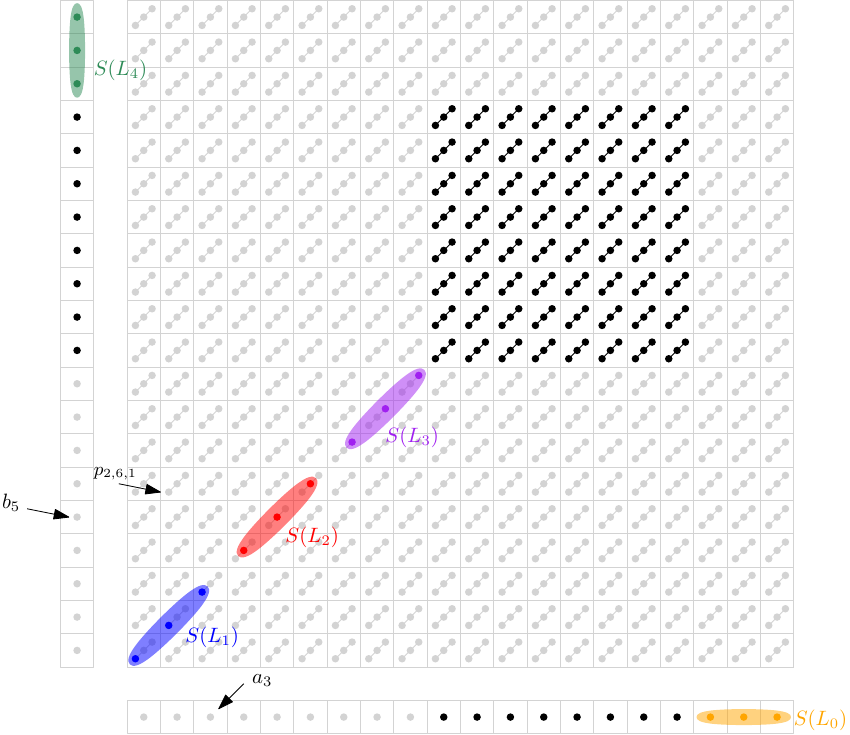}
        \caption{A visualization of how the sets $S(L)$ are embedded into $B_m$ for the case $r=3$.
        To preserve readability in the visualization, each set $S(L)$ has size $3$ instead of $l+1 = 6$.
        The black vertices correspond to the embedding of $B_n$ into $B_m$.
        }
        \label{fig:hardness-undo-flips}
    \end{figure}

    We define $I:= S(L_0) \cup \ldots \cup S(L_{r+1})$.
    By construction, $I$ and $h(G)$ are disjoint.
    Let us now argue that every vertex from $I$ is isolated in the induced subgraph $B'_m := B_m[h(G) \cup I]$.
    \begin{itemize}
        \item Let $a_i \in I$. Then $i > c + n$.
        All the neighbors of $a_i$ in $B_m$ are of the form $p_{i',j,t}$ for some $i' \geq i$.
        All vertices $p_{i',j,t}$ in $B'_m$ satisfy $i' \leq c + n$, so they are non-adjacent to $a_i$.
        \item Let $b_j \in I$. The same reasoning as in the previous case applies.
        \item Let $p_{i,j,t} \in I$. Then $i \le c$ and $j \le c$. By the same reasoning as before, $p_{i,j,t}$ is non-adjacent to all vertices of the form $a_{i'}$ or $b_{j'}$ in $B'_m$.
        Furthermore, any neighbor of $p_{i,j,t}$ of the form $p_{i',j',t'}$ in $B_m$ must satisfy $i' = i$ or $j' = j$. By construction, $B'_m$ contains no vertex $p_{i',j',t'}$ satisfying $i' = i$ or $j' = j$, apart from $p_{i,j,t}$ itself.

    \end{itemize}
    This proves that the set $I$ is indeed isolated in $B'_m$.
    We want to stress that our argument works for each of the three classes $\BB \in \{\Bs_r, \Bc_r, \Bo_r\}$. 
    
    We finally embed $\prep(G)$ into $B'_m$.
    The vertices of $G$ are embedded using $h$.
    The additional vertices $s_{i,1}, \ldots, s_{i,(i+1)}$ that are added to $\prep(G)$ for each color $i\in[k]$ can now be mapped to distinct vertices from a set $S(L)$ to which the layer coloring $\lc$ assigns color $i$.
    As desired, they have color $i$ and are isolated.
\end{claimproof}

\begin{claim}\label{clm:hardness-twinclasses}
    Let $G \in \BB \setminus (\TT \cup \II)$.
    The twin classes of $\flip(\prep(G))$ consist exactly of
    \begin{itemize}
        \item the singleton twin class $\{ v \}$ for every vertex $v \in V(G)$, and
        \item the twin class $T_i := \{s_{i,1}, \ldots, s_{i,i+1}\}$ for every color $i \in [k]$.
    \end{itemize}
\end{claim}

\begin{claimproof}
    The following is easy to see. 
    \begin{equation}\label{eq:hardness-same-color}
        \begin{gathered}
            \text{For all vertices $u,v$ with the same color:}\\
            \text{$u$ and $v$ are twins in $\prep(G)$ if and only if they are twins in $\flip(\prep(G))$.}
        \end{gathered}
    \end{equation}
    Additionally, we argue the following.
    \begin{equation}\label{eq:hardness-diff-color}
            \text{For all vertices $u,v$ with different colors: $u$ and $v$ are not twins in $\flip(\prep(G)).$}
    \end{equation}
    Let $i\neq j$ be the colors of $u$ and $v$. 
    By the assumed minimality of the coloring, there exist a color $d \in [k]$ such that $(i,d) \in R \Leftrightarrow (j,d) \not\in R$.
    There exists at least one vertex $s_d \in \{s_{d,1}, s_{d,2}\}$ that has color $d$ and is non-adjacent and non-equal to both $u$ and $v$ in $\prep(G)$.
    It follows that in $\flip(\prep(G))$ exactly one of $u$ and $v$ will be adjacent to $s_d$. Thus, $u$ and $v$ are no twins in $\flip(\prep(G))$.
    
    Combining \eqref{eq:hardness-same-color} and \eqref{eq:hardness-diff-color}, we have that every two vertices $u$ and $v$ which are no twins in $\prep(G)$ are also no twins in $\flip(\prep(G))$.
    Since $G \notin \TT \cup \II$, in $\prep(G)$ the vertices of $V(G)$ neither have twins among $V(G)$ nor among the isolated vertices added to build $\prep(G)$ from $G$.
    It follows that each vertex from $V(G)$ is contained in a singleton twin class of $\flip(\prep(G))$ as desired.
    Finally, by \eqref{eq:hardness-same-color}, for every color $i \in [k]$ there is a twin class $T_i$ containing the set of isolated vertices $\{s_{i,1}, \ldots, s_{i,i+1}\}$.
    As argued before, $T_i$ contains no vertices from $V(G)$.
    By \eqref{eq:hardness-diff-color},
    $T_i$ is disjoint from $T_j$ for every other color $j \neq i$.
    Then $T_i$ is exactly $\{s_{i,1}, \ldots, s_{i,i+1}\}$, as desired.
\end{claimproof}

\begin{claim}\label{clm:hardness-color}
    For every color $i \in [k]$ there exists a formula $\col_i(x)$ such that
    for every $G \in \BB \setminus (\TT \cup \II)$ and every vertex $v$ in $\prep(G)$ we have
    \[
        \text{$v$ has color $i$}
        \quad
        \Leftrightarrow
        \quad
        \flip(\prep(G))\models \col_i(x).
    \]
\end{claim}

\begin{claimproof}
    We  argue that the following formula does the job.
    \begin{gather*}
        \col_i(x) := \exists z_1, \ldots, z_k: 
        \bigwedge_{j\in k}
        x \neq z_j
        \wedge 
        \counttwins_{=j}(z_j) 
        \wedge 
        \bigl(E(x,z_j) \leftrightarrow (i,j) \in R\bigr)
    \end{gather*}
    The formula quantifies vertices $\bar z = z_1 \ldots z_k$ containing for each color $j \in [k]$ a vertex $z_j$ such that 
    \begin{itemize}
        \item $z_j$ is not equal to $x$,
        \item $z_j$ is from a twin class of size exactly $j+1$,
        \item $z_j$ is adjacent to $x$ if and only if $(i,j) \in R$.
    \end{itemize}
    Let $v$ be a vertex in $\prep(G)$.
    To prove the forwards direction of the claim, assume $v$ has color $i$.
    We can choose a satisfying valuation $\bar w$ of $\bar z$ as follows.
    By \cref{clm:hardness-twinclasses}, for each color $j \in [k]$ the twin class $T_j$ has size exactly $j+1$ and all its vertices are isolated in $\prep(G)$ and have color $j$.
    As $|T_j|\geq 2$, we can pick a vertex $w_j \in T_j$ that is not equal to $v$.
    As $v$ and $w_j$ are non-adjacent in $\prep(G)$ and of color $i$ and $j$ respectively, we have $\flip(\prep(G)) \models E(v,w_j) \Leftrightarrow (i,j) \in R$ as desired.

    For the backwards direction, assume towards contradiction that $v$ has color $i' \neq i$ and there exists a satisfying valuation $\bar w$ of $\bar z$.
    By the assumed minimality of the coloring of $\prep(G)$, there exist a color $d \in [k]$ such that 
    \[
        (i',d) \in R 
        \Leftrightarrow 
        (i,d) \notin R.
    \]
    Again $w_d$ has color $d$ and is non-adjacent to $v$ in $\prep(G)$. 
    By definition of $\flip(\prep(G))$ we have
    \[
        \flip(\prep(G)) \models E(v,w_d) 
        \Leftrightarrow 
        (i',d) \in R.
    \]
    However, for $\bar w$ to be a satisfying valuation of $\bar z$ we must have
    \[
        \flip(\prep(G)) \models E(v,w_d) 
        \Leftrightarrow 
        (i,d) \in R.
    \]
    Combining the three equivalences gives the desired contradiction.
\end{claimproof}

\begin{claim}\label{clm:hardness-rewrite}
    For every formula $\phi(\bar x)$ we can compute a formula $\flip(\phi)(\bar x)$ 
    such that for every graph $G \in \BB \setminus (\TT \cup \II)$ and every tuple $\bar a \in V(G)^{|x|}$,
    \[
        \mathrm{prep}(G) \models \phi(\bar a)
        \Leftrightarrow
        \flip(\mathrm{prep}(G)) \models \flip(\phi)(\bar a).
    \]
\end{claim}

\begin{claimproof}
    Using \cref{clm:hardness-color}, it is easy to see that
    for all graphs $G \in \BB \setminus (\TT \cup \II)$ and vertices $u$ and $v$ in $\prep(G)$,
    \[
        \prep(G) \models E(u,v)
        \quad\Leftrightarrow\quad
        \flip(\prep(G)) \models E(u,v) \text{ XOR } 
        \bigvee_{i,j \in [k]} \col_i(x) \wedge \col_j(y) \wedge (i,j) \in R.
    \]
    For every $\phi(\bar x)$, let $\flip(\phi)(\bar x)$ be the formula obtained by replacing every occurrence of $E(x,y)$ with the formula on the right side of the above equivalence. 
    It now easily follows by structural induction that \(\flip(\phi)(\bar x)\) has the desired properties.
\end{claimproof}

Let $\delta(x) := \flip(\mathrm{hasNeighbor})(x)$ and $\phi(x,y) := \flip(E)(x,y)$, where $\mathrm{hasNeighbor}(x)$ is the formula checking that $x$ is not an isolated vertex.
Using \cref{clm:hardness-rewrite}, we have
\[
    I_{\delta,\phi}\big(\flip(\prep(G))\big) = G
\]
for every graph $G \in \BB \setminus (\TT \cup \II)$.
As $\flip(\prep)(G)$ is contained in $\CC$ and can be computed in polynomial time from $G$, we have that $\CC$ efficiently interprets $\BB \setminus (\TT \cup \II)$.
\end{proof}

\subsection{Encoding Bipartite Graphs}

Having undone the flips, we interpret all bipartite graphs (without isolated vertices) from our intermediate classes $\{\Bs_r, \Bc_r, \Bo_r, \Bg\}$.
Targeting this class of bipartite graphs does not restrict the general case, as the following lemma shows.

\begin{lemma}\label{lem:hardness-bipartite-to-all}
    The class of all bipartite graphs without isolated vertices efficiently interprets the class of all graphs.
\end{lemma}

\begin{proof}
    Let $\delta(x)$ be the formula stating that $x$ has degree at least three and $\phi(x,y)$ be the formula stating that $x$ and $y$ are at distance exactly two.
    For every graph $G$, we build the graph $B_G$ as follows.
    For every vertex $v \in V(G)$ we create a star with three leaves and center $c_v$.
    For every edge $(u,v) \in E(G)$, we add a new vertex adjacent to both $c_u$ and $c_v$.
    It is easy to see that $B_G$ is bipartite, without isolated vertices, and $I_{\delta,\phi}(B_G) = G$.
\end{proof}

The following notation will be convenient.
For every bipartite graph $H$ there exists at least one 
\emph{bipartite representation of $H$}, that is, a tuple 
\[
    H' = (U' \subseteq \N,V' \subseteq \N,E(H') \subseteq U' \times V'),
\]
such that there exist
\begin{itemize}
    \item a bipartition of $V(H)$ into two independent sets $U$ and $V$, and
    \item two bijections $f: U \rightarrow U'$ and $g: V \rightarrow V'$,
\end{itemize}
such that for all $u \in U$ and $v \in V$: 
$(f(u),f(v)) \in E(H') \Leftrightarrow (u,v) \in E(H)$.
Note that $U'$ and $V'$ do not have to be disjoint and $E(H')$ is not necessarily symmetric.

\subsubsection*{Encoding Bipartite Graphs in Star and Clique \(\boldsymbol r\)-Crossings}

\begin{lemma}\label{lem:hardness-sc-to-bipartite}
    For every $r \geq 1$ and $\BB \in \{ \Bs_r, \Bc_r \}$,
    \begin{center}
        $\BB \setminus (\TT \cup \II)$ efficiently interprets the class of all bipartite graphs.    
    \end{center}
\end{lemma}

\begin{proof}
    First assume $\BB = \Bs_r$.
    Let $\delta(x)$ be the formula checking whether $x$ has degree at least three, and let $\phi(x,y)$ be the formula checking whether the distance between $x$ and $y$ is exactly $r+1$.
    To prove the lemma, we show that for every bipartite graph $H$, we can construct a graph $B_H \in \BB \setminus (\TT \cup \II)$ such that $I_{\delta,\phi}(B_H) = H$.
    Let 
    \[
        H' = ([n],[m],E(H')\subseteq [n] \times [m])   
    \]
    be a bipartite representation of $H$ for some $n,m \in \N$.
    We build the graph $B_H$ as follows.
    For every $i \in [n]$ we create a $1$-subdivided star with three leaves consisting of: a center $c_i$ and for every $s \in \{ 0,1,2 \}$ a subdivision vertex $c_{i,s,1}$ and a leaf $c_{i,s,2}$ such that $(c_i, c_{i,s,1}, c_{i,s,2})$ form a path.
    We do the same for every $j \in [m]$, giving us vertices $d_j, d_{j,s,1}, d_{j,s,2}$ for every $s \in \{ 0,1,2 \}$.
    Finally, for every edge $(i,j) \in E(H')$ we add vertices $\{ q_{i,j,t} : t \in [r]\}$ and connect $(c_i,q_{i,j,1},\ldots,q_{i,j,r},d_j)$ to form a path of length $r+1$.
    It is easy to see that $I_{\delta,\phi}(B_H) = H$ and that
    $B_H$ contains neither twins nor isolated vertices.
    It remains to show that $B_H$ is an induced subgraph of a star \(r\)-crossing.

    Let $N := 3(n+m)$ and $B_N$ be the star $r$-crossing of order $N$.
    We  give an embedding $h: V(B_H) \rightarrow V(B_N)$ of $B_H$ into $B_N$.
    Let $f(i) := 3i - 2$ and $g(j) := 3n + 3j - 2$. We define $h$ as follows for all $i \in [n]$, $j \in [m]$, $s \in \{ 0,1,2 \}$, and $t \in [r]$.
    \begin{multicols}{2}\raggedcolumns
        \begin{itemize}
            \item $h(c_i) := a_{f(i)}$
            \item $h(c_{i,s,1}) := p_{f(i),f(i)+s,1}$
            \item $h(c_{i,s,2}) := 
            \begin{cases}
                p_{f(i),f(i)+s,2} & \text{if } r > 1\\ 
                b_{f(i)+s}          & \text{if } r = 1 
            \end{cases}$
            
            \item $h(q_{i,j,t}) := p_{f(i),g(j),t}$ 

            \item $h(d_j) := b_{g(j)}$
            \item $h(d_{j,s,1}) := p_{g(j)+s,g(j),r}$
            \item $h(d_{j,s,2}) :=
            \begin{cases}
                p_{g(j)+s,g(j),r-1} & \text{if } r > 1\\ 
                a_{g(j)+s}          & \text{if } r = 1 
            \end{cases}$
        \end{itemize}
    \end{multicols}
    \noindent
    See \Cref{fig:hardness-sc-to-b} for a visualization of the embedding.
    It is easily checked that $h$ is an embedding. This finishes the case $\BB = \Bs_r$.

    \begin{figure}[htbp]
        \centering
        \includegraphics[scale=1]{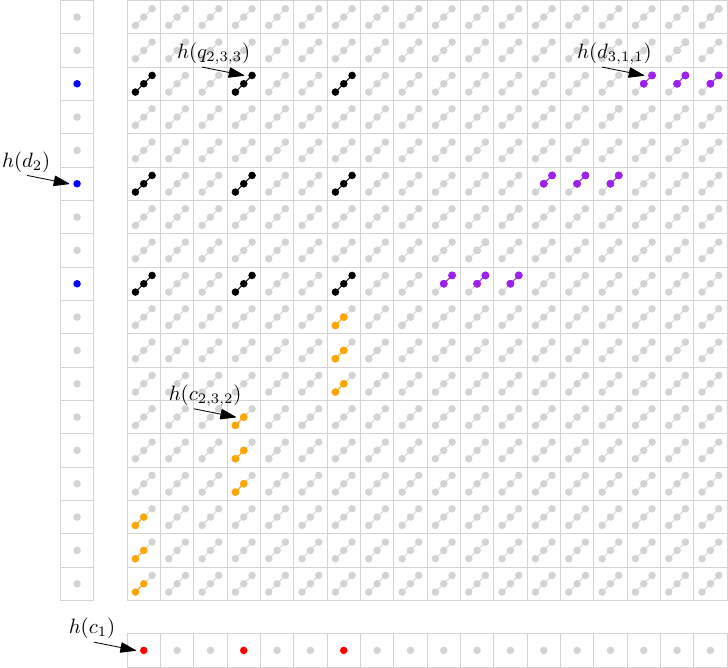}
        \caption{A visualization of how $B_H$ embeds into $B_N$ for the case where $r=3$ and $H$ is the biclique of order $3$.}
        \label{fig:hardness-sc-to-b}
    \end{figure}
    \noindent

    For the case $\BB = \Bc_r$, we take the same edge formula $\phi$ and update the domain formula to state
    \[
        \delta(x) := \text{``$x$ has degree at least three and the neighborhood of $x$ is a clique''}.
    \]
    We build $B_H$ as in the previous case but add additional edges.
    For every $i \in [n]$ we turn the set 
    \[
        \bigcup\bigl\{\{ c_{i,s,1}, c_{i,s,2}, q_{i,j,1} \} : j \in [m], (i,j) \in E(H'), s \in \{ 0,1,2 \}\bigr\}  
    \]
    that contains exactly the neighborhood of $c_i$ into a clique.
    Symmetrically, we do the same for every $j \in [m]$ with the set
    \[
        \bigcup\bigl\{\{ d_{j,s,1}, d_{j,s,2}, q_{i,j,r} \} : i \in [n], (i,j) \in E(H'), s \in \{ 0,1,2 \}\bigr\}  
    \]
    of all neighbors of $d_j$.
    Again $I_{\delta,\phi}(B_H) = H$, $B_H$ contains neither twins nor isolated vertices, and $h$ is an embedding of $B_H$ into the clique \(r\)-crossing of order $3(n+m)$.
\end{proof}

\subsubsection*{Encoding Bipartite Graphs in Half-Graph \(\boldsymbol r\)-Crossings and in Comparability Grids}

\begin{lemma}\label{lem:hardness-o-to-bipartite}
    For every $r \geq 1$ and $\BB \in \{ \Bo_r, \Bg \}$,
    \begin{center}
        $\BB \setminus (\TT \cup \II)$ efficiently interprets the class of all bipartite graphs without isolated vertices.    
    \end{center}
\end{lemma}

\begin{proof}
    We first prove the statement for $\BB = \Bo_r$.
    For every $n \in \N$, we define $\dropbounds{n} := [n] \setminus \{1,n\} = \{ 2,\ldots,n-1 \}$.
    Let $H$ be an arbitrary bipartite graph without isolated vertices and let
    \[
        H' = (\dropbounds{n},\dropbounds{m},E(H')\subseteq \dropbounds{n} \times \dropbounds{m})   
    \]
    be a bipartite representation of $H$ for some $n,m \in \N$.
    We define $B_H$ to be the subgraph of the half-graph \(r\)-crossing of order $\max(n,m)$ induced by the vertices
    \begin{itemize}
        \item $A := \{a_i : i \in \dropbounds{n}\}$ corresponding to the left vertices of $H$,
        \item $B := \{b_j : j \in \dropbounds{m}\}$ corresponding to the right vertices of $H$,
        \item $P:= \{p_{i,j,t} : (i,j) \in E(H'), t \in [r]\}$ corresponding to the edges of $H$,
        \item $\{a_{n}, p_{n,1,1}\}$ and $\{b_{m}, p_{1,m,r}\}$ which we use as auxiliary vertices.
    \end{itemize}
    See the left side of \Cref{fig:hardness-h-to-bipartite} for a visualization.

    \begin{figure}[htbp]
        \centering
        \includegraphics[scale = 1]{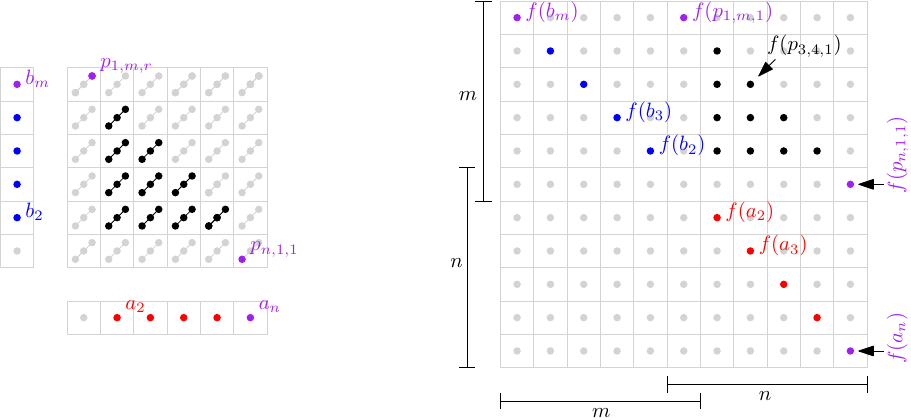}
        \caption{On the left: a visualization of the vertices of $B_H$ where $H$ is the half-graph of order $4$ and $r = 3$.
        The vertices in $A$, $B$, and $P$ are colored red, blue, and black respectively.
        On the right: a visualization of the embedding of $B^\smallblacksquare_H$ into a comparability grid, where again $H$ is the half-graph of order $4$.
        }
        \label{fig:hardness-h-to-bipartite}
    \end{figure}

    Our goal is to interpret $H$ from $B_H$.
    Let 
    \[
        A_\star := A \cup \{a_n\}, \quad
        B_\star := B \cup \{b_m\}, \quad
        P_\star := P \cup \{p_{n,1,1}, p_{1,m,r}\}.
    \]

    \begin{claim}\label{clm:interp-astar}
        $N_{B_H}(p_{n,1,1})= A_\star$
        and
        $N_{B_H}(p_{1,m,r})= B_\star$.
    \end{claim}
    
    \begin{claimproof}
    The claim follows by the definition of the adjacencies in half-graph crossings.
    $A_\star$ is included in the neighborhood of $p_{n,1,1}$ because $i \leq n$ for all $i \in \dropbounds{n} \cup \{n\}$.
    All neighbors of $B_\star$ are of the form $p_{i,j,t}$ for some $j \geq 2$, so no vertex of $B_\star$ is included in the neighborhood of~$p_{n,1,1}$.
    The vertex $p_{n,1,2}$ is not contained in $P_\star$, so no vertex of $P_\star$ is included in the neighborhood of~$p_{n,1,1}$.
    We argue symmetrically to determine the neighborhood of $p_{1,m,r}$.
    \end{claimproof}
    \begin{observation}\label{clm:interp-gateway}
        $N_{B_H}(a_n) = \{ p_{n,1,1} \}$ and 
        $N_{B_H}(b_m) = \{ p_{1,m,r} \}$.
    \end{observation}

    \begin{claim}\label{clm:interp-deg2}
        Except for $a_n$ and $b_m$, all vertices have degree at least two.
    \end{claim}
    
    \begin{claimproof}
        All vertices in $A$ are adjacent to $p_{n,1,1}$ and at least one other $p_{i,j,1}$ vertex since $H$ contains no isolated vertices.
        A symmetric statement holds for $B$ and $p_{1,m,r}$.
        The vertices in $P$ are inner vertices of paths from $A$ to $B$, so they have degree at least two.
        The vertices $p_{n,1,1}$ and $p_{1,m,r}$ have high degree by the previous claim.
    \end{claimproof}

    \begin{observation}\label{obs:interp-no-iso}
        $B_H$ contains no isolated vertices.
    \end{observation}

    \begin{claim}\label{claim:notwins}
        $B_H$ contains no twins.
    \end{claim}
    
    \begin{claimproof}
        The vertices in $A_\star$ can be distinguished from the vertices in $B_\star$ and $P_\star$ by their adjacency to $p_{n,1,1}$.
        Two vertices $a_i$ and $a_{i'}$ from $A_\star$ with $i < i' \leq n$ can be differentiated by a vertex $p_{i,j,1}$ for some $j \in \dropbounds{m}$ which is adjacent to $a_i$ but not to $a_{i'}$.
        This $j$ exist since no vertex in $H$ is isolated.
        It follows that $A_\star$ contains no twins and by a symmetric argument, neither does $B_\star$.

        It remains to distinguish the vertices inside $P_\star$.
        The auxiliary vertices $p_{n,1,1}$ and $p_{1,m,r}$ each have a private neighbor in $a_n$ and $b_m$, so we can focus our attention on the set $P$.
        Let $p_{i,j,t}$ and $p_{i',j',t'}$ be two distinct vertices from $P$. By symmetry, we can assume that either $i<i'$ or $j<j'$ or $t < t'$.
        If $i < i'$ or $t < t'$ then
        the vertex
        \[
            d:= 
            \begin{cases}
                p_{i,j,t-1}& \text{if } t > 1,\\
                a_i& \text{if } t = 1,
            \end{cases}
        \]
        is adjacent to $p_{i,j,t}$ but non-adjacent to $p_{i',j',t'}$.
        If $j<j'$ we argue symmetrically using either $p_{i,j,t+1}$ or $b_j$ if $t = r$.
        We want to stress that the argument works also for the case of $r=1$.
    \end{claimproof}
    We have proven that $B_H \in \Bo_r \setminus (\TT \cup \II)$.
    Let us continue to show that $B_H$ interprets $H$.
    Combining \cref{clm:interp-gateway}, \cref{clm:interp-astar}, \cref{clm:interp-deg2}, we have that the formula
    \[
        \delta(x) := \text{``$x$ has degree at least $2$ and is at distance exactly $2$ from a degree $1$ vertex''}
    \]
    is true exactly on the vertices from $A \cup B$.
    It can therefore act as the domain formula of our interpretation.
    It remains to define the edge relation.
    First notice that the formula
    \[
        \mathrm{sameSide}(x,y) := \exists z: \text{``$z$ has degree $1$ and is at distance exactly $2$ from both $x$ and $y$''}
    \]
    distinguishes $A_\star$ and $B_\star$: for all vertices $u$ and $v$
    \[
        B_H \models \mathrm{sameSide}(u,v)
        \quad
        \Leftrightarrow
        \quad
        \{ u,v \} \in A_\star
        \vee
        \{ u,v \} \in B_\star
        . 
    \]
    We next construct a formula that resolves the half-graphs between $A$ and $P_\star$ and between $B$ and $P_\star$ in the following sense.

    \begin{claim}\label{clm:hardness-e0}
        There exists a formula $E_0(x,y)$ such that for all $i \in \dropbounds{n}$,
        \[
            \{ v \in V(B_H) \mid B_H \models E_0(a_i,v) \}
            =
            \{ p_{i,j,1} : (i,j) \in E(H'), j\in \dropbounds{m} \},
        \]
        and for all $j \in \dropbounds{m}$,
        \[
            \{ v \in V(B_H) \mid B_H \models E_0(b_j,v) \}
            =
            \{ p_{i,j,r} : (i,j) \in E(H'), i\in \dropbounds{n} \}.
        \]
        
    \end{claim}
    
    \begin{claimproof}
        The formula
        \[
            x \prec y := \mathrm{sameSide}(x,y) \wedge \text{``$N(y)$ is a strict subset of $N(x)$''}
        \]
        orders $A_\star$ and $B_\star$ respectively:
        for all $a_i \in A_\star$ and $b_j \in B_\star$ we have
        \begin{gather*}
            \bigl(B_H \models a_i \prec a\bigr)
            \Leftrightarrow
            \bigl(a \in \{ a_{i'} : i < i' \leq n\}\bigr)
            \quad
            \text{and}
            \quad
            \bigl(B_H \models b_j \prec b\bigr)
            \Leftrightarrow
            \bigl(b \in \{ b_{j'} : j < j' \leq m\}\bigr)
            .
        \end{gather*}
        Now the following formula
        \[
            E_0(x,y) := E(x,y) \wedge \neg \bigl(\exists x': x \prec x' \wedge E(x',y)\bigr)
        \]
        has the desired properties.
    \end{claimproof}
    We can finally construct the formula interpreting the edges of $H$
    \[
        \phi(x,y) := \neg \mathrm{sameSide}(x,y) \wedge \exists z_1,\ldots,z_r: E_0(x,z_1) \wedge E_0(y,z_r) \wedge \bigwedge_{t<r}E(z_t,z_{t+1}),
    \]
    which states that $x$ and $y$ are from different sides and connected by a path containing $r+1$ edges and whose first and last edge are $E_0$-edges.
    The formula defines a bipartite graph with sides $A$ and $B$:
    it is symmetric, and we have $B_H \not \models \phi(u,v)$ for all vertices $u$ and $v$ such that $\{ u,v \} \subseteq A$ or $\{ u,v \} \subseteq B$.
    
    \begin{claim}
        For all $i \in \dropbounds{n}$ and $j \in \dropbounds{m}$:
        $
        B_H \models \phi(a_i,b_j)
            \Leftrightarrow
        (i,j) \in E(H')
        $.
    \end{claim}
    
    \begin{claimproof}
        Assume $(i,j) \in E(H')$. Then $p_{i,j,1}, \ldots, p_{i,j,r}$ is a valuation of $z_1,\ldots,z_r$ witnessing that $B_H \models \phi(a_i,b_j)$.
        For the backwards direction assume a satisfying valuation $p_1, \ldots, p_r$ of $z_1,\ldots,z_r$.
        By \cref{clm:hardness-e0} we have $p_1 = p_{i,j',1}$ and $p_r = p_{i',j,r}$ for some $j' \in \dropbounds{m} \setminus \{ j \}$ and $i' \in \dropbounds{n} \setminus \{ i \}$.
        The existence of $p_2,\ldots, p_{r-1}$ implies \(i=i'\) and \(j=j'\).
        Hence, \((i,j) \in E(H')\).
    \end{claimproof}
    It follows that $I_{\delta,\phi}(B_H) = H$.
    Since the definition of $\delta$ and $\phi$ does not depend on $H$, and $H$ was chosen to be an arbitrary bipartite graph without isolated vertices, we have that this interpretation interprets all such $H$ from their corresponding preimage $B_H \in \Bo_r \setminus (\TT \cup \II)$.
    Furthermore, for every $H$ the preimage $B_H$ can be computed in polynomial time, so the interpretation is efficient.
    This finishes the case of $\BB = \Bo_r$.

    \medskip\noindent
    We next show that the interpretation that we constructed for the previous case, where $r =1$ and $\BB = \Bo_1$, 
    also works for the case $\BB = \Bg$.
    In the previous case, we constructed for every bipartite graph without isolated vertices $H$, a graph $B_H \in \Bo_1$ such that $I_{\delta, \phi}(B_H) = H$.
    In this case, since $r=1$, the formula $\phi$ collapses to
    \[
        \phi(x,y) := \neg \mathrm{sameSide}(x,y) \wedge \exists z_1: E_0(x,z_1) \wedge E_0(y,z_1)
    \]
    and the graph $B_H$ consists of vertices
    \[
        \underbrace{\{a_i : i \in \dropbounds{n}\}}_
        {= A} 
        \cup 
        \underbrace{\{b_j : j \in \dropbounds{m}\}}_{=B}
        \cup
        \underbrace{\{p_{i,j,1} : (i,j) \in E(H')\}}_{=P}
        \cup
        \{a_{n}, p_{n,1,1}\}
        \cup
        \{b_{m}, p_{1,m,1}\}
    \]
    where $H'$ is again the bipartite representation of $H$.
    We now build a new graph $B^{\smallblacksquare}_H$ from $B_H$ by adding additional edges as follows.
    We connect each vertex $p_{i,j,1} \in P$ with all the vertices $p_{i',j',1} \in P$ such that $i \leq i'$ and $j \leq j'$ (but not with itself).
    It is easily checked that all the previous claims for $B_H$, still hold true for $B^{\smallblacksquare}_H$:
    \begin{itemize}
        \item We only modified adjacencies inside $P$, 
        so \cref{clm:interp-astar} and \cref{clm:interp-gateway} also hold in $B^{\smallblacksquare}_H$.
        \item The degree of vertices in $P$ only increased, 
        so \cref{clm:interp-deg2} and \cref{obs:interp-no-iso} also hold in $B^{\smallblacksquare}_H$.
        \item As $r=1$, $P$ is pairwise distinguished using only $A$ and $B$, so $B^{\smallblacksquare}_H$ contains no twins (\Cref{claim:notwins}).
        \item The construction of the formulas $\mathrm{sameSide}(x,y)$ and $E_0(x,y)$ only depends on the previous claims and the neighborhoods of $A$ and $B$, so the formulas still work as intended in $B^{\smallblacksquare}_H$.
    \end{itemize}
    It follows that $B^\smallblacksquare_H$ contains neither isolated vertices nor twins and that $I_{\delta,\phi}(B^\smallblacksquare_H) = I_{\delta,\phi}(B_H) = H$.
    It remains to show that $B^\smallblacksquare_H \in \Bg$.
    We do this by showing that $B^\smallblacksquare_H$ is an induced subgraph of
    the comparability grid $G_N$ of order $N := (n+m-1)$ whose vertex set is $\{ a_{i,j} : i,j \in [N] \}$.
    We witness this fact by constructing an embedding $f$ of $B^\smallblacksquare_H$ into $G_N$ as follows:
    \begin{itemize}
        \item $f(a_i) := a_{m+i-1,n-i+1}$ for all $a_i \in A_\star$,
        \item $f(b_j) := a_{m-j+1,n+j-1}$ for all $b_j \in B_\star$,
        \item $f(p_{i,j,1}) := a_{m+i-1,n+j-1}$ for all $p_{i,j,1} \in P_\star$.
    \end{itemize}
    See the right side of \Cref{fig:hardness-h-to-bipartite} for a visualization of the embedding.
    Using this visualization, it is easy to check that $f$ is indeed an embedding of $B^\smallblacksquare_H$ into $G_N$.
    This finishes the case for $\BB = \Bg$ and concludes the proof.
\end{proof}

\subsection{Proof of~\Cref{prop:patternsImplyNotNIP}}

\propPatternsImplyNotNIP*
\begin{proof}
    We first show that $\CC$ efficiently interprets the class of all bipartite graphs without isolated vertices.
    By \cref{obs:pigeonhole-patterns}, we distinguish three cases.
    \begin{itemize}

        \item Either \(\mathcal{C}\) contains a layer-wise flip of each graph from $\BB$, for $\BB \in \{\Bs_r, \Bc_r\}$,
            and we can apply \cref{lem:hardness-undo-flips-sco} and \cref{lem:hardness-sc-to-bipartite},

        \item or \(\mathcal{C}\) contains a layer-wise flip of each graph from $\Bo_r$,
            and we can apply \cref{lem:hardness-undo-flips-sco} and \cref{lem:hardness-o-to-bipartite},

        \item or \(\mathcal{C}\) contains \(\Bg\) and we can apply \cref{lem:hardness-o-to-bipartite}.
    \end{itemize}
    From there, \cref{lem:hardness-bipartite-to-all} brings us to the class of all graphs.
\end{proof}

\section{Almost Bounded Twin-Width}
\label{sec:tww}

\newcommand{\tww}{\mathrm{tww}}

Twin-width is a graph parameter recently introduced in \cite{twwI}.
We quickly recall its definition.
A \emph{contraction sequence} for an $n$-vertex graph $G$ is a sequence of partitions of $V(G)$
\[
    \PP_1 = \{\{ v \} : v \in V(G) \}, \ldots, \PP_n = \{ V(G) \}    
\]
where $\PP_{i+1}$ is obtained from $\PP_i$ by merging two parts. 
The \emph{red-degree} of a part $P \in \PP_i$ is the number of parts $Q \in \PP_i \setminus \{ P \}$, such that $P$ and $Q$ are neither fully adjacent nor fully non-adjacent in $G$.
The \emph{red-degree} of a partition is the largest red-degree among its parts.
The \emph{twin-width}  of a graph $G$, denoted $\tww(G)$, is the smallest integer $d$ such that $G$ has a contraction sequence in which every partition has red-degree at most $d$.
A graph  class $\CC$ has \emph{almost bounded twin-width}
if for every $\varepsilon>0$ we have \[\tww(G)\le O_{\CC,\varepsilon}(n^{\varepsilon})\qquad\text{for every $n$-vertex graph $G\in\CC$.}\]

The goal of this section is to show the following.

\almostBoundedTww*

Note that in Section~\ref{sec:abfw}
we prove that the more general classes of \emph{almost bounded flip-width} are monadically dependent.
This immediately implies Theorem~\ref{thm:almost-bounded-tww}, but we provide a self-contained proof in this section.

We prove the theorem by analyzing the induced subgraphs appearing in classes that are  monadically independent:
flipped star/clique/half-graph $r$-crossings and comparability grids.

\begin{lemma}\label{lem:tww-flipped-crossings}
    Let $G$ be a flipped star/clique/half-graph $r$-crossing of order $n$.
    Then $\tww(G)\geq n^{\frac{1}{r+1}} - 2$.
\end{lemma}

The following proof closely follows the proof of \cite[Proposition 6.2]{twwII}, where it was proven that the $r$-subdivided clique of order $n$ has twin-width $\Omega_r(n^c)$ for some $c>0$.
We adapt their proof to the bipartite setting (which slightly simplifies the setup), and generalize it to also work with half-graph crossings (which makes the proof slightly more complicated).

\begin{proof}[Proof of \cref{lem:tww-flipped-crossings}]
    We use the notation for \(r\)-crossings presented in \cref{def:rcrossings}.
    Let $d$ be the twin-width of $G$ and 
    fix a witnessing contraction sequence.
    By the bipartite symmetry of $G$, we can assume that there is a step in the sequence, where two vertices $a_x, a_y \in L_0$ with $x < y$ are contracted, but no vertices from $L_{r+1}$ have been contracted so far.
    Let $\PP$ be the partition of $V(G)$ at this step, and $P_0 \in \PP$ be the part containing $a_x$ and $a_y$.
    For every set $S \subseteq [n]$ define the sets $V_t(S) := \{ p_{x,j,t} : j \in S \}$ for $t \in [r]$ and
    $V_{r+1}(S) := \{b_j : j \in S\}$.

    \begin{claim}\label{claim:tww}
        For every $t \in [r+1]$ there is $S_t \subseteq [n]$ with $|S_t| \geq \frac{n}{(d+2)^t}$ 
        and $P_t \in \PP$ with $V_t(S_t) \subseteq P_t$.
    \end{claim}
    \begin{claimproof}
        We prove the claim by induction on $t$.
        By definition of a flipped crossing, each vertex in $V_1([n])$ is adjacent exactly one vertex of $a_x$ and $a_y$.
    Consequently, there is a red edge between $P_0$ and every
    part $P \in \PP \setminus \{ P_0 \}$ that intersects $V_1([n])$.
    Hence, at most $d + 1$ parts of $\PP$ intersect $V_1([n])$, so there is a part $P_1$ containing at least $\frac{n}{d+1}$ many vertices from $V_1([n])$, and we can choose $S_1$ accordingly.

    In the inductive step,
    by the structure of a flipped crossing, the vertices from $V_{t+1}(S_t)$ have an inhomogeneous connection to $V_t(S_t)$,
    with the possible exception of the smallest vertex in \(S_{t+1}\) (if \(t=r\) and we consider a flipped half-graph crossing).
    Let~$\mathcal{Q}$ be the set containing $P_t$ and the part of \(\mathcal P\) that contains the smallest vertex in \(S_{t+1}\).
    The aforementioned inhomogeneous connections ensure that
    there is a red edge between $P_t$ and every part $P \in \PP \setminus \mathcal{Q}$ that contains a vertex from $V_{t+1}(S_t)$.
    Then there must be a part $P_{t+1}$ that contains at least 
    $\frac{|V_{t+1}(S_t)|}{d+2} \geq  \frac{n}{(d+2)^{t+1}}$
    vertices from $V_{t+1}(S_t)$. We can choose $S_{t+1}$ accordingly.
    \end{claimproof}
    The above claim yields a part $P_{r+1}$ containing at least $\frac{n}{(d+2)^{r+1}}$ vertices from $L_{r+1}\supseteq V_{r+1}([n])$.
    As no two vertices from $L_{r+1}$ have been contracted yet, $P_{r+1}$ can contain at most one vertex from~$L_{r+1}$. Then we have $d \geq n^{\frac{1}{r+1}} - 2$.
\end{proof}

\begin{lemma}\label{lem:tww-cgrid}
    Let $G$ be a comparability grid of order $n+1$.
    Then $\tww(G)\geq n^{\frac{1}{2}} - 2$.
\end{lemma}
\begin{proof}
    We proceed as in the proof of \Cref{lem:tww-flipped-crossings}.
    Given a comparability grid of order \(n+1\) with vertex set \(\{0,\dots,n\}\times \{0,\dots,n\}\),
    we define \(L_0 = \{0\}\times [n]\) and \(L_{2} = [n]\times \{0\}\).
    By symmetry, we consider the first step in a contraction sequence, where two vertices
    $(0,x), (0,y) \in L_0$ with $x < y$ are contracted, but no vertices from \(L_2\) have been contracted so far.
    For every set $S \subseteq [n]$ define the sets $V_1(S) := S \times \{x\}$ and $V_{2}(S) := S \times \{0\}$.
    The remainder of the proof proceeds as in \Cref{claim:tww} and onwards (with \(r=1\)).
\end{proof}

We finally prove \cref{thm:almost-bounded-tww}, which we rephrase as follows.

\begin{proposition}
    For every hereditary, monadically independent graph class $\CC$  there exists a real $\delta > 0$ such that for every $n \in \N$, there is graph $G_n \in \CC$ with $|V(G_n)| \geq n$ and $\tww(G_n) \geq |V(G_n)|^\delta$.
\end{proposition}

\begin{proof}
    By \cref{thm:main-circle} and hereditariness of $\CC$, there is $r \in \N$ such that for every $n \in \N$, $\CC$ contains either a flipped star/clique/half-graph $r$-crossing of order $n$ or the comparability grid of order $n$.
    We choose $G_n$ to be the flipped $r$-crossing or the comparability grid contained in $\CC$.
    In both cases we have $n \leq |V(G_n)| \leq O_r(n^2)$.
    If $G_n$ is a flipped crossing, we conclude using \cref{lem:tww-flipped-crossings}.
    If $G_n$ is a comparability grid, we conclude using \cref{lem:tww-cgrid}.
\end{proof}

\section{Almost Bounded Flip-Width}
\label{sec:abfw}

The radius-$r$ flip-width of a graph $G$,
denoted $\fw_r(G)$,
is a graph width parameter recently introduced in \cite{flipwidth}.
It is defined in terms of a generalization of the Cops and Robber game to dense graphs.
We refer the reader to \cite{flipwidth} for a definition.
Using this notion,
a hereditary graph  class $\CC$ has \emph{almost bounded flip-width}
if for every fixed radius \(r \ge 1\) and $\varepsilon>0$ we have
\[
   \fw_r(G)\le O_{\CC,\varepsilon}(n^{\varepsilon}) \qquad\text{for every $n$-vertex graph $G\in\CC$.}
\]

In  \cite[Conjecture 10.7]{flipwidth},
the following conjecture is posed:

\conjFW*

In this section, we confirm the forwards implication of Conjecture~\ref{conj:fw}
by proving the following theorem.

\begin{restatable}{theorem}{thmFw}\label{thm:fw}
    Let $\CC$ be a hereditary graph class
    such that for every fixed $r\ge 1$,
    \[\fw_r(G)\le o_{\CC,r}(n^{\frac{1}{2}})\qquad\text{for every $n$-vertex graph $G\in\CC$.}\]
    Then $\CC$ is monadically dependent.
    In particular, every hereditary class of almost bounded flip-width is monadically dependent.
  \end{restatable}

Note that every class of almost bounded twin-width has almost bounded flip-width
by \cite[Thm. 10.21]{flipwidth}.
Therefore, Theorem~\ref{thm:fw}
implies Theorem~\ref{thm:almost-bounded-tww}.

\medskip
To prove \cref{thm:fw}, 
we show the following lemma.

\begin{lemma}\label{lem:fw-patterns}
    Fix $r\ge 1$. Let $\DD$ be either
    \begin{itemize}
        \item the class of star $r$-crossings, or
        \item the class of clique $r$-crossings, or
        \item the class of half-graph $r$-crossings, or
        \item the class of comparability grids.
    \end{itemize}
    Then \(\fw_{r+1}(G)\ge \Omega_{r}({n^{\frac 1 2}})\) for every $n$-vertex graph $G\in \DD$.
  \end{lemma}

 Theorem~\ref{thm:fw} follows easily from Lemma~\ref{lem:fw-patterns}, as we now show.
  \begin{proof}[Proof of Theorem~\ref{thm:fw}]
    Let  $\CC$ be a hereditary graph class, such that 
    for every fixed $r \ge 1$, $G\in\CC$, we have $\fw_r(G)= o_{\CC,r}(|G|^{\frac{1}{2}})$.
    We want to prove that $\CC$ is monadically dependent, so suppose otherwise.
    By Theorem~\ref{thm:main-circle},
    there is some $r,\ell\ge 1$ and a  class $\DD$ as in Lemma~\ref{lem:fw-patterns},
    such that $\CC$ contains some $\ell$-flip $G'$ of every graph $G\in\DD$.
Note that if $G'$ is an $\ell$-flip of $G$,
then from the definition of flip-width, it follows easily that $\fw_{r+1}(G)\le \ell\cdot \fw_{r+1}(G')$ 
(see also \cite[Thm. 8.2]{flipwidth} with \(q=0\)).
    Thus, for every $n$-vertex graph $G\in \DD$ which is a $\ell$-flip of a graph $G'\in\CC$  we have by \Cref{lem:fw-patterns} the contradiction
    \[\Omega_{r}({n^{\frac 1 2}})\le \fw_{r+1}(G)\le \ell\cdot \fw_{r+1}(G') \le o_{\CC,r}(n^{\frac 1 2}).\]
    This proves Theorem~\ref{thm:fw}.
  \end{proof}

It remains to prove Lemma~\ref{lem:fw-patterns}.
The proof follows along the lines of the proof of \cite[Proposition 6.7]{flipwidth},
which implies in particular that the $r$-subdivision $K^{(r)}_n$ of $K_n$ satisfies 
$\fw_{r+1}(K^{(r)}_n)\ge \Omega_r(n^{\frac 12})$.

\medskip

Fix $n,m \ge 1$ and a symbol $S\in\set{\le,=}$.
Let $H^{S}_{n,m}$ denote the bipartite 
graph with parts $L=[n]$ and $R=[n]\times[m]$,
and edges $\set{i, (i',j)}$, for $i,i'\in [n]$ and $j\in [m]$, such that 
$i\,S\,i'$ holds.
For $i\in [n]$ denote $\tilde N(i):=[i]\times [m]\subset R$.

\begin{lemma}\label{lem:half-graph-forest}
    Fix $k,n,m \ge 1$, $S\in\set{\le,=}$, and let $H:=H^{S}_{n,m}=(L,R,E)$ be as defined above.
    Let $H'$ be a $k$-flip of $H$.
    Then there is a set $X\subset L$ 
    of size $n-k$
    and an injective function $f\from X\to L$ such that for every $v\in X$,
    \[|N^{H'}(v)\cap \tilde N({f(v)})|\ge \frac m 2.\]
\end{lemma}
\begin{proof}
    Let $\cal P$ be the partition of $V(H)$ into $k$ parts such that the graph $H'$ is a $\cal P$-flip of $H$.
    First consider the case when all the vertices of $L$ are in one part $A$ of $\cal P$. Let $W$ be the   
  union of the parts $B$ of $\cal P$ that are not flipped with $A$ in the flip that produces $H'$ from $H$.
  Let $X_1$ consist of those vertices $v\in L$ such that $|\tilde N(v)\cap W|\ge \frac{m}2$,
  and let $X_2$ consist of the remaining vertices in $L$.
  
  Every  vertex in $X_1$ is adjacent in $H'$ to at least $\frac{m}2$ vertices in $N^H(v)$ (namely, to $N^H(v)\cap W$), so we can set $f(v)=v$ for all $v\in X_1$.
  On the other hand, if $v'<v\in [n]$ are two distinct vertices in $X_2$, then $N^{H'}(v)\supseteq \tilde N(v')\setminus W$ and 
  $|\tilde N(v')\setminus W|> \frac{m}2$.
  Set $X:=X_1\cup X_2 \setminus \min(X_2)$, and let $f\from X_1\cup X_2 \setminus \min(X_2)\to 
  X_1\cup X_2\setminus \max(X_2)$ be the injective function 
  that maps every vertex in $X_1$ to itself, and every vertex in $X_2-\min(X_2)$ to its predecessor in $X_2$.
Then $|X|\ge |L|-1$, and every $v\in X$
is adjacent in $H'$ to at least $\frac{m}2$ elements of  $\tilde N(f(v))$.

In the general case, partition $L$ as $L=L_1\uplus\ldots\uplus L_s$, for some $s\le k$, following the partition $\cal P$ restricted to $L$. 
For each $1\le i\le s$, let $H_i=H[L_i \cup R]$, and $H_i'=H'[L_i \cup R]$;
then $H_i'$ is a $k$-flip of $H_i$, and they fall into the special case considered above.
Hence, for each $1\le i\le s$ there is a set $X_i\subset L_i$ with $|X_i| \ge|L_i|-1$
and an injection $f_i\from X_i\to L_i$. Set $X=X_1\uplus\ldots\uplus X_s$, and let $f\from X\to L$ be such that $f(x)=f_i(x)$ for $x\in X_i$. Then $X$ and $f$ satisfy the required condition.
\end{proof}

For $r,\ell\ge 1$, let $G_{r,\ell}$ denote the 
disjoint union of 
$\ell$ paths, each of length $r$ (and with $r+1$ vertices). 
In each path, call one of the vertices 
of degree one a \emph{source}, and the other one a \emph{target}.
We can use Lemma~\ref{lem:half-graph-forest} (with $m=1$ and $S$ being $=$)
to prove the following by induction on \(r \ge 1\).

\begin{lemma}\label{lem:path-flips}Fix $k,r,\ell\ge 1$.
  If $G'$ is a $k$-flip of $G_{r,\ell}$, then 
  at least $\ell-rk$ target vertices of $G_{r,\ell}$
are joined by a path of length $r$ in $G'$ with some 
source vertex.
\end{lemma}

We present a modification of the notion of a hideout, defined in \cite{flipwidth}.

\begin{definition}
  Fix $r,k,d\ge 1$ and a graph $G$.
A pair $(A,B)$ of subsets of $V(G)$
is a \emph{$(r,k,d)$-bi-hideout}
if $|A|,|B|>d$ and for every $k$-flip $G'$ of $G$,
\begin{align*}
  |\set{a\in A:|N^{G'}_r[v]\cap B|\le d}|&\le d,\\
  |\set{b\in B:|N^{G'}_r[v]\cap A|\le d}|&\le d.
\end{align*}
\end{definition}

The following is an adaptation of \cite[Lem. 5.16]{flipwidth}.
\begin{lemma}\label{lem:hideouts}
    Fix $r,k\ge 1$. Suppose that $G$ has a $(r,k,d)$-bi-hideout for some $d\ge 1$. Then ${\fw_r(G)>k}$.
\end{lemma}

\begin{proof}Let $(A,B)$ with $A,B\subset V(G)$ be a $(r,k,d)$-bi-hideout.
  We describe a strategy for the runner
  in the flipper game on $G$ with radius $r$ and width $k$, which allows  to elude the flipper indefinitely.
  The strategy is as follows:
  when the flipper announces a $k$-flip $G'$ of $G$ in round $i$,
  if $i$ is odd then the runner moves to some vertex $v\in A$ 
   such that  $|N_r^{G'}(v)\cap B|> d$,
   and if $i$ is even then the runner moves to some vertex $v\in B$ 
   such that  $|N_r^{G'}(v)\cap A|> d$.
   
  In the first move, pick any $v\in A$ with $|N_r^{G}(v)\cap B|> d$.
  Such a vertex exists by considering the flip $G'=G$, since $|A|>d$.
  
We show it is always possible to make a move as described in the strategy.
 Suppose at some point in the game, 
 say in an odd-numbered round (the other case is symmetric), the current position $v\in A$ of the runner
is such that 
\newcommand{\prv}{P}
\newcommand{\nxt}{N}
\begin{align}\label{eq:inv-ineq}
  |N_r^{\prv}(v)\cap B|&> d,
\end{align}
where $\prv$ is the previous $k$-flip of $G$ announced by the flipper (in the first round, $\prv=G$), and that the flipper now announces the next $k$-flip $\nxt$ of $G$. Since $(A,B)$ is a $(r,k,d)$-bi-hideout, 
the set $X\subset B$ of vertices $w\in B$ 
such that $|N_r^{\nxt}(w)\cap A|\le d$
satisfies $|X|\le d$. 
By~\eqref{eq:inv-ineq},
 $N_r^{\prv}(v)$ contains at least one vertex $v'\in B \setminus X$.
The runner moves from $v\in A$ to $v'\in B$ 
along a path of length at most $r$ in $\prv$.
As $v'\in B\setminus X$, the invariant is maintained.
Therefore, playing according to this strategy, the runner can elude the flipper indefinitely, so $\fw_r(G)>k$.
\end{proof}

For a graph \(G\) and disjoint vertex sets \(U,W \subseteq V(G)\), denote by \(G[U,W]\) the bipartite
graph with sides \(U\) and \(W\) and \(u \in U\), \(w \in W\) adjacent if and only if \(uw \in E(G)\).

\begin{restatable}{lemma}{fwCrossings}\label{lem:fw-crossings}
    Fix $r,k\ge 1$ and $n:=2(r+1)k+1$.
    Let $G$ be the star/clique/half-graph $r$-crossing of order $n$.
    Then $\fw_{r+1}(G)>k$.
\end{restatable}

\begin{proof}
    We use the notation of Definition~\ref{def:rcrossings} to denote the vertices of an \(r\)-crossing.
    Let $A=L_0\subset$ and $B=L_{r+1}$  be the two sets of roots of the $r$-crossing $G$,
    with $|A|=|B|=n$. 
    Moreover, \(p_{v,w,1}\) refers to the first vertex on the path from \(v \in A\) to  \(w \in B\).
    We show that $(A, B)$
    forms an $(r+1,k,k)$-bi-hideout in $G$. By Lemma~\ref{lem:hideouts}, this implies that $\fw_{r+1}(G)>k$. 
  
    Denote $A':=L_1$ and  $B':=L_r$.
    Note that $A$ and $A'$ are disjoint, and 
     $G[A,A']$ is a 
    bipartite graph isomorphic to 
    $H^\le_{n,n}$ when $G$ is a half-graph $r$-crossing, and isomorphic to $H^{=}_{n,n}$ in the other two cases.
    Similarly, $G[B,B']$ is also isomorphic to either $H^{\le}_{n,n}$ or $H^=_{n,n}$.

    Consider a \(k\)-flip \(G'\) of \(G\).
    Applying Lemma~\ref{lem:half-graph-forest} to the bipartite graph $H:=G[A,A']$ and the \(k\)-flip $H' := G'[A,A']$ of \(H\)
    gives us a set $X \subseteq A$ of size $n-k$
    and an injective function $f\from X\to A$ such that for all \(v \in X\),
    \[
    P(v) \subseteq N^{G'}(v),
    \qquad \textnormal{ where } \qquad
    P(v) \subseteq \setof{p_{f(v),w,1}}{w\in [n]} \textnormal{ with } |P(v)| \ge \frac n 2.
    \]

    Let \(v \in X\). By definition of \(G\), there are at least \(n/2\) 
    disjoint paths of length \(r\) between the aforementioned set \(P(v)\) and \(B\) in \(G\).
    Hence, by Lemma~\ref{lem:path-flips}, at least \(n/2-rk > k\) vertices in $P(v)$ are joined in \(G'\)
    by a path of length \(r\) with some root in $B$.
    As $P(v) \subseteq N^{G'}(v)$, we have
    \(|N^{G'}_{r+1}[v] \cap B| > k\) for all \(v \in X\).
    Since \(|X| \ge n-k\),
    \[|\set{v\in A:|N^{G'}_{r+1}[v]\cap B|\le k}|\le k.\]

    By reversing the roles of \(A\) and \(B\), we further obtain
    \[|\set{v\in B:|N^{G'}_{r+1}[v]\cap A|\le k}|\le k.\]
  It follows that $(A, B)$ forms an $(r+1,k,k)$-bi-hideout in $G$, and thus $\fw_{r+1}(G)>k$ by  Lemma~\ref{lem:hideouts}.
  \end{proof}

\begin{lemma}\label{lem:fw-comp-grid}
    Let $G$ be a comparability grid of order $n$.
    Then $\fw_{2}(G)> (n-2)/8$.
\end{lemma}
  \begin{proof}
    Consider the comparability grid \(G\) with vertex set \([n] \times [n]\),
    and let $G^+$ denote the $2$-colored graph
    obtained by removing \((1,1)\) and coloring the vertices $\set{2,\ldots,n}\times \set{2,\ldots,n}$ with color $C_1$,
    and the remaining vertices with color~$C_2$.
    The quantifier-free formula 
    \[\phi(x,y)=E(x,y)\land \neg (C_2(x)\land C_2(y)),\]
    when interpreted in the colored graph $G^+$,
    defines a graph $H$ which is isomorphic to the half-graph $1$-crossing of order $n-1$.
    By \cref{lem:fw-crossings},
    we have
    \[
      \fw_2(H)> \frac{(n-1)-1}{2\cdot (1+1)}=(n-2)/4.
    \]
    By \cite[Thm. 8.2]{flipwidth} (with \(q=0\)),
    a quantifier-free interpretation, when applied to a \(c\)-colored graph, can result in an at most \(c\)-fold increase of the flip-width.
    Therefore,
    \[\fw_{2}(G)\ge \frac1 2 \cdot \fw_{2}(H).\qedhere\]
  \end{proof}

 Lemma \ref{lem:fw-patterns} now follows easily.
  \begin{proof}[Proof of Lemma~\ref{lem:fw-patterns}]
    Follows from Lemma~\ref{lem:fw-crossings} and Lemma~\ref{lem:fw-comp-grid}, as $r$-crossings/comparability grids have $O_{r}(n^2)$ vertices.
\end{proof}
This completes the proof of Theorem~\ref{thm:fw}.

\section{Small Classes}\label{sec:small}

\begin{definition}
    A hereditary graph class $\CC$ is \emph{small}
    if it contains at most $n!c^n$ distinct labeled $n$-vertex graphs, for some constant $c$.
\end{definition}

Formally, we implicitly assume that the class $\CC$ is closed under isomorphism. By a labeled $n$-vertex graph we mean a graph $G$ with vertex set equal to $\set{1,\ldots,n}$.
Two  labeled $n$-vertex graphs are then considered equal if their edge sets are identical.
Denoting by $\CC_n$ the set of labeled $n$-vertex graphs in a class $\CC$, the class $\CC$ is small if for some constant $c$ we have $|\CC_n|\le n!c^n$ for all $n$.

The following theorem is a consequence of our results.

\thmSmall*

Let us start by giving some context on small classes.
It is known (see e.g. \cite[Sec. 3.6]{twwII}) that the class of bipartite subcubic graphs (that is, graphs with maximum degree at most $3$) is not small.
Thus, in particular, the converse implication in Theorem~\ref{thm:small} fails,
as the class of all subcubic graphs is nowhere dense and therefore 
monadically dependent.

On the other hand, it is known that every class of bounded twin-width is small \cite{twwII}. In the same paper, it has been conjectured that the converse also holds, but this conjecture has been subsequently refuted in \cite{twwVII}. However, small classes of ordered graphs (that is, graphs equipped with a total order on the vertex set) have bounded twin-width \cite{twwIV}. Note that for classes of ordered graphs, bounded twin-width coincides with monadic dependence.

To prove \Cref{thm:small}, we prove the following lemma.
For a formula $\phi(x,y)$ and graph $G$,
by $\phi(G)$ we denote the graph with vertex set $V(G)$ 
and edges $uv$ such that $u,v\in V(G)$ and $G,u,v\models\phi(u,v)\lor \phi(v,u)$.

\begin{lemma}\label{lem:lin-size}
    Let $\CC$ be a hereditary class which is monadically independent,
    and is closed under isomorphism.
    There is a positive integer $c$ and a formula $\psi(x,y)$ in the signature of $c$-colored graphs, 
    such that for every bipartite graph $G$ with $n$ vertices and $m$ edges there is some $c$-coloring \(G^+\) of a graph from \(\CC\)
    with at most $c\cdot (n+m)$ vertices such that $G=\psi(G^+)[V(G)]$.
\end{lemma}

We remark that by tracing the construction from Section~\ref{sec:hardness}, it is possible to derive a stronger statement, where 
$G^+\in \CC$  is uncolored.
However, the statement above is sufficient for our purpose here, and is easier to argue by analyzing the construction sketched in the proof of \cref{lem:transduceAll}.

Note that we may assume that in the statement of Lemma~\ref{lem:lin-size},
$G$ is \emph{identical}, not just isomorphic, to the subgraph 
of $\psi(G^+)$ induced by $V(G)$. 
This follows from the weaker statement that $\psi(G^+)[W]$ is isomorphic to $G$ for some $W\subset V(G^+)$,
as the class $\CC$ is assumed to be closed under isomorphism.
We stress the distinction between identical and isomorphic graphs, to highlight that \cref{lem:lin-size} implies that if $G$ is a \emph{labeled} graph, then there exists a coloring $G^+$ of a \emph{labeled} graph from $\CC$ with the desired properties.

\begin{proof}[Proof sketch]
    By \cref{thm:main-circle},
    there is some $r\ge 1$ such that 
    $\CC$ contains a flipped star/clique/half-graph $r$-crossing or a comparability grid of order $n$, for every $n\ge 1$.

    Recall the notion of a radius-$r$ encoding from \cref{sec:to-independence}. Observe that for every bipartite graph $G$
    with $n$ vertices and $m$ edges, and every radius-$r$ encoding $H$ of $G$, the graph $H$ has at most $2n+r\cdot m\le O_r(n+m)$ vertices.
Moreover, a flipped star/clique/half-graph $r$-crossing or comparability grid of order $n$ contains some radius-$r$ encoding $H$ of $G$ as an induced subgraph.
It follows that $\CC$ contains some $r$-encoding $H$ of every 
bipartite graph $G$, and moreover $|V(H)|\le O_r(|V(G)|+|E(G)|)$.

Let $\phi(x,y)$ be the formula from \cref{lem:encode},
 allowing to define the $1$-subdivision of $G$ 
in some coloring $H^+$ of any radius-$r$ encoding $H$ of $G$.
Then the formula $$\psi(x,y):=\exists z.\phi(x,z)\land\phi(z,y)$$
allows to define $G$ in $H^+$, as required in the statement.
\end{proof}
% \todo{use Lemma~\ref{lem:encode} instead}
% \begin{proof}[Proof sketch]
%     We modify slightly the construction described in the proof of \cref{lem:transduceAll}. 
%     There, a transduction was described which, given a star/clique/half-graph $r$-crossing with roots $A$ and $B$, or a comparability grid, can output an arbitrary bipartite graph $G=(A,B,E)$.

%     We proceed the same way as in that construction,
%     but instead of representing $G$ in the appropriate $r$-crossing with roots $A$ and $B$,
%     we represent it in its induced subgraph, obtained by removing all vertices on paths connecting roots $a\in A$ and $b\in B$, such that $\set{a,b}\notin E$
%     (that is, removing all vertices that are colored $C_-$ in the construction from \cref{lem:transduceAll}).
%     Denoting by $G^+$ the obtained colored graph, 
%     we have that $$|V(G^+)|\le |A|+|B|+r\cdot |E|\le r\cdot (n+m),$$ as required.

%     In the case of $r$-half-graph crossings,  we additionally mark  (doubling the number of colors) the vertices of $A$ and $B$ that are 
%     isolated in $G$. The key observation is that for two non-isolated roots $a,a'\in A$ 
%     we still have that $a\le a'$ if and only if $N(a)\subset N(A')$,
%     and similarly for non-isolated roots in $B$.
%     This is enough to recover the edges of $E$ using a fixed first-order formula, just as 
%     in \cref{lem:transduceAll}.
% \end{proof}

\begin{proof}[Proof of Theorem~\ref{thm:small}]
    Let $\CC$ be a hereditary graph class which is monadically independent. As we want to show that $\CC$ is not small, 
    assume towards a contradiction that $\CC$ is small.
    We show that the class $\DD$ of all bipartite subcubic graphs is small, which is a contradiction (see \cite[Sec. 3.6]{twwII}).
Let $\psi$ and $c$ be as in Lemma~\ref{lem:lin-size}.

By $\DD_n$ we denote the set of all labeled graphs with exactly $n$ vertices in $\DD$,
and let $\CC_{{\le} n}$ denote the set of all labeled graphs with at most $n$ vertices in $\CC$.
Then $|\CC_{{\le}n}|\le n!2^{O(n)}$ for all $n$.

For every labeled bipartite subcubic graph $G\in \DD_n$ with $n$ vertices, let $G^+$ denote a $c$-colored labeled graph as in Lemma~\ref{lem:lin-size},
so that $G=\psi(G^+)[V(G)]$ and $G^+$ has at most $c(n+3n/2)\le 3cn$ vertices.

The function $f\from \DD_{n}\to \CC_{{\le}3cn} \times c^{[3cn]}\times 2^{[3cn]}$ which maps $G\in\DD_{n}$ to 
$(G^+,V(G))$ is injective,
as we can recover \(G = \psi(G^+)[V(G)]\).
We then have
\[|\DD_n|\le |\CC_{{\le}3cn}| \cdot (2c)^{3cn} \le (3cn)!\cdot 2^{O({3cn})}\le n!\cdot 2^{O(n)}.\]
This proves that the class $\DD$ of bipartite subcubic graphs is small, a contradiction.
\end{proof}

\newpage
\part{Restrictions and Extensions}
\section{Variants of Flip-Breakability}\label{sec:variants}
In this section we show how to characterize nowhere denseness, bounded cliquewidth, bounded treewidth, bounded shrubdepth, and bounded treedepth using natural restrictions of flip-breakability.

\subsection{Deletion-Breakability Characterizes Nowhere Denseness}

\begin{definition}
    A class of graphs $\CC$ is $\emph{deletion-breakable}$,
    if for every radius \(r \in \N\) there exists a function
    \(N_r : \N \to \N\) and a constant \(k_r \in \N\) such that for all \(m \in \N\), \(G \in \CC\) and \(W \subseteq V(G)\) with
    \(|W| \ge N_r(m)\) there exist a set $S\subseteq V(G)$ with $|S| \leq k_r$ and subsets $A,B\subset W \setminus S$ with \(|A|,|B| \ge m\) such that
    \[
        \dist_{G-S}(A, B) > r.
    \]    
\end{definition}

A class of graphs $\CC$ is nowhere dense if for every radius $r \in \N$,
there exists a bound $N_r \in \N$ such that no graph
from $\CC$ contains an $r$-subdivided clique of order $N_r$ as a subgraph.

\begin{theorem}\label{thm:nd}
    A class of graphs is nowhere dense if and only if it is deletion-breakable.
\end{theorem}

\begin{proof}
    It is easy to see that deletion-breakability generalizes the notion of uniform quasi-wideness discussed in the introduction: In every huge set, we find a large set of vertices pairwise of distance greater than $r$ after removing few vertices. We can partition them into two halves $A$ and $B$ to obtain deletion-breakabilitiy.
    Since nowhere dense classes are uniform quasi-wide, this implies that they are also deletion-breakable.

    For the other direction, assume towards a contradiction that $\CC$ is not nowhere dense but deletion-breakable with bounds $N_r(\cdot)$ and $k_r$ for every $r\in \N$. 
    By definition, there exists a radius $r > 1$ such that $\CC$ contains arbitrarily large $(r-1)$-subdivided cliques as subgraphs.
    Let $G \in \CC$ be a graph containing an $(r-1)$-subdivided clique of size $N_{r}(k_{r}+1)$, whose principal vertices we denote with~$W$.
    By deletion-breakability, $W$ contains two subsets $A$ and $B$, each of size $k_{r}+1$, such that 
    \[
        \dist_{G-S}(A, B) > r
    \]    
    for some vertex set $S$ of size at most $k_{r}$.
    Since $W$ is an $(r-1)$-subdivided clique in $G$, there exist $k_{r} + 1$ disjoint paths of length $r$, that each start in $A$ and end in $B$.
    As $|S| \leq k_{r}$, at most one of those paths must survive in $G-S$, witnessing that $\dist_{G-S}(A, B) \leq r$; a contradiction.
\end{proof}

\subsection{Distance-\texorpdfstring{$\infty$}{Infinity} Flip-Breakability Characterizes Bounded Cliquewidth}

\begin{definition}
    A class of graphs $\CC$ is \emph{distance-$\infty$ flip-breakable},
    if there exists a function
    \(N : \N \to \N\) and a constant \(k \in \N\) such that for all \(m \in \N\), \(G \in \CC\) and \(W \subseteq V(G)\) with
    \(|W| \ge N(m)\) there exist subsets $A,B\subset W$ with \(|A|,|B| \ge m\) 
    and a $k$-flip $H$ of $G$ such that in $H$, no two vertices $a\in A$ and $b\in B$ are in the same connected component.
\end{definition}

The goal of this subsection is to prove the following.

\begin{theorem}\label{thm:cw}
    A class of graphs has bounded cliquewidth if and only if it is distance-$\infty$ flip-breakable.
\end{theorem}

To prove the theorem, we work with \emph{rankwidth}, a parameter that is functionally equivalent to cliquewidth.
A graph $G$ 
has rankwidth at most $k$ 
if there is a tree $T$ whose leaves 
are the vertices of $G$,
and inner nodes have degree at most $3$,
such that for every edge $e$ of the tree,
the bipartition $A\uplus B$ of the leaves of $T$
into the leaves on either side of $e$,
has \emph{cut-rank} at most $k$.
The cut-rank of a bipartition $A\uplus B$
of the vertex set of a graph $G$, denoted 
$\rk_G(A,B)$, is defined as the rank, over the two-element field, of the $(0,1)$-matrix with rows $A$ and columns $B$, where the entry at row $a\in A$ and column $b\in B$ is $1$ if $ab\in E(G)$ and $0$ otherwise.

\begin{fact}[{\cite[Proposition 6.3]{oum2006}}]\label{fact:rw}
    A class of graphs has bounded cliquewidth if and only if it has bounded rankwidth.
\end{fact}

\begin{lemma}\label{lem:balanced-edge}
    Let $T$ be a rooted subtree of a binary tree and let $W$ be a subset of the leaves of $T$.
    There exists an edge $e \in E(T)$ such that the two subtrees $T_1$ and $T_2$ obtained by removing $e$ from $T$ each contain at least $\frac{1}{4}|W|$ vertices from $W$.
\end{lemma}

\begin{proof}
    For a vertex $v \in V(T)$, denote by $T(v)$ the subtree rooted at $v$.
    Let $(v_1,\ldots,v_m)$ be a root-to-leaf path in $T$ such that for all $1 \leq i < m$, the vertex $v_{i+1}$ is the child of $v_i$ whose subtree contains the most elements from $W$, where ties are broken arbitrarily.
    Let $i \in [m]$ be the largest index such that $T(v_i)$ contains at least $\frac{1}{4}|W|$ vertices from $W$. 
    By construction, $T(v_2)$ contains at least $\frac{1}{2}|W|$ elements from $W$, and therefore $i > 1$, i.e.\ $v_i$ has a parent $v_{i-1}$.
    $T(v_i)$ contains less than $\frac{1}{2}|W|$ vertices from $W$, as both of its at most two children contain less than $\frac{1}{4}|W|$ elements from $W$.
    Therefore, the edge connecting $v_i$ and $v_{i-1}$ is the desired edge.
\end{proof}

\begin{lemma}\label{lem:cw-fw}
    Every class of graphs with bounded rankwidth is distance-$\infty$ flip-breakable.
\end{lemma}

\begin{proof}
    Fix a number $k$ and let $\CC$ be a class of graphs of rankwidth at most $k$.
    We will show that $\CC$ is distance-$\infty$ flip-breakable for $N(m) := 4m$ using $2^k + 2^{2^k}$ flips.
    For every graph $G\in \CC$ there is 
    a rooted subtree $T$ of a binary tree with leaves $V(G)$,
    such that for every edge $e \in E(T)$,
    the bipartition $X\uplus Y$ of the leaves of $T$
    into the leaves on either side of $e$,
    has cut-rank at most $k$.
    Let $W \subseteq V(G)$ be a set of size $4m$.
    By \cref{lem:balanced-edge}, there exists an edge $e$ such that in the corresponding bipartition $X \uplus Y$ of $V(G)$, both $X$ and $Y$ each contain at least $m$ many vertices of $W$.
    Observe that since $\rk_G(X,Y) \leq k$, $X$ induces at most $2^k$ distinct neighborhoods over $Y$.
    Then there is a $(2^k + 2^{2^k})$-flip $H$ of $G$ in which there are no edges between $X$ and $Y$:
    the corresponding partition of $V(G)$ partitions the vertices of $X$ into $2^k$ parts depending on their neighborhood in $Y$ and partitions the vertices of $Y$ in $2^{2^k}$ parts depending on their neighborhood in $X$.
\end{proof}

A set $W$ of vertices of $G$ is \emph{well-linked}, if for every bipartition $A\uplus B$ of $V(G)$,
the cut-rank of $A\uplus B$ satisfies  $\rk_G(A,B)\ge \min(|A\cap W|,|B\cap W|)$.
We use the following two facts.

\begin{fact}[{\cite[Thm. 5.2]{oum2006}}]\label{fact:well-linked}
Every graph of rankwidth greater than $k$ contains  well-linked set of size $k$.
\end{fact}

\begin{fact}[{\cite[Lem. D.2]{flipwidth}}]\label{lem:rank-flip}
    Let $G$ be a graph 
    and $A\uplus B$ a bipartition of $V(G)$ 
    with $\rk_G(A,B)>k$.
    Then for every $k$-flip $H$ of $G$ there is some  edge $ab\in E(H)$ with $a\in A$ and $b\in B$.
\end{fact}

\begin{lemma}\label{lem:cw-bw}
    Every class of graphs with unbounded rankwidth is not distance-$\infty$ flip-breakable.
\end{lemma}

\begin{proof}
    Let $\CC$ be a class of graphs with unbounded rankwidth.
    Assume towards a contradiction that $\CC$ is distance-$\infty$ flip-breakable with bounds $N(\cdot)$ and $k$.
    By \cref{fact:well-linked}, there exists a graph $G\in \CC$ that contains a well-linked set $W$ of size at least $N(k+1)$.
    By distance-$\infty$ flip-breakability, there exists a $k$-flip $H$ of $G$ and two sets $A, B \subseteq W$ of size $k+1$ each, such that no two vertices $a \in A$ and $b \in B$ are in the same component in $H$.
    We can therefore find a bipartition of $\mathcal{X} \uplus \mathcal{Y}$ of the connected components of $H$, such that $\mathcal{X}$ contains all the components containing a vertex of $A$ and $\mathcal{Y}$ contains all the components containing a vertex of $B$.
    Components containing neither a vertex of $A$ nor of $B$ can be distributed arbitrarily among $\mathcal{X}$ and $\mathcal{Y}$.
    Let $X := \bigcup \mathcal{X}$ and $Y := \bigcup \mathcal{Y}$.
    Then $X \uplus Y$ is a bipartition of $V(H)$, and there is no edge between $X$ and $Y$ in \(H\).
    Since $W$ is well-linked we have that the cut-rank $X \uplus Y$ satisfies
    \[
        \rk_G(X,Y) \geq \min(|X \cap W|, |Y \cap W|) \geq \min(|A|, |B|) = k + 1.    
    \]
    By \cref{lem:rank-flip}, there must be an edge $ab \in E(H)$ with $a\in X$ and $b\in Y$; a contradiction.
\end{proof}

Combining \cref{fact:rw}, \cref{lem:cw-fw}, and \cref{lem:cw-bw} now yields \cref{thm:cw}.

\subsection{Distance-\texorpdfstring{$\infty$}{Infinity} Deletion-Breakability Characterizes Bounded Treewidth}

\begin{definition}
    A class of graphs $\CC$ is \emph{distance-$\infty$ deletion-breakable},
    if there exists a function
    \(N : \N \to \N\) and a constant \(k \in \N\) such that for all \(m \in \N\), \(G \in \CC\) and \(W \subseteq V(G)\) with
    \(|W| \ge N(m)\)
    there exist a set $S\subseteq V(G)$ with $|S| \leq k$ and subsets $A,B\subset W \setminus S$ with \(|A|,|B| \ge m\) such that in
    \(G - S\), no two vertices $a\in A$ and $b\in B$ are in the same connected component.
\end{definition}

The goal of this subsection is to prove the following.

\begin{theorem}\label{thm:tw}
    A class of graphs has bounded treewidth if and only if it is distance-$\infty$ deletion-breakable.
\end{theorem}
We start with the forward direction.
We assume familiarity with treewidth and (nice) tree decompositions. See for example \cite{parameterized_algorithms} for an introduction. 
\begin{lemma}\label{lem:tw-fw}
    Every class of graphs with bounded treewidth is distance-$\infty$ deletion-breakable.
\end{lemma}
\begin{proof}
    Let \(\mathcal{C}\) be a graph class of treewidth at most \(k-1\).
    We show that $\CC$ is distance-$\infty$ deletion-breakable with bounds $N(m) := 4(m+k)$ and $k := k$.

    Consider a graph \(G \in \CC\) of treewidth at most \(k-1\) and a subset \(W\) containing at least \(N(m)= 4(m+k)\) vertices of \(G\).
    We fix a nice tree decomposition of \(G\) of width \(k\), and associate with every bag \(t\) the set \(V(t) \subseteq V(G)\) consisting of all
    vertices that are contained either in \(t\) or in a descendant of \(t\).
    We consider a walk that starts at the root of the nice tree decomposition
    and walks downwards towards the leaves. Whenever the walk reaches a join node, it proceeds
    towards the child \(t\) that maximizes \(|V(t) \cap W|\).
    As each node has at most two children, the cardinality \(|V(t) \cap W|\) of the current node \(t\) of the walk can decrease at most by a factor \(\frac 12\) with each step.
    Hence, as in the proof of \Cref{lem:balanced-edge}, we reach at some point a node \(t\) with 
    \[
        m+k \le \frac 14 |W| \le |V(t) \cap W| < \frac 12 |W| \le 2(m+k).    
    \]
    Let \(S \subseteq V(G)\) be the vertices in this bag \(t\) and choose \(A = (V(t) \cap W)\setminus S\) and \(B = W\setminus V(t)\).
    Note that \(|S| \le k\) and \(|A|,|B| \ge m\).
    By the definition of tree decompositions, the vertices \(S\) act as a separator in the desired sense:
    In \(G - S\), no two vertices \(a \in A\) and \(b \in B\) are in the same connected component.
\end{proof}

The backwards direction will follow easily from the Grid-Minor theorem by Robertson and Seymour.
First some notation.
A graph $H$ is a \emph{minor} of a graph $G$ if there exists a \emph{minor model} $\mu$ of $H$ in $G$.
A minor model is a map $\mu$ that assigns to every vertex $v \in V(H)$ a connected subgraph $\mu(v)$ of $G$ and to every edge $e \in E(H)$ an edge $\mu(e) \in E(G)$ satisfying
\begin{itemize}
    \item for all $u,v \in V(H)$ with $u \neq v$: $V(\mu(u)) \cap V(\mu(v)) = \emptyset$;
    \item for every $(u,v) \in E(H)$: $\mu((u,v)) = (u',v')$ for vertices $u'\in V(\mu(u))$ and $v'\in V(\mu(v))$.
\end{itemize}
Possibly deviating from our notation in previous sections, in this section the \emph{$k$-grid} is the graph on the vertex set $[k] \times [k]$ where two vertices $(i,j)$ and $(i',j')$ are adjacent if and only if $|i - i'| + |j - j'| = 1$.
We can now state the Grid-Minor theorem.
\begin{fact}[{\cite[Thm. 1.5]{robertson1986graph}}]\label{fact:gridminor}
    Let $\CC$ be a class of graphs with unbounded treewidth. 
    Then for every $k \in \N$, $\CC$ contains a graph which contains the $k$-grid as a minor.
\end{fact}

\begin{lemma}\label{lem:tw-bw}
    Every class of graphs with unbounded treewidth is not distance-$\infty$ deletion-breakable.
\end{lemma}

\begin{proof}
    Assume towards a contradiction that $\CC$ has unbounded treewidth but is distance-$\infty$ deletion-breakable with bounds $N(\cdot)$ and $k$.
    Let $t:=N(2k+2)$.
    By \cref{fact:gridminor}, there is a graph $G \in \CC$ such that there exists a minor model $\mu$ of the $t$-grid in $G$.
    Let $W := \{ v_1,\ldots,v_t \}\subseteq V(G)$ be a set \emph{representing} the bottom row of the $t$-grid: we pick one vertex $v_i$ from the subgraph $\mu((i,1))$ for each $i \in [t]$.
    We apply distance-$\infty$ flip-breakability to the set $W$ in $G$, which yields a set $S \subseteq V(G)$ of size $k$ and disjoint sets $A, B \subseteq W \setminus S$, each of size $2k+2$, such that in $G-S$ there is no path from a vertex in $A$ to a vertex in $B$.
    % Let $A_H := \{ u_i : v_i \in A \}$ and 
    Let $i_A \in [t]$ be an index such that each of the sets 
    \[
        A_1 := \{ v_1,\ldots, v_{i_A}\} \cap A
        \quad
        \text{and}
        \quad
        A_2 := \{ v_{i_A+1},\ldots, v_{t}\} \cap A
    \]
    contains $k+1$ elements of $A$. Pick $i_B$, $B_1$, and $B_2$ symmetrically.
    Assume first $i_A \leq i_B$. Then $i < j$ for all $v_i \in A_1$ and $v_j \in B_2$.
    Let $A_\star := \{(i,1) : v_i \in A_1\}$ be the vertices from the bottom row of the $t$-grid represented by $A_1$ and likewise let $B_\star := \{(j,1) : v_j \in B_1\}$.
    In \cref{fig:tw-breakability} it is easy to see that the vertices from $A_\star$ and $B_\star$ can be matched by $k+1$ disjoint paths in the $t$-grid.

    \begin{figure}[htbp]
        \centering
        \includegraphics[scale = 1]{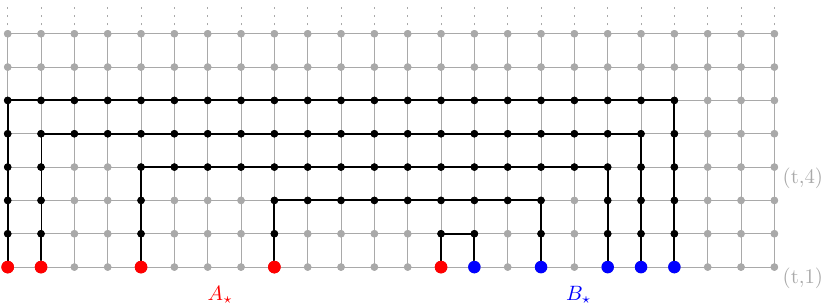}
        \caption{Pairing the vertices of $A_\star$ and $B_\star$ with disjoint paths in the $t$-grid.}
        \label{fig:tw-breakability}
    \end{figure}
    
    By the definition of minor, also $A_1$ and $B_1$ can be matched by $k+1$ disjoint paths in $G$.
    We now reach the desired contradiction, as we assumed that no paths run between $A$ and $B$ in $G-S$, but removing the at most $k$ vertices from $S$ can destroy at most $k$ paths.

    In the case where $i_A > i_B$, we have $i < j$ for all $v_i \in B_1$ and $v_j \in A_2$ and argue symmetrically.
\end{proof}

Combining \cref{lem:tw-fw} and \cref{lem:tw-bw} now yields \cref{thm:tw}.

\subsection{Distance-\texorpdfstring{$\infty$}{Infinity} Flip-Flatness Characterizes Bounded Shrubdepth}

\begin{definition}
    A class of graphs $\CC$ is \emph{distance-$\infty$ flip-flat},
    if there exists a function
    \(N : \N \to \N\) and a constant \(k \in \N\) such that for all \(m \in \N\), \(G \in \CC\) and \(W \subseteq V(G)\) with
    \(|W| \ge N(m)\) there exists a subset $W_\star \subset W$ with \(|W_\star| \ge m\) 
    and a $k$-flip $H$ of $G$ such that in $H$, no two vertices $u,v\in W_\star$ are in the same connected component.
\end{definition}

The goal of this subsection is to prove the following.

\begin{theorem}\label{thm:shrubdepth}
    A class of graphs has bounded shrubdepth if and only if it is distance-$\infty$ flip-flat.
\end{theorem}

To prove that bounded shrubdepth implies distance-$\infty$ flip-flatness, we work with \emph{flipdepth}, a parameter that is functionally equivalent to shrubdepth.
It is defined as follows.
The single vertex graph $K_1$ has flipdepth $0$.
For $k > 0$, a graph $G$ has flipdepth at most $k$, if it 
is a $2$-flip of a disjoint union of (arbitrarily many) graphs of flipdepth at most $k-1$.

\begin{fact}[{\cite[Thm. 3.6]{shrubdepth}}]
    A class of graphs has bounded shrubdepth if and only if it has bounded flipdepth.
\end{fact}

More precisely, \cite[Thm. 3.6]{shrubdepth} shows the functional equivalence of shrubdepth and the graph parameter \emph{SC-depth}.
The definition of SC-depth is obtained from the definition of flipdepth by replacing the $2$-flip with a \emph{set complementation}, that is, the operation of complementing all the edges in an arbitrarily subset of the vertices.
$2$-flips generalize set complementations, but any $2$-flip can be simulated by performing at most three set complementations.
Therefore, flipdepth and SC-depth are functionally equivalent.

\begin{lemma}\label{lem:flipdepth}
    For every graph $G$ of flipdepth at most $k$ and every set $W \subseteq V(G)$,
    there exists a subset $W_\star \subseteq W$ with $|W_\star|\geq |W|^{\frac{1}{2^k}}$ 
    and a $4^{k}$-flip $H$ of $G$, 
    such that in $H$, no two vertices from $W_\star$ are in the same connected component.
\end{lemma}
\begin{proof}
    We prove the lemma by induction on $k$.
    If $k = 0$, we have $G = K_1$ where the statement holds.
    For the inductive step assume $G$ has flipdepth at most $k+1$.
    Then $G$ is $2$-flip of a graph $H$ that is a disjoint union of graphs of flipdepth at most $k$.
    If in $H$ at least $|W|^\frac{1}{2}$ vertices of $W$ are in pairwise different components, we are done.
    Otherwise, there exists a component $C$ of $H$ which contains at least $|W|^\frac{1}{2}$ vertices of $W$.
    By assumption, $H[C]$ has flipdepth at most $k$.
    By induction there is a $4^k$-flip $H_\star$ of $H[C]$ and a set $W_\star \subseteq W$ of size at least $|W_\star| \geq |W|^{\frac{1}{2}\cdot \frac{1}{2^k}}$ such that all vertices from $W_\star$ are in pairwise different components in $H_\star$.
    Refining
    \begin{itemize}
        \item the size $2$ partition of the flip that produced $H$ from $G$,
        \item the size $2$ partition which marks the component $C$ in $H$,
        \item the size $4^k$ partition of the flip that produced $H_\star$ from $H[C]$,
    \end{itemize}
    yields a partition witnessing a $4^{k+1}$-flip of $G$ in which 
    all vertices from $W_\star$ are in pairwise different components as desired.
\end{proof}

\begin{corollary}\label{cor:shrubdepth-implies-ff}
    Every class of bounded shrubdepth is distance-$\infty$ flip-flat.
\end{corollary}

We will use the following proof strategy for the other direction.

\begin{fact}[{\cite[Thm. 1.1]{shrubtrans}}]\label{fact:shrubtrans}
    Every class of unbounded shrubdepth transduces the class of all paths.
\end{fact}
    
\begin{lemma}\label{lem:infty-flip-flat-transduction-closed}
    Let $\CC$ be a class of graphs that is distance-$\infty$ flip-flat.
    Every class that is transducible from $\CC$ is also distance-$\infty$ flip-flat.
\end{lemma}

\begin{lemma}\label{lem:paths-not-infty-flip-flat}
    The class of all paths is not distance-$\infty$ flip-flat.
\end{lemma}

It is easy to see that combining \cref{fact:shrubtrans}, \cref{lem:infty-flip-flat-transduction-closed}, \cref{lem:paths-not-infty-flip-flat} yields the following.

\begin{lemma}
    Every class of graphs with unbounded shrubdepth is not distance-$\infty$ flip-flat.
\end{lemma}

Together with \cref{cor:shrubdepth-implies-ff}, the above lemma proves \cref{thm:shrubdepth}.
It remains to prove \cref{lem:infty-flip-flat-transduction-closed} and \cref{lem:paths-not-infty-flip-flat}.
The former is an immediate consequence of the following fact where
$\phi(x,y)$ is a formula in the language of colored graphs 
and $\phi(G)$ denotes the graph with vertex set $V(G)$ and edge set $\{ uv : G \models \phi(u,v) \lor \phi(v,u) \}$.

\begin{fact}[{\cite[Lem. H.3]{flipwidth}}]\label{fact:transduction-closure}
    For every formula $\phi(x,y)$ and $k \in \N$ there exists $s,\ell \in \N$ such that for every colored graph $G$ and for every $k$-flip $G'$ of $G$, there exists an $\ell$-flip $H'$ of $H := \phi(G)$ such that for every two vertices $u$ and $v$ adjacent in $H'$ we have that $\dist_{G'}(u,v)\leq s$.
\end{fact}

\begin{proof}[Proof of \cref{lem:infty-flip-flat-transduction-closed}]
    Assume the class $\CC$ is distance-$\infty$ flip-flat with bounds $N(\cdot)$ and $k$
    and transduces the class $\DD$ using the formula $\phi(x,y)$.
    We will show that also $\DD$ is flip-flat with bounds $N(\cdot)$ and $\ell$, where $\ell$ is the bound obtained from \cref{fact:transduction-closure} for $\phi$ and $k$.
    Let $H\in \DD$, $m\in \N$, and $W \subseteq V(G)$ be a set of size at least $N(m)$.
    Since $\CC$ transduces $\DD$, we have $H = \phi(G)[V(H)]$ for some colored graph $G \in \CC$.
    By distance-$\infty$ flip-flatness, there is a $k$-flip $G'$ of $G$ and a subset $W_\star \subseteq W$ of size at least $m$ whose vertices are pairwise in different components in $G'$.
    By \cref{fact:transduction-closure}, there is also an $\ell$-flip $H'$ of $\phi(G)$ in which the vertices of $W_\star$ are in pairwise different components.
    It follows that $H'[V(H)]$ is the desired $\ell$-flip of $H$.
\end{proof}

\begin{proof}[Proof of \cref{lem:paths-not-infty-flip-flat}]
    Assume towards a contradiction that the class of all paths is distance-$\infty$ flip-flat with bounds $N(\cdot)$ and $k$.
    Let $G$ be the path containing $N(8k+2)$ vertices.
    By flip-flatness there exists a $k$-coloring $\KK$ of $G$ and $F \subseteq \KK^2$ such that $G_\oplus := G \oplus_\KK F$ contains at least $8k + 2$ components.
    Let $\BB := \{ C \in \KK : |C| \geq 5\}$
    be the \emph{big} color classes and $W := \bigcup (\KK \setminus \BB)$ be the at most $4k$ vertices contained in \emph{small} color classes.
    Let $G'$ be the subgraph of $G$ obtained by isolating $W$ and let $F' := F \cap \BB^2$ be the restriction of $F$ to $\BB$.
    $G'_\oplus := G' \oplus_\KK F'$ is a subgraph of $G_\oplus$: every edge $uv$ in $G'_\oplus$ has no endpoint in $W$ and is therefore also present in~$G_\oplus$.
    Since $G'_\oplus$ and $G_\oplus$ share the same vertex set, $G'_\oplus$ has at least as many components as~$G_\oplus$.
    In order to arrive at the desired contradiction, it remains to bound the number of components in $G'_\oplus$.
    Towards this goal, we first bound the number of components in $G'$.
    As $G'$ is obtained from a path by isolating at most $4k$ vertices, $G'$ contains at most $8k + 1$ components: the path is cut in at most $4k$ places leading to $4k + 1$ components plus the additional at most $4k$ isolated vertices.
    \begin{claim}
        If two vertices $u$ and $v$ are adjacent in $G'$, then they are connected in $G'_\oplus$.
    \end{claim}
    \begin{claimproof}
        Since $u$ and $v$ are adjacent, they are from big color classes $\KK(u)$ and $\KK(v)$.
        Assume the adjacency between $\KK(u)$ and $\KK(v)$ was flipped, as otherwise we are done.
        As $G'$ has maximum degree two, we have $|N_1^{G'}[u] \cup N_1^{G'}[v]| \leq 4$.
        If $\KK(u) = \KK(v)$ then there exists a vertex in that class that is adjacent to none of $u$ and $v$ in $G'$ and therefore adjacent to both of them in~$G'_\oplus$ and we are done.
        Otherwise, there are three vertices $U \subseteq \KK(u)$ non-adjacent to $v$ in $G'$ and three vertices $V \subseteq \KK(v)$ non-adjacent to $u$. 
        Again using the fact that $G'$ has maximum degree two, we find $u' \in U$ and $v' \in V$ that are non-adjacent in $G'$. It follows that $(u,v',u',v)$ is a path in $G'_\oplus$.
    \end{claimproof} 
    It follows that $G'_\oplus$ (and also $G_\oplus$) contains at most $8k + 1$ components; a contradiction.
\end{proof}

\subsection{Distance-\texorpdfstring{$\infty$}{Infinity} Deletion-Flatness Characterizes Bounded Treedepth}

\begin{definition}
    A class of graphs $\CC$ is \emph{distance-$\infty$ deletion-flat},
    if there exists a function
    \(N : \N \to \N\) and a constant \(k \in \N\) such that for all \(m \in \N\), \(G \in \CC\) and \(W \subseteq V(G)\) with
    \(|W| \ge N(m)\) there exists
    a set $S\subseteq V(G)$ with $|S| \leq k$ and a subset $W_\star\subset W \setminus S$ with \(|W_\star| \ge m\) such that in 
    $G - S$, no two vertices $u,v\in W_\star$ are in the same connected component.
\end{definition}

In this subsection we relate distance-$\infty$ deletion-flatness
to the graph parameter \emph{treedepth}.
The single vertex graph $K_1$ has treedepth $1$.
For $k > 1$, a graph $G$ has treedepth at most $k$ if there exists a vertex
whose deletion splits $G$ into a disjoint union of (arbitrarily many) graphs of treedepth at most $k-1$.

\begin{theorem}\label{thm:treedepth}
    A class of graphs has bounded treedepth if and only if it is distance-$\infty$ deletion-flat.
\end{theorem}
\begin{proof}
    Essentially, the definitions of treedepth and distance-$\infty$ deletion-flatness are obtained from the definitions of flipdepth and distance-$\infty$ flip-flatness by replacing flips with vertex deletions.
    Proving that every class of bounded treedepth is distance-$\infty$ flip-flat is therefore analogous to the proof of \cref{lem:flipdepth}. 
    The other direction follows by combining the following easy facts.
    \begin{enumerate}
        \item If a class has bounded treedepth then so does its closure under taking subgraphs.
        \item Every class of unbounded treedepth contains all paths as subgraphs. \cite[Proposition 6.1]{nevsetvril2012sparsity}
        \item The class of all paths is not distance-$\infty$ flip-flat.
    \end{enumerate}
\end{proof}

\section{Binary Structures}\label{sec:binary}
In this section, we lift the some of the 
 characterizations of monadically dependent graph classes to the more general setting of binary structures, that is, structures equipped with a finite number of binary relations.
More precisely, with the appropriate notion of flips for binary structures defined further below, we prove the following.
\begin{theorem}\label{thm:flip-breakability-binary}
    Let $\Sigma$ be a finite signature consisting of binary relation symbols, and let $\cal C$ be a class of $\Sigma$-structures. Then the following conditions are equivalent:
    \begin{enumerate}
        \item $\cal C$ is monadically dependent,
        \item $\cal C$ is flip-breakable,
         \item for every $h\in\N$ there is some $n\in\N$ such that 
          no structure $G\in \cal C$  contains a (minimal) transformer of length $h$ and order $n$.
    \end{enumerate}
\end{theorem}

As an application, we demonstrate that in the setting of ordered graphs, this result allows us to deduce a key result of \cite{twwIV},
stating that monadically dependent classes of ordered graphs have bounded grid rank, and hence bounded twin-width. Moreover, on the non-structure side,
we derive that monadically independent classes of ordered graphs contain some specific patterns that were also exhibited in \cite{twwIV}.
Those patterns form the main ingredient for the forbidden subgraph characterization of monadic dependence for classes of ordered graphs given in \cite{twwIV}.

The proof of Theorem~\ref{thm:flip-breakability-binary} follows by using standard ideas.
We use the fact that for any monadically dependent class of binary structures we can find a monadically dependent graph class,
such that one can transduce back and forth between those classes via a straightforward transduction.
We moreover use the observation that flip-breakability and transformers are preserved under (quantifier-free) transductions.

\subsection{Notions}\label{sec:notions}
We first introduce the setting of binary structures, and define monadic dependence and flip-breakability for  classes of binary structures.

\paragraph{Structures.}
Fix a signature $\Sigma$ consisting of one or more relation symbols, where each symbol has a specified \emph{arity}, which is a positive integer. The signature $\Sigma$ is \emph{binary}
if every relation symbol $R\in\Sigma$ has arity exactly two.
A \emph{$\Sigma$-structure} $G$ consists of a domain, denoted $V(G)$, and an interpretation of each relation symbol  $R\in\Sigma$ of arity $k$ as a relation $R_G\subset V(G)^k$.
A graph $G$ is represented as a structure over the signature $\set{E}$
consisting of one binary relation symbol $E$, which is interpreted 
as the set of pairs $(u,v)\in V(G)^2$ such that $uv\in E(G)$.
In particular, each class of graphs can be viewed as a class of $\{ E \}$-structures.

\paragraph{Gaifman Graph and Distances.}
The \emph{Gaifman graph} of a $\Sigma$-structure $G$ 
is the graph  with vertex set $V(G)$ and
edges between vertices \(u \neq v\)
if and only if there is some symbol $R\in\Sigma$ and tuple $\bar u\in R_G$ such that both $u$ and $v$
occur in $\bar u$.
For two elements $u,v\in V(G)$, we write $\dist_{G}(u,v)$ to denote the distance between $u$ and $v$ in the Gaifman graph of $G$.

\paragraph{Colored Structures.}
If $G$ is a $\Sigma$-structure and $k\in\N$,
then a \emph{$k$-coloring} $G^+$  of $G$ 
is specified by a function mapping each element of $V(G)$ to exactly one color in $\set{1,\ldots,k}$.
Denoting by $\Sigma^{(k)}$ 
the signature $\Sigma$ together with the $k$ unary predicates $U_1,\ldots,U_k$,
we may view a $k$-coloring of $G$ as the $\Sigma^{(k)}$-structure $G^+$ 
 obtained from $G$ by expanding $G$ by unary relation symbols 
representing the respective color classes.

\paragraph{Transductions and Monadic Dependence.}
Fix two relational signatures $\Sigma$ and $\Gamma$ relation
symbols. An \emph{abstract transduction} $T\from\Sigma\to\Gamma$ is a binary relation between $\Sigma$-structures and $\Gamma$-structures.
For a $\Sigma$-structure $G$ we write $T(G)$ to denote the class of 
$\Gamma$-structures that are related to $G$ via $T$,
and say that $T$ \emph{transduces} $H$ from $G$ if 
$H\in T(G)$.
For a class $\CC$ of $\Sigma$-structures we write $T(\CC)$ 
to denote $\bigcup\setof{T(G)}{G\in\CC}$.

A (first-order) \emph{transduction} $T\from\Sigma\to\Gamma$ is specified 
by a number $k$ of colors, and a collection of first-order formulas $\phi_R(x_1,\ldots,x_r)$,
for each symbol $R\in \Gamma$ with arity $r$,
where each of the formulas $\phi_R$ is in the signature $\Sigma^{(k)}$. The transduction $T$ is defined as follows.
First, for a $\Sigma^{(k)}$-structure $G^+$,
let $T_0(G^+)$ denote the structure $H$ with domain $V(G^+)$,
  in which each relation $R\in \Gamma$ of arity $r$ is interpreted as 
$$R_H\quad :=\quad \setof{\ (v_1,\ldots,v_r)\in V(G)^r\ }{\ G^+\models\phi(v_1,\ldots,v_r)\ }.$$
For a $\Sigma$-structure $G$, $T(G)$
consists of all structures that can be obtained (up to isomorphism) in the following steps:
\begin{enumerate}
    \item color $G$ using $k$ colors,
    obtaining a $\Sigma^{(k)}$-structure $G^+$,
    \item construct the structure $H:=T_0(G^+)$ as defined above,
    \item take an induced substructure of $H$.
\end{enumerate}
This finishes the description of $T$.

By a \emph{transduction} we mean a first-order transduction.

\begin{definition}
    A class $\cal C$ of $\Sigma$-structures \emph{transduces} a class $\cal D$ of $\Gamma$-structures
    if there exists a first-order transduction $T\from \Sigma\to \Gamma$ 
    such that $\DD\subset T(\CC)$.        
\end{definition}

\begin{definition}
    A class $\cal C$ of $\Sigma$-structures is \emph{monadically dependent}
    if $\cal C$ does not transduce the class of all (finite) graphs.        
\end{definition}

\paragraph{Flips.}
We use the notion of flips for binary structures defined in \cite[Sec. 7]{flipwidth}.
For a $\Sigma$-structure $G$ and partition $\cal P$ of $V(G)$,
a \emph{$\cal P$-flip} $F$ of $G$ is specified by a $\Sigma$-structure 
$F$ with domain $\cal P$. Applying $F$ to $G$ results 
in the structure $G'$ with domain $V(G)$, in which each
symbol $R\in\Sigma$ is interpreted as the relation
$$R_{G'}~~:=~~ R_{G}~~\triangle \bigcup_{(P,Q)\in  R_F}P\times Q,$$
where $\triangle$ denotes the symmetric difference.

A \emph{$k$-flip} of $G$ is a structure $G'$ obtained as above, 
for some partition $\cal P$ of $V(G)$ into $k$ parts.

\paragraph{Flip-breakability.}
\begin{definition}Fix a signature $\Sigma$ consisting of binary relation symbols.
    A class $\cal C$ of $\Sigma$-structures is \emph{flip-breakable} if for every radius $r\in\N$ there exists an unbounded function $f_r\from \N\to \N$ and a constant $k_r\in \N$ such that for all $G\in\cal C$ and $W\subset V(G)$ there exist subsets $A,B\subset W$ with $|A|,|B|\ge f_r(|W|)$ and a $k_r$-flip $G'$ of $G$ such that $$\dist_{G'}(A,B)>r.$$
\end{definition}
In other words, 
for every $r\in\N$, $G\in\CC$ and $W\subset V(G)$ there are subsets $A,B\subset W$ with $|A|,|B|\ge U_{\CC,r}(|W|)$ and a $\const(\CC,r)$-flip $G'$ of $G$ such that $\dist_{G'}(A,B)>r.$

\paragraph{Transformers and Minimal Transformers.}
All the notions
from Section~\ref{sec:transformers} 
-- of \emph{meshes}, \emph{vertical/horizontal meshes},
\emph{conducting pairs}, \emph{(minimal) transformers} --
can be 
interpreted verbatim in the setting of binary structures.
Indeed, all those notions are defined in terms 
of the atomic types of pairs of vertices, $\atp_G(u,v)$,
in a (colored) graph $G$, which makes sense in any structure.

\medskip
We have defined all the notions involved in the statement of Theorem~\ref{thm:flip-breakability-binary}. We now introduce some tools that will be used in its proof.

\subsection{Flip Transfer Lemma}
We repeat a lemma from \cite{flipwidth}, which allows to ``transfer'' flips 
along a transduction, without decreasing distances too much.
We state the lemma in greater generality than \cite[Lem. H.3]{flipwidth},
namely for binary structures. However, the proof remains the same,
and is a straightforward consequence of locality of first-order logic.

\begin{lemma}[{\cite[Lem. H.3]{flipwidth}}]\label{lem:flip-transfer}
    Fix two finite signatures $\Gamma,\Sigma$ consisting of binary relation symbols and numbers $k,c,q\ge 1$.
    Let $T\from \Sigma\to\Gamma$ be a transduction involving $c$ colors
    and formulas of quantifier rank at most $q$.
    Let $G$ be a $\Sigma$-structure and let $H\in T(G)$.
    For every $k$-flip $G'$ of $G$ there is a $\const(k,c,q)$-flip $H'$ of 
    $H$ such that:
    \begin{align}\label{eq:flip-transfer}
        \dist_{G'}(u,v)&\le 2^q\cdot  \dist_{H'}(u,v)\qquad\text{for all $u,v\in V(G)$.}
    \end{align}
\end{lemma}

The next lemma is an easy consequence.
\begin{lemma}\label{lem:breakability-transfer}
    Fix two finite signatures $\Sigma,\Gamma$ consisting of binary relation symbols.
    Let $\CC$ be a class of $\Sigma$-structures,
    which transduces a class $\DD$ of $\Gamma$-structures.
    If $\CC$ is flip-breakable then $\DD$ is flip-breakable.
\end{lemma}
\begin{proof}
     Let $T\from \Sigma\to \Gamma$ be a transduction
     witnessing that $\CC$ transduces $\DD$, where $T$ involves $c$ colors.
     Let $q$ be maximal quantifier rank of the formulas involved in $T$.
     Without loss of generality, we may assume that 
     $\DD$ is the class of structures of the form $T_0(G^+)$,
     where $G\in\CC$ and  $G^+$ is a $c$-coloring of $G$, since flip-breakability is preserved by taking the hereditary closure of a class of structures.

     Fix a radius $r$. We prove that $\DD$ satisfies the 
     condition in flip-breakability for  radius $r$.
     By assumption, $\CC$ is flip-breakable for radius $r\cdot 2^q$.
     This means that there is a constant $k=\const(\CC,q,r)$
such that for every graph $G\in \CC$ and set $W\subset V(G)$
there is a $k$-flip $G'$ of $G$ and two sets $A,B\subset W$ of size $U_{\CC,q,r}(|W|)$ such that 
\begin{align}\label{eq:flip-distant}
    \dist_{G'}(A,B)>r\cdot 2^q.    
\end{align}

Let $H\in\DD$, so that $H=T_0(G^+)$ for some $G\in\CC$ and $c$-coloring $G^+$ of $G$. Let $W\subset V(H)=V(G)$.
Then there is a $k$-flip $G'$ of $G$ and two sets $A,B\subset W$ as described above.
Let $H'$ be the $\ell$-flip of $H$ as described in Lemma~\ref{lem:flip-transfer}, for some $\ell=\const(k,c,q)\le \const(\CC,q,r,T)\le\const(\DD,r)$.
Then combining \eqref{eq:flip-distant} with \eqref{eq:flip-transfer}
we conclude that 
\begin{align}\label{eq:flip-distant'}
    \dist_{H'}(A,B)>r.
\end{align}
This proves that the class $\DD$ is flip-breakable with radius $r$. 
As $r$ is arbitrary, the lemma follows.
\end{proof}

\subsection{Encoding Binary Structures in Graphs}
We now show how binary structures can be encoded in graphs, 
in a way which preserves monadic dependence. This encoding is folklore,
but we need to analyze some of its properties related to our context.

To this end, it is convenient to use \emph{transductions with copying}, defined below.
First, the \emph{copying operation}, parameterized by a signature $\Sigma$ and number $\ell$ of copies and denoted $C_\ell$, is an abstract transduction 
from $\Sigma$-structures to $\Sigma\cup\set{M}$-structures,
where $M$ is a binary relation symbol not in $\Sigma$,
and is defined as follows.
Given an input $\Sigma$-structure $G$,
define the $\Sigma\cup\set{M}$-structure $C_\ell(G)$,
with vertex set $V(G)\times [\ell]$,
such that 
\begin{align*}
    C_\ell(G)_M&=\setof{((v,i),(v,j))}{v\in V(G), i,j\in[\ell]},\\
    C_\ell(G)_R&=\setof{((v_1,i),(v_2,i),\ldots,(v_r,i))}{(v_1,\ldots,v_r)\in R_G, i\in[\ell]}\qquad\textit{for $R\in\Sigma$ of arity $r$}.
\end{align*}
A (first-order) \emph{transduction with copying} $T\from\Sigma\to\Gamma$
is a composition of a copying operation
$C_\ell\from \Sigma\to (\Sigma\cup\set{M})$, for some $\ell\in\N$, and a transduction $S\from(\Sigma\cup\set{M})\to \Gamma$.
More precisely, given a $\Sigma$-structure $G$,
we set $T(G):=S(C_\ell(G))$.
We say that $T$ is \emph{quantifier-free}
if all the formulas defining the underlying transduction $S$ 
are quantifier-free formulas.

Transductions with copying are closed under composition, as expressed by the following  
folklore result 
(see, e.g., \cite[Lem. 2.5]{gajarsky2020first}).
\begin{fact}\label{fact:copying-composition}
Let $T\from\Sigma\to\Gamma$ and $S\from \Gamma\to\Delta$ be two transductions with copying. Then the composition $S\circ T\from \Sigma\to\Delta$ is (definable by) a transduction with copying.
\end{fact}

Allowing transductions with copying instead of usual transductions in  the definition of monadic dependence would not affect the notion. This is also folklore, but we sketch a proof below for completeness.
\begin{fact}\label{fact:copying}
Let $\CC$ be a class of $\Sigma$-structures such that 
for some transduction with copying $T$,
the class $T(\CC)$ contains every graph.
Then $\CC$ transduces the class of all graphs
via a transduction without copying.
\end{fact}
\begin{proof}[Proof sketch]
    This can be deduced from the result of Baldwin and Shelah~\cite[Thm. 8.1.8]{baldwin1985second}.
We sketch a more direct argument. 
    Suppose that $\CC$ transduces the class of all graphs using a transduction $T$ with copying. Then \(T\) is of the form $T=S\circ C_\ell$ for some transduction $S$ 
    and number $\ell$, where $C_\ell$ is the $\ell$-times copying operation.
    Let $K^{(1)}_n$ denote the $1$-subdivision of the $n$-clique.
    In particular, 
    \begin{align*}\label{eq:subdivided-cliques}
        \setof{K^{(1)}_n}{n\in\N}\subset T(\CC)    
    \end{align*}
    
    Fix $n\in\N$.
    Then there is a structure $G\in\CC$ 
    such that $K^{(1)}_n\in T(G)$.
We prove that there is a fixed collection  $(T_{p,q}:p,q\in [\ell])$ of $\ell^2$ many transductions (depending only on $T$) such that $K^{(1)}_m\in T_{p,q}(G)$ for some $m\ge U_T(n)$ and $p,q\in [\ell]$.
By the pigeonhole principle, since $n$ is arbitrarily large, and since \(\mathcal{C}\) is hereditary, this proves that 
there is a fixed transduction $T_{p,q}$ 
such that $\setof{K^{(1)}_m}{m\in \N}\subset T_{p,q}(\CC)$.
Therefore, $\CC$ transduces the class of $1$-subdivided cliques, which in turn transduces the class of all graphs.
The conclusion follows by composing the two transductions
without copying, obtaining a single transduction without copying.

\medskip
Denote the principal vertices 
of $K^{(1)}_n$ by $(a_i:i\in[n])$,
and the subdivision vertices by $(b_{i,j}:i,j\in[n], i< j)$.
     As \(T\) creates \(\ell\) copies, we can assume that  $V(K^{(1)}_n)\subset V(G)\times[\ell]$.    
     Let $\pi_1\from V(G)\times[\ell]\to V(G)$
     and $\pi_2\from V(G)\times[\ell]\to [\ell]$
     be the projection mappings selecting the first and second entry of a pair, respectively.

     Color each edge $\set{i,j}$, where $i<j\in [n]$, of the clique $K_n$
     with the color $(\pi_2(a_i),\pi_2(b_{ij}))$.
    By Ramsey's theorem,
    there is a subset $I\subset [n]$ of size $U_\ell(n)$ 
    such that all edges have the same color.
    It follows that there are $p,q\in [\ell]$
    such that $\pi_2(a_i)=p$ for all $i\in I$ 
    and $\pi_2(b_{i,j})=q$ for all $i<j\in I$.

    Let $K^{(1)}_I$ be the subgraph of $K^{(1)}_n$
induced by the vertices $A=\setof{a_{i}}{i\in I}$ and $B=\setof{b_{i,j}}{i<j\in~I}$. 
Then $K^{(1)}_I$ is (isomorphic to) the $1$-subdivision  of the clique on $|I|\ge U_\ell(n)$ vertices.
By another Ramsey argument,
    it is not difficult to show that
    by replacing $I$ by a subset of size $U(|I|)$ 
    (which we denote $I$ as well),
    we may ensure that the mapping $\pi_1$, restricted to 
    $V(K^{(1)}_I)\subset V(G)\times[\ell]$,
    is injective
    (this is similar to the argument in Lemma~\ref{lem:identical-meshes}). 

    For $i\in[\ell]$ 
let $c_i\from V(G)\to V(G)\times[\ell]$
denote the function that maps each vertex $v\in V(G)$
to the copy $(v,i)$ of $v$ in $V(G)\times[\ell]$.
Let $\phi(x,y)$ be the formula underlying the transduction $S$, and let $k$ denote the number of colors that are involved in this formula.
Define sets $A',B'\subset V(G)$ as
\begin{align}
    A':=\pi_1(A)\qquad B':=\pi_1(B).
\end{align}
It follows from the above that $|A'|=|I|, |B'|= {|I|\choose 2} $, $A'\cap B'=\emptyset$.
Furthermore, there is some 
$k$-coloring $\wh H$ of $H:=C_\ell(G)$ 
such that 
\begin{align}\label{eq:1-subd-clique}
    (A'\cup B',\, \setof{ab}{a\in A',b\in B', \wh H\models \phi(c_p(a),c_q(b))})\quad\cong\quad K^{(1)}_I.
\end{align}
% is isomorphic to $K^{(1)}_I$.

The $k$-coloring $H^+$ of $H$ 
induces a $k^\ell$-coloring $G^+$ of $G$,
where each vertex $v\in V(G)$ is colored
by the $\ell$-tuple of colors of the $\ell$ copies of the vertex $v$ in $\wh H$.

\begin{claim}\label{cl:copylessing}
There is a  formula $\phi_{p,q}(x,y)$, depending only on $p,q$, and $\phi$, such that 
the following holds for all $a,b\in V(G)$:
$$\wh G\models\phi_{p,q}(a,b)\quad\iff\quad \wh H \models \phi(c_p(a),c_q(b)).$$
\end{claim}
\begin{claimproof}[Proof sketch]
We show more generally that for every formula $\phi(x_1,\ldots,x_t)$  and for each $f\from[t]\to [\ell]$,
there is a formula $\phi_f(x_1,\ldots,x_t)$, depending only on $\phi$ and $f$,
such that the following holds for all $a_1,\ldots,a_t\in V(G)$:
$$\wh G\models \phi_f(a_1,\ldots,a_t)\quad\iff\quad \wh H\models \phi(c_{f(1)}(a_1),\ldots,c_{f(t)}(a_t)).$$
This is proved by induction on the size of $\phi$. In the base case $\phi$ is an atomic formula and the formula $\phi_f$ is constructed directly. In the inductive step, we only need to consider formulas $\phi$ of the form $\phi_1\lor\phi_2$, $\neg\phi'$ or $\exists x.\phi'$, where $\phi_1,\phi_2,\phi'$ are formulas smaller than $\phi$.
The first two cases, of boolean combinations, follow immediately from the inductive hypothesis.
In the case of a formula $\phi(x_1,\ldots,x_t)=\exists x_{t+1}.\phi'(x_1,\ldots,x_{t+1})$ define $$\phi_{f}(x_1,\ldots,x_t):=\exists x_{t+1}.\bigvee_{f'\supset f}\phi'_{f'}(x_1,\ldots,x_{t+1}),$$
where the disjunction ranges over all functions $f'\from[t+1]\to[\ell]$ extending $f\from[t]\to[\ell]$,
and the formula $\phi'_{f'}$ is obtained by inductive assumption.
It is straightforward to verify that the formula $\phi_f$ has the required property. This completes the induction.
The claim then follows by setting $\phi_{p,q}(x,y):=\phi_f(x,y)$, where $f(1)=p$ and ${f(2)=q}$.
\end{claimproof}

From \cref{cl:copylessing} and 
\eqref{eq:1-subd-clique}
it follows that 
$$(A'\cup B',\, \setof{ab}{a\in A',b\in B', \wh G\models \phi_{p,q}(a,b)})\quad\cong\quad K^{(1)}_I.$$
Moreover, the formula $\phi_{p,q}$ only depends on $\phi$ and the indices $p,q\in[\ell]$.

Let $T_{p,q}$ denote the transduction without copying which first colors the given input structure using $k^\ell$ colors,
then applies the formula $\phi_{p,q}(x,y)$,
defining the edges of a graph, and finally, takes an arbitrary induced subgraph.
The argument above proves that for every $n\in\N$ there are $p,q\in [\ell]$ and $m\ge U_T(n)$ such that 
\[K^{(1)}_m\in T_{p,q}(G).\qedhere\]

\end{proof}

\begin{corollary}\label{cor:mon-dep-copying}
    If $\CC$ is a monadically dependent class of $\Sigma$-structures 
    and $T\from\Sigma\to\Gamma$ is a transduction with copying then 
    $T(\CC)$ is monadically dependent.
\end{corollary}
\begin{proof}
    Assume towards a contradiction that $T(\CC)$ transduces the class of all graphs,
    that is, there is a transduction $S\from \Gamma\to\set{E}$
    such that $(S\circ T)(\CC)$ contains the class of all graphs.
As $S\circ T$ is a transduction with copying by Fact~\ref{fact:copying-composition},
by Fact~\ref{fact:copying} it follows that $\CC$ transduces the class of all graphs, that is, $\CC$ is monadically independent; a contradiction.
\end{proof}

The following lemma allows us to encode binary structures in graphs.

\begin{lemma}\label{lem:encoding-in-graphs}
    Fix a finite signature $\Sigma$ consisting of binary relational symbols, and let $\Gamma=\set{E}$ denote the signature of graphs.
    There are
    \begin{itemize}
        \item a quantifier-free transduction with copying $F\from \Sigma\to \Gamma$,
        \item a transduction (without copying) $T\from \Gamma\to \Sigma$,
    \end{itemize}
 such that 
    \begin{align}\label{eq:recover}
        G\in T(F(G))\qquad\text{for every $\Sigma$-structure $G$.}
    \end{align}
\end{lemma}

\begin{proof}[Proof sketch]    
Fix a $\Sigma$-structure $G$, and let $V$ denote its vertex set.
We define the graph $F(G)$ as follows.
The vertex set is $$V(F(G)):=V_0\uplus V_1\uplus \biguplus_{R\in\Sigma}V_R,$$
where $V_0,V_1$, and $V_R$, for $R\in\Sigma$, are disjoint copies of 
$V$.
That is, $V(F(G))$ consists of $|\Sigma|+2$ copies of $V(G)$.
Write $c_i(v)\in V_i$ for the copy of vertex $v\in V$ in $V_i$,
for $i\in\set{0,1}\cup\Sigma$.
The edges of $F(G)$ are
\begin{itemize}
    \item $\{c_0(u),c_R(v)\}$, for every $R\in\Sigma$ and $(u,v)\in R_G$,
    \item $\{c_0(v),c_1(v)\}$, for every $v\in V(G)$, and
    \item $\{c_1(v),c_R(v)\}$, for every $v\in V(G)$ and $R\in\Sigma$.
\end{itemize}

We skip the straightforward argument that the operation $F$ can be implemented as a transduction with copying.
We now argue that there is a fixed transduction $T$ 
which recovers $G$ from $F(G)$. 

By identifying $V(G)$ with $V_0$ by via the bijection $c_0$,
we observe that
for every pair $u,v\in V(G)=V_0$  and relation symbol $R\in\Sigma$, 
the following conditions are equivalent:
\begin{itemize}
    \item $(u,v)\in R_G$,
    \item there are vertices $v'\in V_R$ and $v''\in V_1$ 
    such that $\{u,v'\},\{v,v''\},\{v',v''\} \in E(F(G))$.
\end{itemize}
The transduction $T$ colors each of the $|\Sigma|+2$ copies of $V$ comprising $V(F(G))$ by a separate color, using $|\Sigma|+2$ many colors jointly, represented by unary predicates.
For each relation \(R \in \Sigma\), the transduction contains a formula $\phi_R(x,y)$ in the signature of $|\Sigma|+2$-colored graphs.
The formula $\phi_R(x,y)$ recovers the relation $R$ 
by verifying the second condition described above,
which is clearly expressible by a first-order formula involving the unary predicates representing the copies and a binary relation representing adjacency in $F(G)$.

It follows from the equivalence of the two conditions above that $G\in T(F(G))$. As the transduction $T$ does not depend on the $\Sigma$-structure $G$,
the statement follows.
\end{proof}

The following lemma can be seen as a technical amendment to the statement of Lemma~\ref{lem:encoding-in-graphs}.

\begin{lemma}\label{lem:copy-transformer}
    Let $F$ be the transduction constructed in Lemma~\ref{lem:encoding-in-graphs}.
    For every $\Sigma$-structure $G$, if $G$ contains a transformer of length 
    $h$ and order $n$, then there is a graph $H \in F(G)$ that contains a transformer of length at most $3h$ and order $n$.
\end{lemma}
\begin{proof}[Proof sketch]
    We use the transformer notation introduced in \Cref{sec:meshesandtransformers}.
    Fix a $\Sigma$-structure $G$ and let $H\in F(G)$ be the graph denoted as $F(G)$ in the proof of \cref{lem:encoding-in-graphs}.
    Let $c_i\from V(G)\to V(H)$ be the functions that maps each vertex $v\in V(G)$
    to the corresponding copy of $v$ in $V(H)$, for $i\in\set{0,1}\cup\Sigma$.
    
    Since each $c_i$ is injective, it follows that 
for every mesh $\mesh\from I\times J\to V(G)$ in $G$
    and  each $i\in\set{0,1}\cup\Sigma$,
    the composition $c_i\circ \mesh$ is a mesh in $H$.
    Furthermore, it is easy to verify that if $\mesh,\mesh'$ is a regular pair of meshes 
    in $G$, then the pair $c_i\circ \mesh,c_j\circ \mesh'$ is a regular pair of meshes in $H$, for each $i,j\in\set{0,1}\cup\Sigma$.
    
    \medskip
    We now argue that if $\mesh,\mesh'\from I\times J\to V(G)$ is a conducting pair 
    of meshes in $G$, then there is $R\in\Sigma$ such that each pair of consecutive meshes in the sequence $$c_0\circ \mesh,c_R\circ \mesh',c_1\circ\mesh',c_0\circ \mesh'$$  is conducting.

    Namely, as the pair $\mesh,\mesh'$ is conducting, it is regular and not homogeneous. It follows that there is some binary relation symbol $R\in\Sigma$ such that $R_G(\mesh(i,j),\mesh'(i,j))$ is not the same for all $i\in I$ and $j\in J$.
    By definition of the edge relation in $H=F(G)$,
    it follows that 
    $E(c_0(\mesh(i,j)),c_R(\mesh'(i,j)))$ is not the same for all $i\in I$ and $j\in J$.
It follows that the pair $c_0\circ \mesh,c_R\circ \mesh'$ is not homogeneous in $H$. As it is regular (see above), it is conducting.
Furthermore, by construction of $F(G)$, the pairs 
$(c_R\circ \mesh',c_1\circ\mesh')$ and $(c_1\circ \mesh',c_0\circ\mesh')$ are conducting, as $E(H)$ forms a matching between the corresponding copies $V_R$ and $V_1$ of $V$, and between the copies $V_1$ and $V_0$ of $V$.

\medskip
By a similar argument as above, it is easy to verify that 
if $\mesh\from I\times J\to V(G)$ is vertical (resp. horizontal) in $G$,
then $c_0\circ\mesh\from I\times J\to V(G)$ is vertical (resp. horizontal) in $H$.

\medskip
Finally, suppose that $\mesh_1,\ldots,\mesh_h$ is a transformer in $G$.
By the arguments above,
there are $R_1,\ldots,R_{h-1}\in \Sigma$ such that the following sequence of meshes is a transformer in $H$:       
\[c_0\circ \mesh_1, c_{R_1}\circ\mesh_2,c_1\circ\mesh_2,c_0\circ\mesh_2,
                     c_{R_2}\circ\mesh_3,c_1\circ\mesh_3,c_0\circ\mesh_3,
                     \ldots,
                     c_{R_{h-1}}\circ\mesh_h,c_1\circ\mesh_h,c_0\circ\mesh_h
\]          
\end{proof}

\begin{lemma}\label{lem:transformers-through-qftransduction}
    Let $F\from\Sigma\to\Gamma$  be a quantifier-free transduction with copying, where $\Sigma$ and \(\Gamma\) consist of binary relational symbols.
    Let $G$ be a $\Sigma$-structure and $H\in F(G)$.
    If $H$ contains a transformer of length $h$ and order $n$ then $G$ contains a transformer of length $h$ and order $U_{h,F}(n)$.
\end{lemma}
\begin{proof}[Proof sketch]
Suppose $F$ is a composition of the $\ell$-copying operation with a quantifier-free transduction $S$ without copying that uses $k$ colors.
Let $H\in F(G)$. By definition, there is a $k$-coloring $G_\ell^+$ of $C_\ell(G)$ such that $H$ is an induced subgraph of $S_0(G_\ell^+)$. In particular $V(H)\subset V(G)\times[\ell]$.

Consider the mapping $\pi\from V(H)\to V(G)$ that maps each copy of a vertex $v$ to $v$ itself.
This mapping is injective when restricted to a single copy \(V(G)\times \{i\}\), for some \(i \in [\ell]\).
Every mesh $\mesh\from I\times J\to V(H)$ in $H$ 
induces a function $\pi\circ\mesh\from I\times J\to V(G)$.
Note that $\pi\circ\mesh$ might not be injective, thus it is not necessarily a mesh in $G$.
However, if $\mesh$ is a \emph{monocopy mesh} in $H$, that is, $V(\mesh)\subset V(G)\times\set{i}$ for some $i\in[\ell]$,
then $\pi\circ\mesh$ is a mesh in $G$.

We call a mesh $\mesh$ in $H$ \emph{monochromatic}, if all of its vertices have the same color in $G^+_\ell$.
Using the fact that $S$ is quantifier-free,
it is straightforward to verify that if a monochromatic monocopy mesh $\mesh$ is vertical (resp.\ horizontal) in \(H\),
then $\pi\circ\mesh$ is vertical (resp.\ horizontal) in $G$,
and that if $(\mesh,\mesh')$ is a conducting pair of monochromatic monocopy meshes in $H$,
then $(\pi\circ\mesh,\pi\circ\mesh')$ is a conducting pair of meshes in $G$.

Note that if $H$ contains a transformer $T$ of length $h$ and order $n$, then it also contains a transformer $T'=(\mesh_1,\ldots,\mesh_h)$ of length $h$ and order $U_{h,\ell,k}(n)$ in which each mesh $\mesh_i$ is a monochromatic monocopy mesh,
by a Ramsey argument similar to Lemma~\ref{lem:regularize}.
It follows from the above that the sequence \(\pi\circ \mesh_1,\ldots,\pi\circ\mesh_h\) is a transformer in $G$
of length $h$ and order $U_{h,\ell,k}(n)\ge U_{h,F}(n)$.
\end{proof}

\subsection{Proof of Theorem~\ref{thm:flip-breakability-binary}}
We now prove Theorem~\ref{thm:flip-breakability-binary}.
\begin{proof}[Proof of Theorem~\ref{thm:flip-breakability-binary}]
    Fix a finite signature $\Sigma$ consisting of binary relational symbols, and let $\CC$ be a class of $\Sigma$-structures.

    Let $F\from \Sigma\to\set{E}$ and $T\from\set{E}\to\Sigma$ be the transductions from Lemma~\ref{lem:encoding-in-graphs}. 
    Denote $\DD:=F(\CC)$. Note that $\DD$ is a class of graphs and $\CC\subset T(\DD)$ by \eqref{eq:recover}.
    Then $\CC$ is monadically dependent if and only if 
    $\DD$ is monadically dependent, by Corollary~\ref{cor:mon-dep-copying}.

    We now show that $\CC$ is monadically dependent if and only if it is flip-breakable.
    Suppose first that $\CC$ is monadically dependent.
    Hence, $\DD$ is also monadically dependent, and therefore is flip-breakable 
    by  Theorem~\ref{thm:main-circle}.
    Since $\CC\subset T(\DD)$, it follows that 
    $\CC$ is flip-breakable by 
     Lemma~\ref{lem:breakability-transfer}.
    Conversely, assume now that $\CC$ is flip-breakable.
    By Lemma~\ref{lem:breakability-transfer}, every graph class transducible from $\CC$ is flip-breakable.
    As the class of all graphs is not flip-breakable,
    it follows that $\CC$ is monadically dependent.

    The fact that the class of all graphs is not flip-breakable
    follows immediately from Theorem~\ref{thm:main-circle},
    as clearly the class of all graphs is  monadically independent.
    One could also argue more directly, by showing that already the class 
    of $1$-subdivided cliques is not flip-breakable with radius~$2$,
    by a combinatorial argument similar as in the proof Lemma~\ref{lem:fb-shatter}.

    \medskip

    We now show that if $\CC$ is  monadically  independent then there is some $h\in\N$ such that $\CC$ contains transformers of length $h$ and arbitrarily large order.
    Since $\CC$ is  monadically independent, neither is $\DD$.
    Therefore, by Theorem~\ref{thm:main-circle},
    $\DD$ is not prepattern-free. By Proposition~\ref{prop:transformer}, there is some $h\in\N$ such that for all $n\in\N$, some graph $H_n\in \DD$ contains a  transformer of length $h$ and  order $n$.
    As $H_n\in\DD$ and $\DD=F(\CC)$, for every $n\in\N$ there 
    is some $G_n\in\CC$ such that $H_n \in F(G_n)$.
By Lemma~\ref{lem:transformers-through-qftransduction},
$G_n\in \CC$ contains a transformer of length $h$ and order $U_{F,h}(n)$.
It follows that $\CC$ contains transformers of length $h$ and arbitrarily large order, as required. Moreover, the transformer can be converted to a minimal transformer of order $U_{F,h}(n)$, by Lemma~\ref{lem:to-minimal}.

\medskip

 Finally, suppose that $\CC$ contains transformers of length $h$ and arbitrarily large order. 
 By Lemma~\ref{lem:copy-transformer}, $\DD$ contains transformers of length at most $3h$ and arbitrarily large order.
 Therefore, $\DD$ is monadically independent, and thus neither is $\CC$.
 
 One way to argue that $\DD$ is  monadically independent is by showing directly that $\DD$ transduces the class of all graphs, using the transformers. Another way is 
 to observe that $\DD$ contains one of the families of patterns,
 as described in Proposition~\ref{prop:patterns-main},
 by the same proof as in that proposition. Hence, $\DD$ is  monadically independent by Theorem~\ref{thm:main-circle}.

 \medskip
 This concludes the proof of Theorem~\ref{thm:flip-breakability-binary}.
\end{proof}

\subsection{Two-Set Variant of Flip-Breakability}
The following is a two-set variant of flip-breakability.

\begin{theorem}\label{thm:flip-breakable2-binary}
    Let $\CC$ be a monadically dependent class of binary structures.
    For every $r\ge 1$, structure $G\in\CC$ and sets $A,B\subset V(G)$ with $|A|=|B|=m$,
    there are subsets $A'\subset A$ and $B'\subset B$ of size $U_{r,\CC}(m)$,
    and a $\const(r,\CC)$-flip $G'$ of $G$ such that 
    $$\dist_{G'}(A',B')>r.$$
\end{theorem}
\begin{proof}
    By Theorem~\ref{thm:flip-breakability-binary}, the class $\CC$ is flip-breakable.
Applying flip-breakability to the set $A$ and radius $2r$,
we get two subsets $A_0,A_1\subset A$ 
of size $U_{r,\CC}(|A|)$
and a $\const(r,\CC)$-flip $G'$ of $G$ such that 
$$\dist_{G'}(A_0,A_1)>2r.$$
By the triangle inequality,  every $b\in B$ 
is at distance  greater than $r$ from either $A_0$ or $A_1$ in (the Gaifman graph of) $G'$.
By the pigeonhole principle, there is  a subset $B'$ of $B$ with $|B'|\ge |B|/2$ 
and $A'\in\set{A_0,A_1}$,
such that all elements of $B'$ are at distance larger than $r$ from $A'$ in $G'$. This yields the conclusion.
\end{proof}

\subsection{Ordered Graphs and Twin-Width}
We now derive two consequences of Theorem~\ref{thm:flip-breakability-binary} 
in the context of ordered graphs.
An \emph{ordered graph} is a structure $G=(V,E,<)$, where $(V,E)$ is a graph and $<$ is a total order on $V$. This is naturally viewed as a binary structure over the signature consisting of two binary relation symbols $E$ and $<$.

We reprove the following result of \cite{twwIV}.
For two sets $A,B$ of vertices of a (ordered) graph $G$,
let $r_G(A,B)$ denote the maximum
of the cardinalities of the following two sets
$$\setof{N(a)\cap B}{a\in A},\quad \setof{N(b)\cap A}{b\in B}.$$
In other words, $r_G(A,B)$ is the maximum of the number of distinct rows, and the number of distinct columns, of the $A\times B$ adjacency matrix  between the sets $A$ and $B$.
A \emph{convex partition} of an ordered graph $G=(V,E,<)$ is a partition $\cal P$ of $V$ such that every part of $\cal P$ is an interval with respect to~$<$.

A class $\CC$ of ordered graphs has \emph{bounded grid rank}
if there are constants $\ell,m$ such that for all ${G=(V,E,<)\in\CC}$ 
and convex partitions $\cal L,\cal R$ of size $\ell$, there 
are two intervals $A\in\cal L$ and $B\in \cal R$ such that 
$r_G(A,B)\le m$. We remark that this definition of bounded grid rank is equivalent to the definition we give in the introduction.

\begin{theorem}[{\cite[Thm. 1, (\emph{iv})$\rightarrow$(\emph{ii})]{twwIV}}]\label{thm:grid-rank}
    Every monadically dependent class of ordered graphs has bounded grid rank.
\end{theorem}

In \cite{twwIV} it is furthermore proved that for  classes of ordered graphs, bounded grid rank implies bounded twin-width. 
This is shown using a separate argument involving the Marcus-Tardos theorem.
 Theorem~\ref{thm:grid-rank}
can be deduced from Theorem~\ref{thm:flip-breakability-binary}
(actually, the two-set variant, Theorem~\ref{thm:flip-breakable2-binary}),
as we show further below.

We also show how another of the central results of \cite{twwIV} can be deduced from Theorem~\ref{thm:flip-breakability-binary}.
Whereas Theorem~\ref{thm:grid-rank} above describes the structure that is found in monadically dependent classes of ordered graphs,
the statement below is on the non-structure side, and describes 
the patterns
that can be found in monadically independent classes of ordered graphs.
The following is a reformulation of \cite[Thm. 5, (\emph{iii})$\rightarrow$(\emph{v})]{twwIV}.

\begin{theorem}\label{thm:tww-patterns}
    Let $\CC$ be a class of ordered graphs.
    If $\CC$ is monadically independent then
    for every $n$ and bijection $\sigma\from [n]\to [n]$,
    there is an ordered graph $G=(V,E,<)\in\CC$,  increasing vertex sequences  $a_1<\cdots<a_n$ and $b_1<\cdots<b_n$ in $V$,
    and a symbol $\alpha\in\set{=,\neq,\le,\ge}$
    such that
    \[
        \{a_i,b_j\}\in E(G)\iff j\, \alpha \, \sigma(i) \qquad\text{for all \(i,j \in [n]\).}
    \]
\end{theorem}

The following two subsections are dedicated to deriving these two theorems as consequences of \Cref{thm:flip-breakability-binary}.

\subsection{Proof of Theorem~\ref{thm:grid-rank}}
We use a simple lemma from \cite[Lem. G.2]{flipwidth},
which allows us to analyze flips of a totally ordered set $(V,<)$, viewed as a binary structure.

\begin{lemma}[\cite{flipwidth}]\label{lem:order-flips}
Let $L=(V,<)$ be a totally ordered set (viewed as binary structure) and $L'$ be a $k$-flip of $L$. Then there is a set $S$ with $|S|\le k$ such that any two vertices of $V-S$ with no vertex of $S$ between them in the order are at distance at most $2$ in the Gaifman graph of $L'$.
\end{lemma}

\begin{proof}[Proof of Theorem~\ref{thm:grid-rank}]

    Let \(\mathcal{C}\) be a monadically dependent class of ordered graphs.
    Let $G=(V,E,<)\in \CC$, and let $\cal L,\cal R$ be two convex partitions of $G$ with $|\cal L|=|\cal R|=\ell$ for some constant $\ell=\const(\CC)$ that will be specified later.
    Let $L=\setof{\min(A)}{A\in \cal L}$ and $R=\setof{\min(B)}{B\in \cal R}$ denote the sets of minimal elements of the intervals in $\cal L$ and in $\cal R$, respectively. Then $|L|=|R|=\ell$.

    Apply Theorem~\ref{thm:flip-breakable2-binary}
    with radius $r=5$,
    yields sets $L'\subset L$ and $R'\subset R$ 
    of size $U_{\CC}(\ell)$,
    and a $k$-flip $G'$ of $G$, for some $k\le \const(\CC)$,
    such that 
    \begin{align}\label{eq:dist-5}
        \dist_{G'}(L',R')>5.    
    \end{align}

    Denote $G'=(V,E',<')$.
    Note that $(V,<')$ is a $k$-flip of $(V,<)$.
    Let $S$ be the set given by Lemma~\ref{lem:order-flips}.
    Then $|S|\le k \le \const(\CC)$.
    There is a fixed constant $\ell$ depending only on $\CC$ such that 
    $|\cal L|=|\cal R|= \ell$ implies
    $|L'|=|R'|>|S|$.
    This is because 
    $$|L'|=|R'|\ge U_{\CC}(\ell)> |S|$$
    for sufficiently large $\ell$, as $|S|\le \const(\CC)$.
    Fixing this sufficiently large constant $\ell=\const(\CC)$, we henceforth assume that $|L'|=|R'|>|S|$.
    
    It follows that there is some interval $A \in \cal L$ with \(\min(A) \in L'\) and \(A \cap S = \emptyset\).
    By Lemma~\ref{lem:order-flips},
    the elements of $L$ are pairwise at distance at most $2$ in the Gaifman graph of $G'$.
    Similarly, there is an interval $B \in \cal R$ 
    with \(\min(B) \in R'\) so that the elements of $B$ are pairwise at distance at most $2$ in the Gaifman graph of $G'$.
    Suppose some $u\in A$ and $v\in B$  are adjacent in the Gaifman graph of $G'$. 
    It follows that $\min(A)$ and $\min(B)$ are at distance at most $5$ in the Gaifman graph of $G'$.
    As $\min(A)\in L'$ and $\min(B)\in R'$, this contradicts \eqref{eq:dist-5}.

    Therefore, the sets $A \in \cal L$ and $B \in \cal R$
    span no edges in $G'$.
    As $(V,E')$ is a $k$-flip of $(V,E)$,
    it follows that $r_{G}(A,B)\le \const(k)\le \const(\CC)$.
\end{proof}

\subsection{Proof of Theorem~\ref{thm:tww-patterns}}
We now prove Theorem~\ref{thm:tww-patterns}, by analyzing the minimal transformers in ordered graphs.
We use the following lemma.
If $I,J$ are two indexing sequences
then the lexicographic order $<_{IJ}$ on $I\times J$ is defined as usual.
Denoting by $I^R$ and $J^R$ the reverse sequences of $I$ and $J$, then we have in total 
$8$ lexicographic orders on $I\times J$, 
namely $<_{KL}$ where $K\in \set{I,I^R}$ and $L\in\set{J,J^R}$, or 
$K\in \set{J,J^R}$ and $L\in\set{I,I^R}$.

\begin{lemma}\label{lem:meshes-in-orders}
    Let $I,J$ be two indexing sequences of length $|I|=|J|\ge 3$,
    and let $(V,<)$ be a totally ordered set.
    Let $\mesh\from I\times J\to V$ be a regular mesh in $(V,<)$.
    Then for one of the eight lexicographic orders $<_{\text{lex}}$ described above, 
    we have that $$\mesh(i,j)<\mesh(i',j')\quad\iff \quad (i,j)<_{\text{lex}} (i',j')\qquad\text{for all $(i,j),(i',j')\in I\times J$}.$$
    Furthermore, 
    for $I'= I-\set{\min(I),\max(I)}$, $J'= J-\set{\min(J),\max(J)}$ with $|I'|=|J'| > 3$,
    we have that $\mesh_{I'\times J'}$
     is either horizontal and not vertical,
    or is vertical and not horizontal in $(V,<)$.

\end{lemma}
\begin{proof}
    Let $i_{-1}<i_0<i_1\in I$ and $j_{-1}<j_0<j_1\in J$.
    As $<$ is a total order,
    we have that $\mesh(i_0,j_0)<\mesh(i_1,j_0)$
    or  $\mesh(i_0,j_0)>\mesh(i_1,j_0)$.
    Replacing $I$ with $I^R$ if necessary, we may assume that 
    \begin{align}\label{eq:lex1}
        \mesh(i_0,j_0)<\mesh(i_1,j_0).
    \end{align}
    Similarly, replacing $J$ with $J^R$ if necessary, we may assume
    \begin{align}\label{eq:lex2}
        \mesh(i_0,j_0)<\mesh(i_0,j_1).
    \end{align}
    By regularity of $\mesh$ we have that 
     $$\mesh(i_{0},j_{0})<\mesh(i_1,j_{-1})\quad\iff \quad  \mesh(i_{-1},j_{1})<\mesh(i_0,j_0).$$ 
     Thus, exactly one of the following holds:
     $$\mesh(i_{0},j_{0})<\mesh(i_1,j_{-1})\quad\text{or} \quad  \mesh(i_0,j_0)<\mesh(i_{-1},j_{1}).$$ 
    Exchanging the roles of $I$ and $J$ if necessary, we may assume
    \begin{align}\label{eq:lex3}
        \mesh(i_{0},j_{0})<\mesh(i_1,j_{-1}).
    \end{align}
    Now, by regularity of $\mesh$, we derive that
    \begin{align*}
        \mesh(i_0,j_0)\stackrel{\eqref{eq:lex2}}<
        \mesh(i_0,j_1)\stackrel{\eqref{eq:lex3}}<
        \mesh(i_1,j_0)\stackrel{\eqref{eq:lex2}}<    
        \mesh(i_1,j_1).
    \end{align*}
    Therefore, we have
    \begin{align}
        \label{eq:lex4}
        \mesh(i_0,j_0)<\mesh(i_1,j_1).
    \end{align}

    In the following, we use \eqref{eq:lex1},\eqref{eq:lex2},\eqref{eq:lex3},\eqref{eq:lex4}
    to verify that for all $i,i'\in I$ and $j,j'\in J$
$$\mesh(i,j)<\mesh(i',j')\qquad\iff\qquad\text{$i=i'$ and $j<j'$, or $i<i'$}.$$
For the right-to-left implication, assume first that  $i=i'$ and $j<j'$.
Then $\mesh(i,j)<\mesh(i',j')$
follows by regularity and \eqref{eq:lex2}.
If $i<i'$ then $\mesh(i,j)<\mesh(i',j')$ follows from 
\eqref{eq:lex1},\eqref{eq:lex3},\eqref{eq:lex4}.
We prove the left-to-right implication by contrapositive.
Observe that the negation of the condition in the right-hand-side is equivalent to either $i> i'$, or $i=i'$ and $j\ge j'$.
Assuming additionally that $(i,j)\neq(i',j')$, this is equivalent 
to the right-hand side condition above, with $(i,j)$ swapped with $(i',j')$.
Thus, by the implication proved already, we conclude that 
$\mesh(i,j)>\mesh(i',j')$, obtaining the negation of the left-hand side condition, and proving the contrapositive.
Therefore,  $<$ agrees with the lexicographic order $<_{IJ}$ on $V(\mesh)$, proving the first part of the statement.

We now verify the second part of the statement. By symmetry, we may assume that $<$ agrees with $<_{IJ}$ on $V(\mesh)$.
Let $I'= I-\set{\min(I),\max(I)}$ and $J'= J-\set{\min(J),\max(J)}$ and assume $|I'|=|J'|> 3$.
We observe that $\mesh|_{I'\times J'}$ is vertical and not horizontal in $(V,<)$.
Verticality is witnessed by the function $a\from I'\to V$ defined by $a(i):=\mesh(i,\min(J))$.

We show that $\mesh_{I'\times J'}$ is not horizontal.
Assume $b\from J'\to V$
is such that  
\begin{align}\label{eq:depon}
    \atp(\mesh(i,j),b(j'))\qquad\text {depends only on 
    $\otp(j,j')$ for all $i\in I',j,j'\in J'$}
\end{align}
We show that $\atp(\mesh(i,j),b(j'))$
is the same for all $i\in I'$ and $j,j'\in J$.
This will prove that $\mesh_{I'\times J'}$ is not horizontal.

Fix any $j,j'\in J'$ with $\min(J')<j'<\max(J')$.
As $\mesh_{I'\times J'}$ is injective
and $|I'|>1$,
we have that $\mesh(i,j)\neq b(j')$ for some $i\in I'$.
Suppose that $$\mesh(i,j)<b(j')\text{\quad for some $i\in I'$},$$
the other case being symmetric.
By \eqref{eq:depon}, we have that 
$$\mesh(\max(I'),j)<b(j').$$
Pick any $i'\in I'$ with $i'<\max(I')$.
As $<$ agrees with $<_{IJ}$ on $V(\mesh)$, for 
 all $j''\in J'$ we have:
$$\mesh(i',j'')<\mesh(\max(I'),j)<b(j').$$
In particular, $\mesh(i',j'')<b(j')$ for all $j''\in J'$.
As $\min(J')<j'<\max(J')$,
combined with \eqref{eq:depon}, this implies that 
 $\atp(\mesh(i,j''),b(j'))$ is the same for all $i\in I',j',j''\in J'$.
\end{proof}

\begin{proof}[Proof sketch for Theorem~\ref{thm:tww-patterns}]
    Suppose $\CC$ is a class of ordered graphs which is monadically independent.
    By Theorem~\ref{thm:flip-breakability-binary},
    there is some $h\in\N$ such that $\CC$ contains minimal transformers of 
    length $h$ and arbitrarily large order.
Let $G=(V,E,<)$ be an ordered graph
and let  $T=(\mesh_1,\ldots,\mesh_h)$ be a minimal transformer in $G$ of order at least $6$ and length $h$.
By Lemma~\ref{lem:meshes-in-orders} and reducing the order of $T$ by at most $2$,
we may assume that every mesh in $T$ 
is either horizontal and not vertical, or vertical and not horizontal in $(V,<)$.
In particular, every mesh in $T$ is either horizontal or vertical in $G$ (possibly both).
By minimality, the meshes $\mesh_2,\ldots,\mesh_{h-1}$ 
are neither horizontal nor vertical, so it follows that $h\le 2$. 

Let \(n \ge 6\).
As argued above, we can find a transformer \(T\) in $\CC$ of length $h \in \{1,2\}$ and order \(n\),
where each mesh in $T$ is either horizontal and not vertical, or vertical and not horizontal in $(V,<)$.
We show that the following holds for some $m\ge U(n)$.
For every permutation 
$\sigma\from[m]\to[m]$ 
there are strictly monotone functions
$a\from [m]\to V$ and 
$b\from [m]\to V$ 
such that  one of the following cases holds
(uniformly) for all $i,j\in [m]$:
    \begin{description}
        \item[$(=)$]    $\{a(i),b(j)\}\in E(G)\iff j=\sigma(i)$,
        \item[$(\neq)$] $\{a(i),b(j)\}\in E(G)\iff j\neq \sigma(i)$,        
        \item[$(\le)$]  $\{a(i),b(j)\}\in E(G)\iff j\le  \sigma(i)$,
         \item[$(>)$]  $\{a(i),b(j)\}\in E(G)\iff j>  \sigma(i)$.
    \end{description}
    As $n$ is arbitrarily large, this will prove the theorem (observe that the case $(>)$ can be reduced to the analogous case $(\ge)$ by shifting $b$).

Suppose first that $h=2$, that is, $T=(\mesh,\mesh')$
for some conducting pair of meshes 
 $\mesh,\mesh'\from I\times J\to V(G)$,
 where  
$\mesh$ is vertical and not horizontal, and $\mesh'$ is horizontal and not vertical.
By the same argument as in Lemma~\ref{lem:regular-pairs},
we conclude that 
the atomic type $\atp_G(\mesh(i,j),\mesh'(i',j'))$ is 
not the same for all $i,i'\in I$ and $j,j'\in J$, but is the 
same for all $i,i'\in I$ and $j,j'\in J$ with $(i,j)\neq (i',j')$.
It is not difficult to see that in the setting of ordered graphs,
this implies $\mesh$ and $\mesh'$ to be either matched or co-matched in the (unordered) graph $(V,E)$.

Up to reindexing, assume that the indexing sequences of $\mesh$ and \(\mesh'\) are $I=[n]$ and $J=[n]$.
Let $\sigma\from [n]\to [n]$ be any permutation; we view $\sigma$ as a bijection $\sigma\from I\to J$.

As $\mesh$ is not horizontal in \((V,<)\), the linear order $<$ agrees on $V(\mesh)$ 
with one of the four lexicographic orders $<_{KL}$ on $I\times J$ with $K\in \set{I,I^R}$ and $L\in\set{J,J^R}$.
Reversing the order of $I$ if necessary,
we may assume that the sequence $$(\mesh(i,\sigma(i)):i\in I)\subset V(\mesh)$$ 
is strictly increasing.
Similarly, as \(M'\) is not vertical, by reversing the order of $J$ if necessary, we may assume that the sequence
$$(\mesh'(\sigma^{-1}(j),j):j\in J)\subset V(\mesh')$$ 
is strictly increasing. 
Moreover, if $\mesh$ and $\mesh'$ are matched,
for all $i\in I, j\in J$ we have that
$$(\mesh(i,\sigma(i)),\ \mesh'(\sigma^{-1}(j),j))\in E(G)\quad\iff\quad j=\sigma(i).$$
If $\mesh$ and $\mesh'$ are co-matched,
the condition becomes $j\neq\sigma(i)$.

Set $a(i):=\mesh(i,\sigma(i))$ and 
$b(j):=\mesh'(\sigma^{-1}(j)j)$ for $i\in I$ and $J\in J$.
Then $$\{a(i),b(j)\}\in E(G)\iff j=\sigma(i)\qquad\text{for all $i\in I,j\in J$}.$$
Moreover, $(a(i):i\in I)$ and $(b(j):j\in J)$ 
are strictly increasing. 
Therefore, the case $(=)$ or $(\neq)$ described above occurs.
This completes the analysis in the case when $T$ is a transformer of length~$2$.

\medskip
Suppose now that $T$ is a transformer of length $1$ and order $n$,
that is, $T$ consists of a single mesh $\mesh$ which is regular, and is both vertical and horizontal in $G=(V,E,<)$.
We argued already that \(\mesh\) is either vertical or horizontal in the totally ordered set \((V,<)\), but not both.
Suppose (by symmetry) that $\mesh$ is vertical and not horizontal in $(V,<)$.
As $\mesh$ is horizontal in $(V,E,<)$ but not in \((V,<)\),
it needs to be horizontal in \((V,E)\).
By \Cref{lem:vertical2}, 
$\mesh^\trans$ is capped in the graph $(V,E)$.
This means
that  for some $\alpha\in\set{=,\le}$, the mesh
$\mesh^\trans$ is $\alpha$-capped in $H$, where $H$ is either the graph $(V,E)$,
or its edge-complement: 
 There is some function $b\from J'\to V$, where $J'=J-\set{\min(J),\max(J)}$,
so that
\begin{align}\label{eq:mesh-perm}
    \set{\mesh(i,j),b(j')}\in E(H)\quad\iff\quad j\,\alpha\, j'
    \qquad \text{for all $i\in I$ and $j,j'\in J$}.        
\end{align}

As \(M\) is not horizontal in \((V,<)\),
the order \(<\) agrees on $V(\mesh)$ with $<_{KL}$ for $K\in\set{I,I^R}$ and $L\in\set{J,J^R}$.
Reversing the order of $I$ if necessary,
we have that $<$ agrees with $<_{IJ}$ or with $<_{IJ^R}$.

By the Erd\H os-Szekeres theorem, replacing $J'$ by a subsequence of length $m\ge U(n)$, we may assume that the function $b\from J'\to V$ 
is strictly monotone with respect to $<$.
Reversing the order of $J$ if necessary,
we may assume that $b\from J'\to V$ is strictly increasing.

Let $I'\subset I$ be any subsequence such that $|I'|=|J'|=m\ge U(n)$.
Up to reindexing, assume that $I'=[m]$ and $J'=[m]$.
Let $\sigma\from [m]\to[m]$ be an arbitrary permutation; we view $\sigma$ as a bijection $\sigma\from I'\to J'$.

Set $a(i):=\mesh(i,\sigma(i))$ for $i\in I'$.
Then $a\from I'\to V$ is strictly increasing,
as $<$ agrees on $V(M)$ with $<_{IJ}$ or with $<_{IJ^R}$.
By \eqref{eq:mesh-perm} we have $$\set{a(i),b(j)}\in E(H)\quad\iff \quad \sigma(i)\,\alpha\,j\qquad\text{ for }i\in I',j\in J'.$$
We conclude that one of the two cases $(=)$ or $(\neq)$
holds if $\alpha$ is $=$
(depending on whether $H$ is $(V,E)$ or its edge-complement)
and one of the two cases $(\le)$ or $(>)$ holds if $\alpha$ is $\le$.
This completes the sketch of the proof.
\end{proof}

\newpage
\bibliographystyle{plain}
\bibliography{ref}

\end{document}